\documentclass[12pt]{article}
\usepackage{amssymb,amsmath,amsthm,tikz,multirow}
\usetikzlibrary{arrows,calc}

\title{Tilings of the Sphere by Congruent Pentagons II: Edge Combination $a^3b^2$}
\author{Erxiao Wang\thanks{Research was supported by ZJNU Shuang-Long Distinguished Professorship Fund No. YS304319159.}, 
Zhejiang Normal University \\
Min Yan\thanks{Research was supported by Hong Kong RGC General Research Fund 16303515 and 16305920.}, 
Hong Kong University of Science and Technology}

\usepackage[hidelinks, hyperindex]{hyperref}

\hyperref[sec:function]{}

\newcommand{\mc}{\mathcal}

\newcommand{\thin}{\hspace{0.1em}\rule{0.7pt}{0.8em}\hspace{0.1em}}
\newcommand{\thick}{\hspace{0.1em}\rule{1.5pt}{0.8em}\hspace{0.1em}}

\newcommand{\pentagon}{\tikz \foreach \a in {0,...,4} \draw[rotate=72*\a] (18:0.17) -- (90:0.17);}

\newcommand{\arcThroughThreePoints}[4][]{
\coordinate (middle1) at ($(#2)!.5!(#3)$);
\coordinate (middle2) at ($(#3)!.5!(#4)$);
\coordinate (aux1) at ($(middle1)!1!90:(#3)$);
\coordinate (aux2) at ($(middle2)!1!90:(#4)$);
\coordinate (center) at ($(intersection of middle1--aux1 and middle2--aux2)$);
\draw[#1] 
 let \p1=($(#2)-(center)$),
      \p2=($(#4)-(center)$),
      \n0={veclen(\p1)},       
      \n1={atan2(\y1,\x1)}, 
      \n2={atan2(\y2,\x2)},
      \n3={\n2>\n1?\n2:\n2+360}
    in (#2) arc(\n1:\n3:\n0);
}

\newtheorem{theorem}{Theorem}
\newtheorem{lemma}[theorem]{Lemma}

\newtheorem{proposition}[theorem]{Proposition}
\newtheorem*{theorem*}{Theorem}

\theoremstyle{definition}
\newtheorem*{definition*}{Definition}
\newtheorem*{case*}{Case}
\newtheorem*{subcase*}{Subcase}

\theoremstyle{remark}

\numberwithin{equation}{section}

\begin{document}

\maketitle

\begin{abstract}
There are fifteen edge-to-edge tilings of the sphere by congruent pentagons with the edge combination $a^3b^2$:  five one-parameter families of pentagonal subdivision tilings, and ten flip modifications of three special cases of two pentagonal subdivision tilings. 

\end{abstract}

\section{Introduction}

In an edge-to-edge tiling of the sphere by congruent pentagons, the pentagon can have five possible edge combinations \cite[Lemma 9]{wy1}: $a^2b^2c,a^3bc,a^3b^2,a^4b,a^5$. We classified tilings for $a^2b^2c$ and $a^3bc$ in \cite{wy1}. In this paper, we classify tilings for $a^3b^2$. 

By \cite[Lemma 9]{wy1}, we know the pentagon is given by Figure \ref{pentagon}, with normal edge $a$ and thick edge $b$. We always denote the angles by $\alpha,\beta,\gamma,\delta,\epsilon$ as indicated. The second picture is the flip of the first. We call the first arrangement of angles {\em positively} oriented, and the second arrangement {\em negatively} oriented. The angles determine the orientation. Conversely, the edge lengths and the orientation also determine the angles. The orientation gives a convenient way of presenting the tiling without indicating the angles.

\begin{figure}[htp]
\centering
\begin{tikzpicture}[>=latex,scale=1]


\foreach \a in {0,...,4}
\fill[gray!50, xshift=3cm, rotate=72*\a]
	(0,0) -- (18:1) -- (90:1);

\foreach \a in {0,1,2}
\foreach \b in {0,1}
{
\begin{scope}[xshift=3*\b cm, rotate=-72*\a]

\draw
	(18:1) -- (-54:1);

\node at (-18:1) {\small $a$};

\end{scope}
}

\foreach \a in {0,1}
\foreach \b in {0,1}
{
\begin{scope}[xshift=3*\b cm, rotate=72*\a]

\draw[line width=1.5]
	(18:1) -- (90:1);

\node at (54:1) {\small $b$};

\end{scope}
}

\node at (90:0.75) {$\alpha$};
\node at (162:0.75) {$\beta$};
\node at (15:0.75) {$\gamma$};
\node at (234:0.75) {$\delta$};
\node at (-54:0.75) {$\epsilon$};

\node at (0,0) {$+$};

\begin{scope}[xshift=3 cm]

\node at (90:0.75) {$\alpha$};
\node at (14:0.75) {$\beta$};
\node at (162:0.75) {$\gamma$};
\node at (-54:0.75) {$\delta$};
\node at (234:0.75) {$\epsilon$};

\node at (0,0) {$-$};

\end{scope}

\end{tikzpicture}
\caption{Pentagon with the edge combination $a^3b^2$.}
\label{pentagon}
\end{figure}

\begin{theorem*}
Edge-to-edge tilings of the sphere by congruent pentagons with the edge combination $a^3b^2$ ($a,b$ distinct) are the following:
\begin{enumerate}
\item Five one parameter families of the pentagonal subdivisions of the Platonic solids.
\item Three families of tilings, obtained as the flip modifications of three special cases of two pentagonal subdivision tilings, and each family has the same unique pentagon:
\begin{itemize}
\item $T(4\beta\gamma\epsilon^2,2\epsilon^4)$, $T(4\beta^2\gamma^2,2\epsilon^4)$, $24$ tiles.
\item $T(5\beta\gamma\epsilon^3,7\epsilon^5)$, $T(10\beta\gamma\epsilon^3,2\epsilon^5)$, $T(2\beta^2\gamma^2\epsilon, 6\beta\gamma\epsilon^3,4\epsilon^5)$,  \newline
$T(6\beta^2\gamma^2\epsilon, 3\beta\gamma\epsilon^3,3\epsilon^5)$, $60$ tiles.
\item $T(5\beta\gamma\epsilon^2,5\delta\epsilon^3,7\epsilon^5)$,
$T(10\beta\gamma\epsilon^2,10\delta\epsilon^3,2\epsilon^5)$, $T(10\beta\gamma\epsilon^2,6\delta\epsilon^3,4\epsilon^5)$, \newline 
$T(15\beta\gamma\epsilon^2,3\delta\epsilon^3,3\epsilon^5)$, $60$ tiles.
\end{itemize}
\end{enumerate}
\end{theorem*} 

\subsubsection*{Pentagonal Subdivision Tiling}

We introduced the pentagonal subdivision in \cite[Section 3.1]{wy1}, and studied its combinatorial structure in \cite{yan2}. The pentagonal subdivisions of the Platonic solids are all the edge-to-edge tilings of the sphere by congruent pentagons with the edge combination $a^2b^2c$. The pentagon in such a tiling is the center of Figure \ref{reduction}, with $n=3$ for tetrahedron, $n=4$ for octahedron (or cube), and $n=5$ for icosahedron (or dodecahedron). Moreover, the sum of three unlabelled angles should be $2\pi$. The number of tiles is respectively $f=12,24,60$ for $n=3,4,5$. The pentagonal subdivision tilings allow two free parameters, and therefore form two dimensional moduli \cite{wy3}.

The classification in \cite{wy1} assumes distinct edge lengths $a,b,c$. However, the pentagonal subdivision construction allows some edge lengths to become equal. This leads to four possible reductions: $c=a\ne b$, $c=b\ne a$, $a=b\ne c$, $a=b=c$. The first three reductions allow one free parameter and correspond to three lines in the moduli. The last reduction has no free parameter and corresponds to points in the moduli.

\begin{figure}[htp]
\centering
\begin{tikzpicture}[>=latex,scale=1]


\foreach \a in {0,...,4}
\draw[rotate=72*\a]
	(18:1) -- (90:1);
\node at (126:1) {\small $a$};
\node at (198:1) {\small $a$};

\node at (54:1) {\small $b$};
\node at (-18:1) {\small $b$};

\node at (-90:1) {\small $c$};

\draw[shift={(162:1)}]
	(36:0.2) arc (36:-72:0.2);
\node at (162:0.5) {\small $\frac{2}{3}\pi$};

\fill[shift={(18:1)}]
	(144:0.2) arc (144:252:0.2) -- (0,0);
\node at (18:0.5) {\small $\frac{2}{n}\pi$};


\draw[thick, ->] 
	(-1.5,0) -- node[above] {\small $c=a$} (-2.5,0);

\begin{scope}[xshift=-4cm]

\foreach \a in {1,2,3}
\draw[rotate=72*\a]
	(18:1) -- (90:1);

\foreach \a in {0,-1}
\draw[line width=1.5,rotate=72*\a]
	(18:1) -- (90:1);

\draw[shift={(162:1)}]
	(36:0.2) arc (36:-72:0.2);
\node at (162:0.5) {\small $\frac{2}{3}\pi$};

\fill[shift={(18:1)}]
	(144:0.2) arc (144:252:0.2) -- (0,0);
\node at (18:0.5) {\small $\frac{2}{n}\pi$};

\end{scope}


\draw[thick, ->] 
	(1.5,0) -- node[above] {\small $c=b$} (2.5,0);

\begin{scope}[xshift=4cm]

\foreach \a in {0,-1,-2}
\draw[rotate=72*\a]
	(18:1) -- (90:1);

\foreach \a in {1,2}
\draw[line width=1.5,rotate=72*\a]
	(18:1) -- (90:1);

\draw[shift={(162:1)}]
	(36:0.2) arc (36:-72:0.2);
\node at (162:0.5) {\small $\frac{2}{3}\pi$};

\fill[shift={(18:1)}]
	(144:0.2) arc (144:252:0.2) -- (0,0);
\node at (18:0.5) {\small $\frac{2}{n}\pi$};

\end{scope}

\end{tikzpicture}
\caption{Reduction of pentagonal subdivision tiling.}
\label{reduction}
\end{figure}

Figure \ref{reduction} shows that the reductions $c=a\ne b$ and $c=b\ne a$ give tilings with the edge combination $a^3b^2$, subject to relabelling of the edges. The two reductions are the same for $n=3$ and are different for $n=4,5$. This gives the five families in the first part of the main theorem. Figure \ref{subdivision_tiling} gives their pictures, with vertices of degree $>3$ indicated by $\bullet$. The angles are determined by two facts: all angles at $\bullet$ are $\frac{2}{n}\pi$, and angles in every tile are arranged as (the left or right of) Figure \ref{reduction}.

\begin{figure}[htp]
\centering
\begin{tikzpicture}[>=latex,scale=1]


\foreach \a in {0,...,4}
\draw[rotate=-72*\a]
	(-54:0.33) -- (18:0.33) -- (18:0.65) -- (-18:0.8) -- (-54:0.65)
	(54:0.8) -- (54:1.2) -- (-18:1.2);

\draw[line width=1.5]
	(18:0.33) -- (90:0.33) -- (162:0.33)
	(90:0.33) -- (90:0.65)
	(18:0.65) -- (-18:0.8) -- (-54:0.65)
	(-18:0.8) -- (-18:1.2)
	(198:0.8) -- (234:0.65) -- (270:0.8)
	(234:0.33) -- (234:0.65)
	(54:1.2) -- (126:1.2) -- (198:1.2)
	(126:1.2) -- (126:0.8);


\foreach \a in {0,1,2,3}
{

\foreach \b in {1,2}
{
\begin{scope}[xshift=3.5*\b cm, rotate=90*\a]

\draw
	(0,0) -- (0.4,0) -- (0.6,0.25) -- (0.25,0.6) -- (0,0.4)
	(-15:1) -- (15:1) -- (45:1) -- (75:1)
	(0.6,0.25) -- (15:1)
	(0.6,-0.25) -- (-15:1)
	(0:1.3) -- (30:1.3) -- (60:1.3) -- (90:1.3)
	(-15:1) -- (0:1.3)
	(45:1) -- (30:1.3)
	(60:1.3) -- (60:1.6);

\fill
	(0,0) circle (0.1)
	(75:1) circle (0.1)
	(60:1.6) -- ++(-75:0.1) arc (-75:-165:0.1);

\end{scope}
}

\draw[line width=1.5, xshift=3.5cm, rotate=90*\a]
	(0,0) -- (0.4,0)
	(-45:1) -- (-15:1) -- (15:1)
	(0.6,-0.25) -- (-15:1) -- (0:1.3)
	(60:1.3) -- (60:1.6);

\node at (3.5,-1.7) {\small $c=a\ne b$};

\draw[line width=1.5, xshift=7cm, rotate=90*\a]
	(0.4,0) -- (0.6,0.25) -- (0.25,0.6)
	(0.6,0.25) -- (15:1)
	(0:1.3) -- (30:1.3) -- (60:1.3)
	(45:1) -- (30:1.3);

\node at (7,-1.7) {\small $c=b\ne a$};
	
}


\begin{scope}[shift={(0.75cm,-4.5cm)}]

\foreach \a in {0,...,4}
{

\foreach \b in {0,1}
{
\begin{scope}[xshift=5.5*\b cm, rotate=72*\a]

\draw
	(0,0) -- (18:0.45) -- (36:0.7) -- (72:0.7) -- (90:0.45)
	(0:0.7) -- (6:1.1)
	(30:1.1) -- (36:0.7)
	(6:1.1) -- (30:1.1) -- (54:1.1) -- (78:1.1)
	(-9:1.4) -- (6:1.1) -- (21:1.4) -- (39:1.4) -- (54:1.1)
	(-9:1.4) -- (0:1.6) -- (12:1.6) -- (21:1.4)
	(39:1.4) -- (48:1.6) -- (63:1.4)
	(12:1.6) -- (30:1.9) -- (48:1.6)
	(54:1.9) -- (30:1.9) -- (6:1.9) -- (-18:1.9) -- (0:1.6)
	(6:1.9) -- (-6:2.2)
	(30:1.9) -- (42:2.2)
	(-54:2.2) -- (-30:2.2) -- (-6:2.2) -- (18:2.2) -- (18:2.5);

\fill
	(0,0) circle (0.1)
	(6:1.1) circle (0.1)
	(30:1.9) circle (0.1)
	(18:2.5) -- ++(234:0.1) arc (234:162:0.1);

\end{scope}
}


\draw[line width=1.5, rotate=72*\a]
	(0,0) -- (18:0.45)
	(0:0.7) -- (6:1.1)
	(-18:1.1) -- (6:1.1) -- (30:1.1)
	(-9:1.4) -- (6:1.1) -- (21:1.4)
	(12:1.6) -- (30:1.9) -- (48:1.6)
	(54:1.9) -- (30:1.9) -- (6:1.9)
	(30:1.9) -- (42:2.2)
	(90:2.2) -- (90:2.5);

\node at (0,-2.6) {\small $c=a\ne b$};


\draw[line width=1.5, xshift=5.5cm, rotate=72*\a]
	(18:0.45) -- (36:0.7) -- (72:0.7)
	(30:1.1) -- (36:0.7)
	(21:1.4) -- (39:1.4) -- (54:1.1)
	(39:1.4) -- (48:1.6)
	(-9:1.4) -- (0:1.6) -- (12:1.6)
	(-18:1.9) -- (0:1.6)
	(-30:2.2) -- (-6:2.2) -- (18:2.2)
	(6:1.9) -- (-6:2.2);

\node at (5.5,-2.6) {\small $c=b\ne a$};
	
}

\end{scope}
	
\end{tikzpicture}
\caption{Pentagonal subdivision tilings for $a^3b^2$.}
\label{subdivision_tiling}
\end{figure}

The reduction $a=b\ne c$ (almost equilateral) will be studied in \cite{ly1,ly2}, and the reduction $a=b=c$ will be studied in \cite{awy}.

\subsubsection*{Flip Modification}

The flip modification is applied only to the reduction $c=b\ne a$. The tiles in the pentagonal subdivision tilings are given by the right of Figure \ref{reduction}. This means that, in terms of the pentagon in Figure \ref{pentagon}, we have $\alpha=\frac{2}{3}\pi$ and $\epsilon=\frac{2}{n}\pi$. 

On the left of Figure \ref{subdivision2} are the regular triangular faces of the regular octahedron and icosahedron. The middle is a pentagonal subdivision of one regular triangular face. If we replace all triangular faces by this pentagonal subdivision, then we get tilings of the sphere on the right, which are the reductions $c=b\ne a$ in Figure \ref{subdivision_tiling}. 

\begin{figure}[htp]
\centering
\begin{tikzpicture}[>=latex,scale=1]


\begin{scope}[yshift=4.2cm]

\foreach \a in {0,...,3}
\foreach \b in {0,1}
\draw[xshift=7*\b cm,rotate=90*\a]
	(0,0) -- (1.6,1.6)
	(0.9,0.9) -- (0.9,-0.9)
	;

\foreach \a in {0,...,3}
\draw[line width=1.5, xshift=7 cm,rotate=90*\a]
	(0.6,0.6) -- (0.6,0) -- (0.3,-0.3)
	(0.6,0) -- (0.9,-0.3)
	(-1.4,1.4) -- (0,1.3) -- (1.2,1.2)
	(-0.3,0.9) -- (0,1.3);
	
\end{scope}


\foreach \a in {0,1,2}
{
\begin{scope}[xshift=3.5 cm, yshift=2.7cm, rotate=120*\a]

\draw
	(-30:1.2) -- (90:1.2);
	
\draw[line width=1.5]
	(0,0) -- (0.693,0);

\node at (60:0.2) {\scriptsize $\alpha$};
\node at (0.4,0.2) {\scriptsize $\beta$};
\node at (0.6,-0.2) {\scriptsize $\gamma$};
\node at (0.2,0.6) {\scriptsize $\delta$};
\node at (-30:0.95) {\scriptsize $\epsilon$};

\end{scope}
}

\draw[very thick, ->]
	(2.3,1.7) -- (4.7,1.7);
\node at (3.5,1.4) {subdivision};


\foreach \a in {0,...,4}
\foreach \b in {0,1}
\draw[xshift=7*\b cm,rotate=72*\a]
	(0,0) -- (-18:0.9) -- (54:0.9) -- (18:1.8) -- (-18:0.9)
	(90:1.8) -- (18:1.8) -- (18:2.5)
	;

\foreach \a in {0,...,4}
\draw[line width=1.5, xshift=7 cm,rotate=72*\a]
	(54:0.3) -- (0,0.485) -- (126:0.6)
	(0,0.485) -- ++(54:0.3)
	(0,1.085) -- (-0.176,0.728)
	(-0.176,1.443) -- (0,1.085) -- (0.353,1.085)
	(54:1.271) -- (0.176,1.443)
	(126:1.271) -- (-0.353,1.085)
	(54:1.271) -- (1.141,0.971)
	(18:2) -- (54:1.8) -- (90:2.2)
	(54:1.8) -- (0.571,1.385);
	
\end{tikzpicture}
\caption{Pentagonal subdivision for the reduction $c=b\ne a$.}
\label{subdivision2}
\end{figure}

Each vertex $V=\epsilon^4$ ($\epsilon^5$) of the octahedron (icosahedron) has a neighbourhood $N(V)$ consisting of four (five) triangular faces. After the subdivision, the neighbourhood becomes the disk tiling $N(\epsilon^4)$ or $N(\epsilon^5)$ in the top of Figure \ref{e-nd}. The disk tiling is the ``extended neighbourhood'' of the center vertex $\epsilon^4$ or $\epsilon^5$. We also indicate the vertex in the original octahedron or icosahedron by $\bullet$ in Figure \ref{e-nd}.

In the middle of Figure \ref{e-nd}, we indicate the angles and edges along the boundary of $N(\epsilon^4)$ or $N(\epsilon^5)$. Inside the pentagonal subdivision tiling, we also indicate the angles and edges on the other side of the boundary. 

If the equality $\delta=2\epsilon$ is satisfied by the pentagonal subdivision of the octahedron, then we know the angle sums of $\beta\gamma\epsilon^2$ and $\delta\epsilon^2$ are $2\pi$. Therefore we may flip $N(\epsilon^4)$ along the line $L_{\delta=2\epsilon}$ to get a new tiling of the sphere. See the first in the bottom of Figure \ref{e-nd}.

If the equality $\beta+\gamma=2\epsilon$ is satisfied by the pentagonal subdivision of the icosahedron, then we know the angle sums of $\beta\gamma\epsilon^3$ and $\delta\epsilon^2$ are $2\pi$. Therefore we may flip $N(\epsilon^5)$ along the line $L_{\beta+\gamma=2\epsilon}$ to get a new tiling of the sphere. See the second in the bottom of Figure \ref{e-nd}.

If the equality $\delta=2\epsilon$ is satisfied by the pentagonal subdivision of the icosahedron, then we know the angle sums of $\beta\gamma\epsilon^2$ and $\delta\epsilon^3$ are $2\pi$. Therefore we may flip $N(\epsilon^5)$ along the line $L_{\delta=2\epsilon}$ to get a new tiling of the sphere. See the third in the bottom of Figure \ref{e-nd}.

\begin{figure}[htp]
\centering
\begin{tikzpicture}[>=latex,scale=1]

\begin{scope}[yshift=4.2cm]


\foreach \a in {0,...,3}
{
\begin{scope}[rotate=90*\a]

\draw
	(0,0) -- (0.6,0) -- (1,0.5) -- (0.5,1) -- (0,0.6)
	(0.5,1) -- (0.5,1.6) -- (-0.5,1.6) -- (-0.5,1) 
	(0.5,1.6) -- (1.3,1.3) -- (1.6,0.5);
	
\draw[line width=1.5]
	(0.6,0) -- (1,0.5) -- (0.5,1)
	(-0.5,1) -- (-0.5,1.6);

\fill
	(0,0) circle (0.1)
	(0.5,1.6) circle (0.1);
	
\node at (0.75,0.45) {\small $\alpha$}; 
\node at (0.47,0.75) {\small $\beta$};
\node at (0.52,0.2) {\small $\gamma$};
\node at (0.2,0.52) {\small $\delta$};
\node at (0.15,0.15) {\small $\epsilon$};

\node at (-0.33,1.05) {\small $\alpha$};
\node at (0,0.85) {\small $\beta$};
\node at (-0.33,1.43) {\small $\gamma$};
\node at (0.33,1.05) {\small $\delta$};
\node at (0.33,1.43) {\small $\epsilon$}; 

\node at (-0.7,1.05) {\small $\alpha$};
\node at (-0.7,1.35) {\small $\beta$}; 
\node at (-1.05,0.7) {\small $\gamma$}; 
\node at (-1.15,1.15) {\small $\delta$};
\node at (-1.35,0.7) {\small $\epsilon$};
	
\end{scope}
}

\node at (0,-1.9) {$N(\epsilon^4)$};


\begin{scope}[xshift=5cm]

\foreach \a in {0,...,4}
{
\begin{scope}[rotate=72*\a]

\draw
	(0,0) -- (90:0.7) -- (70:1.2) -- (75:1.8)
	(-18:2) -- (3:1.8) -- (33:1.8) -- (54:2)
	;

\draw[line width=1.5]
	(90:0.7) -- (110:1.2) -- (142:1.2)
	(110:1.2) -- (105:1.8);

\fill
	(0,0) circle (0.1)
	(3:1.8) circle (0.1);
	
\node at (43:1) {\small $\alpha$};
\node at (65:0.95) {\small $\beta$};
\node at (75:0.65) {\small $\delta$};
\node at (33:0.65) {\small $\gamma$};
\node at (54:0.25) {\small $\epsilon$};

\node at (102:1.25) {\small $\alpha$};
\node at (90:1) {\small $\beta$};
\node at (100:1.6) {\small $\gamma$};
\node at (78:1.25) {\small $\delta$};
\node at (80:1.6) {\small $\epsilon$};

\node at (44:1.35) {\small $\alpha$};
\node at (41:1.65) {\small $\beta$};
\node at (64:1.35) {\small $\gamma$};
\node at (54:1.8) {\small $\delta$};
\node at (69:1.65) {\small $\epsilon$};

\end{scope}
}

\node at (1.6,-1.8) {$N(\epsilon^5)$};

\end{scope}

\end{scope}


\foreach \a in {0,...,3}
\foreach \b/\c in {0/0,-1.5/-4}
{
\begin{scope}[xshift=\b cm, yshift=\c cm, rotate=90*\a]

\draw
	(-15:1.3) -- (15:1.3) -- (45:1.3) -- (75:1.3)
	;

\draw[line width=1.5]
	(45:1.3) -- (45:1.6);

\node at (75:1.47) {\scriptsize $\epsilon^2$};

\node at (15:1.45) {\scriptsize $\delta$};

\node at (52:1.45) {\scriptsize $\beta$};	
\node at (38:1.45) {\scriptsize $\gamma$};

\end{scope}
}


\foreach \a in {0,...,3}
{
\begin{scope}[rotate=90*\a]

\draw
	(75:1.3) -- (75:1)
	;

\draw[line width=1.5]
	(15:1.3) -- (15:1);

\node at (80:1.15) {\scriptsize $\epsilon$};	
\node at (70:1.15) {\scriptsize $\epsilon$};
\node at (45:1.15) {\scriptsize $\delta$};
\node at (24:1.1) {\scriptsize $\beta$};	
\node at (6:1.1) {\scriptsize $\gamma$};

\end{scope}
}

\draw[gray]
	(-30:1.7) -- (-30:-1.7);

\node at (150:2.1) {\small $L_{\delta=2\epsilon}$};

\draw[line width=4pt, ->]
	(-0.5,-1.6) -- ++(-0.4,-0.8);
\node at (-1.4,-1.7) {\small $\delta=2\epsilon$};


\foreach \a in {0,...,3}
{
\begin{scope}[xshift=-1.5 cm, yshift=-4 cm, rotate=90*\a]

\draw
	(45:1.3) -- (45:1)
	;

\draw[line width=1.5]
	(15:1.3) -- (15:1);

\node at (50:1.15) {\scriptsize $\epsilon$};	
\node at (40:1.15) {\scriptsize $\epsilon$};
\node at (75:1.15) {\scriptsize $\delta$};
\node at (6:1.1) {\scriptsize $\beta$};	
\node at (24:1.1) {\scriptsize $\gamma$};

\end{scope}
}

\node at (-1.5,-3.7) {flipped};
\node at (-1.5,-4.3) {$N(\epsilon^4)$};


\foreach \a in {0,...,4}
\foreach \b/\c in {5/0,3/-4,7/-4}
{
\begin{scope}[xshift=\b cm, yshift=\c cm, rotate=72*\a]

\draw
	(6:1.5) -- (30:1.5) -- (54:1.5) -- (78:1.5) 
	;

\draw[line width=1.5]
	(54:1.5) -- (54:1.8);

\node at (6:1.67) {\scriptsize $\epsilon^3$};
\node at (30:1.65) {\scriptsize $\delta$};
	
\node at (49:1.65) {\scriptsize $\gamma$};
\node at (60:1.65) {\scriptsize $\beta$};

\end{scope}
}


\begin{scope}[xshift=5cm]

\foreach \a in {0,...,4}
{
\begin{scope}[rotate=72*\a]

\draw
	(6:1.5) -- (6:1.2)
	;

\draw[line width=1.5]
	(30:1.5) -- (30:1.2);
	
\node at (0:1.35) {\scriptsize $\epsilon$};
\node at (12:1.35) {\scriptsize $\epsilon$};

\node at (23:1.35) {\scriptsize $\gamma$};
\node at (38:1.35) {\scriptsize $\beta$};

\node at (54:1.35) {\scriptsize $\delta$};

\end{scope}
}

\draw[gray]
	(-18:2) -- (-18:-2)
	(30:2) -- (30:-2);
	
\node at (-15:2.65) {\small $L_{\beta+\gamma=2\epsilon}$};
\node at (30:2.3) {\small $L_{\delta=2\epsilon}$};

\draw[line width=4pt, ->]
	(-1,-1.6) -- ++(-0.4,-0.8);
\node at (-2.2,-1.7) {\small $\beta+\gamma=2\epsilon$};

\draw[line width=4pt, ->]
	(1,-1.6) -- ++(0.4,-0.8);
\node at (1.9,-1.7) {\small $\delta=2\epsilon$};

\foreach \a in {-2,2}
{
\node at (\a,-3.7) {flipped};
\node at (\a,-4.3) {$N(\epsilon^5)$};
}

\end{scope}


\foreach \a in {0,...,4}
{
\begin{scope}[shift={(3cm,-4cm)}, rotate=72*\a]

\draw
	(30:1.5) -- (30:1.2)
	;

\draw[line width=1.5]
	(6:1.5) -- (6:1.2);

\node at (-2:1.35) {\scriptsize $\beta$};
\node at (14:1.35) {\scriptsize $\gamma$};

\node at (24:1.35) {\scriptsize $\epsilon$};
\node at (36:1.35) {\scriptsize $\epsilon$};

\node at (54:1.35) {\scriptsize $\delta$};

\end{scope}
}


\foreach \a in {0,...,4}
{
\begin{scope}[shift={(7cm,-4cm)}, rotate=72*\a]

\draw
	(54:1.5) -- (54:1.2)
	;

\draw[line width=1.5]
	(30:1.5) -- (30:1.2);

\node at (22:1.35) {\scriptsize $\beta$};
\node at (36:1.35) {\scriptsize $\gamma$};

\node at (59:1.35) {\scriptsize $\epsilon$};
\node at (49:1.35) {\scriptsize $\epsilon$};

\node at (6:1.35) {\scriptsize $\delta$};

\end{scope}
}

\end{tikzpicture}
\caption{Flips of the extended neighbourhoods $N(\epsilon^4)$ and $N(\epsilon^5)$.}
\label{e-nd}
\end{figure}

Since the pentagonal subdivision tiling for the reduction $c=b\ne a$ allows one free parameter, the introduction of one additional equality $\delta=2\epsilon$ or $\beta+\gamma=2\epsilon$ actually uniquely determines the pentagon. We calculate the pentagon in detail in the second step of the proof of Proposition \ref{special1}. 

The octahedron is the union of two $N(V)$ along an ``equator''. Therefore the $c=b\ne a$ reduction of the pentagonal subdivision of the octahedron is the union of two $N(\epsilon^4)$ along the equator. The edges and angles along the equator are given by the first in the middle of Figure \ref{e-nd}. In case $\delta=2\epsilon$ is satisfied, we may flip one $N(\epsilon^4)$ to get the first of Figure \ref{subdivision_tiling_modify24}, or flip both $N(\epsilon^4)$ to get the second of Figure \ref{subdivision_tiling_modify24}. We note that the flip of two $N(\epsilon^4)$ actually changes the vertex $\beta\gamma\epsilon^2$ to $\beta^2\gamma^2$. We indicate the vertex $\epsilon^4$ by $\bullet$, and indicate vertices $\beta\gamma\epsilon^2$ and $\beta^2\gamma^2$ by $\circ$. 

\begin{figure}[htp]
\centering
\begin{tikzpicture}[>=latex,scale=1]

\foreach \a in {0,...,11}
{
\fill[gray!50, rotate=30*\a]
	(0,0) -- (-15:1) -- (15:1);

\fill[gray!50, xshift=3.5 cm, rotate=30*\a]
	(0,0) -- (0:1.6) -- (30:1.6);
}	
	
\foreach \a in {0,1,2,3}
{

\foreach \b in {0,1,2}
\draw[xshift=3.5*\b cm, rotate=90*\a]
	(0,0) -- (0.4,0) -- (0.6,0.25) -- (0.25,0.6) -- (0,0.4)
	(-15:1) -- (15:1) -- (45:1) -- (75:1)
	(0.6,0.25) -- (15:1)
	(0.6,-0.25) -- (-15:1)
	(-30:1.3) -- (0:1.3) -- (30:1.3) -- (60:1.3);

\foreach \b in {0,1}
\draw[xshift=3.5*\b cm, rotate=90*\a]
	(-15:1) -- (0:1.3)
	(45:1) -- (30:1.3)
	(60:1.3) -- (60:1.6);

\foreach \b in {0,1}
\draw[line width=1.5, xshift=3.5*\b cm, rotate=90*\a]
	(0.6,0.25) -- (0.25,0.6) -- (0,0.4)
	(0.6,-0.25) -- (-15:1);

\draw[line width=1.5, rotate=90*\a]
	(0:1.3) -- (30:1.3) -- (60:1.3)
	(45:1) -- (30:1.3);

\draw[line width=1.5, xshift=3.5 cm, rotate=90*\a]
	(-30:1.3) -- (0:1.3) -- (30:1.3)
	(-15:1) -- (0:1.3);

\draw[xshift=7cm, rotate=90*\a]	
	(30:1.3) -- (30:1.6);

\draw[line width=1.5, xshift=7cm, rotate=90*\a]
	(0.4,0) -- (0.6,0.25) -- (0.25,0.6)
	(0.6,0.25) -- (15:1)(-30:1.3) -- (0:1.3) -- (30:1.3)
	(15:1) -- (0:1.3);

}

\foreach \a in {0,1,2,3}
{

\begin{scope}[rotate=90*\a]

\filldraw[fill=white]
	(75:1) circle (0.1);

\fill
	(60:1.6) -- ++(-75:0.1) arc (-75:-165:0.1);

\end{scope}

\begin{scope}[xshift=3.5cm, rotate=90*\a]

\filldraw[fill=white]
	(75:1) circle (0.1);

\fill
	(60:1.6) -- ++(-75:0.1) arc (-75:-165:0.1);
	
\end{scope}

\begin{scope}[xshift=7cm, rotate=90*\a]

\filldraw[fill=white]
	(15:1) circle (0.1);

\draw
	(45:1) -- (60:1.3);

\fill
	(120:1.6) -- ++(-15:0.1) arc (-15:-105:0.1);
	
\end{scope}

}

\foreach \b in {0,1,2}
\fill
	(3.5*\b,0) circle (0.1);

\node at (0,-1.8) 
	{$T(4\beta\gamma\epsilon^2\circ,
	2\epsilon^4\bullet)$};
\node at (5.25,-1.8) 
	{$T(4\beta^2\gamma^2\circ,
	2\epsilon^4\bullet)$};

\end{tikzpicture}
\caption{Two flip modifications of the subdivision of the octahedron.}
\label{subdivision_tiling_modify24}
\end{figure}

The shaded tiles in Figure \ref{subdivision_tiling_modify24} indicate the fact that the flip changes the positive orientation (the first of Figure \ref{pentagon}) to the negative orientation (the second of Figure \ref{pentagon}). We may flip the whole second tiling again, so that all tiles become positively oriented. This is the third of Figure \ref{subdivision_tiling_modify24}, considered as the same tiling as the second. The two tilings in Figure \ref{subdivision_tiling_modify24} form the first family in the second part of the main theorem. 

Next we consider the flip modifications of the $c=b\ne a$ reduction of the pentagonal subdivision of the icosahedron. We may find one, two, or even three non-overlapping $N(V)$ in the icosahedron. There are altogether four possibilities, given by Figure \ref{mod60A}, and with the choices of $V$ indicated by $\bullet$. For each possibility, under the additional assumption $\beta+\gamma=2\epsilon$ or $\delta=2\epsilon$, we may apply the corresponding flip of $N(\epsilon^5)$ in Figure \ref{e-nd} to these $N(V)$. 

\begin{figure}[htp]
\centering
\begin{tikzpicture}[>=latex,scale=1]

\foreach \a in {0,...,4}
\foreach \b in {0,...,3}
\fill[gray!50, xshift=3*\b cm, rotate=72*\a]
	(0,0) -- (-18:0.5) -- (54:0.5);

\foreach \a in {0,...,4}
\fill[gray!50, xshift=3 cm, rotate=72*\a]
	(18:1) -- (18:1.3) -- (90:1.3) -- (90:1);	

\foreach \b/\c in {2/0,3/1,3/-1}
\fill[gray!50, xshift=3*\b cm, rotate=72*\c]
	(162:1.3) -- (162:1) -- (126:0.5) -- (54:0.5) -- (18:1) -- (18:1.3) -- (90:1.3);	
	
\foreach \a in {0,...,4}
\foreach \b in {0,...,3}
\draw[xshift=3*\b cm,rotate=72*\a]
	(0,0) -- (-18:0.5) -- (54:0.5) -- (18:1) -- (-18:0.5)
	(90:1) -- (18:1) -- (18:1.3)
	;

\foreach \b in {0,...,3}
\fill
	(3*\b,0) circle (0.1);

\foreach \a in {0,...,4}
\fill[xshift=3 cm,rotate=72*\a]
	(90:1.3) -- ++(-54:0.1) arc (-54:-126:0.1);

\fill
	(6,1) circle (0.1);
	
\fill[xshift=9cm]
	(18:1) circle (0.1)
	(162:1) circle (0.1);	

\end{tikzpicture}
\caption{Non-overlapping $N(\bullet)$ in the icosahedron.}
\label{mod60A}
\end{figure}

Similar to the two ways of drawing $T(4\beta^2\gamma^2,2\epsilon^4)$ in Figure \ref{subdivision_tiling_modify24}, we may flip the whole tiling again, when too many tiles become shaded (i.e., negatively oriented). This means the second and fourth of Figure \ref{mod60A} are turned into the first and second of Figure \ref{mod60B}. For the case $\beta+\gamma=2\epsilon$, the four flip modifications of the pentagonal subdivision of the icosahedron are given by the first four of Figure \ref{subdivision_tiling_modify60A}, corresponding to the first of Figure \ref{mod60A}, the first of Figure \ref{mod60B}, the third of Figure \ref{mod60A}, and the second of Figure \ref{mod60B}. We note that all four tilings are combinatorially still the pentagonal subdivision of the icosahedron. These form the second family in the second part of the main theorem.

\begin{figure}[htp]
\centering
\begin{tikzpicture}[>=latex,scale=1]

\begin{scope}[xshift=-6.2 cm]

\foreach \a in {0,...,4}
\fill[gray!50, rotate=72*\a]
	(-18:0.5) -- (18:1) -- (90:1) -- (54:0.5);

\fill[gray!50, xshift=3 cm]
	(54:0.5) -- (126:0.5) -- (90:1)
	(-54:1.3) -- (-54:1) -- (-18:0.5) -- (-90:0.5) -- (-162:0.5) -- (-126:1) -- (-126:1.3);
		
\foreach \a in {0,...,4}
\foreach \b in {0,1}
\draw[xshift=3*\b cm, rotate=72*\a]
	(0,0) -- (-18:0.5) -- (54:0.5) -- (18:1) -- (-18:0.5)
	(90:1) -- (18:1) -- (18:1.3)
	;

\fill 
	(0,0) circle (0.1)
	(3,0) circle (0.1);
	
\foreach \a in {0,...,4}
\fill[rotate=72*\a]
	(90:1.3) -- ++(-54:0.1) arc (-54:-126:0.1);

\fill[xshift=3cm]
	(18:1) circle (0.1)
	(162:1) circle (0.1);
			
\end{scope}


\foreach \a in {1,-1}
\foreach \b in {1,-1}
\fill[gray!50, xscale=\a, yscale=\b]
	 (0,0.25) -- (0.3,0.5) -- (1.3,0) -- (1.6,0) -- (0,1.3);

\foreach \a in {1,-1}
\foreach \b in {1,-1}
\draw[xscale=\a, yscale=\b]
	(0,0) -- (0,0.25) -- (0.5,0) -- (0.3,0.5) -- (0,0.25)
	(0.5,0) -- (1.3,0) -- (0,1) -- (0,1.3)
	(1.3,0) -- (0.3,0.5) -- (-0.3,0.5) -- (0,1) -- (1.3,0);

\fill 
	(0.5,0) circle (0.1)
	(-0.5,0) circle (0.1);
		

\begin{scope}[shift={(3.4 cm,0.5 cm)}, yscale=-1]
	
\foreach \a in {0,1,2}
\fill[gray!50, rotate=120*\a]
	(30:0.2) -- (150:0.2) -- (0,0)
	(-30:1.5) -- (30:0.5) -- (90:1.5) -- (90:1.8) -- (-30:1.8);

\foreach \a in {0,1,2}
\draw[rotate=120*\a]
	(30:0.2) -- (150:0.2)
	(30:0.2) -- (-30:0.7) -- (-90:0.2)
	(30:0.2) -- (30:0.5) -- (90:0.7) -- (150:0.5)
	(90:0.7) -- (90:1.5) -- (30:0.5) -- (-30:1.5) -- (90:1.5)
	;

\foreach \a in {0,1,2}
\fill[rotate=120*\a] 
	(90:0.7) circle (0.1);
	
\end{scope}

\end{tikzpicture}
\caption{Alternative views of Figure \ref{mod60A}.}
\label{mod60B}
\end{figure}

The third family in the second part of the main theorem is the four flip modifications for the case $\delta=2\epsilon$. The modifications are no longer combinatorially pentagonal subdivisions. Moreover, we notice the first and second of Figure \ref{mod60A} have the five fold symmetry, the third of Figure \ref{mod60A} has the two fold symmetry, and the fourth of Figure \ref{mod60A} has the three fold symmetry. The later two symmetries are made explicit by being redrawn as the third and fourth of Figure \ref{mod60B} (and whole flips are applied afterwards).

\begin{figure}[htp]
\centering
\begin{tikzpicture}[>=latex,scale=1]

\foreach \a in {0,...,14}
{
\foreach \c in {0,1}
\fill[gray!50, yshift= -5.6*\c cm, rotate=24*\a]
	(0,0) -- (6:1.1) -- (30:1.1);

\fill[gray!50, xshift=5.5cm, rotate=24*\a]
	(6:1.1) -- (30:1.1) -- (30:1.9) -- (6:1.9);
}

\fill[gray!50, yshift=-5.6cm]
	(78:1.1) -- (63:1.4) -- (48:1.6) -- (30:1.9) -- (42:2.2)  -- (18:2.2) -- (18:2.5) -- (42:2.5) -- (66:2.5) -- (90:2.5) -- (114:2.5) -- (138:2.5) -- (162:2.5) -- (162:2.2) -- (186:2.2) -- (174:1.9) -- (156:1.6) -- (165:1.4);
	
\fill[gray!50, xshift=5.5cm, yshift=-5.6cm]
	(78:1.9) -- (96:1.6) -- (87:1.4) -- (102:1.1) -- (78:1.1) -- (54:1.1) -- (30:1.1) -- (45:1.4) -- (60:1.6)
	(-90:1.1) -- (-66:1.1) -- (-42:1.1) -- (-57:1.4) -- (-48:1.6) -- (-66:1.9) -- (-78:2.2) -- (-102:2.2) -- (-126:2.2) -- (-150:2.2) -- (-138:1.9) -- (-156:1.6) -- (189:1.4) -- (174:1.1) --  (198:1.1) -- (222:1.1) -- (246:1.1)
	;
	

\foreach \a in {0,...,4}
\foreach \b/\c/\d in {0/0/1,1/0/-1,0/1/1,1/1/-1}
{
\begin{scope}[shift={(5.5*\b cm, -5.6*\c cm)}, xscale=\d, rotate=72*\a]

\draw
	(0,0) -- (18:0.45) -- (36:0.7) -- (72:0.7) -- (90:0.45)
	(0:0.7) -- (6:1.1)
	(30:1.1) -- (36:0.7)
	(6:1.1) -- (30:1.1) -- (54:1.1) -- (78:1.1)
	(-9:1.4) -- (6:1.1) -- (21:1.4) -- (39:1.4) -- (54:1.1)
	(-9:1.4) -- (0:1.6) -- (12:1.6) -- (21:1.4)
	(39:1.4) -- (48:1.6) -- (63:1.4)
	(12:1.6) -- (30:1.9) -- (48:1.6)
	(54:1.9) -- (30:1.9) -- (6:1.9) -- (-18:1.9) -- (0:1.6)
	(6:1.9) -- (-6:2.2)
	(30:1.9) -- (42:2.2)
	(-54:2.2) -- (-30:2.2) -- (-6:2.2) -- (18:2.2) -- (18:2.5);

\fill
	(0,0) circle (0.1);

\end{scope}
}

\foreach \a in {0,...,4}
{


\begin{scope}[rotate=72*\a]

\draw[line width=1.5]
	(36:0.7) -- (72:0.7) -- (90:0.45)
	(0:0.7) -- (6:1.1)
	(21:1.4) -- (39:1.4) -- (54:1.1)
	(39:1.4) -- (48:1.6)
	(-9:1.4) -- (0:1.6) -- (12:1.6)
	(-18:1.9) -- (0:1.6)
	(-30:2.2) -- (-6:2.2) -- (18:2.2)
	(6:1.9) -- (-6:2.2);

\fill
	(30:1.9) circle (0.1)
	(18:2.5) -- ++(234:0.1) arc (234:162:0.1);
	
\filldraw[fill=white]
	(6:1.1) circle (0.1);

\end{scope}

\node at (0,-2.7) 
	{$T(5\beta\gamma\epsilon^3\circ,
	7\epsilon^5\bullet)$};
	

\begin{scope}[xshift=5.5cm, rotate=72*\a]

\draw[line width=1.5]
	(18:0.45) -- (36:0.7) -- (72:0.7)
	(30:1.1) -- (36:0.7)
	(87:1.4) -- (69:1.4) -- (54:1.1)
	(69:1.4) -- (60:1.6)
	(45:1.4) -- (36:1.6) -- (24:1.6)
	(54:1.9) -- (36:1.6)
	(-30:2.2) -- (-6:2.2) -- (18:2.2)
	(6:1.9) -- (-6:2.2)
	;

\fill
	(18:2.5) -- ++(234:0.1) arc (234:162:0.1);
	
\filldraw[fill=white]
	(6:1.9) circle (0.1)
	(30:1.1) circle (0.1);
	
\end{scope}

\node at (5.5,-2.7) 
	{$T(10\beta\gamma\epsilon^3\circ,
	2\epsilon^5\bullet)$};

}


\begin{scope}[yshift=-5.6cm]

\foreach \a in {0,...,4}
\draw[line width=1.5,rotate=72*\a]
	(36:0.7) -- (72:0.7) -- (90:0.45)
	(0:0.7) -- (6:1.1);
	
\foreach \a in {2,3,4}
\draw[line width=1.5,rotate=72*\a]	
	(21:1.4) -- (39:1.4) -- (54:1.1)
	(39:1.4) -- (48:1.6)
	(63:1.4) -- (72:1.6) -- (84:1.6)
	(54:1.9) -- (72:1.6)
	(42:2.2) -- (66:2.2) -- (90:2.2)
	(78:1.9) -- (66:2.2);

\draw[line width=1.5]
	(111:1.4) -- (93:1.4) -- (78:1.1)
	(93:1.4) -- (84:1.6)
	(150:1.1) -- (135:1.4) -- (120:1.6)
	(135:1.4) -- (144:1.6)
	(21:1.4) -- (39:1.4) -- (48:1.6)
	(39:1.4) -- (54:1.1)
	(78:1.9) -- (54:1.9) -- (30:1.9)
	(54:1.9) -- (72:1.6)
	(174:1.9) -- (150:1.9) -- (126:1.9)
	(150:1.9) -- (138:2.2)
	(66:2.2) -- (90:2.2) -- (114:2.2)
	(90:2.2) -- (90:2.5);

\foreach \a in {0,...,4}
\filldraw[fill=white,rotate=72*\a]
	(18:2.5) -- ++(234:0.1) arc (234:162:0.1) -- cycle;

\fill
	(102:1.9) circle (0.1)
	(-42:1.9) circle (0.1)
	(-114:1.9) circle (0.1);
	
\filldraw[fill=white]
	(6:1.1) circle (0.1)
	(-66:1.1) circle (0.1)
	(-138:1.1) circle (0.1)
	(30:1.9) circle (0.1)
	(174:1.9) circle (0.1);

\filldraw[fill=gray]
	(78:1.1) circle (0.1)
	(150:1.1) circle (0.1);

\node at (0,-2.7) 
	{$T(2\beta^2\gamma^2\epsilon{\color{gray} \bullet},
	6\beta\gamma\epsilon^3\circ,
	4\epsilon^5\bullet)$};
	
\end{scope}

	
\begin{scope}[xshift=5.5cm, yshift=-5.6cm]

\foreach \a in {0,...,4}
\draw[line width=1.5,rotate=72*\a]
	(18:0.45) -- (36:0.7) -- (72:0.7)
	(30:1.1) -- (36:0.7)
	;

\foreach \a in {0,1}
\draw[line width=1.5, rotate=144*\a]
	(30:1.1) -- (15:1.4) -- (24:1.6)
	(15:1.4) -- (-3:1.4)
	(54:1.1) -- (69:1.4) -- (60:1.6)
	(87:1.4) -- (69:1.4)
	(-42:1.1) -- (-27:1.4) -- (-12:1.6)
	(-27:1.4) -- (-36:1.6)
	(15:1.4) -- (-3:1.4)
	(78:1.9) -- (54:1.9) -- (30:1.9)
	(54:1.9) -- (36:1.6)
	(-18:1.9) -- (-42:1.9) -- (-66:1.9)
	(-42:1.9) -- (-30:2.2)
	(-6:2.2) -- (18:2.2) -- (42:2.2)
	(18:2.2) -- (18:2.5);

\draw[line width=1.5]
	(-90:1.1) -- (-75:1.4) -- (-84:1.6)
	(-57:1.4) -- (-75:1.4)
	(-90:1.9) -- (-108:1.6) -- (-120:1.6)
	(-108:1.6) -- (-99:1.4)
	(-78:2.2) -- (-102:2.2) -- (-126:2.2)
	(-102:2.2) -- (-114:1.9)
	;

\foreach \a in {0,...,4}
\filldraw[fill=gray,rotate=72*\a]
	(18:2.5) -- ++(234:0.1) arc (234:162:0.1) -- cycle;

\fill
	(6:1.9) circle (0.1)
	(150:1.9) circle (0.1);
	
\filldraw[fill=white]
	(-114:1.1) circle (0.1)
	(-66:1.9) circle (0.1)
	(-138:1.9) circle (0.1);

\filldraw[fill=gray]
	(-42:1.1) circle (0.1)
	(30:1.1) circle (0.1)
	(102:1.1) circle (0.1)
	(174:1.1) circle (0.1)
	(78:1.9) circle (0.1);

\node at (0,-2.7) 
	{$T(6\beta^2\gamma^2\epsilon{\color{gray} \bullet},
	3\beta\gamma\epsilon^3\circ,
	3\epsilon^5\bullet)$};

\end{scope}


\begin{scope}[yshift=-11.2 cm]

\foreach \a in {0,...,14}
{
\fill[gray!50, rotate=24*\a]
	(0,0) -- (6:1.1) -- (30:1.1);

\fill[gray!50, xshift=5.5cm, rotate=24*\a]
	(6:1.1) -- (30:1.1) -- (30:1.9) -- (6:1.9);
}


\foreach \a in {0,...,4}
\foreach \b in {0,1}
\draw[xshift=5.5*\b cm, rotate=72*\a]
	(0,0) -- (18:0.45) -- (36:0.7) -- (72:0.7) -- (90:0.45)
	(0:0.7) -- (6:1.1)
	(30:1.1) -- (36:0.7)
	(6:1.1) -- (30:1.1) -- (54:1.1) -- (78:1.1);

\fill
	(0,0) circle (0.1)
	(5.5,0) circle (0.1);

\foreach \a in {0,...,4}
{
	

\begin{scope}[rotate=72*\a]

\draw[line width=1.5]
	(36:0.7) -- (72:0.7) -- (90:0.45)
	(0:0.7) -- (6:1.1);

\draw[rotate=-24]
	(-9:1.4) -- (6:1.1) -- (21:1.4) -- (39:1.4) -- (54:1.1)
	(-9:1.4) -- (0:1.6) -- (12:1.6) -- (21:1.4)
	(39:1.4) -- (48:1.6) -- (63:1.4)
	(12:1.6) -- (30:1.9) -- (48:1.6)
	(54:1.9) -- (30:1.9) -- (6:1.9) -- (-18:1.9) -- (0:1.6)
	(6:1.9) -- (-6:2.2)
	(30:1.9) -- (42:2.2)
	(-54:2.2) -- (-30:2.2) -- (-6:2.2) -- (18:2.2) -- (18:2.5);

\draw[line width=1.5,rotate=-24]
	(21:1.4) -- (39:1.4) -- (54:1.1)
	(39:1.4) -- (48:1.6)
	(-9:1.4) -- (0:1.6) -- (12:1.6)
	(-18:1.9) -- (0:1.6)
	(-30:2.2) -- (-6:2.2) -- (18:2.2)
	(6:1.9) -- (-6:2.2);

\end{scope}


\begin{scope}[xshift=5.5 cm, rotate=72*\a]

\draw[line width=1.5]
	 (18:0.45) -- (36:0.7) -- (72:0.7)
	(36:0.7) -- (30:1.1);

\draw[rotate=24, xscale=-1]
	(-9:1.4) -- (6:1.1) -- (21:1.4) -- (39:1.4) -- (54:1.1)
	(-9:1.4) -- (0:1.6) -- (12:1.6) -- (21:1.4)
	(39:1.4) -- (48:1.6) -- (63:1.4)
	(12:1.6) -- (30:1.9) -- (48:1.6)
	(54:1.9) -- (30:1.9) -- (6:1.9) -- (-18:1.9) -- (0:1.6);

\draw[line width=1.5,rotate=24, xscale=-1]
	(21:1.4) -- (39:1.4) -- (54:1.1)
	(39:1.4) -- (48:1.6)
	(-9:1.4) -- (0:1.6) -- (12:1.6)
	(-18:1.9) -- (0:1.6);
		
\draw[rotate=-24]	
	(6:1.9) -- (-6:2.2)
	(30:1.9) -- (42:2.2)
	(-54:2.2) -- (-30:2.2) -- (-6:2.2) -- (18:2.2) -- (18:2.5);

\draw[line width=1.5,rotate=-24]
	(-30:2.2) -- (-6:2.2) -- (18:2.2)
	(6:1.9) -- (-6:2.2);

\end{scope}	
	
}

\foreach \a in {0,...,4}
{


\begin{scope}[rotate=72*\a]

\fill
	(6:1.9) circle (0.1)
	(-6:2.5) -- ++(210:0.1) arc (210:138:0.1);
	
\filldraw[fill=gray]
	(-18:1.1) circle (0.1);
	
\filldraw[fill=white]
	(30:1.1) circle (0.1);
	
\end{scope}

\node at (0,-2.7) 
	{$T(5\beta\gamma\epsilon^2\circ,
	5\delta\epsilon^3{\color{gray} \bullet},
	7\epsilon^5\bullet)$};


\begin{scope}[xshift=5.5cm, rotate=72*\a]

\fill
	(-6:2.5) -- ++(210:0.1) arc (210:138:0.1);
	
\filldraw[fill=gray]
	(-18:1.1) circle (0.1)
	(30:1.9) circle (0.1);
	
\filldraw[fill=white]
	(6:1.1) circle (0.1)
	(6:1.9) circle (0.1);
	
\end{scope}

\node at (5.5,-2.7) 
	{$T(10\beta\gamma\epsilon^2\circ,
	10\delta\epsilon^3{\color{gray} \bullet},
	2\epsilon^5\bullet)$};

}

\end{scope}

\end{tikzpicture}
\caption{Six flip modifications of the subdivision of the icosahedron.}
\label{subdivision_tiling_modify60A}
\end{figure}

For the case $\delta=2\epsilon$, therefore, we choose to draw the flip modifications according to the first of Figure \ref{mod60A}, and the first, third and fourth of Figure \ref{mod60B}. These are the fifth and sixth of Figure \ref{subdivision_tiling_modify60A}, and the first and second of Figure \ref{subdivision_tiling_modify60B}. We may also redraw the third and fourth of Figure \ref{subdivision_tiling_modify60A}, as the third and fourth of Figure \ref{subdivision_tiling_modify60B}, to explicitly reveal the two fold and three fold symmetries.

\begin{figure}[htp]
\centering
\begin{tikzpicture}[>=latex,scale=1]



\foreach \a in {-1,1}
\foreach \b in {-1,1}
\foreach \c in {0,1}
\fill[gray!50, yshift=-6*\c cm, xscale=\a,yscale=\b]
	(0,0.6) -- (1.2,1.2) -- (2.1,0) -- (2.7,0) arc (0:90:2.7);

\foreach \a in {-1,1}
\foreach \b in {-1,1}
\foreach \c in {0,1}
\draw[yshift=-6*\c cm, xscale=\a,yscale=\b]
	(0,0) -- (0,0.6) -- (1.2,1.2) -- (2.1,0)
	(0,1.2) -- (1.2,1.2) -- (0,2.1) -- (0,2.1) -- (0,2.7)
	(0,0) circle (2.1);

\foreach \a in {-1,1}
\foreach \c in {0,1}
\draw[line width=1.5, yshift=-6*\c cm, scale=\a]
	(-0.4,0.8) -- (0,1) -- (0.8,1)
	(0,1) -- (-0.4,1.2)
	(0.4,1.2) -- (0,1.5) -- (-0.8,1.5)
	(0,1.5) -- (0.4,1.8)
	(1.8,0.4) to[out=90,in=-50] 
	(35:1.85) to[out=130,in=0]  
	(0.8,1.5)
	(35:1.85) to[out=50,in=240] (60:2.1)
	(1.5,-0.8) to[out=-90,in=40] 
	(-55:1.85) to[out=220,in=0] (0.4,-1.8)
	(-55:1.85) to[out=-40,in=150] (-30:2.1)
	(30:2.1) to[out=-30,in=75]
	(-15:2.3) to[out=255,in=0] 
	(-60:2.1)
	(-15:2.3) to[out=15,in=-45]
	(45:2.4) to[out=135,in=0] 
	(0,2.4);	

\foreach \a in {-1,1}
\foreach \c in {0,1}
\fill
	(0.9*\a,-6*\c) circle (0.1)
	(0,2.1*\a-6*\c) circle (0.1);

		
\foreach \a in {-1,1}
{
\begin{scope}[scale=\a]

\foreach \b in {-1,1}
\draw[yscale=\b]
	(0,0.2) -- (0.3,0.2) -- (0.55,0.6)
	(0.3,0.2) -- (0.55,0) -- (0.9,0) -- (0.8,0.4)
	(0.9,0) -- (1.2,0.25)
	(0.4,0.8) -- (0.55,0.6) -- (0.8,0.4) -- (0.95,0.8)
	(0.8,1) -- (0.95,0.8) -- (1.2,0.65)
	(1.5,0.8) -- (1.2,0.65) -- (1.2,0.25) -- (1.5,0.25)
	(1.8,0.4) -- (1.5,0.25) -- (1.5,0);

\draw[line width=1.5]
	(0.4,0.8) -- (0.55,0.6) -- (0.8,0.4)
	(0.3,0.2) -- (0.55,0.6) 
	(0,-0.2) -- (0.3,-0.2) -- (0.55,-0.6)
	(0.3,-0.2) -- (0.55,0)
	(1.5,0.8) -- (1.2,0.65) -- (1.2,0.25)
	(0.95,0.8) -- (1.2,0.65)
	(0.8,-0.4) -- (0.95,-0.8) -- (0.8,-1)
	(0.95,-0.8) -- (1.2,-0.65)
	(1.5,0.25) -- (1.5,-0.25) -- (1.8,-0.4)
	(1.2,-0.25) -- (1.5,-0.25);

\filldraw[fill=white]
	(0.8,1) circle (0.1)
	(0.4,-0.8) circle (0.1)
	(1.8,0.4) circle (0.1)
	(1.5,-0.8) circle (0.1);	

\filldraw[fill=gray]
	(1.2,1.2) circle (0.1)
	(1.2,-1.2) circle (0.1)
	(2.1,0) circle (0.1);

\end{scope}
}

\filldraw[fill=white]
	(0,0.2) circle (0.1)
	(0,-0.2) circle (0.1);
	
\node at (0,-3.05) 
	{$T(10\beta\gamma\epsilon^2\circ, 
	6\delta\epsilon^3{\color{gray} \bullet},
	4\epsilon^5\bullet)$ };
	

\begin{scope}[yshift=-6 cm]

\foreach \a in {-1,1}
{
\begin{scope}[scale=\a]

\draw
	(0.9,0) -- (0,0.6) 
	(0.9,0) -- (1.2,1.2)
	(0.9,0) -- (0,-0.6) 
	(0.9,0) -- (1.2,-1.2)  
	(0.9,0) -- (2.1,0)
	(0,-0.2) -- (0.3,0) -- (0.6,-0.2)
	(0.3,0) -- (0.3,0.4)
	(0.4,0.8) -- (0.7,0.6) -- (0.6,0.2)
	(0.7,0.6) -- (1.1,0.8)
	(0.8,-1) -- (0.7,-0.6) -- (1,-0.4)
	(0.7,-0.6) -- (0.3,-0.4)
	(1,0.4) -- (1.4,0.4) -- (1.5,0.8)
	(1.4,0.4) -- (1.7,0)
	(1.3,0) -- (1.4,-0.4) -- (1.8,-0.4)
	(1.4,-0.4) -- (1.1,-0.8);
	
\draw[line width=1.5]
	(0.6,0.2) -- (0,0.6) 
	(1,0.4) -- (1.2,1.2)
	(0.6,-0.2) -- (0,-0.6) 
	(1,-0.4) -- (1.2,-1.2)  
	(1.3,0) -- (2.1,0)
	(0.3,0) -- (0.3,0.4)
	(0.7,0.6) -- (1.1,0.8)
	(0.7,-0.6) -- (0.3,-0.4)
	(1.4,0.4) -- (1.7,0)
	(1.4,-0.4) -- (1.1,-0.8);

\filldraw[fill=white]
	(1.2,1.2) circle (0.1)
	(1.2,-1.2) circle (0.1)
	(2.1,0) circle (0.1);
	
\end{scope}
}

\filldraw[fill=gray]
	(0,0.6) circle (0.1)
	(0,-0.6) circle (0.1);	
	
\node at (0,-3.05) 
	{$T(2\beta^2\gamma^2\epsilon{\color{gray} \bullet},
	6\beta\gamma\epsilon^3\circ,
	4\epsilon^5\bullet)$};	

\end{scope}


\begin{scope}[xshift=6cm]


\foreach \a in {0,1,2} 
\fill[gray!50, rotate=120*\a]
	(0,0) -- (90:0.5) -- (50:0.5) -- (10:0.5) -- (-30:0.5)
	(150:2.4) -- (90:2) -- (30:2.4) -- (30:2.7) arc (30:150:2.7) ;

\foreach \a in {0,1,2} 
{
\begin{scope}[rotate=120*\a]

\draw
	(-30:0.5) -- (-30:2) -- (30:2.4) -- (90:2)
	(-30:0.5) -- (10:0.5) -- (50:0.5) -- (90:0.5);
	
\foreach \x in {-1,1}
\draw[xscale=\x]
	(-30:1.5) -- (1,-1)
	(-30:1) -- (0.55,-0.8)
	(1.2,-1.43) -- (0.9,-1.3)
	(0.6,-1.91) -- (0.3,-1.7) -- (-0.3,-1.7)
	(-70:0.5) -- (0.3,-0.7) -- (0.55,-0.8) -- (0.6,-1) -- (1,-1) -- (0.9,-1.3) -- (0.5,-1.4) -- (0.3,-1.7)
	(0.3,-0.7) -- (0,-0.9) -- (0,-1.3)
	(0,-1.3) -- (0.6,-1)
	(0,-1.3) -- (0.5,-1.4)
	(0,0) circle (2.4);

\draw[line width=1.5]
	(0,0) -- (50:0.5)
	(-110:0.5) -- (-0.3,-0.7) -- (-0.55,-0.8)
	(-0.3,-0.7) -- (0,-0.9)
	(0.3,-0.7) -- (0.55,-0.8) -- (0.6,-1)
	(-30:1) -- (0.55,-0.8)
	(1,-1) -- (0.9,-1.3) -- (0.5,-1.4)
	(1.2,-1.43) -- (0.9,-1.3)
	(0.3,-1.7) -- (-0.3,-1.7) -- (-0.6,-1.91)
	(-0.5,-1.4) -- (-0.3,-1.7)
	(-0.6,-1) -- (-1,-1) -- (-0.9,-1.3)
	(-150:1.5) -- (-1,-1)
	(0.6,-1.91) to[out=0, in=-60] (1.84,-0.32)
	(-10:2.4) to[out=200,in=0] (1.7,-1.32)
	(-50:2.4) -- (-50:2.7);
	
\end{scope}
}

\foreach \a in {0,1,2} 
{
\begin{scope}[rotate=120*\a]	
	
\fill
	(0,-1.3) circle (0.1);

\filldraw[fill=white]
	(50:0.5) circle (0.1)
	(90:1) circle (0.1)
	(90:1.5) circle (0.1)
	(-1.2,-1.43) circle (0.1) 
	(0.6,-1.91) circle (0.1);	

\filldraw[fill=gray]
	(30:2.4) circle (0.1);
	
\end{scope}
}

\node at (0,-3.05) 
	{$T(15\beta\gamma\epsilon^2\circ,
	3\delta\epsilon^3{\color{gray} \bullet},
	3\epsilon^5\bullet)$};

\begin{scope}[yshift=-6cm]


\foreach \a in {0,1,2} 
\fill[gray!50, rotate=120*\a]
	(0,0) -- (90:0.5) -- (50:0.5) -- (10:0.5) -- (-30:0.5)
	(150:2.4) -- (90:2) -- (30:2.4) -- (30:2.7) arc (30:150:2.7) ;

\foreach \a in {0,1,2} 
{
\begin{scope}[rotate=120*\a]

\draw
	(-30:0.5) -- (-30:2) -- (30:2.4) -- (90:2)
	(-30:0.5) -- (10:0.5) -- (50:0.5) -- (90:0.5)
	(-30:0.5) -- (0.4,-0.6) -- (0.3,-0.9) -- (0,-1.3)
	(-150:0.5) -- (-0.4,-0.6) -- (-0.3,-0.9) -- (0,-1.3)
	(-30:2) -- (0,-1.3)
	(0,-2.4) -- (0,-1.3)
	(-150:2) -- (0,-1.3)
	(0.4,-0.6) -- (0,-0.8) -- (-0.3,-0.9)
	(-110:0.5) -- (0,-0.8)
	(0.3,-0.9) -- (0.722,-0.85) -- (-30:1)
	(1.155,-1.1) -- (0.722,-0.85)
	(-0.577,-1.2) -- (-0.722,-0.85) -- (-150:1.5)
	(-0.4,-0.6) -- (-0.722,-0.85)
	(1.155,-1.467) -- (0.577,-1.567) -- (0.577,-1.2)
	(0,-2.033) -- (0.577,-1.567)
	(0,-1.667) -- (-0.577,-1.567) -- (-0.577,-1.933)
	(-1.155,-1.1) -- (-0.577,-1.567)
	(0,0) circle (2.4);

\draw[line width=1.5]
	(0,0) -- (50:0.5)
	(-30:0.5) -- (0.4,-0.6) -- (0.3,-0.9) 
	(-150:0.5) -- (-0.4,-0.6) -- (-0.3,-0.9) 
	(0.4,-0.6) -- (0,-0.8)
	(1.732,-1) -- (0.577,-1.2)
	(-1.732,-1) -- (-0.577,-1.2)
	(1.155,-1.1) -- (0.722,-0.85)
	(-0.4,-0.6) -- (-0.722,-0.85)
	(-90:2.4) -- (-90:1.67)
	(0,-2.033) -- (0.577,-1.567)
	(-1.155,-1.1) -- (-0.577,-1.567)
	(0.6,-1.91) to[out=0, in=-45] (1.84,-0.32)
	(-10:2.4) to[out=200,in=-30] (1.7,-1.35)
	(-50:2.4) -- (-50:2.7);
		
\end{scope}
}

\foreach \a in {0,1,2} 
{
\begin{scope}[rotate=120*\a]	
	
\fill
	(0,-1.3) circle (0.1);

\filldraw[fill=white]
	(30:2.4) circle (0.1);	

\filldraw[fill=gray]
	(-30:0.5) circle (0.1)
	(-30:2) circle (0.1);
	
\end{scope}
}

\node at (0,-3.05) 
	{$T(6\beta^2\gamma^2\epsilon{\color{gray} \bullet},
	3\beta\gamma\epsilon^3\circ,
	3\epsilon^5\bullet)$};
	
\end{scope}

\end{scope}

\end{tikzpicture}
\caption{Four flip modifications of the subdivision of the icosahedron.}
\label{subdivision_tiling_modify60B}
\end{figure}

\subsubsection*{Companion}

In the second part of the main theorem, the pentagons in the second and third families are related as in the second and third of Figure \ref{calculation_division}. The second family subdivides a triangular face of icosahedron by using the thick lines, and the third family by using the dashed lines. According to the end of \cite[Section 3.1]{wy1}, this means that the two corresponding pentagonal subdivision tilings are the {\em companions} of each other. The angles $\alpha',\beta',\gamma',\delta',\epsilon'$ of the companion pentagon are related to the angles $\alpha,\beta,\gamma,\delta,\epsilon$ of the original pentagon by
\[
\alpha'=\alpha,\;
\beta'+\gamma'=\delta,\;
\beta+\gamma=\delta',\;
\delta+\delta'=2\pi,\;
\epsilon'=\epsilon.
\]
Moreover, the orientation $\alpha'\to \beta'\to \gamma'\to \delta'\to \epsilon'$ is opposite to the orientation $\alpha\to \beta\to \gamma\to \delta\to \epsilon$.

The flip modification and the companion operation are compatible. The standard pentagonal subdivision of the icosahedron is the union of twenty $n(\alpha^3)$ (the middle of Figure \ref{subdivision2}, and the first of Figure \ref{special1C}), and each $N(\epsilon^5)$ consists of five $n(\alpha^3)$. The companion operation is applied to each $n(\alpha^3)$, by changing thick edges to dashed edges, like the middle of Figure \ref{companion}. The operation does not change $n(\alpha^3)$ as a spherical polygon. On the other hand, the flip modification is applied to $N(\epsilon^5)$, by effectively redistributing the five constituent $n(\alpha^3)$, without changing $N(\epsilon^5)$ as a spherical polygon. Therefore we may apply the companion operation to each $n(\alpha^3)$ in the flip modified tiling, and still get a tiling. 

Figure \ref{companion} shows the companion operation on the flip modification tiling $T(4\beta\gamma\epsilon^2,2\epsilon^4)$ in the first of Figure \ref{subdivision_tiling_modify24}. The shaded and non-shaded regions are switched because the companion operation changes the orientation.

\begin{figure}[htp]
\centering
\begin{tikzpicture}[>=latex,scale=1]


\foreach \a in {0,...,11}
{
\fill[gray!50, rotate=30*\a]
	(0,0) -- (-15:1) -- (15:1);
}	
	
\foreach \a in {0,1,2,3}
{

\draw[rotate=90*\a]
	(0,0) -- (0.4,0) -- (0.6,0.25) -- (0.25,0.6) -- (0,0.4)
	(-15:1) -- (15:1) -- (45:1) -- (75:1)
	(0.6,0.25) -- (15:1)
	(0.6,-0.25) -- (-15:1)
	(-30:1.3) -- (0:1.3) -- (30:1.3) -- (60:1.3)
	(-15:1) -- (0:1.3)
	(45:1) -- (30:1.3)
	(60:1.3) -- (60:1.6);

\draw[line width=1.5, rotate=90*\a]
	(0.6,0.25) -- (0.25,0.6) -- (0,0.4)
	(0.6,-0.25) -- (-15:1)
	(0:1.3) -- (30:1.3) -- (60:1.3)
	(45:1) -- (30:1.3);
	
}

\foreach \a in {0,1,2,3}
{
\begin{scope}[rotate=90*\a]

\filldraw[fill=white]
	(75:1) circle (0.1);

\fill
	(60:1.6) -- ++(-75:0.1) arc (-75:-165:0.1);

\end{scope}
}

\fill
	(0,0) circle (0.1);


\draw[very thick,->]
	(1.8,0) -- (3.2,0);

\foreach \a in {0,1,2}
{	
\begin{scope}[shift={(2.5cm,0.6cm)}, rotate=120*\a]

\draw
	(-30:0.693) -- (90:0.693);

\draw[line width=1.5]
	(0,0) -- (0.4,0);
	
\draw[dash pattern=on 1pt off 1pt]
	(0,0) -- (-0.4,0);

\end{scope}
}	
	

\begin{scope}[xshift=5cm]

\foreach \a in {0,...,11}
{
\fill[gray!50, rotate=30*\a]
	(-15:1) -- (15:1) -- (0:1.6) -- (-30:1.6);
}

\foreach \a in {0,1,2,3}
{

\draw[rotate=90*\a]
	(-0.25,0.6) -- (0,0.4) -- (0,0)
	(-15:1) -- (15:1) -- (45:1) -- (75:1)
	(0.6,0.25) -- (15:1)
	(-15:1) -- (0:1.3) -- (-30:1.3) -- (-30:1.6);

\draw[dash pattern=on 1pt off 1pt,rotate=90*\a]
	(0.4,0) -- (0.25,0.6) -- (-0.25,0.6)
	(0.25,0.6) -- (45:1)
	(15:1) -- (30:1.3) -- (90:1.3)
	(30:1.3) to[out=-40, in=40] (-30:1.3);

}

\end{scope}	
			
\end{tikzpicture}
\caption{Companion of flip modification tiling.}
\label{companion}
\end{figure}

For the pentagonal subdivision of the icosahedron, the companion operation exchanges $\beta+\gamma$ and $\delta$. Therefore the operation changes the flip with respect to $L_{\beta+\gamma=2\epsilon}$ (used for the second family) to the flip with respect to $L_{\delta=2\epsilon}$ (used for the third family). Then the companion operation changes the four tilings in the second family to the four tilings in the third family.

We do not include the companion of the first family in the main theorem, because the first family is the same as its own companion. For example, the flip of the whole right tiling in Figure \ref{companion} is combinatorially the same as the left one. Geometrically, they are also the same because the pentagon is given by the first of Figure \ref{calculation_division}. By $\beta+\gamma=\delta=\pi$, the companion has the same pentagon.

\subsubsection*{Outline of the Paper}

The classification for the edge combinations $a^2b^2c$ and $a^3bc$ in \cite{wy1} is mainly the analysis of the neighbourhood of a special tile with four degree $3$ vertices and the fifth vertex of degree $3,4$ or $5$. While this is still a key technique of this paper, the classification for the edge combination $a^3b^2$ is more about allowable combinations of angles at degree $3$ vertices in a tiling. This is illustrated in the series of Propositions \ref{special1} through \ref{special9}. The classification of pentagonal tilings of the sphere cannot be accomplished by a single technique. Depending on the edge combination, we actually have a spectrum of techniques, with geometrical constraints and spherical trigonometry playing more and more important roles as the edges become less and less varied. In the third of our series \cite{awy}, the classification for the equilateral case is actually dominated by calculations using spherical trigonometry. 

This paper is organized as follows. Section \ref{constraint1} is a key geometrical property for spherical pentagons that will become extremely useful throughout our classification project. Section \ref{basic_facts} develops basic techniques needed for the classification work. This includes general results from \cite{wy1} and some technical results specific to the edge combination $a^3b^2$. Sections \ref{special1tiling} and \ref{special2tiling} prove two key classification results, which provide all the special tilings in our main theorem. Section \ref{general} gives seven cases of angle combinations at degree $3$ vertices that do not admit tilings. The final Section \ref{classification} analyses the neighbourhood of a special tile and completes the classification.

We provide precise calculation when the pentagon is unique. There are two such cases. The first is the pentagon suitable for the flip modification, which appears in Step 2 of the proof of Proposition \ref{special1}. The second is the symmetric pentagon suitable for tiling, which appears in Section \ref{symmetric}. 

The values of angles and edge lengths are expressed as multiples of $\pi$. When the multiple is a rational fraction, such as $\alpha=\frac{2}{3}\pi$, we mean the precise value. When the multiple is a decimal expression, such as $\beta=0.38831\pi$, we mean an approximate value. In other words, $0.38831\pi\le\beta < 0.38832\pi$.

\subsubsection*{Acknowledgement}

We would like to thank Ka Yue Cheuk and Ho Man Cheung. Some of their initial work on the pentagonal subdivision tilings are included in this paper.

\section{Non-symmetric Pentagon}
\label{constraint1}

The pentagon in Figure \ref{geom} has two pairs of edges of equal lengths, denoted $a$ and $b$ (we allow $a=b$). The following is \cite[Lemma 21]{gsy}. 

\begin{lemma}\label{geometry}
For the pentagon in the second of Figure \ref{geom}, we have $\beta=\gamma$ if and only of $\delta=\epsilon$.
\end{lemma}

\begin{figure}[htp]
\centering
\begin{tikzpicture}

\draw
	(0,0.7) -- node[fill=white,inner sep=1] {\small $b$}
	(-1,0) -- node[fill=white,inner sep=1] {\small $a$}
	(-0.7,-1) -- 
	(0.7,-1) -- node[fill=white,inner sep=1] {\small $a$}
	(1,0) -- node[fill=white,inner sep=1] {\small $b$}
	cycle;

\node at (0,0.45) {\small $\alpha$};	
\node at (-0.7,-0.1) {\small $\beta$};
\node at (0.75,-0.15) {\small $\gamma$};	
\node at (-0.55,-0.8) {\small $\delta$};	
\node at (0.55,-0.8) {\small $\epsilon$};	

\end{tikzpicture}
\caption{Pentagon with two pairs of equal edges.}
\label{geom}
\end{figure}

Of course the equalities in Lemma \ref{geometry} mean that the pentagon is symmetric. If we know the pentagon is simple, then we have the following stronger result.

\begin{lemma}\label{geometry1}
If the spherical pentagon in Figure \ref{geom} is simple and has two pairs of edges of equal lengths $a$ and $b$, then $\beta>\gamma$ is equivalent to $\delta<\epsilon$.
\end{lemma}

By simple polygon, we mean the boundary does not intersect itself. By \cite[Lemma 1]{gsy}, in an edge-to-edge tiling of the sphere by congruent pentagons, the pentagon must be simple.

To simplify the discussion, we always use strict inequalities in this section. For example, this means that the opposite of $\alpha>\pi$ is $\alpha<\pi$. The special case of equalities can be easily analysed and therefore will be omitted.

Unlike the plane, the concept of the {\em interior} of a simple polygon on the sphere is relative. Once we designate the ``inside'', then we also have the ``outside''. Moreover, the inside is the ``outside of the outside''. We note that the reformulation of Lemma \ref{geometry1} in terms of the outside is equivalent to the original formulation in terms of the inside. 

We will also use the sine law and the following well known result in the spherical trigonometry: If a spherical triangle has angles $\alpha,\beta,\gamma<\pi$ and corresponding edges $a,b,c$ opposite to the angles, then $\alpha>\beta$ if and only if $a>b$.

Let $A,B,C,D,E$ be the vertices at angles $\alpha,\beta,\gamma,\delta,\epsilon$. Among the two great arcs connecting $B$ and $C$, we take the one with length $<\pi$ to form the edge $BC$, and indicate $BC$ by the dashed line in Figure \ref{geom2}. Since $AB$ and $AC$ intersect only at one point $A$ and have the same length $b$, we get $b<\pi$. Since all three edges $AB,AC,BC$ have lengths $<\pi$, by the sine law, among the two triangles bounded by the three edges, one has all three angles $<\pi$. We denote this triangle by $\triangle ABC$. The angle $\angle BAC$ of $\triangle ABC$ is either $\alpha$ of the pentagon, or its complement $2\pi-\alpha$. Since the inside and outside versions of the lemma are equivalent, we will always assume $\alpha=\angle BAC<\pi$ in the subsequent discussion.

\begin{figure}[htp]
\centering
\begin{tikzpicture}


\draw
	(0,1) -- (-1.2,-1) -- (-0.3,-1.5) -- (0.3,-0.5) -- (1.2,-1) -- cycle;
\draw[dashed]
	(1.2,-1) -- (-1.2,-1)
	(0.3,-0.5) -- ++(-0.9,0.5);

\node at (0,1.2) {\small $A$};
\node at (-1.4,-1) {\small $B$};
\node at (1.4,-1) {\small $C$};
\node at (-0.3,-1.7) {\small $D$};
\node at (0.05,-0.55) {\small $E$};
\node at (0.1,-1.2) {\small $F$};
\node at (-0.8,0) {\small $G$};

\node[fill=white,inner sep=1] at (0.4,0.3) {\small $b$};
\node[fill=white,inner sep=1] at (-0.4,0.3) {\small $b$};
\node[fill=white,inner sep=1] at (-0.7,-1.3) {\small $a$};
\node[fill=white,inner sep=1] at (0.5,-0.6) {\small $a$};

\node at (0,0.7) {\small $\alpha$};
\node at (-0.9,-0.9) {\small $\beta$};
\node at (0.75,-0.6) {\small $\gamma$};
\node at (-0.35,-1.25) {\small $\delta$};
\node at (0.3,-0.35) {\small $\epsilon$};


\begin{scope}[xshift=3.5cm]

\draw
	(0,1) -- (-1.2,-0.2) -- (-0.8,-1.5) -- (0.8,-1.5) -- (1.2,-0.2) -- cycle;
\draw[dashed]
	(1.2,-0.2) -- (-1.2,-0.2);

\node at (0,1.2) {\small $A$};
\node at (-1.4,-0.2) {\small $B$};
\node at (1.4,-0.2) {\small $C$};
\node at (-0.9,-1.65) {\small $D$};
\node at (0.9,-1.65) {\small $E$};

\node[fill=white,inner sep=1] at (-1,-0.9) {\small $a$};
\node[fill=white,inner sep=1] at (1,-0.9) {\small $a$};
\node[fill=white,inner sep=1] at (-0.5,0.5) {\small $b$};
\node[fill=white,inner sep=1] at (0.5,0.5) {\small $b$};

\node at (0,0.7) {\small $\alpha$};  
\node at (-0.65,-1.3) {\small $\delta$};
\node at (0.7,-1.3) {\small $\epsilon$};  
\node at (-0.9,-0.4) {\small $\beta'$};
\node at (0.95,-0.4) {\small $\gamma'$};

\end{scope}


\begin{scope}[xshift=7cm]

\draw
	(0,1) -- (-1.2,-1.5) -- (-0.35,-0.7) -- (0.35,-0.7) -- (1.2,-1.5) -- cycle;
\draw[dashed]
	(1.2,-1.5) -- (-1.2,-1.5);
	
\node at (0,1.2) {\small $A$};
\node at (-1.4,-1.5) {\small $B$};
\node at (1.4,-1.5) {\small $C$};

\node at (0,0.65) {\small $\alpha$};
\node at (-0.8,-0.95) {\small $\beta$};
\node at (0.75,-0.9) {\small $\gamma$};
\node at (-0.3,-0.5) {\small $\delta$};
\node at (0.3,-0.55) {\small $\epsilon$};
		
\node at (-0.7,-1.3) {\small $\beta'$};
\node at (0.6,-1.25) {\small $\gamma'$};
\node at (-0.3,-0.9) {\small $\delta'$};
\node at (0.25,-0.9) {\small $\epsilon'$};

\node[fill=white,inner sep=1] at (0.5,0) {\small $b$};
\node[fill=white,inner sep=1] at (-0.5,0) {\small $b$};

\end{scope}

\end{tikzpicture}
\caption{Reduce Lemma \ref{geometry1} to Lemma \ref{geometry2}.}
\label{geom2}
\end{figure}

The pentagon is obtained by choosing $D,E$, and then connecting $B$ to $D$, $C$ to $E$, and $D$ to $E$ by great arcs. By $BC<\pi$, we know $BC$ intersects $BD$ and $CE$ only at $B$ and $C$. Moreover, $BC$ and $DE$ intersect at no more than one point. 

Suppose $BC$ and $DE$ intersect at one point $F$. Then one of $D,E$ is inside $\triangle ABC$ and one is outside (we omit the ``equality case'' of $D$ or $E$ being on $BC$). The first of Figure \ref{geom2} shows the case $D$ is outside $\triangle ABC$ and $E$ is inside $\triangle ABC$. Then by $BC$ intersecting $BD$ and $CE$ only at $B$ and $C$, and the isosceles triangle $\triangle ABC$, we get 
\[
\beta>\angle ABC=\angle ACB> \gamma.
\]
On the other hand, by $AC=b<\pi$ and $BC<\pi$, the prolongation of $CE$ intersects the boundary of $\triangle ABC$ at a point $G$ on $AB$. By $AB<\pi$, $\alpha<\pi$, $\gamma<\angle ACB<\pi$, and applying the sine law to $\triangle ACG$, we get $CG<\pi$. Therefore $a=CE<CG<\pi$. By $a<\pi$, $BF<BC<\pi$, $CF<BC<\pi$, $\angle BFD=\angle CFE<\pi$, and applying the sine law to $\triangle BDF$ and $\triangle CEF$, we get $\angle BDF<\pi$ and $\angle CEF<\pi$. Therefore
\[
\delta=\angle BDF<\pi<2\pi-\angle CEF=\epsilon.
\] 
This proves that $D$ outside and $E$ inside imply $\beta>\gamma$ and $\delta<\epsilon$. Similarly, $D$ inside and $E$ outside imply $\beta<\gamma$ and $\delta>\epsilon$. 

Suppose $BC$ and $DE$ are disjoint. By the simple pentagon, we know $BC$ does not intersect $AB$ and $AC$. Therefore either both $B,C$ are outside $\triangle ABC$, as in the second of Figure \ref{geom2}, or both $B,C$ are inside $\triangle ABC$, as in the third of Figure \ref{geom2}. By the isosceles triangle $\triangle ABC$, we get $\beta-\beta'=\gamma-\gamma'$ in the second picture, and $\beta+\beta'=\gamma+\gamma'$ in the third picture. Therefore we get $\beta>\gamma$ if and only if $\beta'>\gamma'$ in the second picture, and $\beta>\gamma$ if and only if $\beta'<\gamma'$ in the third picture. Moreover, in the third picture, we have $\delta<\epsilon$ if and only if $\delta'>\epsilon'$. In both cases, the proof of Lemma \ref{geometry1} is then reduced to the proof of the following similar result for quadrilaterals.

\begin{lemma}\label{geometry2}
If the spherical quadrilateral in Figure \ref{geom3} is simple and has a pair of equal edges $a$, then $\beta>\gamma$ is equivalent to $\delta<\epsilon$.
\end{lemma}

\begin{figure}[htp]
\centering
\begin{tikzpicture}


\draw
	(-1.2,0.6) 
	-- node[fill=white,inner sep=1] {\small $a$}
	(-0.8,-0.6)
	-- (0.8,-0.6)
	-- node[fill=white,inner sep=1] {\small $a$}
	(1.2,0.6)
	-- cycle;

\node at (-1.4,0.6) {\small $B$};
\node at (1.4,0.6) {\small $C$};
\node at (-1,-0.6) {\small $D$};
\node at (1,-0.6) {\small $E$};

\node at (-0.85,0.35) {\small $\beta$};
\node at (0.9,0.35) {\small $\gamma$};
\node at (-0.65,-0.4) {\small $\delta$};
\node at (0.65,-0.4) {\small $\epsilon$};

\end{tikzpicture}
\caption{Quadrilateral with a pair of equal edges.}
\label{geom3}
\end{figure}

Similar to pentagon, we note that the reformulation of Lemma \ref{geometry2} in terms of the outside is equivalent to the original formulation in terms of the inside.

To prove Lemma \ref{geometry2}, we use the conformally accurate way of drawing great circles on the sphere. Let the circle $\Gamma$ be the stereographic projection (from the north pole to the tangent space of the south pole) of the equator. Then antipodal points on the equator are projected to antipodal points on $\Gamma$. We denote the antipodal point of $P$ by $P^*$. Since the intersection of any great arc with the equator is a pair of antipodal points on the equator, the great circles of the sphere are in one-to-one correspondence with the circles (and straight lines) on the plane that intersect $\Gamma$ at pairs of antipodal points.

\begin{proof}
Suppose $a>\pi$. In Figure \ref{geom4}, we draw great circles $\bigcirc BPP^*$ and $\bigcirc CPP^*$ containing the two $a$-edges. The two circles intersect at a pair of antipodal points $P,P^*$ and divide the sphere into four $2$-gons. Since $a>\pi$, and the quadrilateral is simple, we know $P$ and $P^*$ lie in different $a$-edges. Up to symmetry, therefore, there are two ways the four vertices $B,C,D,E$ of the two $a$-edges can be located, described by the first and second of Figure \ref{geom4}. Moreover, by $a>\pi$, the antipodal point $B^*$ of $B$ lies in the $a$-edge $BP^*D$.

In the first of Figure \ref{geom4}, we consider two great arcs connecting $B$ and $C$. One great arc (the solid one) is completely contained in the indicated $2$-gon. The other great arc (the dashed one) intersects the $a$-edge $BP^*D$ at the antipodal point $B^*$ and therefore cannot be an edge of the quadrilateral. We conclude that the edge $BC$ is the solid one. By the same reason, the edge $DE$ is also the solid one, that is completely contained in the indicated $2$-gon. Then the picture implies
\[
\beta<\pi<\gamma,\quad
\delta>\pi>\epsilon.
\]
Similar argument gives the $BC$ edge and $DE$ edges of the quadrilateral in the second of Figure \ref{geom4}, and we get the same inequalities above.

This completes the proof for the case $a>\pi$. Next we may assume $a<\pi$. 

\begin{figure}[htp]
\centering
\begin{tikzpicture}[scale=0.9]


\draw[dotted]
	(0,0) circle (1);
	
\coordinate (P) at (0,1);
\coordinate (Q) at (0,-1);
\coordinate (P1) at (160:1);
\coordinate (Q1) at (-20:1);
\coordinate (P2) at (180:1);
\coordinate (Q2) at (0:1);

\node at (0.05,1.25) {\small $P$};	
\node at (0,-1.3) {\small $P^*$};
\node at (45:1.2) {\small $\Gamma$};	


\begin{scope}[xshift=-1cm]

\coordinate (B) at (100:1.414);
\coordinate (D) at (10:1.414);
\coordinate (BB) at (-30:1.414);

\node at (105:1.6) {\small $B$};
\node at (110:1.2) {\small $\beta$};
\node at (0:1.6) {\small $D$};
\node at (20:1.3) {\small $\delta$};
\node at (-20:1.65) {\small $B^*$};

\end{scope}


\begin{scope}[xshift=1cm]

\coordinate (C) at (186.3:1.414);
\coordinate (E) at (-54.1:1.414);

\node at (-50:1.6) {\small $E$};
\node at (-45:1.25) {\small $\epsilon$};

\end{scope}


\begin{scope}
    \clip (-2.5,-1.8) rectangle (2.5,1.8);

\arcThroughThreePoints{B}{BB}{D};
\arcThroughThreePoints[dashed]{D}{P}{B};

\arcThroughThreePoints{E}{P}{C};
\arcThroughThreePoints[dashed]{C}{Q}{E};

\arcThroughThreePoints{B}{P1}{C};
\arcThroughThreePoints[dashed]{C}{Q1}{B};

\arcThroughThreePoints{E}{Q2}{D};
\arcThroughThreePoints[dashed]{D}{Q2}{E};

\end{scope}

\node[fill=white,inner sep=2] at (-2.4,0) {\small $a$};
\node[fill=white,inner sep=2] at (2.4,0) {\small $a$};

\node[fill=white,inner sep=0] at (-0.57,0.25) {\small $C$};
\node[fill=white,inner sep=1] at (-0.35,-0.35) {\small $\gamma$};


\begin{scope}[xshift=5.5cm]

\draw[dotted]
	(0,0) circle (1);
	
\coordinate (P) at (0,1);
\coordinate (Q) at (0,-1);
\coordinate (P1) at (135:1);
\coordinate (Q1) at (-45:1);
\coordinate (P2) at (180:1);
\coordinate (Q2) at (0:1);
	

\begin{scope}[xshift=-1cm]

\coordinate (B) at (100:1.414);
\coordinate (D) at (10:1.414);
\coordinate (BB) at (-30:1.414);

\node at (95:1.6) {\small $B$};
\node at (110:1.55) {\small $\beta$};
\node at (5:1.2) {\small $D$};
\node at (15:1.6) {\small $\delta$};

\end{scope}


\begin{scope}[xshift=1cm]

\coordinate (C) at (170:1.414);
\coordinate (EE) at (-54.1:1.414);
\coordinate (E) at (-69.4:1.414);

\node at (160:1.3) {\small $\epsilon$};
\node at (180:1.57) {\small $E$};

\end{scope}


\begin{scope}
    \clip (-2.5,-1.8) rectangle (2.5,1.8);

\arcThroughThreePoints{B}{BB}{D};
\arcThroughThreePoints[dashed]{D}{P}{B};

\arcThroughThreePoints{E}{P}{C};
\arcThroughThreePoints[dashed]{C}{Q}{E};

\arcThroughThreePoints{E}{P1}{B};
\arcThroughThreePoints[dashed]{B}{P1}{E};

\arcThroughThreePoints{D}{EE}{C};
\arcThroughThreePoints[dashed]{C}{EE}{D};

\end{scope}

\node[fill=white,inner sep=2] at (-2.4,0) {\small $a$};
\node[fill=white,inner sep=2] at (2.4,0) {\small $a$};

\node[fill=white,inner sep=0] at (1.9,-1.33) {\small $C$};
\node[fill=white,inner sep=1] at (1.3,-1.3) {\small $\gamma$};

\end{scope}


\begin{scope}[xshift=10cm]

\draw
	(0,0) circle (1.4);

\coordinate (D) at (0,1.4);
\coordinate (DD) at (0,-1.4);
\coordinate (E) at (-60:1.4);
\coordinate (EE) at (120:1.4);

\coordinate (B) at (50:0.8);

\coordinate (C) at (200:0.52);

\coordinate (P) at (220:1.4);
\coordinate (Q) at (40:1.4);


\begin{scope}
    \clip (-1.5,-1.8) rectangle (1.5,1.8);
    
\arcThroughThreePoints{B}{DD}{D};
\arcThroughThreePoints[dashed]{D}{DD}{B};

\arcThroughThreePoints{B}{P}{C};
\arcThroughThreePoints[dashed]{C}{P}{B};

\arcThroughThreePoints{C}{EE}{E};
\arcThroughThreePoints[dashed]{E}{EE}{C};

\end{scope}

\node at (0.1,1.6) {\small $D$};
\node at (0.1,-1.6) {\small $D^*$};
\node at (93:1.2) {\small $\delta$};

\node[fill=white,inner sep=1] at (40:0.95) {\small $B$};
\node at (65:0.7) {\small $\beta$};

\node at (215:0.3) {\small $C$};
\node[fill=white,inner sep=1] at (195:0.68) {\small $\gamma$};

\node at (-55:1.2) {\small $E$};
\node at (-75:1.25) {\small $\epsilon$};

\end{scope}
           
\end{tikzpicture}
\caption{Lemma \ref{geometry2}: $a>\pi$, and $DE<\pi$ in case $\delta,\gamma<\pi$.}
\label{geom4}
\end{figure}

We argue that, if $\delta,\epsilon<\pi$, then we may further assume $DE<\pi$. Suppose $DE>\pi$. We draw the great circle $\bigcirc DE$ containing $DE$ in the third of Figure \ref{geom4}. By $\delta,\epsilon<\pi$, the part of $DB$ near $D$ and the part of $EC$ near $E$ lie in the same hemisphere $H$ bounded by the circle $\bigcirc DE$. If $DB$ is not in $H$, then it passes through the antipodal point $D^*$ of $D$. By $DB>\pi$, we know $D^*$ lies on $DB$, contradicting the simple quadrilateral. Therefore $DB$ lies in $H$. By the same reason, $EC$ also lies in $H$.  By simple quadrilateral, this implies $BC$ also lies in $H$. 

Since all three edges connecting $D,B,C,E$ lie in $H$, we have the complement of the quadrilateral in $H$. The complement is still a quadrilateral, which differs from the original one by replacing $DE$ with the other great arc (of length $2\pi-DE<\pi$) connecting $D$ and $E$, and by replacing $\beta,\gamma,\delta,\epsilon$ with $2\pi-\beta,2\pi-\gamma,\pi-\delta,\pi-\epsilon$. Moreover, we still have $\pi-\delta,\pi-\epsilon<\pi$. It is easy to see that the lemma for the new quadrilateral is equivalent to the lemma for the original quadrilateral.

Now we prove the lemma by dividing into two cases.
\begin{enumerate}
\item At least three angles $<\pi$. By considering the complement quadrilateral, this also covers the case of at least three angles $>\pi$. 
\item Two angles $>\pi$ and two angles $<\pi$. 
\end{enumerate}

Suppose at least three angles $<\pi$. Up to symmetry, we may assume $\beta,\delta,\epsilon<\pi$. By the earlier argument, we may further assume that $a<\pi$ and $DE<\pi$. In Figure \ref{geom5}, we draw the great circles $\bigcirc DB$ and $\bigcirc DE$ containing $DB$ and $DE$. The two great circles divide the sphere into four $2$-gons. By $\delta<\pi$, we may assume that $\delta$ is an angle of the middle $2$-gon. By $DB=a<\pi$ and $DE<\pi$, $B$ and $E$ lie on the two edges of the middle $2$-gon. By $\beta,\epsilon<\pi$, we know $BC$ and $EC$ are inside the middle $2$-gon. We also prolong the $a$-edge $EC$ to intersect the boundary of middle $2$-gon at $T$. The two pictures in Figure \ref{geom5} refer to $\gamma>\pi$ and $\gamma<\pi$.

\begin{figure}[htp]
\centering
\begin{tikzpicture}


\draw[dashed]
	(-0.5,0) circle (1.6)
	(0.5,0) circle (1.6);
	
\draw[dotted]
	(0,0) circle (1.52);
			
\coordinate (DD) at (0,1.52);
\coordinate (D) at (0,-1.52);
\coordinate (C) at (-0.108,-0.27);
\coordinate (P1) at (145:1.52);
\coordinate (Q1) at (-35:1.52);
\coordinate (P2) at (200:1.52);
\coordinate (Q2) at (20:1.52);

\node at (0.05,1.75) {\small $D^*$};	
\node at (0,-1.7) {\small $D$};	
\node at (0,-1.3) {\small $\delta$};


\begin{scope}[xshift=-0.5cm]

\coordinate (E) at (10:1.6);

\node at (15:1.4) {\small $E$};
\node at (3:1.45) {\small $\epsilon$};

\end{scope}


\begin{scope}[xshift=0.5cm]

\coordinate (B) at (160:1.6);

\node at (153:1.4) {\small $B$};
\node at (175:1.45) {\small $\beta$};

\end{scope}


\begin{scope}
    \clip (-2.5,-1.8) rectangle (2.5,1.8);

\arcThroughThreePoints{D}{DD}{E};
\arcThroughThreePoints{B}{DD}{D};

\arcThroughThreePoints{B}{P1}{C};
\arcThroughThreePoints[dashed]{C}{P1}{B};

\arcThroughThreePoints{C}{Q2}{E};
\arcThroughThreePoints[dashed]{E}{Q2}{C};

\end{scope}

\node at (-0.05,-0.05) {\small $C$};
\node at (-0.15,-0.5) {\small $\gamma$};
\node[fill=white,inner sep=0] at (-1.25,-0.5) {\small $T$};

\node[fill=white,inner sep=2] at (0.4,-0.1) {\small $a$};
\node[fill=white,inner sep=2] at (-0.7,-1) {\small $a$};


\begin{scope}[xshift=5.5cm]

\draw[dashed]
	(-0.5,0) circle (1.6)
	(0.5,0) circle (1.6);
\draw[dotted]
	(0,0) circle (1.52);
			
\coordinate (DD) at (0,1.52);
\coordinate (D) at (0,-1.52);
\coordinate (C) at (0.05,1.07);
\coordinate (P1) at (145:1.52);
\coordinate (Q1) at (-35:1.52);
\coordinate (P2) at (200:1.52);
\coordinate (Q2) at (20:1.52);

\node at (0.05,1.75) {\small $D^*$};	
\node at (0,-1.7) {\small $D$};	
\node at (0,-1.3) {\small $\delta$};


\begin{scope}[xshift=-0.5cm]

\coordinate (E) at (10:1.6);

\node at (14:1.8) {\small $E$};
\node at (10:1.45) {\small $\epsilon$};

\end{scope}


\begin{scope}[xshift=0.5cm]

\coordinate (B) at (160:1.6);

\node at (160:1.8) {\small $B$};
\node at (160:1.45) {\small $\beta$};

\node at (135:1.8) {\small $T$};
\node at (-75:1.4) {\small $T^*$};

\end{scope}


\begin{scope}
    \clip (-2.5,-1.8) rectangle (2.5,1.8);

\arcThroughThreePoints{D}{DD}{E};
\arcThroughThreePoints{B}{DD}{D};

\arcThroughThreePoints{C}{P2}{B};
\arcThroughThreePoints[dashed]{B}{P2}{C};

\arcThroughThreePoints{E}{P1}{C};
\arcThroughThreePoints[dashed]{C}{P1}{E};

\end{scope}

\node at (0.1,1.25) {\small $C$};
\node at (0,0.85) {\small $\gamma$};

\node[fill=white,inner sep=2] at (0.6,0.8) {\small $a$};
\node[fill=white,inner sep=2] at (-1,-0.6) {\small $a$};

\end{scope}
           
\end{tikzpicture}
\caption{Lemma \ref{geometry2}: At least three angles $>\pi$.}
\label{geom5}
\end{figure}

In the first picture, we have $DT<DB=a=EC<ET$. Since all angles of $\triangle DET$ are $<\pi$, this implies that $\epsilon=\angle DET<\angle EDT=\delta$. We also have $\beta<\pi<\gamma$. 

In the second picture, the great circles $\bigcirc DB$ and $\bigcirc EC$ containing the two $a$-edges intersect at antipodal points $T$ and $T^*$, and the two antipodal points do not lie on the two $a$-edges. This implies $BT+a+DT^*=\pi=CT+a+ET^*$. Since all angles in $\triangle BCT^*$ and $\triangle DET$ are $<\pi$, we then have 
\[
\beta>\gamma
\iff a+ET^*>a+DT^*
\iff a+CT<a+BT
\iff \delta<\epsilon.
\]

Suppose two angles $>\pi$ and two angles $<\pi$. Then up to symmetry, we need to consider the following three subcases:
\begin{enumerate}
\item $\beta,\gamma>\pi$ and $\delta,\epsilon<\pi$.
\item $\beta,\delta>\pi$ and $\gamma,\epsilon<\pi$.
\item $\beta,\epsilon>\pi$ and $\gamma,\delta<\pi$.
\end{enumerate}

In the first subcase, by the earlier argument, we may additionally assume $a<\pi$ and $DE<\pi$. We draw great circles $\bigcirc BD$ and $\bigcirc CE$ containing the two $a$-edges. Since $a<\pi$ and the two $a$-edges do not intersect, we may assume that they are inside the two edges of the middle $2$-gon bounded by $\bigcirc BD$ and $\bigcirc CE$, as in the first of Figure \ref{geom6}. By $DE<\pi$, we know $DE$ lies inside the middle $2$-gon. By $\delta,\epsilon<\pi$, the two angles also lie in the middle $2$-gon. Then by $\beta,\gamma>\pi$, we know $BC$ lies outside the middle $2$-gon. If we replace the $BC$ edge by the other great arc $c$ connecting the two points, then we get a new quadrilateral with angles $\beta-\pi,\gamma-\pi,\delta,\epsilon$. Therefore the lemma for the new quadrilateral is equivalent to the lemma for the original quadrilateral. All four angles of the new quadrilateral are $<\pi$. By the first case, the lemma holds for the new quadrilateral. This completes the proof of the first subcase.

\begin{figure}[htp]
\centering
\begin{tikzpicture}


\draw[dashed]
	(-0.5,0) circle (1.6)
	(0.5,0) circle (1.6);

\fill 
	(0.96,-0.65) circle (0.03);
	
\coordinate (AA) at (0,1.52);
\coordinate (A) at (0,-1.52);
\coordinate (E) at (0.96,-0.65);
\coordinate (B) at (-0.96,0.65);
\coordinate (P1) at (145:1.52);
\coordinate (Q1) at (-35:1.52);
\coordinate (P2) at (200:1.52);
\coordinate (Q2) at (20:1.52);



\begin{scope}[xshift=-0.5cm]

\coordinate (C) at (40:1.6);

\node at (36:1.4) {\small $C$};
\node at (-25:1.8) {\small $E$};
\node at (-22:1.45) {\small $\epsilon$};

\end{scope}


\begin{scope}[xshift=0.5cm]

\coordinate (D) at (220:1.6);

\node at (160:1.45) {\small $B$};
\node at (220:1.8) {\small $D$};
\node at (217:1.4) {\small $\delta$};

\end{scope}


\begin{scope}
    \clip (-2.5,-1.8) rectangle (2.5,1.8);

\arcThroughThreePoints{E}{AA}{C};
\arcThroughThreePoints{B}{AA}{D};

\arcThroughThreePoints{D}{P2}{E};
\arcThroughThreePoints[dashed]{E}{P2}{D};

\arcThroughThreePoints{B}{Q2}{C};
\arcThroughThreePoints[dashed]{C}{Q2}{B};

\end{scope}

\node[fill=white,inner sep=0] at (-0.85,0.85) {\small $\beta$};
\node[fill=white,inner sep=0] at (0.6,1.1) {\small $\gamma$};

\node[fill=white,inner sep=2] at (1.05,0.3) {\small $a$};
\node[fill=white,inner sep=2] at (-1.05,-0.3) {\small $a$};

\node[fill=white,inner sep=2] at (-0.1,1.1) {\small $c$};


\begin{scope}[xshift=5cm]

\draw
	(0,0) circle (1.6);
	
\coordinate (EE) at (0,1.6);
\coordinate (E) at (0,-1.6);
\coordinate (CC) at (200:1.6);
\coordinate (C) at (20:1.6);
\coordinate (B) at (0.7,0);
\coordinate (D) at (-0.31,0.98);
\coordinate (P) at (130:1.6);
\coordinate (Q) at (-50:1.6);		


\begin{scope}
    \clip (-1.8,-1.8) rectangle (1.8,1.8);

\arcThroughThreePoints{B}{CC}{C};
\arcThroughThreePoints[dashed]{C}{CC}{B};

\arcThroughThreePoints{D}{EE}{E};
\arcThroughThreePoints[dashed]{E}{EE}{D};

\arcThroughThreePoints{B}{P}{D};
\arcThroughThreePoints[dashed]{D}{P}{B};

\end{scope}

\node at (92:1.8) {\small $E^*$};	
\node at (-94:1.8) {\small $E$};	
\node at (16:1.8) {\small $C$};
\node at (205:1.85) {\small $C^*$};
\node at (0.8,0.3) {\small $B$};
\node at (-0.4,1.25) {\small $D$};

\node at (0.62,-0.25) {\small $\beta$};
\node at (10:1.45) {\small $\gamma$};
\node at (-0.25,0.7) {\small $\delta$};
\node at (-87:1.45) {\small $\epsilon$};

\node[fill=white,inner sep=2] at (-30:1.6) {\small $a$};
\node[fill=white,inner sep=1] at (0.3,0.5) {\small $a$};

\end{scope}
         
\end{tikzpicture}
\caption{Lemma \ref{geometry2}: Two angles $<\pi$ and two angles $>\pi$.}
\label{geom6}
\end{figure}

In the second subcase, we may again assume $a<\pi$. We draw the circle $\bigcirc CE$ containing the $a$-edge $CE$, as in the second of Figure \ref{geom6}. By $\gamma,\epsilon<\pi$, the part of $CB$ near $C$ and the part of $ED$ near $E$ lie in the same hemisphere $H_1$ bounded by $\bigcirc CE$. The prolongations of $CB$ and $ED$ intersect at a point inside the hemisphere $H_1$. Since $CB$ and $ED$ do not intersect, we must have either $B,C$ lie in the same hemisphere $H_2$ bounded by the circle $\bigcirc DEE^*$, or $D,E$ lie in the same hemisphere bounded by the circle $\bigcirc BCC^*$. Without loss of generality, we may assume the first scenario happens, as in the second of Figure \ref{geom6}. By $a<\pi$, we also know that both $BC$ and $CE$ lie in $H_2$, as in the second of Figure \ref{geom6}. Then $BD=a<\pi$ implies that $BD$ also lies inside $H_1$. This implies that $\delta$ is the angle between an edge $BD$ inside $H_2$ and the boundary $\bigcirc DEE^*$ of $H_2$. Such angle is always $<\pi$. This contradicts the assumption that $\delta>\pi$.

Finally, the third subcase is consistent with the conclusion of lemma.
\end{proof}

\section{Basic Technique}
\label{basic_facts}

We use the notations and techniques in \cite[Section 2]{wy1}, and add some discussion specific to the edge combination $a^3b^2$.

By \cite[Lemma 1]{gsy}, the pentagon is simple. By \cite[Lemma 9]{wy1}, the pentagon is given by Figure \ref{pentagon}. There is one $b^2$-angle $\alpha$, two $ab$-angles $\beta,\gamma$, and two $a^2$-angles $\delta,\epsilon$. By \cite[Lemma 4]{wy1}, we have the {\em angle sum equation for pentagon}
\[
\alpha+\beta+\gamma+\delta+\epsilon
=(3 + \tfrac{4}{f})\pi.
\]

A {\em special tile} has four degree $3$ vertices, and the fifth vertex has degree $3,4$ or $5$. We call them respectively $3^5$-tile, $3^44$-tile, or $3^45$-tile. Then we may combine \cite[Lemmas 1, 2, 3]{wy1} and the idea of \cite[Proposition 5]{gsy} to get the following. Pentagonal subdivision is introduced in \cite[Section 3.1]{wy1}.

\begin{lemma}\label{basic}
An edge-to-edge tiling of the sphere by pentagons has a special tile.
\begin{enumerate}
\item The number of tiles $f\ge 12$ and is even. Moreover, if $f=12$, then the tiling is the dodecahedron, i.e., the pentagonal subdivision of the tetrahedron.
\item If there is no $3^5$-tile, then $f\ge 24$. Moreover, if $f=24$, then every tile is a $3^44$-tile, and the tiling is the pentagonal subdivision of the octahedron (and cube).
\item If there is no $3^5$-tile and no $3^44$-tile, then $f\ge 60$. Moreover, if $f=60$, then every tile is a $3^45$-tile, and the tiling is the pentagonal subdivision of the icosahedron (and dodecahedron).
\end{enumerate}
\end{lemma}

Pentagonal tilings of the sphere with the minimal number $f=12$ of tiles has been classified in \cite{ay1,gsy}. As explained in \cite[Section 2.1]{wy1}, we may assume $f$ is an even integer $\ge 16$ throughout this paper. In fact, only $f>12$ is used in this paper.

The collection of all possible combinations of angles at vertices is the {\em anglewise vertex combination}, or AVC. The AVC is introduced in \cite[Section 2.4]{wy1}. The AVC we derive initially may contain some combinations that actually do not appear, and may become more refined (i.e., smaller) after further argument.

To derive the AVC, we use various geometrical and combinatorial constraints. The first constraint is given by Lemmas \ref{geometry} and \ref{geometry1}, which implies three possibilities:
\begin{enumerate}
\item $\beta=\gamma$ and $\delta=\epsilon$.
\item $\beta>\gamma$ and $\delta<\epsilon$.
\item $\beta<\gamma$ and $\delta>\epsilon$.
\end{enumerate}
The first is the symmetric case. In the non-symmetric case (i.e., $\beta\ne\gamma$ and $\delta\ne\epsilon$), we have 
\[
\beta>\gamma
\iff
\delta<\epsilon.
\]
This implies how two degree $3$ vertices with only one $b$-edge can appear simultaneously. We note that, by the edge length consideration, such vertices are $\beta\gamma\delta,\beta\gamma\epsilon,\beta^2\delta,\beta^2\epsilon,\gamma^2\delta,\gamma^2\epsilon$.  

\begin{lemma}\label{geometry3}
If the spherical pentagon in Figure \ref{pentagon} is simple and not symmetric, then we cannot have three different degree $3$ vertices with only one $b$-edge in a tiling by the pentagon. Moreover, the following are the only possible combinations of two degree $3$ vertices with only one $b$-edge:
\begin{enumerate}
\item $\beta\gamma\delta,\gamma^2\epsilon$.
\item $\beta\gamma\epsilon,\beta^2\delta$.
\item $\beta^2\delta,\gamma^2\epsilon$.
\end{enumerate}
\end{lemma}

The following is \cite[Lemma 8]{wy1}, about an angle not appearing at degree $3$ vertices.

\begin{lemma}\label{hdeg}
Suppose an angle $\theta$ does not appear at degree $3$ vertices in a tiling of the sphere by pentagons with the same angle combination. Then one of $\theta^3\rho,\theta^4,\theta^5$ is a vertex, where $\rho\ne\theta$.
\end{lemma}

\begin{lemma}\label{geometry4}
Suppose two spherical pentagons in Figure \ref{pentagon} share the $a$-edge opposite to $\alpha$. If the two shared vertices have degree $3$, then the arrangement of one pentagon implies the arrangement of the other pentagon.
\end{lemma} 

\begin{figure}[htp]
\centering
\begin{tikzpicture}[>=latex,scale=1]


\foreach \a in {-1,1}
{
\begin{scope}[xshift=0.566*\a cm,scale=\a]

\draw
	(-72:0.7) -- (0:0.7) -- (72:0.7) -- (144:0.7) -- (-144:0.7) -- (-72:0.7);

\draw[line width=1.5]
	(-72:0.7) -- (0:0.7) -- (72:0.7);

\end{scope}
}

\node at (1,0) {\small $\alpha$};
\node at (0.7,0.4) {\small $\beta$};
\node at (0.15,0.25) {\small $\delta$};
\node at (0.15,-0.3) {\small $\epsilon$};
\node at (0.7,-0.45) {\small $\gamma$};

\node at (-1,0) {\small $\alpha$};
\node at (-0.7,-0.4) {\small $\beta$};
\node at (-0.15,-0.25) {\small $\delta$};
\node at (-0.15,0.3) {\small $\epsilon$};
\node at (-0.7,0.45) {\small $\gamma$};

\end{tikzpicture}
\caption{One pentagon determines the other.}
\label{classify0}
\end{figure}

\begin{proof}
The arrangement of the angles in a pentagon is determined by the locations of its $\delta,\epsilon$.  If the arrangements of the two pentagons are not related as in Figure \ref{classify0}, then the two vertices shared by the two pentagons are $\delta^2\cdots$ and $\epsilon^2\cdots$. These are different combinations of three angles from $\delta,\epsilon$. By comparing the angle sums of the two vertices, we get $\delta=\epsilon$. By Lemma \ref{geometry}, the pentagon is symmetric, and we still get Figure \ref{classify0}.
\end{proof}

\begin{lemma}\label{klem4}
In a spherical tiling by congruent pentagons in Figure \ref{pentagon}, if a vertex has no $\beta,\gamma$, then the vertex is $\alpha^k$ or $\delta^k\epsilon^l$.
\end{lemma}

\begin{proof}
If a vertex has both $a$ and $b$, then it has $ab$-angle $\beta$ or $\gamma$. If a vertex has only $a$, then it has only $a^2$-angles $\delta,\epsilon$. Therefore the vertex is $\delta^k\epsilon^l$. If a vertex has only $b$, then it has only $b^2$-angle $\alpha$. Therefore the vertex is $\alpha^k$.
\end{proof}

\begin{lemma}\label{klem5}
Let $\theta,\rho$ be $ab$-angles in a spherical tiling by congruent pentagons in Figure \ref{pentagon}.
\begin{enumerate}
\item If a vertex is $\theta\thin\rho\cdots$, and the remainder has no $ab$-angle, then the vertex has only $\alpha$. In other words, $\theta\thin\rho\cdots=\alpha^k\theta\rho$.
\item If a vertex is $\theta\thick\rho\cdots$, and the remainder has no $ab$-angle, then the vertex has only $\delta,\epsilon$. In other words, $\theta\thick\rho\cdots=\theta\rho\delta^k\epsilon^l$.
\end{enumerate}
\end{lemma}

\begin{proof}
We have $\theta\thin\rho\cdots=\thick\theta\thin\rho\thick\cdots$. If the remainder has no $ab$-angle, then there is no $a$ in the remainder. Then the remainder has only $b^2$-angle $\alpha$.

We have $\theta\thick\rho\cdots=\thin\theta\thick\rho\thin\cdots$. If the remainder has no $ab$-angle, then there is no $b$ in the remainder. Then the remainder has only $a^2$-angles $\delta,\epsilon$.
\end{proof}

\begin{lemma}[Parity Lemma]\label{beven}
In a spherical tiling by congruent pentagons in Figure \ref{pentagon}, the total number of $ab$-angles $\beta$ and $\gamma$ at any vertex is even. 
\end{lemma}

The lemma is used so often that we choose to give it a name.

\begin{proof}
When we go around a vertex, the number of times the edges switch between $a$ and $b$ must be even. The $ab$-angles $\beta$ or $\gamma$ appear exactly when we have such a switch. 
\end{proof}

In the non-symmetric case, the two $ab$-angles $\beta,\gamma$ can be distinguished by their values (same for the two $a^2$-angles $\delta,\epsilon$). Then it makes sense to say that the total number of times the $ab$-angle $\beta$ appears in the tiling is $f$, and the similar statements for all the other angles. Such kind of counting argument and the parity lemma have the following consequence. Again the lemma is used so often, that we give it a name. 

\begin{lemma}[Balance Lemma]\label{balance}
In a spherical tiling by congruent non-symmetric pentagons in Figure \ref{pentagon}, if either $\beta^2\cdots$ or $\gamma^2\cdots$ is not a vertex, then any vertex either has no $\beta,\gamma$, or is of the form $\beta\gamma\cdots$ with no more $\beta,\gamma$ in the remainder. 
\end{lemma}

\begin{proof}
If $\beta^2\cdots$ is not a vertex, then any vertex is  of the form $\beta^k\gamma^l\cdots$, $k=0,1$, with no $\beta,\gamma$ in the remainder. If $k=0$, then $l\ge k$. If $k=1$, then by the parity lemma, we know $1+l$ is even, which implies $l\ge 1=k$. Therefore at each vertex, the number of times $\beta$ appears is always no bigger than the number of times $\gamma$ appears. On the other hand, the total number of times the two angles appear in the whole tiling is the same $f$. This implies $l=k$ at every vertex. 
\end{proof}

\begin{lemma}\label{klem6}
In a spherical tiling by congruent non-symmetric pentagons in Figure \ref{pentagon}, if $\theta\beta\gamma$ ($\theta=\alpha,\delta,\epsilon$) is a vertex, then $\theta^2\cdots$ is $\alpha^k$ or $\delta^k\epsilon^l$.
\end{lemma}

\begin{proof}
If $\theta^2\beta^2\cdots$ is a vertex, then the angle sum implies $\theta+\beta\le\pi$. By (the angle sum of) $\theta\beta\gamma$, this implies $\gamma\ge \pi$. Therefore $\gamma^2\cdots$ is not a vertex. This contradicts the vertex  $\theta^2\beta^2\cdots$ and the balance lemma. By the same argument, we also know $\theta^2\gamma^2\cdots$ is not a vertex. By $\theta\beta\gamma$, we know $\theta^2\beta\gamma\cdots$ is not a vertex. Then by the parity lemma, we conclude that $\theta^2\cdots$ has no $\beta,\gamma$. Then by Lemma \ref{klem4}, we know $\theta^2\cdots$ is $\alpha^k$ or $\delta^k\epsilon^l$.
\end{proof}

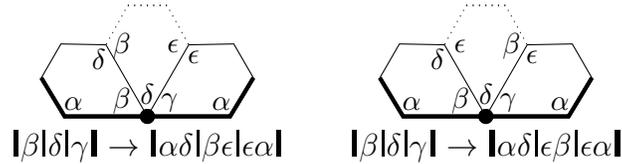
We introduced the very useful tool {\em adjacent angle deduction} (abbreviated as AAD) in \cite[Section 2.5]{wy1}. The following is \cite[Lemma 10]{wy1}.

\begin{lemma}\label{aadlemma}
Suppose $\lambda$ and $\mu$ are the two angles adjacent to $\theta$ in a pentagon.
\begin{itemize}
\item If $\lambda\thin\lambda\cdots$ is not a vertex, then $\theta^n$ has the unique AAD $\thin^{\lambda}\theta^{\mu}\thin^{\lambda}\theta^{\mu}\thin^{\lambda}\theta^{\mu}\thin\cdots$.
\item If $n$ is odd, then we have the AAD $\thin^{\lambda}\theta^{\mu}\thin^{\lambda}\theta^{\mu}\thin$ at $\theta^n$.
\end{itemize}
\end{lemma}

For the edge combination $a^3b^2$, we may apply the lemma to $\thick^{\beta}\alpha^{\gamma}\thick$, $\thin^{\beta}\delta^{\epsilon}\thin$, $\thin^{\gamma}\epsilon^{\delta}\thin$. The next result is specific to the edge combination $a^3b^2$, and the proof uses the reciprocity property of AAD. 

\begin{lemma}\label{klem2}
Consider a spherical tiling by congruent non-symmetric pentagons in Figure \ref{pentagon}.
\begin{enumerate}
\item If $\delta\thin\delta\cdots$ is not a vertex, then there is no consecutive $\beta\epsilon\cdots\epsilon\beta$ ($\cdots$ consists of only $\epsilon$) at a vertex. In case the sequence is $\beta\beta$, this means that $\beta\thin\beta\cdots$ is not a vertex.
\item If $\gamma\thin\delta\cdots,\delta\thin\delta\cdots$ are not vertices, then there is no consecutive $\beta\epsilon,\epsilon\epsilon\epsilon$ at a vertex.
\item If $\gamma\thin\delta\cdots,\delta\thin\delta\cdots$ are not vertices, and $\delta\epsilon\cdots=\delta\epsilon^2$, then there is no consecutive $\gamma\epsilon\gamma,\gamma\epsilon\epsilon$ at a vertex. Moreover, $\thick\beta\thin\cdots$ must be $\thick\beta\thin\delta\thin\cdots$.
\end{enumerate}
\end{lemma}

\begin{proof}
By $\thick^{\alpha}\beta^{\delta}\thin$ and $\thin^{\gamma}\epsilon^{\delta}\thin$, the AAD of consecutive $\beta\epsilon\cdots\epsilon\beta$ gives $\delta\thin\delta\cdots$. This proves the first part.

The AAD of $\thick\beta\thin\epsilon\thin$ is $\thick^{\alpha}\beta^{\delta}\thin^{\gamma}\epsilon^{\delta}\thin$ or $\thick^{\alpha}\beta^{\delta}\thin^{\delta}\epsilon^{\gamma}\thin$. We also have $\thin^{\gamma}\epsilon^{\delta}\thin\epsilon\thin=\thin^{\gamma}\epsilon^{\delta}\thin^{\gamma}\epsilon^{\delta}\thin$ or $\thin^{\gamma}\epsilon^{\delta}\thin^{\delta}\epsilon^{\gamma}\thin$. Therefore the assumption of the second part implies no $\beta\epsilon$ and no $\thin^{\gamma}\epsilon^{\delta}\thin\epsilon\thin$. Then no $\thin^{\gamma}\epsilon^{\delta}\thin\epsilon\thin$ further implies no $\epsilon\epsilon\epsilon=\thin\epsilon\thin^{\gamma}\epsilon^{\delta}\thin\epsilon\thin$.

For the third part, we still do not have  $\thin^{\gamma}\epsilon^{\delta}\thin\epsilon\thin$.

The AAD of $\thick\gamma\thin\epsilon\thin$ is $\thick^{\alpha}\gamma^{\epsilon}\thin^{\delta}\epsilon^{\gamma}\thin$ or $\thick^{\alpha}\gamma^{\epsilon}\thin^{\gamma}\epsilon^{\delta}\thin$. We know $\thick^{\alpha}\gamma^{\epsilon}\thin^{\delta}\epsilon^{\gamma}\thin$ gives $\thin^{\beta}\delta^{\epsilon}\thin^{\gamma}\epsilon^{\delta}\thin\cdots$. By $\delta\epsilon\cdots=\delta\epsilon^2$, we have $\thin^{\beta}\delta^{\epsilon}\thin^{\gamma}\epsilon^{\delta}\thin\cdots=\thin^{\beta}\delta^{\epsilon}\thin^{\gamma}\epsilon^{\delta}\thin\epsilon\thin$, contradicting no $\thin^{\gamma}\epsilon^{\delta}\thin\epsilon\thin$. This proves the unique AAD $\thick^{\alpha}\gamma^{\epsilon}\thin^{\gamma}\epsilon^{\delta}\thin$ of $\thick\gamma\thin\epsilon\thin$.

The unique AAD $\thick^{\alpha}\gamma^{\epsilon}\thin^{\gamma}\epsilon^{\delta}\thin$ implies $\thick\gamma\thin\epsilon\thin\gamma\thick=\thick^{\alpha}\gamma^{\epsilon}\thin^{\gamma}\epsilon^{\delta}\thin^{\epsilon}\gamma^{\alpha}\thick$, and $\thick\gamma\thin\epsilon\thin\epsilon\thin=\thick^{\alpha}\gamma^{\epsilon}\thin^{\gamma}\epsilon^{\delta}\thin\epsilon\thin$. The first AAD contains $\thin^{\gamma}\epsilon^{\delta}\thin^{\epsilon}\gamma^{\alpha}\thick$, contradicting the unique AAD $\thick^{\alpha}\gamma^{\epsilon}\thin^{\gamma}\epsilon^{\delta}\thin$. The second AAD contradicts no $\thin^{\gamma}\epsilon^{\delta}\thin\epsilon\thin$. This proves no consecutive $\gamma\epsilon\gamma,\gamma\epsilon\epsilon$.

Consider a vertex $\thick\beta\thin\theta\cdots$. By the first part, we have $\theta\ne\beta$. By the second part, we have $\theta\ne\epsilon$. If $\theta=\gamma$, then the AAD $\thick^{\alpha}\beta^{\delta}\thin^{\epsilon}\gamma^{\alpha}\thick$ gives $\thin^{\epsilon}\delta^{\beta}\thin^{\gamma}\epsilon^{\delta}\thin\cdots=\delta\epsilon^2=\thin^{\epsilon}\delta^{\beta}\thin^{\gamma}\epsilon^{\delta}\thin\epsilon\thin$, contradicting no $\thin^{\gamma}\epsilon^{\delta}\thin\epsilon\thin$. This proves $\theta=\delta$.
\end{proof}

\begin{lemma}\label{klem3}
In a spherical tiling by congruent non-symmetric pentagons in Figure \ref{pentagon}, if there is a vertex with unequal number of $\beta$ and $\gamma$, then one of $\beta^2\delta\cdots,\delta^2\cdots$ is a vertex, and one of $\gamma^2\epsilon\cdots,\epsilon^2\cdots$ is a vertex.
\end{lemma}

\begin{proof}
The lemma is equivalent to the following statement: If $\beta^2\delta\cdots$ and $\delta^2\cdots$ are not vertices, then the numbers of $\beta$ and $\gamma$ at any vertex are equal. We prove this statement.

At any vertex, the $b$-edges divide all the angles  into consecutive sequences $\beta\theta_1\cdots\theta_k\beta$, $\beta\theta_1\cdots\theta_k\gamma$, $\gamma\theta_1\cdots\theta_k\gamma$, $\alpha$, where $\theta_i$ are $a^2$-angles $\delta$ or $\epsilon$. See Figure \ref{vertex_bedge}. By no $\beta^2\delta\cdots$, the sequence $\beta\theta_1\cdots\theta_k\beta$ must be $\beta\epsilon\cdots\epsilon\beta$. By no $\delta^2\cdots$, the sequence $\beta\epsilon\cdots\epsilon\beta$ contradicts Lemma \ref{klem2}. Therefore there is no consecutive $\beta\theta_1\cdots\theta_k\beta$ at any vertex. This implies that the number of $\beta$ is no bigger than the number of $\gamma$ at any vertex. Since the total numbers of $\beta$ and $\gamma$ in the tiling are equal, this further implies that the numbers of $\beta$ and $\gamma$ at any vertex are equal. 
\end{proof}

\begin{figure}[htp]
\centering
\begin{tikzpicture}[>=latex,scale=1]


\foreach \a in {20,50,90,120,180,210}
\draw
	(0,0) -- (\a:1.4);

\foreach \b in {0,70,140,160,230}
\draw[line width=1.5]
	(0,0) -- (\b:1.4);

\fill (0,0) circle (0.1);

\foreach \x in {10,60,220}
\node at (\x:1) {\footnotesize $\beta$};

\foreach \y in {80,130,170}
\node at (\y:1) {\footnotesize $\gamma$};

\node at (150:1) {\footnotesize $\alpha$};

\foreach \z in {35,105,195}
\node at (\z:1) {\footnotesize $\theta_*^k$};

\draw[dotted]
	(240:1) -- (-10:1);

\end{tikzpicture}
\caption{Consecutive angles between $b$-edges at a vertex.}
\label{vertex_bedge}
\end{figure}

For the calculation in spherical trigonometry, we often use the cosine law for triangles. The following extends the cosine law to quadrilaterals. The lemma is proved in \cite[Lemma 3]{ay1}, and is also stated as \cite[Lemma 11]{wy1}.

\begin{lemma}\label{fourth}
For the spherical quadrilateral in Figure \ref{quad2}, we have
\begin{align*}
\cos x
&=\cos a\cos b\cos c
+\sin a\sin b\cos c\cos\theta
+\cos a\sin b\sin c\cos\rho \nonumber \\
&\quad 
+\sin a\sin c\sin\theta\sin\rho
-\sin a\cos b\sin c\cos \theta\cos\rho.
\end{align*}
\end{lemma}

\begin{figure}[htp]
\centering
\begin{tikzpicture}[scale=0.8]

\draw
	(-1,1.8) -- node[fill=white,inner sep=1] {\small $a$}
	(0,0) -- node[fill=white,inner sep=1] {\small $b$}
	(2,0) -- node[fill=white,inner sep=1] {\small $c$}
	(2.8,1.5);
\draw[dashed]
	(-1,1.8) -- node[fill=white,inner sep=1] {\small $x$} 
	(2.8,1.5);
	
\node at (0.15,0.2) {\small $\theta$};
\node at (1.9,0.2) {\small $\rho$};

\end{tikzpicture}
\caption{The length of fourth edge in a quadrilateral.}
\label{quad2}
\end{figure}

\section{Tiling with Vertices $\alpha^3,\beta\gamma\delta,\delta\epsilon^2$}
\label{special1tiling}

We adopt the notation in \cite{wy1} for the construction of tilings. We denote by $T_i$ the tile labeled $i$, by $E_{ij}$ the edge shared by $T_i,T_j$, and by $V_{ijk}$ the vertex shared by $T_i,T_j,T_k$. We denote by $\theta_i$ the angle $\theta$ in $T_i$. When we say by $\alpha\beta\gamma$, we mean by the angle sum $\alpha+\beta+\gamma=2\pi$ of the vertex. When we say by $\theta_i$, we mean by the angles adjacency to $\theta_i$. When we say a tile is {\em determined}, we mean that we know all the edges and angles of the tile.

We will use Lemma/Proposition $n'$ to denote the use of Lemma/Proposition $n$ after exchanging $\beta\leftrightarrow \gamma$ and $\delta\leftrightarrow \epsilon$.

\begin{proposition}\label{special1}
Tilings of the sphere by congruent non-symmetric pentagons in Figure \ref{pentagon}, such that $\alpha^3,\beta\gamma\delta$ are vertices and $\delta+2\epsilon=2\pi$, are the following:
\begin{enumerate}
\item $f=24$: $T(6\epsilon^4)$, $T(4\beta\gamma\epsilon^2,2\epsilon^4)$, $T(4\beta^2\gamma^2,2\epsilon^4)$. The pentagon for the three tilings is the same unique one.
\item $f=60$: $T(12\epsilon^5)$, $T(5\beta\gamma\epsilon^3,7\epsilon^5)$, $T(10\beta\gamma\epsilon^3,2\epsilon^5)$, $T(2\beta^2\gamma^2\epsilon, 6\beta\gamma\epsilon^3,4\epsilon^5)$, 
$T(6\beta^2\gamma^2\epsilon, 3\beta\gamma\epsilon^3,3\epsilon^5)$. The pentagon for the five tilings is the same unique one.
\end{enumerate}
\end{proposition}

The condition $\delta+2\epsilon=2\pi$ is satisfied if $\delta\epsilon^2$ is a vertex (and hence the title of the section). In fact, in the proof of the proposition, whenever we say ``the angle sum of $\delta\epsilon^2$'', we really mean the equality $\delta+2\epsilon=2\pi$, and never imply that $\delta\epsilon^2$ appears as a vertex. The catch here is that the degree $3$ vertices in $T(6\epsilon^4)$ and $T(12\epsilon^5)$ are exactly $\alpha^3,\beta\gamma\delta$, and $\delta\epsilon^2$ appears only in the other six tilings. 

We remark that $\gamma^2\epsilon$ is not a vertex in the tilings in the proposition.

The notation $T(4\beta\gamma\epsilon^2,2\epsilon^4)$ means the tiling has $4$ vertex $\beta\gamma\epsilon^2$ and $2$ vertex $\epsilon^4$, and these are all the vertices of degree $>3$. The tilings $T(6\epsilon^4)$ and $T(12\epsilon^5)$ are the pentagonal subdivision tilings (and satisfy the additional equality $\delta+2\epsilon=2\pi$), and are the third and fifth of Figure \ref{subdivision_tiling}. The other tilings are the two tilings in Figure \ref{subdivision_tiling_modify24}, and the first four tilings in Figure \ref{subdivision_tiling_modify60A}. The third and fourth of Figure \ref{subdivision_tiling_modify60A} have alternative drawings as the third and fourth of Figure \ref{subdivision_tiling_modify60B}.

The proof is divided into three steps.

\subsubsection*{Step 1: Determine Angles and AVC}

The angle sums of $\alpha^3,\beta\gamma\delta,\delta\epsilon^2$ (actually $\delta+2\epsilon=2\pi$) and the angle sum for pentagon imply
\[
\alpha=\tfrac{2}{3}\pi,\;
\beta+\gamma=(\tfrac{2}{3}+\tfrac{8}{f})\pi,\;
\delta=(\tfrac{4}{3}-\tfrac{8}{f})\pi,\;
\epsilon=(\tfrac{1}{3}+\tfrac{4}{f})\pi.
\]
By $f>12$ , we have $\delta>\alpha>\epsilon$. By Lemma \ref{geometry1}, this implies $\beta<\gamma$. By $\beta+\gamma=2\epsilon$, we get $\beta<\epsilon<\gamma$. 

By $\beta\gamma\delta$ and Lemma \ref{klem6}, we know $\delta^2\cdots=\delta^k\epsilon^l$. By $\delta\epsilon^2$ and $\delta>\epsilon$, this implies $\delta^2\cdots$ is not a vertex. Then by the first part of Lemma \ref{klem2}, this implies no consecutive $\beta\epsilon\cdots\epsilon\beta$.

By $\beta<\gamma$, the parity lemma, and $R(\alpha^2\cdots)=\alpha<\beta+\gamma,\delta$, we get $\alpha^2\cdots=\alpha^3,\alpha^2\beta^k\epsilon^l$. By no consecutive $\beta\epsilon\cdots\epsilon\beta$, we know $\alpha^2\beta^k\epsilon^l$ is not a vertex. Therefore $\alpha^2\cdots=\alpha^3$. 

If $\delta\cdots$ has no $\beta,\gamma$, then by $\delta\epsilon^2$, no $\delta^2\cdots$, and Lemma \ref{klem4}, we get $\delta\cdots=\delta\epsilon^2$. If $\delta\cdots$ has $\beta,\gamma$, then by $\beta\gamma\delta$, $\beta<\gamma$, no $\delta^2\cdots$, and the parity lemma, we get $\delta\cdots=\beta\gamma\delta,\beta^k\delta\cdots$, with $k\ge 2$ and even, and only $\alpha,\epsilon$ in $R(\beta^k\delta\cdots)$. Then by $\delta\epsilon^2$ and $\alpha>\epsilon$, we get $\beta^2\delta\cdots=\alpha\beta^k\delta,\beta^k\delta,\beta^k\delta\epsilon$. If $k\ge 4$, then we have consecutive $\beta\epsilon\cdots\epsilon\beta$ in $\beta^2\delta\cdots$, a contradiction. Therefore $\beta^2\delta\cdots=\alpha\beta^2\delta,\beta^2\delta,\beta^2\delta\epsilon$. By $\beta\gamma\delta$ and $\beta<\gamma$, we conclude $\delta\cdots=\beta\gamma\delta,\delta\epsilon^2,\alpha\beta^2\delta,\beta^2\delta\epsilon$.

\subsubsection*{Case. $\alpha\beta^2\delta$ is a vertex}

The angle sum of $\alpha\beta^2\delta$ further implies
\[
\alpha=\tfrac{2}{3}\pi,\;
\beta=\tfrac{4}{f}\pi,\;
\gamma=(\tfrac{2}{3}+\tfrac{4}{f})\pi,\;
\delta=(\tfrac{4}{3}-\tfrac{8}{f})\pi,\;
\epsilon=(\tfrac{1}{3}+\tfrac{4}{f})\pi.
\] 
We have the AAD $\thin^{\delta}\beta^{\alpha}\thick^{\beta}\alpha^{\gamma}\thick^{\alpha}\beta^{\delta}\thin$ at $\alpha\beta^2\delta$. This gives a vertex $\alpha\gamma\cdots$. By $R(\alpha\gamma\cdots)=(\tfrac{2}{3}-\tfrac{4}{f})\pi<\alpha,\gamma,\delta,2\epsilon$, we get $\alpha\gamma\cdots=\alpha\beta^k\gamma,\alpha\beta^k\gamma\epsilon$. By $\alpha+\beta+\gamma<2\pi$, and no consecutive $\beta\epsilon\cdots\epsilon\beta$, we get $\alpha\gamma\cdots=\alpha\beta\gamma\epsilon$. 

In Figure \ref{special1A}, we consider four tiles around $\alpha\beta\gamma\epsilon$. We may first determine $T_2,T_4$. By $\epsilon_1$ and no $\delta^2\cdots$, we get $\delta_2\cdots=\gamma_1\delta_2\cdots=\beta_5\gamma_1\delta_2$. Then $\gamma_1,\epsilon_1$ determine $T_1$, and $\beta_5$ determines $T_5$. 

On the other hand, by $\alpha_3$ and $\alpha^2\cdots=\alpha^3$, we know $R(\alpha_2\cdots)$ has no $\alpha$. This implies $\alpha_6$, and we get $\alpha_6\gamma_2\cdots=\alpha\beta\gamma\epsilon$. Then by $\delta_5\epsilon_2\cdots=\delta\epsilon^2,\beta^2\delta\epsilon$, we find that either two $\epsilon$ adjacent, or $\beta,\epsilon$ adjacent, both contradictions.

\begin{figure}[htp]
\centering
\begin{tikzpicture}[>=latex,scale=1]

\foreach \a in {0,1,2,3}
\draw[rotate=90*\a]
	(0,0) -- (0,0.8) -- (0.6,1.2) -- (1.2,0.6) -- (0.8,0);

\draw
	(0.6,1.2) -- (0.6,1.8) -- (-0.6,1.8) -- (-0.6,1.2)
	(-1.8,1.8) -- (-1.2,0.6) -- (-1.8,0.6) -- (-1.8,-0.4) -- (-1.5,-0.4) -- (-0.8,0);
	
\draw[line width=1.5]
	(0,0.8) -- (0.6,1.2) -- (1.2,0.6)
	(0.6,1.2) -- (0.6,1.8)
	(-1.8,0.6) -- (-1.2,0.6) -- (-0.8,0) -- (0,0) -- (0,-0.8) -- (0.6,-1.2);

\node at (-0.7,0.2) {\small $\alpha$}; 
\node at (-0.2,0.2) {\small $\beta$};
\node at (-1,0.55) {\small $\gamma$};
\node at (-0.2,0.7) {\small $\delta$};
\node at (-0.6,1) {\small $\epsilon$};

\node at (-1.3,0.4) {\small $\alpha$};
\node at (-1.5,0.8) {\small $\beta$};
\node at (-1.15,0.9) {\small $\epsilon$};

\node at (-0.2,-0.2) {\small $\alpha$}; 

\node at (0.2,-0.7) {\small $\alpha$}; 
\node at (0.6,-0.95) {\small $\beta$};
\node at (0.2,-0.2) {\small $\gamma$};
\node at (1,-0.6) {\small $\delta$};
\node at (0.7,-0.2) {\small $\epsilon$}; 

\node at (0.6,0.95) {\small $\alpha$}; 
\node at (0.95,0.55) {\small $\beta$}; 
\node at (0.2,0.7) {\small $\gamma$};
\node at (0.7,0.2) {\small $\delta$};
\node at (0.2,0.2) {\small $\epsilon$};

\node at (0.4,1.3) {\small $\alpha$}; 
\node at (0,1) {\small $\beta$};  
\node at (0.4,1.6) {\small $\gamma$};
\node at (-0.4,1.3) {\small $\delta$};
\node at (-0.4,1.65) {\small $\epsilon$};


\node[inner sep=1,draw,shape=circle] at (0.55,0.55) {\small $1$};
\node[inner sep=1,draw,shape=circle] at (-0.55,0.55) {\small $2$};
\node[inner sep=1,draw,shape=circle] at (-0.55,-0.55) {\small $3$};
\node[inner sep=1,draw,shape=circle] at (0.55,-0.55) {\small $4$};
\node[inner sep=1,draw,shape=circle] at (0,1.45) {\small $5$};
\node[inner sep=1,draw,shape=circle] at (-1.45,0) {\small $6$};

\end{tikzpicture}
\caption{$\{\alpha^3,\beta\gamma\delta,\delta\epsilon^2\}$:  $\alpha\beta\gamma\epsilon$, $f=36$.}
\label{special1A}
\end{figure}

\subsubsection*{Case. $\beta^2\delta\epsilon$ is a vertex}

The angle sum of $\beta^2\delta\epsilon$ further implies
\[
\alpha=\tfrac{2}{3}\pi,\;
\beta=(\tfrac{1}{6}+\tfrac{2}{f})\pi,\;
\gamma=(\tfrac{1}{2}+\tfrac{6}{f}),\;
\delta=(\tfrac{4}{3}-\tfrac{8}{f})\pi,\;
\epsilon=(\tfrac{1}{3}+\tfrac{4}{f})\pi.
\]
We have the AAD $\thin^{\delta}\beta^{\alpha}\thick^{\alpha}\beta^{\delta}\thin$ at $\beta^2\delta\epsilon$. This gives a vertex $\thick^{\gamma}\alpha^{\beta}\thick^{\beta}\alpha^{\gamma}\thick\cdots=\alpha^3$. By Lemma \ref{aadlemma}, this implies $\gamma\thick\gamma\cdots$ is a vertex. By the parity lemma and $R(\alpha\gamma^2\cdots)=(\frac{1}{3}-\frac{12}{f})\pi<\alpha,2\beta,\gamma,\delta,\epsilon$, we get $\alpha\gamma^2\cdots=\alpha\gamma^2=\thin\gamma\thick\alpha\thick\gamma\thin$. This implies $\gamma\thick\gamma\cdots$ has no $\alpha$. Then by $R(\gamma^2\cdots)=(1-\frac{12}{f})\pi<\delta$, we know $R(\gamma\thick\gamma\cdots)$ has only $\beta,\gamma,\epsilon$. By the angle values, no $\beta\epsilon\cdots\epsilon\beta$, and the parity lemma, we get $\gamma\thick\gamma\cdots=\beta^2\gamma^2,\gamma^2\epsilon, \beta\gamma^3,\beta^2\gamma^2\epsilon,\gamma^2\epsilon^2$. Moreover, any of the first two vertices implies $f=24$, and any of the last three implies $f=60$. Then we get the corresponding angles 
\begin{align}
f=24 &\colon 
\alpha=\tfrac{2}{3}\pi,\;
\beta=\tfrac{1}{4}\pi,\;
\gamma=\tfrac{3}{4}\pi,\;
\delta=\pi,\;
\epsilon=\tfrac{1}{2}\pi; \label{special1_24B} \\
f=60 &\colon 
\alpha=\tfrac{2}{3}\pi,\;
\beta=\tfrac{1}{5}\pi,\;
\gamma=\tfrac{3}{5}\pi,\;
\delta=\tfrac{6}{5}\pi,\;
\epsilon=\tfrac{2}{5}\pi. \label{special1_60B}
\end{align}

\subsubsection*{Case. $\delta\cdots=\beta\gamma\delta$ or $\delta\epsilon^2$}

The assumption implies that $\beta^2\delta\cdots,\delta^2\cdots$ are not vertices. By Lemma \ref{klem3}, besides $\delta\cdots=\beta\gamma\delta,\delta\epsilon^2$, the only other vertices are $\alpha^a\beta^b\gamma^b\epsilon^e$. The angle sum of $\alpha^a\beta^b\gamma^b\epsilon^e$ gives
\[
\tfrac{2}{3}a+(\tfrac{1}{3}+\tfrac{4}{f})(2b+e)=2.
\]
Besides $\alpha^3$, we look for non-negative integer solutions satisfying $a<3$, and $a>0\implies b>0$ (by the edge length consideration). The solutions we get are $\beta^2\gamma^2,\beta\gamma\epsilon^2,\epsilon^4$ for $f=24$, $\alpha\beta\gamma\epsilon$ for $f=36$, and $\beta^2\gamma^2\epsilon,\beta\gamma\epsilon^3,\epsilon^5$ for $f=60$. Moreover, we get the corresponding angles and AVCs
\begin{align}
f=24 &\colon 
\alpha=\tfrac{2}{3}\pi,\;
\beta+\gamma=\pi,\;
\delta=\pi,\;
\epsilon=\tfrac{1}{2}\pi, \nonumber \\
&\text{AVC}
=\{\alpha^3,\beta\gamma\delta,\delta\epsilon^2,\beta^2\gamma^2,\beta\gamma\epsilon^2,\epsilon^4\}; \label{special1_24A} \\
f=36 &\colon 
\alpha=\tfrac{2}{3}\pi,\;
\beta+\gamma=\tfrac{8}{9}\pi,\;
\delta=\tfrac{10}{9}\pi,\;
\epsilon=\tfrac{4}{9}\pi, \nonumber  \\
&\text{AVC}
=\{\alpha^3,\beta\gamma\delta,\delta\epsilon^2,\alpha\beta\gamma\epsilon\}; \label{special1_36} \\
f=60 &\colon 
\alpha=\tfrac{2}{3}\pi,\;
\beta+\gamma=\tfrac{4}{5}\pi,\;
\delta=\tfrac{6}{5}\pi,\;
\epsilon=\tfrac{2}{5}\pi, \nonumber \\
&\text{AVC}
=\{\alpha^3,\beta\gamma\delta,\delta\epsilon^2,\beta^2\gamma^2\epsilon,\beta\gamma\epsilon^3,\epsilon^5\}. \label{special1_60A}
\end{align}

The case \eqref{special1_36} has the property $\alpha\gamma\cdots=\alpha\beta\gamma\epsilon$. By the same argument (actually simpler argument due to simpler AVC) for the case $\alpha\beta^2\delta$ is a vertex, and using Figure \ref{special1A}, we get the same contradiction.

\subsubsection*{Step 2: Calculate the Pentagon}

The pentagonal subdivision tiling has the edge combination $a^2b^2c$. The right of Figure \ref{reduction} shows that the reduction $c=b$ gives a tiling of the sphere by congruent pentagons with the edge combination $a^3b^2$. The angles of the reduced tiling satisfy
\begin{equation}\label{rightdivision}
\alpha=\tfrac{2}{3}\pi,\;
\beta+\gamma+\delta=2\pi,\;
\epsilon=\tfrac{2}{n}\pi.
\end{equation}
We note that \eqref{special1_24A} (with $n=4$) and \eqref{special1_60A} (with $n=5$) are special cases of \eqref{rightdivision}. Moreover, \eqref{special1_24B} and \eqref{special1_60B} are further specialisations of \eqref{special1_24A} and \eqref{special1_60A}. Therefore the pentagons we are interested in can tile the pentagonal subdivision tilings. We calculate the pentagons by using this fact.

The general pentagonal subdivision tiling allows two free parameters. The reduction $c=b$ and the equality $\delta=\pi$ in \eqref{special1_24A} and \eqref{special1_60A} are the two additional conditions that should uniquely determine the pentagon. It would be coincidental for the more restrictive conditions \eqref{special1_24B} and \eqref{special1_60B} to be satisfied.

Now we calculate for \eqref{special1_24A}. A triangular face of the regular octahedron has angle $\epsilon=\frac{1}{2}\pi$. The pentagonal subdivision divides each face into three pentagons, as in the first of Figure \ref{calculation_division}. The condition $\delta=\pi$ means that $B,D$ lie on an edge of the triangle, and actually divides the edge into three equal parts. This implies $l=\frac{1}{2}\pi=3a$ and $AD=AB=b$. Then by
\begin{align*}
\cos BC
&=\cos a\cos 2a+\sin a\sin 2a\cos\epsilon
=\tfrac{\sqrt{3}}{4}, \\
\cos BC
&=\cos^2b+\sin^2b\cos\alpha
=\cos^2b-\tfrac{1}{2}\sin^2b,
\end{align*}
we get ($b=0.21076\pi$ means $0.21076\pi\le b<0.21077\pi$)
\[
\cos b=\frac{3+\sqrt{3}}{6},\quad
\sin b=\frac{\sqrt{4-\sqrt{3}}}{\sqrt{6}},\quad
b=0.21076\pi.
\]
By $AB=AD=b$, we get
\[
\cos b=\cos a\cos b+\sin a\sin b\cos\beta.
\]
Substituting the sine and cosine of $a$ and $b$, we get
\[
\cos\beta=-\cos\gamma=\frac{1}{\sqrt{5+2\sqrt{3}}},\quad
\beta=0.38831\pi.
\]
This implies $\beta\ne\frac{1}{4}\pi$, and the case \eqref{special1_24B} is impossible. Therefore we only need to consider \eqref{special1_24A}.

\begin{figure}[htp]
\centering
\begin{tikzpicture}[>=latex,scale=1]


\foreach \a in {0,1,2}
\draw[rotate=120*\a, dotted]
	(-30:2) -- (90:2);	

\draw
	(-0.577,-1) -- (-30:2) -- (1.154,0);

\draw[line width=1.5]
	(1.154,0) -- (0,0) -- (-0.577,-1);

\draw[dashed]
	(-0.577,-1) -- (1.154,0);

\draw[dotted]
	(0,0) -- (0.577,-1);
		
\fill
	(0.577,-1) circle (0.05);

\node at (0.1,-0.15) {\small $\alpha$};	
\node at (-0.3,-0.8) {\small $\beta$};
\node at (1.05,-0.2) {\small $\gamma$};
\node at (0.577,-0.75) {\small $\delta$};
\node at (-30:1.7) {\small $\epsilon$};

\node at (-0.15,0.2) {\small $A$};
\node at (-0.6,-1.2) {\small $B$};
\node at (1.3,0.17) {\small $C$};
\node at (0.6,-1.2) {\small $D$};
\node at (-30:2.2) {\small $E$};

\node at (1.55,-0.4) {\small $a$};
\node at (0.55,0.2) {\small $b$};
\node at (150:1.2) {\small $l$};

\node at (90:1.45) {\small $\tfrac{1}{2}\pi$};


\begin{scope}[xshift=4.5cm]

\foreach \a in {0,1,2}
\draw[rotate=120*\a,dotted]
	(-30:2) -- (90:2);	
	
\draw
	(-30:2) -- ++(168:1.2) -- ++(204:1.2)
	(-30:2) -- ++(108:1.2)
	(210:2) -- (246:1.37);

\draw[line width=1.5]
	(6:1.37) -- (0,0) -- (246:1.37);
	

	
\node at (0.15,-0.15) {\small $\alpha$};	
\node at (-0.25,-0.92) {\small $\beta$};
\node at (1.23,-0.1) {\small $\gamma$};
\node at (0.57,-0.57) {\small $\delta$};
\node at (-27:1.7) {\small $\epsilon$};

\node at (-0.15,0.2) {\small $A$};
\node at (-0.6,-1.43) {\small $B$};
\node at (1.5,0.3) {\small $C$};
\node at (0.55,-0.95) {\small $D$};
\node at (-30:2.2) {\small $E$};

\node at (1.7,-0.4) {\small $a$};
\node at (0.6,0.27) {\small $b$};
\node at (150:1.2) {\small $l$};

\node at (90:1.45) {\small $\tfrac{2}{5}\pi$};

\end{scope}


\begin{scope}[xshift=9cm]

\foreach \a in {0,1,2}
{
\begin{scope}[rotate=120*\a]

\draw[dotted]
	(-30:2) -- (90:2);	
	
\draw
	(-30:2) -- ++(168:1.2) -- ++(204:1.2)
	(210:2) -- (246:1.37);

\draw[line width=1.5]
	(0,0) -- (246:1.37);
	
\draw[dashed]
	(0,0) -- (-53.5:0.93);



\end{scope}
}

\node[fill=white, inner sep=1.5] at (-0.15,-0.25) {\small $\alpha'$};
\node at (0.25,-0.7) {\small $\beta'$};
\node at (-0.9,-0.35) {\small $\gamma'$};
\node[fill=white, inner sep=1] at (-0.47,-0.95) {\small $\delta'$};
\node at (-1.4,-0.8) {\small $\epsilon'$};

\node[fill=white, inner sep=1] at (-0.45,-0.05) {\small $b'$};

\end{scope}

\end{tikzpicture}
\caption{$\{\alpha^3,\beta\gamma\delta,\delta\epsilon^2\}$: Calculation of the pentagon.}
\label{calculation_division}
\end{figure}

For $f=60$, we carry out the similar calculation. The second of Figure \ref{calculation_division} describes the pentagonal subdivision of a regular face of the regular icosahedron. The angle of the face is $\epsilon=\frac{2}{5}\pi$, and the side length of the face is given by $\cos l=\frac{1}{\sqrt{5}}$. 

In fact, we also calculate the companion pentagon (see \cite[Section 3.1]{wy1}) in the third of Figure \ref{calculation_division}, bounded by two dashed edges $b'$ and three normal edges $a$, and with angles $\alpha'=\alpha=\frac{2}{3}\pi,\; \beta',\; \gamma',\; \delta'=2\pi-\delta,\; \epsilon'=\epsilon=\frac{2}{5}\pi$. We calculate the companion because it will appear in the next Proposition \ref{special2}. 

We do not calculate the companion for $f=24$ because $\delta=\pi$ implies the pentagon is its own companion.

The two ends of the side length $l$ can be reached by three segments of length $a$ at alternating angles $\delta,\delta$ (at $B,D$). By Lemma \ref{fourth}, we get
\[
\cos l
=\cos^3a(1-\cos\delta)^2
+\cos^2a\sin^2\delta
+\cos a(2\cos\delta-\cos^2\delta)
-\sin^2\delta.
\]
Substituting $\cos l=\frac{1}{\sqrt{5}}$ and $\delta=\tfrac{6}{5}\pi$, we get a cubic equation for $\cos a$. By \cite{wy3}, for the pentagon to be simple, we need  $a,b<\frac{1}{2}\pi$. Among the three real roots of the cubic equation above, the only positive solution is 
\[
\cos a
=\tfrac{1}{2\sqrt{5}}
(3-\sqrt{5}+\sqrt{2}{\textstyle \sqrt{3\sqrt{5}-1}}),\quad
a=0.12293\pi.
\]

To find the length $b$, we use Lemma \ref{fourth} and the cosine law to get
\begin{align*}
\cos BC
&=\cos^3a(1-\cos\delta)(1-\cos\epsilon)
-\cos^2a\sin\delta\sin\epsilon \\
&\quad +\cos a(\cos\delta+\cos\epsilon-\cos\delta\cos\epsilon)
+\sin\delta\sin\epsilon \\
&=\tfrac{1}{\sqrt{10}}(\sqrt{2}(\sqrt{5}-2)+(3-\sqrt{5}){\textstyle \sqrt{3\sqrt{5}-1}}), \\
\cos BC
&=\cos^2b+\sin^2b\cos\alpha
=\cos^2b-\tfrac{1}{2}\sin^2b.
\end{align*}
By $b<\frac{1}{2}\pi$, we get
\[
\cos b
=\sqrt{\tfrac{1}{3\sqrt{5}}(3\sqrt{5}-4+\sqrt{2}(3-\sqrt{5}){\textstyle \sqrt{3\sqrt{5}-1}})}, \quad
b= 0.15211\pi.
\]
By the similar method, and $\alpha'=\tfrac{2}{3}\pi,\delta'=\tfrac{4}{5}\pi,\epsilon'=\tfrac{2}{5}\pi$, we get
\[
\cos b'=\sqrt{\tfrac{1}{3\sqrt{10}}(\sqrt{2}(3\sqrt{5}-2)+(3-\sqrt{5}){\textstyle \sqrt{3\sqrt{5}-1}})},\quad
b'=0.10543\pi.
\]

Figure \ref{calculate_pentagon} is the scheme relating the pentagon $\pentagon ABDCE$ and its companion $\pentagon A'B'D'C'E'$. They are different unions of the same three triangles. The isosceles triangle $\triangle CDE$ gives 
\begin{align*}
\cos d
&=\cos^2a+\sin^2a\cos\epsilon \\
&=\tfrac{1}{2\sqrt{5}}(1+\sqrt{5}
+\sqrt{2}(\sqrt{5}-2){\textstyle \sqrt{3\sqrt{5}-1}}),\quad
d=0.14212\pi.
\end{align*}

\begin{figure}[htp]
\centering
\begin{tikzpicture}[>=latex]

\draw
	(-2.4,2) -- (-3.6,0) -- (1.2,0) -- (3.6,0) -- (2.4,2);

\draw[line width=1.5]
	(-1.2,0) -- (0,2) -- (2.4,2);

\draw[dashed]
	(1.2,0) -- (0,2) -- (-2.4,2);

\draw[dotted]
	(2.4,2) -- (1.2,0)
	(-2.4,2) -- (-1.2,0);

\node at (0,2.25) {\small $A=A'$};
\node at (-1.2,-0.2) {\small $B=D'$};
\node at (1.2,-0.2) {\small $B'=D$};
\node at (3.8,0) {\small $E$};
\node at (-3.8,0) {\small $E'$};
\node at (2.6,2.1) {\small $C$};
\node at (-2.6,2.1) {\small $C'$};

\node[fill=white, inner sep=1.5] at (0,0) {\small $a$};
\node[fill=white, inner sep=1.5] at (2.4,0) {\small $a$};
\node[fill=white, inner sep=1.5] at (-2.4,0) {\small $a$};
\node[fill=white, inner sep=1.5] at (3,1) {\small $a$};
\node[fill=white, inner sep=1.5] at (-3,1) {\small $a$};
\node[fill=white, inner sep=1.5] at (-0.6,1) {\small $b$};
\node[fill=white, inner sep=1] at (0.6,1.05) {\small $b'$};
\node[fill=white, inner sep=1.5] at (1.2,2) {\small $b$};
\node[fill=white, inner sep=1.5] at (-1.2,2) {\small $b'$};
\node[fill=white, inner sep=1.5] at (1.8,1) {\small $d$};
\node[fill=white, inner sep=1.5] at (-1.8,1) {\small $d$};

\node at (0,1.55) {\small $\alpha_1$};
\node at (-0.45,1.75) {\small $\alpha_2$};
\node at (0.45,1.75) {\small $\alpha_2$};
\node at (2,1.75) {\small $\rho$};
\node at (-1.2,0.4) {\small $\rho$};
\node at (-2,1.75) {\small $\rho'$};
\node at (1.25,0.4) {\small $\rho'$};
\node at (2.4,1.65) {\small $\theta$};
\node at (1.5,0.2) {\small $\theta$};
\node at (-1.5,0.2) {\small $\theta$};
\node at (2.4,1.65) {\small $\theta$};
\node at (-2.4,1.65) {\small $\theta$};
\node at (-0.85,0.2) {\small $\beta$};
\node at (0.8,0.2) {\small $\beta'$};
\node at (-3.3,0.2) {\small $\epsilon$};
\node at (3.3,0.2) {\small $\epsilon$};

\end{tikzpicture}
\caption{$\{\alpha^3,\beta\gamma\delta,\delta\epsilon^2\}$ and $\{\alpha^3,\beta\gamma\delta,\delta^2\epsilon\}$: Companion pentagons.}
\label{calculate_pentagon}
\end{figure}

We note that $A$ is the center of a triangular face, and $E$ is a vertex of the face. We know the distance satisfies $\cos AE=\frac{\sqrt{5}+1}{\sqrt{6(5-\sqrt{5})}}$. Using this and the exact values of all the edge lengths in Figure \ref{calculate_pentagon}, we may calculate the exact values of all the angles. For example, we may use
\begin{align*}
\cos b'
&=\cos a\cos b+\sin a\sin b\cos\beta, \\
\cos AE
&=\cos a\cos b
+\sin a\sin b \cos\gamma,
\end{align*}
to get

\begin{align*}
\cos\beta
&=\tfrac{1}{32}(9+3\sqrt{5}+\sqrt{2}{\textstyle \sqrt{3\sqrt{5}-1}})\sqrt{4+(\sqrt{5}-3)\sqrt{2}{\textstyle \sqrt{3\sqrt{5}-1}}}, \\
\cos\gamma
&=\tfrac{1}{16}(\sqrt{5}-1+\sqrt{2}(2\sqrt{5}-5){\textstyle \sqrt{3\sqrt{5}-1}}){\textstyle \sqrt{11+5\sqrt{5}}}.\end{align*}
Similarly, we get
\begin{align*}
\cos\beta'
&=\tfrac{1}{304}(61+27\sqrt{5})(-39+19\sqrt{5}-\sqrt{2}{\textstyle \sqrt{3\sqrt{5}-1}}) \\
&\qquad \sqrt{-4+3\sqrt{5}+\sqrt{2}(3-\sqrt{5}){\textstyle \sqrt{3\sqrt{5}-1}}},  \\
\cos\gamma'
&=\tfrac{1}{4}(2-\sqrt{5}){\textstyle \sqrt{3\sqrt{5}-1}}{\textstyle \sqrt{11+5\sqrt{5}}}.
\end{align*}
The approximate values are
\[
\beta =0.24847\pi, \;
\gamma =0.55152\pi, \;
\beta' =0.46882\pi, \;
\gamma' =0.73117\pi.
\]
In particular, we have $\beta\ne \frac{1}{5}\pi$. This means the case \eqref{special1_60B} is impossible. Therefore we only need to consider \eqref{special1_60A}. 

The pentagons $\pentagon ABDCE$ and $\pentagon A'B'D'C'E'$ exist by the pentagonal subdivision construction. The problem is to show that the pentagons are simple, so that they are suitable for tiling. This can be argued by the general way in \cite{wy3}. For example, according to the notation in \cite{wy3}, we have $D$ in the region $\Omega_1$, which means $C$ (which is $V$ for $\pentagon ABDCE$ in \cite{wy3}) is in the region $\Omega_7$. This implies $\pentagon ABDCE$ is simple. Similarly, the vertex $V$ for $\pentagon A'B'D'C'E'$ is $C'$. We see $C'$ lies in the region $\Omega_1$, which also implies $\pentagon A'B'D'C'E'$ is simple.

Alternatively, we may use the method in \cite[Section 3.2]{wy1} to give another explanation that the pentagons are simple. In Figure \ref{calculate_pentagon}, we may use the known values of all the edges to get approximate values
\[
\alpha_1=0.3033\pi,\;
\alpha_2=0.3633\pi,\;
\theta=0.3114\pi,\;
\rho=0.2400\pi,\;
\rho'=0.4197\pi.
\]
This implies $\delta'=\beta+\theta+\rho<0.81\pi<\pi$. Then by the known values $\alpha',\beta',\gamma',\epsilon'<\pi$, we know the companion pentagon $\pentagon A'B'D'C'E'$ is convex. Therefore the pentagon is simple. 

We also have $\beta'+\rho'<0.89\pi<\pi$. Then by $\alpha,\beta,\rho<\pi$, we know $\square ABDC$ is convex. Then $\square ABDC$ and $\triangle CDE$ are convex and lie in the separate hemispheres divided by the great circle $\bigcirc CD$. This implies the pentagon $\pentagon ABDCE$ is simple.

\subsubsection*{Step 3: Construct the Tiling}

We construct tilings for \eqref{special1_24A} and  \eqref{special1_60A}. By the first step of the proof, we do not have consecutive $\beta\epsilon\cdots\epsilon\beta$. By the AVC (we will henceforth omit mentioning AVC), this implies $\gamma\thin\gamma\cdots$ is not a vertex.

Figure \ref{special1B} describes the neighbourhood of a $3^5$-tile, with $T_1$ arranged as in Figure \ref{pentagon}. The five degree $3$ vertices in the neighbourhood can only be $\alpha^3,\beta\gamma\delta,\delta\epsilon^2$. We have $\alpha_1\cdots=\alpha_1\alpha_2\alpha_6$, $\beta_1\cdots=\beta_1\gamma_6\delta_5$, $\gamma_1\cdots=\beta_2\gamma_1\delta_3$. Then $\alpha_2,\beta_2$ determine $T_2$, and $\alpha_6,\gamma_6$ determine $T_6$. By $\delta_3$, we get $\epsilon_1\cdots=\delta_4\epsilon_1\epsilon_3$. Then $\delta_3,\epsilon_3$ determine $T_3$. By $\delta_4,\delta_5$, we know $\delta_1\cdots\ne\beta\gamma\delta$. Therefore $\delta_1\cdots=\delta_1\epsilon_4\epsilon_5$. Then $\delta_4,\epsilon_4$ determine $T_4$, and $\delta_5,\epsilon_5$ determine $T_5$. Then we have $\gamma_4\thin\gamma_5\cdots$, a contradiction.

\begin{figure}[htp]
\centering
\begin{tikzpicture}[>=latex,scale=1]

\foreach \x in {1,...,5}
\draw[rotate=72*\x]
	(90:0.7) -- (18:0.7) -- (18:1.3) -- (54:1.7) -- (90:1.3);

\draw[line width=1.5]
	(162:1.3) -- (198:1.7) -- (234:1.3) -- (-90:1.7) -- (-54:1.3) -- (-18:1.7) -- (18:1.3)
	(18:0.7) -- (90:0.7) -- (162:0.7)
	(90:0.7) -- (90:1.3);
	
\node at (90:0.45) {\small $\alpha$}; 
\node at (162:0.45) {\small $\beta$};
\node at (18:0.45) {\small $\gamma$};
\node at (234:0.45) {\small $\delta$};
\node at (-54:0.45) {\small $\epsilon$};

\node at (76:0.8) {\small $\alpha$}; 
\node at (30:0.85) {\small $\beta$};
\node at (28:1.15) {\small $\delta$};
\node at (82:1.15) {\small $\gamma$};
\node at (54:1.5) {\small $\epsilon$};

\node at (-18:1.45) {\small $\alpha$}; 
\node at (7:1.15) {\small $\beta$};
\node at (-46:1.15) {\small $\gamma$};
\node at (6:0.75) {\small $\delta$};
\node at (-40:0.75) {\small $\epsilon$};

\node at (-68:0.8) {\small $\delta$}; 
\node at (-116:0.8) {\small $\epsilon$};
\node at (-116:1.18) {\small $\gamma$};
\node at (-90:1.45) {\small $\alpha$};
\node at (-64:1.2) {\small $\beta$};

\node at (198:1.45) {\small $\alpha$}; 
\node at (173:1.2) {\small $\beta$}; 
\node at (226:1.15) {\small $\gamma$};
\node at (174:0.8) {\small $\delta$};
\node at (220:0.75) {\small $\epsilon$};

\node at (104:0.8) {\small $\alpha$}; 
\node at (148:0.8) {\small $\gamma$};
\node at (154:1.2) {\small $\epsilon$};
\node at (99:1.1) {\small $\beta$};
\node at (128:1.5) {\small $\delta$};

\node[draw,shape=circle, inner sep=1] at (0,0) {\small 1};

\foreach \x in {2,...,6}
\node[draw,shape=circle, inner sep=1] at (198-72*\x:1) {\small \x};

\end{tikzpicture}
\caption{$\{\alpha^3,\beta\gamma\delta,\delta\epsilon^2\}$: Neighborhood of a $3^5$-tile.}
\label{special1B}
\end{figure}

We conclude there is no $3^5$-tile. For $f=24$, by Lemma \ref{basic}, the tiling is combinatorially the pentagonal subdivision of the octahedron. For $f=60$, by Lemma \ref{basic} and the lack of degree $4$ vertex in the AVC, the tiling is combinatorially the pentagonal subdivision of the icosahedron. 

Since we already know the combinatorial structure of the tiling, the construction of the tiling means assigning edges and angles to the combinatorial tiling in the way compatible with the AVC. We do this in two steps. First, we find all the possible pentagonal subdivisions of a triangular face of the octahedron or icosahedron. Second, we find how these pieces fit together to form the whole tiling.

Figure \ref{special1C} gives two pentagonal subdivisions of a triangular face. We use $\bullet$ to indicate the three vertices (of degree $4$ or $5$) of the face, which are also vertices of the octahedron or icosahedron. All other vertices have degree $3$. By the AVC, therefore, the center vertex is $\alpha^3,\beta\delta\gamma,\delta\epsilon^2$, and $\alpha,\delta$ do not appear at $\bullet$.

Suppose the center vertex is $\alpha^3$. Then the angles at $\bullet$ are not adjacent to $\alpha$, which means $\delta$ or $\epsilon$. By no $\delta$ at $\bullet$, all three angles at $\bullet$ are $\epsilon$. This determines $n(\alpha^3)$ in Figure \ref{special1C}.

Suppose the center vertex is $\beta\delta\gamma$. Then by no $\delta$ at $\bullet$, we determine the two tiles containing $\beta,\gamma$, including $\beta$ and $\epsilon$ at two $\bullet$. See the second of Figure \ref{special1C}. The angle at the third $\bullet$ is not adjacent to $\delta$. Therefore the angle is $\alpha$ or $\gamma$. By no $\alpha$ at $\bullet$, the angle is $\gamma$. This determines the tile containing $\delta$, and then determines $n(\beta\gamma\delta)$ in Figure \ref{special1C}.

Suppose the center vertex is $\delta\epsilon^2$. By the similar reason of not adjacent to $\delta$ or $\epsilon$, and no $\alpha$ at $\bullet$, we may determine all angles at three $\bullet$, and then determine all three tiles. Then we get a degree $3$ vertex $\gamma\epsilon\cdots$, contradicting the AVC. Therefore $\delta\epsilon^2$ cannot be the center vertex.

\begin{figure}[htp]
\centering
\begin{tikzpicture}[>=latex,scale=1]

\foreach \b in {0,1}
\foreach \a in {0,1,2}
{
\begin{scope}[xshift=6.5*\b cm, rotate=120*\a]

\draw
	(-30:1.6) -- (90:1.6)
	(0,0) -- (0.93,0);

\fill
	(-30:1.6) circle (0.1);
	
\end{scope}
}


\foreach \a in {0,1,2}
{
\begin{scope}[rotate=120*\a]

\draw[line width=1.5]
	(0,0) -- (0.93,0);


\node at (0,1.3) {\small $\epsilon$};
\node at (0.12,0.2) {\small $\alpha$};
\node at (0.6,0.2) {\small $\beta$};
\node at (0.8,-0.2) {\small $\gamma$};
\node at (0.45,-0.63) {\small $\delta$};


\end{scope}
}

\node at (1.5,-1.3) {$n(\alpha^3)$};


\begin{scope}[xshift=6.5 cm]



\draw[line width=1.5]
	(-0.455,0.8) -- (210:1.6)
	(0.455,0.8) -- (-30:1.6)
	(0,0) -- (0.93,0)
	;

\node at (0,1.3) {\small $\epsilon$}; 
\node at (0.6,0.2) {\small $\alpha$};
\node at (0.12,0.2) {\small $\beta$};
\node at (0.3,0.7) {\small $\gamma$};
\node at (-0.3,0.8) {\small $\delta$}; 

\node at (1,-0.6) {\small $\beta$}; 
\node at (0.8,-0.2) {\small $\alpha$};
\node at (0.1,-0.2) {\small $\gamma$};
\node at (0.45,-0.6) {\small $\delta$};
\node at (-0.2,-0.63) {\small $\epsilon$};

\node at (-1.05,-0.6) {\small $\gamma$}; 
\node at (-0.75,-0.05) {\small $\alpha$};
\node at (-0.48,0.45) {\small $\beta$};
\node at (-0.2,0) {\small $\delta$}; 
\node at (-0.55,-0.63) {\small $\epsilon$};


\node at (1.5,-1.3) {$n(\beta\gamma\delta)$};

\end{scope}


\begin{scope}[shift={(3cm,0.3 cm)}]

\foreach \a in {0,1,2}
{
\begin{scope}[rotate=120*\a]

\draw
	(30:0.7) -- (70:1.2) -- (110:1.2) -- (150:0.7);

\draw[line width=1.5]
	(0,0) -- (30:0.7);

\node at (90:0.25) {\small $\alpha$};
\node at (50:0.6) {\small $\beta$};
\node at (10:0.6) {\small $\gamma$};
\node at (75:1) {\small $\delta$};
\node at (105:1) {\small $\epsilon$};

\end{scope}
}

\begin{scope}[xshift=6.5cm]

\foreach \a in {0,1,2}
\draw[rotate=120*\a]
	(0,0) -- (30:0.7) -- (70:1.2) -- (110:1.2) -- (150:0.7);

\draw[line width=1.5]
	(-10:1.2) -- (30:0.7) -- (70:1.2)
	(0,0) -- (30:0.7)
	(150:0.7) -- (190:1.2) -- (230:1.2);

\node at (50:0.6) {\small $\alpha$};
\node at (90:0.25) {\small $\beta$};
\node at (75:1) {\small $\gamma$};
\node at (130:0.6) {\small $\delta$};
\node at (105:1) {\small $\epsilon$};

\node at (10:0.6) {\small $\alpha$};
\node at (-17:1) {\small $\beta$};
\node at (-30:0.25) {\small $\gamma$};
\node at (-45:1) {\small $\delta$};
\node at (-70:0.6) {\small $\epsilon$};

\node at (170:0.6) {\small $\beta$};
\node at (195:1) {\small $\alpha$};
\node at (210:0.25) {\small $\delta$};
\node at (225:1) {\small $\gamma$};
\node at (250:0.6) {\small $\epsilon$};

\end{scope}

\end{scope}

\end{tikzpicture}
\caption{$\{\alpha^3,\beta\gamma\delta,\delta\epsilon^2\}$: Subdivision tilings of  a triangular face.}
\label{special1C}
\end{figure}

We conclude that, up to the symmetry of rotation and flip, Figure \ref{special1C} gives all the possible pentagonal subdivisions of a triangular face of the octahedron or icosahedron. We may use the fact that all non-$\bullet$-vertices have degree $3$ to determine how adjacent triangular faces are glued together. For example, if we glue one $n(\alpha^3)$ to the bottom of a non-flipped $n(\beta\gamma\delta)$, then $n(\alpha^3)$ is also non-flipped. Such argument shows that, if one triangular face is not flipped, then all triangular faces are not flipped. Therefore all the pentagonal subdivisions of triangular faces in the subsequence discussion are exactly the ones in Figure \ref{special1C}, i.e., all are not flipped. 

For $n(\beta\gamma\delta)$, we also use $b$-edges to determine the adjacent triangular faces. For example, to the left or right of $n(\beta\gamma\delta)$, we can only glue another $n(\beta\gamma\delta)$. 

Each vertex $V$ of the octahedron or icosahedron has a neighbourhood $N(V)$ consisting of $4$ or $5$ triangular faces. Each face is either $n(\alpha^3)$ or $n(\beta\gamma\delta)$. Moreover, by the AVC, we know $V=\beta^2\gamma^2,\beta\gamma\epsilon^2,\epsilon^4$ for $f=24$, and $V=\beta^2\gamma^2\epsilon,\beta\gamma\epsilon^3,\epsilon^5$ for $f=60$. This gives us all the possible $N(V)$ in Figure \ref{special1D}. We remark that $N(\epsilon^4)$ and $N(\epsilon^5)$ are the extended neighbourhood tilings in 
Figure \ref{e-nd}, and $N'(\epsilon^4)$ and $N'(\epsilon^5)$ are the flips of $N(\epsilon^4)$ and $N(\epsilon^5)$. 
 
\begin{figure}[htp]
\centering
\begin{tikzpicture}[>=latex,scale=1]


\foreach \b in {0,...,3}
\foreach \a in {0,...,3}
{
\begin{scope}[xshift=3*\b cm,rotate=90*\a]

\draw
	(0,0) -- (-0.9,0.9) -- (0.9,0.9);

\fill
	(0,0) circle (0.1)
	(0.9,0.9) circle (0.1);

\end{scope}
}


\foreach \a in {1,-1}
{
\begin{scope}[scale=\a]

\draw[line width=1.5]
	(-0.3,-0.9) -- (0.9,-0.9) -- (0.9,0.3)
	(-0.6,-0.6) -- (0.6,0.6);
	
\node at (0.3,0) {\small $\gamma$};
\node at (0,0.3) {\small $\beta$};
\node at (0.4,-0.7) {\small $\gamma$};
\node at (0.7,-0.4) {\small $\beta$};
\node at (0.75,0.5) {\small $\epsilon$};
\node at (0.5,0.75) {\small $\epsilon$};

\end{scope}
}

\node at (0,-1.3) {$N(\beta^2\gamma^2)$};


\begin{scope}[xshift=3cm]

\draw[line width=1.5]
	(0.9,-0.9) -- (0.9,0.3)
	(0,0) -- (0.6,0.6)
	(-0.9,0.9) -- (0.3,0.9);
	
\node at (0.3,0) {\small $\gamma$};
\node at (0,0.27) {\small $\beta$};
\node at (-0.4,0.7) {\small $\gamma$};
\node at (0.7,-0.4) {\small $\beta$};
\node at (0.75,0.5) {\small $\epsilon$};
\node at (0.5,0.75) {\small $\epsilon$};
\node at (-0.75,-0.5) {\small $\epsilon$};
\node at (-0.5,-0.75) {\small $\epsilon$};
\node at (-0.75,0.5) {\small $\epsilon$};
\node at (0.5,-0.75) {\small $\epsilon$};
\node at (-0.27,0) {\small $\epsilon$};
\node at (0,-0.27) {\small $\epsilon$};

\node at (0,-1.3) {$N(\beta\gamma\epsilon^2)$};

\end{scope}


\foreach \a in {0,1,2,3}
{
\begin{scope}[xshift=6cm, rotate=90*\a]
	
\node at (0.75,0.5) {\small $\epsilon$};
\node at (0.5,0.75) {\small $\epsilon$};
\node at (-0.27,0) {\small $\epsilon$};

\end{scope}
}

\node at (6,-1.3) {$N(\epsilon^4)$};


\foreach \a in {0,1,2,3}
{
\begin{scope}[xshift=9cm, rotate=90*\a]

\draw[line width=1.5]
	(0.3,0.3) -- (0.9,0.9);
	
\node at (0.4,-0.7) {\small $\beta$};
\node at (0.7,-0.4) {\small $\gamma$};
\node at (-0.27,0) {\small $\epsilon$};

\end{scope}
}

\node at (9,-1.3) {$N'(\epsilon^4)$};


\begin{scope}[yshift=-3cm]

\foreach \b in {0,...,3}
\foreach \a in {0,...,4}
{
\begin{scope}[xshift=3*\b cm,rotate=72*\a]

\draw
	(0,0) -- (18:1.2) -- (90:1.2);

\fill
	(0,0) circle (0.1)
	(18:1.2) circle (0.1);

\end{scope}
}


\draw[line width=1.5]
	(140:1) -- (90:1.2) -- (40:1)
	(-54:1.2) -- (-4:1)
	(234:1.2) -- (184:1)
	(18:0.8) -- (0,0) -- (162:0.8);

\node at (54:0.3) {\small $\beta$};
\node at (126:0.3) {\small $\gamma$};
\node at (210:0.3) {\small $\beta$};
\node at (-24:0.3) {\small $\gamma$};
\node at (-90:0.25) {\small $\epsilon$};

\node at (78:0.85) {\small $\gamma$};
\node at (102:0.85) {\small $\beta$};
\node at (225:0.9) {\small $\gamma$};
\node at (-40:0.9) {\small $\beta$};

\foreach \a in {1,2,4}
\node at (100+72*\a:0.9) {\small $\epsilon$};

\foreach \a in {1,3,4}
\node at (80+72*\a:0.9) {\small $\epsilon$};

\node at (0,-1.4) {$N(\beta^2\gamma^2\epsilon)$};


\begin{scope}[xshift=3cm]

\draw[line width=1.5]
	(18:1.2) -- (67:1)
	(162:1.2) -- (113:1)
	(0,0) -- (90:0.8);

\node at (54:0.3) {\small $\gamma$};
\node at (126:0.3) {\small $\beta$};

\node at (150:0.85) {\small $\gamma$};
\node at (30:0.8) {\small $\beta$};

\foreach \a in {-1,0,1}
\node at (-90+72*\a:0.25) {\small $\epsilon$};

\foreach \a in {0,2,3,4}
\node at (80+72*\a:0.9) {\small $\epsilon$};

\foreach \a in {0,1,2,3}
\node at (100+72*\a:0.9) {\small $\epsilon$};

\node at (0,-1.4) {$N(\beta\gamma\epsilon^3)$};

\end{scope}


\begin{scope}[xshift=6cm]

\foreach \a in {0,...,4}
{
\begin{scope}[rotate=72*\a]

\node at (54:0.25) {\small $\epsilon$};
\node at (80:0.9) {\small $\epsilon$};
\node at (100:0.9) {\small $\epsilon$};

\end{scope}
}

\node at (0,-1.4) {$N(\epsilon^5)$};

\end{scope}


\begin{scope}[xshift=9cm]

\foreach \a in {0,...,4}
{
\begin{scope}[rotate=72*\a]

\draw[line width=1.5]
	(18:1.2) -- (18:0.4);

\node at (54:0.25) {\small $\epsilon$};

\end{scope}
}

\node at (30:0.85) {\small $\gamma$};
\node at (102:0.85) {\small $\gamma$};
\node at (176:0.85) {\small $\gamma$};
\node at (-40:0.9) {\small $\gamma$};
\node at (247:0.85) {\small $\gamma$};

\node at (4:0.9) {\small $\beta$};
\node at (78:0.85) {\small $\beta$};
\node at (147:0.85) {\small $\beta$};
\node at (223:0.85) {\small $\beta$};
\node at (-67:0.85) {\small $\beta$};

\node at (0,-1.4) {$N'(\epsilon^5)$};

\end{scope}

\end{scope}
	
\end{tikzpicture}
\caption{$\{\alpha^3,\beta\gamma\delta,\delta\epsilon^2\}$: Extended neighbourhood tilings.}
\label{special1D}
\end{figure}

All the vertices $\bullet$ on the boundary of $N(V)$ are vertices of the octahedron or icosahedron. Therefore they are the centers of the other $N(V')$. For example, the boundary vertices $\beta\thin\gamma\cdots$ of $N(\beta^2\gamma^2)$ must be the center of other two  $N(\beta^2\gamma^2)$. This adds more $\epsilon$ to the boundary vertices $\epsilon^2\cdots$ of $N(\beta^2\gamma^2)$. Therefore the boundary vertices $\epsilon^2\cdots=\epsilon^4$, and the lower left half and the upper right half of $N(\beta^2\gamma^2)$ are parts of two respective $N'(\epsilon^4)$, and the tiling is two $N'(\epsilon^4)$ glued together. Note that the pentagonal subdivision of the octahedron is two $N(\epsilon^4)$ glued together. Since $N'(\epsilon^4)$ is the flip of $N(\epsilon^4)$, the tiling we get from one $N(\beta^2\gamma^2)$ is obtained by flipping both halves of the pentagonal subdivision of the octahedron. This is the second tiling $T(4\beta^2\gamma^2,2\epsilon^4)$ in Figure \ref{subdivision_tiling_modify24}. 

Applying the similar argument to $N(\beta\gamma\epsilon^2)$, we find the lower left half of $N(\beta\gamma\epsilon^2)$ is part of $N(\epsilon^4)$, and upper right half of $N(\beta\gamma\epsilon^2)$ is part of $N'(\epsilon^4)$. Therefore the tiling is one $N(\epsilon^4)$ and one $N'(\epsilon^4)$ glued together. This is obtained by flipping one half of the pentagonal subdivision of the octahedron, and is the first tiling $T(4\beta\gamma\epsilon^2,2\epsilon^4)$ in Figure \ref{subdivision_tiling_modify24}. 

For $f=24$, it remains to consider the case of no  $N(\beta^2\gamma^2)$ and no $N(\beta\gamma\epsilon^2)$. In other words, the extended neighbourhoods are only $N(\epsilon^4)$ and $N'(\epsilon^4)$. This implies that $N'(\epsilon^4)$ cannot appear, and only $N(\epsilon^4)$ remains. Then the tiling is two $N(\epsilon^4)$ glued together. This is the pentagonal subdivision $T(6\epsilon^4)$ of the octahedron, or the third of Figure \ref{subdivision_tiling}. 

For $f=60$, by looking at the boundary vertex $\beta\thin\gamma\cdots$, we know that one $N(\beta^2\gamma^2\epsilon)$ must be paired with another $N(\beta^2\gamma^2\epsilon)$. This gives eight heavily shaded triangles in the first and second of Figure \ref{special1G}. Then only $N'(\epsilon^5)$ fits the three triangles at the two vertices $\epsilon^3\cdots$ on the boundary of the region. This adds four lightly shaded triangles above and below the region. The result is the hexagon in the middle. Then we have four vertices $\beta\gamma\epsilon\cdots$ on the boundary of the combined shaded region. If one of these is $\beta^2\gamma^2\epsilon$, then we have $N(\beta^2\gamma^2\epsilon)$ around the vertex, and it is easy to further derive the first tiling in Figure \ref{special1G}. If all four vertices are $\beta\gamma\epsilon^3$, then we have four $N(\beta\gamma\epsilon^3)$ around these vertices, and we can further get the second in Figure \ref{special1G}. 

The first of Figure \ref{special1G} has three non-overlapping $N'(\epsilon^5)$. The complement of these consists of five $n(\epsilon^3)$. If we flip all three $N'(\epsilon^5)$ to $N(\epsilon^5)$, then the tiling consists of twenty $n(\epsilon^3)$, which is the pentagonal subdivision of the icosahedron. Therefore the tiling is the triple flip modification of the pentagonal subdivision tiling. This is $T(6\beta^2\gamma^2\epsilon,3\beta\gamma\epsilon^3,3\epsilon^5)$, the fourth of Figure \ref{subdivision_tiling_modify60A} and the fourth of Figure \ref{subdivision_tiling_modify60B}.

Similarly, the second of Figure \ref{special1G} has two non-overlapping $N'(\epsilon^5)$, and the complement consists of ten $n(\epsilon^3)$. Therefore the tiling is the double flip modification of the pentagonal subdivision tiling. This is $T(2\beta^2\gamma^2\epsilon,6\beta\gamma\epsilon^3,4\epsilon^5)$, the third of Figure \ref{subdivision_tiling_modify60A} and the third of Figure \ref{subdivision_tiling_modify60B}.

Now we may assume that $\beta^2\gamma^2\epsilon$ is not a vertex. Then the tiling has no $N(\beta^2\gamma^2\epsilon)$. This implies that an $N(\beta\gamma\epsilon^3)$ must be paired with an $N'(\epsilon^5)$, and all five boundary vertices $\beta\gamma\cdots$ of $N'(\epsilon^5)$ are $\beta\gamma\epsilon^3$. The union of the five corresponding $N(\beta\gamma\epsilon^3)$ is a disk tiling ${\mc T}$ by fifteen triangles. See the third and fourth of Figure \ref{special1G}. 

The disk tiled by ${\mc T}$ is bounded by the shaded circle, and we have five vertices $\epsilon^3\cdots$ on the boundary. If one of the five is $\beta\gamma\epsilon^3$, then we find that all five should be $\beta\gamma\epsilon^3$. Therefore the whole tiling is completed by glueing one $N'(\epsilon^5)$ to ${\mc T}$. This is the third of Figure \ref{special1G}. If none of the five is $\beta\gamma\epsilon^3$, then all five are $\epsilon^5$. Therefore the whole tiling is completed by glueing one $N(\epsilon^5)$ to ${\mc T}$. This is the fourth of Figure \ref{special1G}.

The third of Figure \ref{special1G} is the flip modification according to the second of Figure \ref{mod60A}, and is $T(10\beta\gamma\epsilon^3,2\epsilon^5)$ in the second of Figure \ref{subdivision_tiling_modify60A}. The fourth of Figure \ref{special1G} is the flip modification according to the first of Figure \ref{mod60A}, and is $T(5\beta\gamma\epsilon^3,7\epsilon^5)$ in the first of Figure \ref{subdivision_tiling_modify60A}.  

\begin{figure}[htp]
\centering
\begin{tikzpicture}[>=latex,scale=1]

\foreach \a in {0,1}
\foreach \b in {1,-1}
\foreach \c in {1,-1}
{
\fill[gray!50,xshift=6*\a cm, xscale=\b, yscale=\c]
	(0,0) -- (0,0.9) -- (1.6,0.7) -- (1.6,0);

\fill[gray!20,xshift=6*\a cm, xscale=\b, yscale=\c]	
	(0,0.9) -- (0,1.6) -- (1.6,0.7);	
}

\foreach \a in {0,1}
{
\begin{scope}[xshift=6*\a cm]

\foreach \b in {1,-1}
\foreach \c in {1,-1}
{
\begin{scope}[xscale=\b, yscale=\c]

\draw
	(0,0) -- (0.8,0) -- (0,0.9) -- (1.6,0.7) -- (1.6,0)
	(0.8,0) -- (1.6,0.7)
	(0,0.9) -- (0,1.6) -- (1.6,0.7) -- (2.4,0)
	(0,0) ellipse (2.4 and 1.6);

\draw[line width=1.5]
	(0.8,0) -- (0.267,0.6)
	(0.533,0.833) -- (1.6,0.7)
	(0,1.6) -- (0,1.133);

\end{scope}
}

\foreach \b in {1,-1}
{
\begin{scope}[scale=\b]

\node at (-0.35,0.2) {\small $\gamma$};
\node at (0.35,0.17) {\small $\beta$};
\node at (0,0.7) {\small $\epsilon$};

\node at (0.8,0.25) {\small $\gamma$};
\node at (1.15,0.5) {\small $\beta$};
\node at (0.35,0.7) {\small $\epsilon$};

\node at (1.15,-0.55) {\small $\gamma$};
\node at (0.8,-0.25) {\small $\beta$};
\node at (0.35,-0.7) {\small $\epsilon$};

\node at (1.45,0.4) {\small $\epsilon$};
\node at (1.45,-0.4) {\small $\epsilon$};
\node at (1,0) {\small $\epsilon$};

\node at (0.8,0.95) {\small $\gamma$};
\node at (0.2,1.25) {\small $\beta$};
\node at (0.2,1) {\small $\epsilon$};

\node at (0.2,-1.3) {\small $\gamma$};
\node at (0.8,-1) {\small $\beta$};
\node at (0.2,-1) {\small $\epsilon$};

\end{scope}
}

\end{scope}
}


\draw
	(2.4,0) -- (2.8,0);

\foreach \b in {1,-1}
\draw[line width=1.5, yscale=\b]
	(1.6,0.7) -- ++(0.533,-0.467)
	(0,1.6) arc (90:30: 2.4 and 1.6)
	(-2.4,0) -- (-2.8,0);


\node at (1.8,0.3) {\small $\gamma$};
\node at (1.8,-0.3) {\small $\beta$};
\node at (2.2,0) {\small $\epsilon$};

\node at (1.7,0.9) {\small $\beta$};
\node at (0.75,1.35) {\small $\gamma$};
\node at (2.2,0.35) {\small $\epsilon$};

\node at (1.7,-0.85) {\small $\gamma$};
\node at (0.75,-1.35) {\small $\beta$};
\node at (2.2,-0.35) {\small $\epsilon$};

\node at (0,1.8) {\small $\beta$};
\node at (2.55,0.15) {\small $\epsilon$};
\node at (2.55,-0.15) {\small $\epsilon$};


\node at (-1.75,0.35) {\small $\epsilon$};
\node at (-1.75,-0.35) {\small $\epsilon$};
\node at (-2.2,0) {\small $\epsilon$};

\node at (-1.65,0.85) {\small $\epsilon$};
\node at (-0.75,1.35) {\small $\epsilon$};
\node at (-2.2,0.35) {\small $\epsilon$};

\node at (-1.65,-0.85) {\small $\epsilon$};
\node at (-0.75,-1.35) {\small $\epsilon$};
\node at (-2.2,-0.35) {\small $\epsilon$};

\node at (0,-1.8) {\small $\gamma$};
\node at (-2.55,0.2) {\small $\gamma$};
\node at (-2.55,-0.25) {\small $\beta$};

\node at (0,-2.2) {$T(6\beta^2\gamma^2\epsilon, 3\beta\gamma\epsilon^3,3\epsilon^5)$};


\foreach \b in {1,-1}
\foreach \c in {1,-1}
{
\begin{scope}[xshift=6cm, xscale=\b, yscale=\c]

\draw
	(2.4,0) -- (2.8,0);

\node at (1.65,0.85) {\small $\epsilon$};
\node at (0.75,1.35) {\small $\epsilon$};
\node at (2.2,0.35) {\small $\epsilon$};

\node at (1.75,0.35) {\small $\epsilon$};
\node at (1.75,-0.35) {\small $\epsilon$};
\node at (2.2,0) {\small $\epsilon$};

\node at (2.55,0.15) {\small $\epsilon$};
\node at (0,1.75) {\small $\epsilon$};

\end{scope}
}

\node at (6,-2.2) {$T(2\beta^2\gamma^2\epsilon, 6\beta\gamma\epsilon^3,4\epsilon^5)$};


\begin{scope}[xshift=-1.2cm, yshift=-4.5cm]

\foreach \b in {0,1}
\draw[gray!50, line width=3, xshift=4.2*\b cm]
	(0,0) circle (1.6);

\foreach \a in {0,...,4}
{

\foreach \b in {0,1,2}
{
\begin{scope}[xshift=4.2*\b cm, rotate=72*\a]

\draw
	(0,0) -- (18:1) -- (90:1)
	(90:1) -- (54:1.6) -- (18:1)
	(0,0) circle (1.6)
	(54:1.6) -- (54:2);

\node at (54:0.25) {\small $\epsilon$};
\node at (34:1) {\small $\epsilon$};
\node at (74:1) {\small $\epsilon$};
\node at (54:1.4) {\small $\epsilon$};
\node at (18:1.15) {\small $\epsilon$};
\node at (40:1.45) {\small $\epsilon$};
\node at (66:1.45) {\small $\epsilon$};

\end{scope}
}	


\begin{scope}[rotate=72*\a]

\draw[line width=1.5]
	(18:0.333) -- (18:1)
	(54:1.6) -- (54:1.85);

\node at (74:0.65) {\small $\beta$};
\node at (32:0.65) {\small $\gamma$};
\node at (48:1.8) {\small $\gamma$};
\node at (61:1.8) {\small $\beta$};

\end{scope}

\node at (0,-2.3) {$T(10\beta\gamma\epsilon^3,2\epsilon^5)$};


\begin{scope}[xshift=4.2cm, rotate=72*\a]
	
\draw[line width=1.5]
	(18:0.333) -- (18:1);

\node at (74:0.65) {\small $\beta$};
\node at (32:0.65) {\small $\gamma$};
\node at (48:1.75) {\small $\epsilon$};
\node at (60:1.75) {\small $\epsilon$};

\end{scope}

\node at (4.2,-2.3) {$T(5\beta\gamma\epsilon^3,7\epsilon^5)$};


\begin{scope}[xshift=8.4cm, rotate=72*\a]

\node at (78:0.7) {\small $\epsilon$};
\node at (28:0.7) {\small $\epsilon$};
\node at (48:1.75) {\small $\epsilon$};
\node at (60:1.75) {\small $\epsilon$};

\end{scope}

\node at (8.4,-2.3) {$T(12\epsilon^5)$};

}

\end{scope}

\end{tikzpicture}
\caption{$\{\alpha^3,\beta\gamma\delta,\delta\epsilon^2\}$: Five tilings for $f=60$.}
\label{special1G}
\end{figure}

It remains to consider the case all vertices are $\epsilon^5$. This is the fifth of Figure \ref{special1G}, and is the pentagonal subdivision $T(12\epsilon^5)$ of the icosahedron, or the fifth of Figure \ref{subdivision_tiling}.

\section{Tiling with Vertices $\alpha^3,\beta\gamma\delta,\delta^2\epsilon$}
\label{special2tiling}

\begin{proposition}\label{special2}
Tilings of the sphere by congruent non-symmetric pentagons in Figure \ref{pentagon}, such that $\alpha^3,\beta\gamma\delta$ are vertices and $2\delta+\epsilon=2\pi$, are the following:
\begin{enumerate}
\item $f=24$: $T(6\epsilon^4)$. The pentagon for the tiling is unique.
\item $f=60$: $T(12\epsilon^5)$, 
$T(5\beta\gamma\epsilon^2,5\delta\epsilon^3,7\epsilon^5)$, $T(10\beta\gamma\epsilon^2,10\delta\epsilon^3,2\epsilon^5)$, 
\newline
$T(10\beta\gamma\epsilon^2,6\delta\epsilon^3,4\epsilon^5)$, $T(15\beta\gamma\epsilon^2,3\delta\epsilon^3,3\epsilon^5)$. The pentagon for the five tilings is the same unique one.
\end{enumerate}
\end{proposition}

Similar to Proposition \ref{special1}, the condition $2\delta+\epsilon=2\pi$ is satisfied if $\delta^2\epsilon$ is a vertex, and ``the angle sum of $\delta^2\epsilon$'' in the proof never implies that $\delta^2\epsilon$ appears as a vertex. 

We remark that $\beta^2\cdots$ and $\gamma^2\cdots$ are not vertices in the tilings in the proposition.

The tilings $T(6\epsilon^4)$ and $T(12\epsilon^5)$ are the pentagonal subdivision tilings. The other tilings are the last two tilings in Figure \ref{subdivision_tiling_modify60A}, and the first two of Figure \ref{subdivision_tiling_modify60B}.

The proof is also divided into three steps.

\subsubsection*{Step 1: Determine Angles and AVC}

The angle sums of $\alpha^3,\beta\gamma\delta,\delta^2\epsilon$ and the angle sum for pentagon imply 
\[
\alpha=\tfrac{2}{3}\pi,\;
\beta+\gamma=(\tfrac{7}{6}+\tfrac{2}{f})\pi,\;
\delta=(\tfrac{5}{6}-\tfrac{2}{f})\pi,\;
\epsilon=(\tfrac{1}{3}+\tfrac{4}{f})\pi.
\]
By $f>12$, we get $\delta>\alpha>\epsilon$. By Lemma \ref{geometry1}, we get $\beta<\gamma$. By $\beta+\gamma=(\frac{7}{6}+\frac{2}{f})\pi$, we get $\gamma> (\frac{7}{12}+\frac{1}{f})\pi$.

By the parity lemma, we have $\alpha\gamma\cdots=\alpha\beta\gamma\cdots,\alpha\gamma^2\cdots$. By $\beta<\gamma$, we have $R(\alpha\gamma^2\cdots)<R(\alpha\beta\gamma\cdots)= (\frac{1}{6}-\frac{2}{f})\pi<\alpha,\gamma,\delta,\epsilon$. Therefore the remainders have only $\beta$. If one of the remainders is not empty, then $\beta\le (\frac{1}{6}-\frac{2}{f})\pi$. This implies $\gamma=(\frac{7}{6}+\frac{2}{f})\pi-\beta\ge \pi$, and $\gamma^2\cdots$ is not a vertex. Then $\alpha\gamma\cdots=\alpha\beta\gamma\cdots$, and the remainder has $\beta$. Therefore $\beta^2\cdots$ is a vertex, while $\gamma^2\cdots$ is not a vertex, contradicting the balance lemma. This proves that the remainders are empty, and we get $\alpha\gamma\cdots=\alpha\beta\gamma,\alpha\gamma^2$. Since the angle sum of $\alpha\beta\gamma$ further implies $f=12$, we get $\alpha\gamma\cdots=\alpha\gamma^2$. 

The angle sum of $\alpha\gamma^2$ further implies
\[
\alpha=\gamma=\tfrac{2}{3}\pi,\;
\beta=(\tfrac{1}{2}+\tfrac{2}{f})\pi,\;
\delta=(\tfrac{5}{6}-\tfrac{2}{f})\pi,\;
\epsilon=(\tfrac{1}{3}+\tfrac{4}{f})\pi.
\]
The unique AAD $\thin^{\epsilon}\gamma^{\alpha}\thick^{\beta}\alpha^{\gamma}\thick^{\alpha}\gamma^{\epsilon}\thin$ of $\alpha\gamma^2$ gives $\alpha\beta\cdots$. By $\alpha\gamma\cdots=\alpha\gamma^2$ and the parity lemma, we get $\alpha\beta\cdots=\alpha\beta^2\cdots$. By $0\ne R(\alpha\beta^2\cdots)=(\tfrac{1}{3}-\tfrac{4}{f})\pi<$ all angles, we get a contradiction. This proves that $\alpha\gamma\cdots$ is not a vertex.

Consider a vertex $\gamma^2\cdots$. By no $\alpha\gamma\cdots$, the vertex has no $\alpha$. By $\beta<\gamma$, $\beta+\gamma>\pi$, and the parity lemma, we know $R(\gamma^2\cdots)$ has no $\beta,\gamma$. By $\beta\gamma\delta$ and $\beta<\gamma$, we have $R(\gamma^2\cdots)<\delta<3\epsilon$. This implies $\gamma^2\cdots=\gamma^2\epsilon,\gamma^2\epsilon^2$.

If $\gamma^2\epsilon$ is a vertex, then the angle sum of $\gamma^2\epsilon$ further implies
\[
\alpha=\tfrac{2}{3}\pi,\;
\beta=\epsilon=(\tfrac{1}{3}+\tfrac{4}{f})\pi,\;
\gamma=\delta=(\tfrac{5}{6}-\tfrac{2}{f})\pi.
\]
By the balance lemma, we know $\beta^2\cdots$ is a vertex. If the vertex has $\alpha$, then $0\ne R(\alpha\beta^2\cdots)=(\tfrac{2}{3}-\tfrac{8}{f})\pi<\alpha,2\beta,\gamma,\delta,2\epsilon$. By the parity lemma, we get $\alpha\beta^2\cdots=\alpha\beta^2\epsilon$. We have the AAD $\thin^{\delta}\beta^{\alpha}\thick^{\beta}\alpha^{\gamma}\thick^{\alpha}\beta^{\delta}\thin\cdots$ at $\alpha\beta^2\epsilon$, contradicting no $\alpha\gamma\cdots$. Therefore $\beta^2\cdots$ has no $\alpha$. By the values of $\beta,\gamma$, and the parity lemma, we then get $\beta^2\cdots=\beta^3\gamma\cdots,\beta^4\cdots,\beta^2\cdots$, with only $\delta,\epsilon$ in the remainders. By $R(\beta^3\gamma\cdots)=(\tfrac{1}{6}-\tfrac{10}{f})\pi<\delta,\epsilon$, we get $\beta^3\gamma\cdots=\beta^3\gamma$. By $R(\beta^4\cdots)=(\tfrac{2}{3}-\tfrac{16}{f})\pi<\delta,2\epsilon$, we get $\beta^4\cdots=\beta^4,\beta^4\epsilon$. By $R(\beta^2\cdots)=(\tfrac{4}{3}-\tfrac{8}{f})\pi<2\delta,4\epsilon,\delta+2\epsilon$, we get $\beta^2\cdots=\beta^2\delta,\beta^2\delta\epsilon,\beta^2\epsilon,\beta^2\epsilon^2,\beta^2\epsilon^3$. The angle sum of any of $\beta^2\delta,\beta^2\epsilon$ implies $f=12$, the angle sum of any of $\beta^4,\beta^2\epsilon^2$ implies $f=24$, and the angle sum of any of $\beta^3\gamma,\beta^4\epsilon,\beta^2\delta\epsilon,\beta^2\epsilon^3$ implies $f=60$. Then we dismiss the case $f=12$ and get the corresponding angles
\begin{align}
f=24 &\colon
\alpha=\tfrac{2}{3}\pi,\;
\beta=\epsilon=\tfrac{1}{2}\pi,\;
\gamma=\delta=\tfrac{3}{4}\pi;  \label{special2_24B} \\
f=60 &\colon
\alpha=\tfrac{2}{3}\pi,\;
\beta=\epsilon=\tfrac{2}{5}\pi,\;
\gamma=\delta=\tfrac{4}{5}\pi. \label{special2_60B}
\end{align}

If $\gamma^2\epsilon^2$ is a vertex, then the angle sum of $\gamma^2\epsilon^2$ further implies
\[
\alpha=\tfrac{2}{3}\pi,\;
\beta=(\tfrac{1}{2}+\tfrac{6}{f})\pi,\;
\gamma=(\tfrac{2}{3}-\tfrac{4}{f})\pi,\;
\delta=(\tfrac{5}{6}-\tfrac{2}{f})\pi,\;
\epsilon=(\tfrac{1}{3}+\tfrac{4}{f})\pi.
\]
By $\beta<\gamma$, we get $f>60$. By the balance lemma, we know $\beta^2\cdots$ is a vertex. We have 
$R(\beta^2\cdots)=(1-\frac{12}{f})\pi<\alpha+\epsilon,2\beta,\delta+\epsilon,3\epsilon$. By $f>60$, we also have $R(\beta^2\cdots)>\delta,2\epsilon$. By $\beta<\gamma$, $R(\beta^2\cdots)<2\beta$, and the parity lemma, we know $R(\beta^2\cdots)$ has no $\beta,\gamma$. By $0<R(\alpha\beta^2\cdots)=(\frac{1}{3}-\frac{12}{f})\pi<\alpha,\delta,\epsilon$, we know $\beta^2\cdots$ has no $\alpha$. Therefore $R(\beta^2\cdots)$ has only $\delta,\epsilon$. Then by $\delta>\epsilon$ and $\delta,2\epsilon<R(\beta^2\cdots)<\delta+\epsilon,3\epsilon$, we conclude $\beta^2\cdots$ is not a vertex. By the balance lemma, this implies $\gamma^2\epsilon^2$ is not a vertex. 

Finally, it remains to consider the case that $\gamma^2\cdots$ is not a vertex. By the balance lemma, a vertex involving $\beta,\gamma$ is $\beta\gamma\cdots$, with no $\beta,\gamma$ in the remainder. Then by no $\alpha\gamma\cdots$, we know $R(\beta\gamma\cdots)$ has only $\delta,\epsilon$. Then by $\epsilon<R(\beta\gamma\cdots)=\delta<3\epsilon$, we conclude that $\beta\gamma\delta,\beta\gamma\epsilon^2$ are the only vertices involving $\beta,\gamma$. This implies no $\beta\thin\gamma\cdots$. Moreover, by $\delta^2\epsilon$ and $\delta>\epsilon$, the only degree $3$ vertex not involving $b$-edge is $\delta^2\epsilon$. Therefore the only degree $3$ vertices are $\alpha^3,\beta\gamma\delta,\delta^2\epsilon$.

If both vertices $\gamma\cdots,\epsilon\cdots$ of a tile have degree $3$, then $\gamma\cdots=\beta\gamma\delta$, and $\epsilon\cdots=\delta^2\epsilon$. This implies two $\delta$ in the adjacent tile across the $\gamma\epsilon$-edge. The contradiction implies that there is no $3^5$-tile. 

By Lemma \ref{basic} and no $3^5$-tile, the tiling has a vertex $H$ of degree $4$ or $5$. If $H$ has $\beta,\gamma$, which we know are $\beta\gamma\delta,\beta\gamma\epsilon^2$, then $H=\beta\gamma\epsilon^2$. If $H$ has no $\beta,\gamma$, then by $\alpha^3$ and Lemma \ref{klem4}, we know $H=\delta^k\epsilon^l$. Then by the values of $\delta,\epsilon$, we get $H=\delta\epsilon^3,\epsilon^4,\epsilon^5$. The angle sum of $\epsilon^4$ implies $f=24$, and the angle sum of and of  $\beta\gamma\epsilon^2,\delta\epsilon^3,\epsilon^5$ implies $f=60$. Then we get the corresponding angles, and further get the AVCs by using $\beta\gamma\delta,\beta\gamma\epsilon^2$ being the only vertices involving $\beta,\gamma$
\begin{align}
f=24 &\colon
\alpha=\tfrac{2}{3}\pi,\;
\beta+\gamma=\tfrac{5}{4}\pi,\;
\delta=\tfrac{3}{4}\pi,\;
\epsilon=\tfrac{1}{2}\pi,  \nonumber \\
&\text{AVC}
=\{\alpha^3,\beta\gamma\delta,\delta^2\epsilon,\epsilon^4\}; \label{special2_24A} \\
f=60 &\colon
\alpha=\tfrac{2}{3}\pi,\;
\beta+\gamma=\tfrac{6}{5}\pi,\;
\delta=\tfrac{4}{5}\pi,\;
\epsilon=\tfrac{2}{5}\pi, \nonumber \\
&\text{AVC}
=\{\alpha^3,\beta\gamma\delta,\delta^2\epsilon,\beta\gamma\epsilon^2,\delta\epsilon^3,\epsilon^5\}. \label{special2_60A}
\end{align}

\subsubsection*{Step 2: Calculate the Pentagon}

The calculation is the same as the second step of the proof of Proposition \ref{special1}. The idea is that any pentagon satisfying \eqref{special2_24A} or \eqref{special2_60A} is the tile of a pentagonal subdivision tiling. Moreover, \eqref{special2_24B} and \eqref{special2_60B} are special cases of \eqref{special2_24A} and \eqref{special2_60A}.

For \eqref{special2_24A}, we may use the idea for \eqref{special1_60A} (\eqref{special1_24A} is too special because $B$ and $D$ exactly trisect the edge $l$) to get 
\[
\cos^3a+(3-2\sqrt{2})\cos^2a+(5-4\sqrt{2})\cos a-3+2\sqrt{2}=0, \quad
a=0.1973\pi.
\]
Then we get 
\[
\cos b
=\sqrt{-\tfrac{2}{3}\cos^2a+\tfrac{2}{3}(\sqrt{2}-1)\cos a+1},\quad
b= 0.1640\pi,
\]
and
\[
\tfrac{1}{\sqrt{3}}
=\cos a\cos b
+\sin a\sin b \cos\gamma,\quad
\gamma= 0.7051\pi.
\]
This implies $\gamma\ne\frac{3}{4}\pi$. Therefore \eqref{special2_24B} is impossible, and we only need to consider \eqref{special2_24A} and the associated AVC. We do not calculate the exact values, because we will prove in the third step that the pentagon only leads to the ordinary pentagonal subdivision tiling.

For \eqref{special2_60A}, we note that the sum of $\delta=\frac{4}{5}\pi$ in \eqref{special2_60A} and $\delta=\frac{6}{5}\pi$ in \eqref{special1_60A} is $2\pi$. This means that the pentagon for \eqref{special1_60A} and the pentagon for \eqref{special2_60A} are the companions as defined in \cite[Section 3.1]{wy1}. The calculation has been done in the second step of the proof of Proposition \ref{special1}, with $a,b,\alpha,\beta,\gamma,\delta,\epsilon$ labeled $a,b',\alpha',\beta',\gamma',\delta',\epsilon'$ in the earlier calculation. In particular, we have $\beta=0.46882\pi\ne\frac{2}{5}\pi$ here. Therefore \eqref{special2_60B} is impossible, and we only need to consider \eqref{special2_60A} and the associated AVC.

\subsubsection*{Step 3: Construct the Tiling}

We construct tilings for \eqref{special2_24A} and  \eqref{special2_60A}. The AVC says that the only vertices involving $\beta,\gamma$ are $\beta\gamma\delta,\beta\gamma\epsilon^2$ (and $\beta\gamma\epsilon^2$ is not a vertex for \eqref{special2_24A}). This implies that the neighbourhood of $\alpha^3$ is given by $n(\alpha^3)$ in Figure \ref{special1C}. By $\alpha\cdots=\alpha^3$, any one tile in an $n(\alpha^3)$ determines all three tiles.

For \eqref{special2_24A}, we consider the neighbourbhood of $\epsilon^4$ in the first of Figure \ref{special2C}. We may assume $T_1$ is arranged as indicated. Then $T_1$ determines one $n(\alpha^3)$ consisting of $T_1,T_5,T_6$. We have $\beta_6\gamma_1\cdots=\beta_6\gamma_1\delta_2$. Then $\delta_2,\epsilon_2$ determine $T_2$, and further determine another $n(\alpha^3)$, including $T_7,T_8$. Repeating the argument with $T_2$ in place of $T_1$, we determine $T_3$ and get the third $n(\alpha^3)$. Repeating again, we determine $T_4$ and get the fourth $n(\alpha^3)$. The four $n(\alpha^3)$ form an extended neighborhood $N(\epsilon^4)$ in Figure \ref{e-nd}. We also find four vertices $\epsilon^2\cdots=\epsilon^4$ on the boundary of the extended neighbourhood, such as $\epsilon_6\epsilon_7\cdots$ in the picture (and also in Figure \ref{e-nd}). The argument can then be repeated around the four new $\epsilon^4$. More repetitions of the argument gives the pentagonal subdivision $T(6\epsilon^4)$ of the octahedron.

\begin{figure}[htp]
\centering
\begin{tikzpicture}[>=latex,scale=1]


\foreach \a in {0,...,3}
{
\begin{scope}[rotate=90*\a]

\draw
	(0,0) -- (0.8,0) -- (1.2,0.6) -- (0.6,1.2) -- (0,0.8);
	
\draw[line width=1.5]
	(0.8,0) -- (1.2,0.6) -- (0.6,1.2);

\node at (0.95,0.55) {\small $\alpha$};   
\node at (0.6,0.95) {\small $\beta$};
\node at (0.7,0.2) {\small $\gamma$};
\node at (0.2,0.7) {\small $\delta$};
\node at (0.2,0.2) {\small $\epsilon$};
	
\end{scope}
}

\draw 
	(0.6,1.2) -- (0.6,2.2)
	(-1.2,0.6) -- (-1.8,0.6) -- (-1.8,1.8) --  (1.8,1.8) -- (1.8,-0.6) -- (1.2,-0.6);

\draw[line width=1.5]
	(-0.6,1.2) -- (-0.6,1.8)
	(1.2,0.6) -- (1.8,0.6);
	
\fill (0,0) circle (0.1);

\node[inner sep=1,draw,shape=circle] at (-0.55,0.55) {\small $1$};
\node[inner sep=1,draw,shape=circle] at (0.55,0.55) {\small $2$};
\node[inner sep=1,draw,shape=circle] at (0.55,-0.55) {\small $3$};
\node[inner sep=1,draw,shape=circle] at (-0.55,-0.55) {\small $4$};
\node[inner sep=1,draw,shape=circle] at (-1.3,1.3) {\small $5$};
\node[inner sep=1,draw,shape=circle] at (0,1.45) {\small $6$};
\node[inner sep=1,draw,shape=circle] at (1.3,1.3) {\small $7$};
\node[inner sep=1,draw,shape=circle] at (1.45,0) {\small $8$};

\node at (-0.8,1.6) {\small $\beta$};  
\node at (-1.6,1.6) {\small $\delta$};
\node at (-0.8,1.3) {\small $\alpha$};
\node at (-1.3,0.8) {\small $\gamma$};
\node at (-1.6,0.8) {\small $\epsilon$};

\node at (0.4,1.65) {\small $\epsilon$};  
\node at (-0.4,1.65) {\small $\gamma$};
\node at (0.4,1.3) {\small $\delta$};
\node at (-0.4,1.3) {\small $\alpha$};
\node at (0,1.05) {\small $\beta$};

\node at (0.8,1.65) {\small $\epsilon$};  
\node at (1.6,1.6) {\small $\delta$};
\node at (0.8,1.3) {\small $\gamma$};
\node at (1.3,0.8) {\small $\alpha$};
\node at (1.6,0.8) {\small $\beta$};

\node at (1.6,-0.4) {\small $\epsilon$};  
\node at (1.6,0.4) {\small $\gamma$};
\node at (1.3,-0.4) {\small $\delta$};
\node at (1.05,0) {\small $\beta$};
\node at (1.3,0.4) {\small $\alpha$};

\node at (0.4,2) {\small $\epsilon$};
\node at (0.8,2) {\small $\epsilon$};


\begin{scope}[xshift=4cm]

\foreach \a in {1,-1}
{
\begin{scope}[xscale=\a]

\draw
	(0.6,1.2) -- (0,0.8) -- (0,0) -- (0.8,0);

\draw[dotted]
	(0.6,1.2) -- (0.6,1.8) -- (0,1.8);

\draw[line width=1.5]
	(0.8,0) -- (1.2,0.6) -- (0.6,1.2);

\node at (0.55,0.95) {\small $\beta$};  
\node at (0.2,0.7) {\small $\delta$};
\node at (0.95,0.6) {\small $\alpha$};
\node at (0.2,0.2) {\small $\epsilon$};
\node at (0.7,0.2) {\small $\gamma$};

\end{scope}
}

\node at (0.4,1.3) {\small $\theta$};
\node at (-0.4,1.3) {\small $\rho$};

\node at (0,1) {\small $\epsilon$};

\node[inner sep=1,draw,shape=circle] at (0.55,0.55) {\small $1$};
\node[inner sep=1,draw,shape=circle] at (-0.55,0.55) {\small $2$};
\node[inner sep=1,draw,shape=circle] at (0,1.4) {\small $3$};

\end{scope}

\end{tikzpicture}
\caption{$\{\alpha^3,\beta\gamma\delta,\delta^2\epsilon\}$: Tiling for \eqref{special2_24A}, and the AAD of $\thin\epsilon\thin\epsilon\thin$.}
\label{special2C}
\end{figure}

For \eqref{special2_60A}, by no $\gamma^2\cdots$, the AAD of $\thin\epsilon\thin\epsilon\thin$ is $\thin^{\gamma}\epsilon^{\delta}\thin^{\gamma}\epsilon^{\delta}\thin$ or $\thin^{\gamma}\epsilon^{\delta}\thin^{\delta}\epsilon^{\gamma}\thin$. The second of Figure \ref{special2C} describes the AAD $\thin^{\gamma}\epsilon^{\delta}\thin^{\delta}\epsilon^{\gamma}\thin$. We have $\delta_1\delta_2\cdots=\delta_1\delta_2\epsilon_3$. Then by $\beta\cdots=\beta\gamma\delta^2,\beta\gamma\epsilon$, we get $\theta,\rho=\delta,\epsilon$. Then $T_3$ has three $a^2$-angles, a contradiction. Therefore we have the unique AAD $\thin^{\gamma}\epsilon^{\delta}\thin^{\gamma}\epsilon^{\delta}\thin$.

By the unique AAD $\thin^{\gamma}\epsilon^{\delta}\thin^{\gamma}\epsilon^{\delta}\thin$ and no $\beta\thin\gamma\cdots$, we get the unique AAD $\thin^{\beta}\delta^{\epsilon}\thin^{\gamma}\epsilon^{\delta}\thin^{\gamma}\epsilon^{\delta}\thin^{\gamma}\epsilon^{\delta}\thin$ of $\delta\epsilon^3$. The unique AAD $\thin^{\gamma}\epsilon^{\delta}\thin^{\gamma}\epsilon^{\delta}\thin$ also implies that the AAD of $\beta\gamma\epsilon^2$ is $\thick^{\alpha}\beta^{\delta}\thin^{\delta}\epsilon^{\gamma}\thin^{\delta}\epsilon^{\gamma}\thin^{\epsilon}\gamma^{\alpha}\thick$ or $\thick^{\alpha}\beta^{\delta}\thin^{\gamma}\epsilon^{\delta}\thin^{\gamma}\epsilon^{\delta}\thin^{\epsilon}\gamma^{\alpha}\thick$. See the left of Figure \ref{special2D}. We label the two AADs by $\beta\gamma\epsilon^2[1]$ and $\beta\gamma\epsilon^2[2]$, and indicate them by $\bullet$ and $\circ$.

\begin{figure}[htp]
\centering
\begin{tikzpicture}[>=latex,scale=1]


\begin{scope}[shift={(-6.5cm,2.2cm)}]

\foreach \b in {0,1}
{
\begin{scope}[yshift=-2.8*\b cm]

\foreach \a in {0,...,3}
\draw[rotate=90*\a]
	(0,0) -- (0.8,0) -- (1.2,0.6) -- (0.6,1.2) -- (0,0.8);

\draw[line width=1.5]
	(0.6,1.2) -- (1.2,0.6) -- (0.8,0)
	(-1.2,0.6) -- (-0.6,1.2) -- (0,0.8)
	(0.6,-1.2) -- (0,-0.8) -- (-0.6,-1.2)
	(0,-0.8) -- (0,0);

\node at (0.55,0.95) {\small $\beta$};  
\node at (0.2,0.7) {\small $\delta$};
\node at (0.2,0.2) {\small $\epsilon$};
\node at (0.95,0.55) {\small $\alpha$};
\node at (0.7,0.2) {\small $\gamma$};

\node at (-0.6,0.95) {\small $\alpha$};  
\node at (-0.2,0.7) {\small $\gamma$};
\node at (-0.2,0.2) {\small $\epsilon$};
\node at (-0.95,0.6) {\small $\beta$};
\node at (-0.7,0.2) {\small $\delta$};

\end{scope}
}


\fill (0,0) circle (0.1);

\node at (-0.55,-0.95) {\small $\gamma$};  
\node at (-0.2,-0.7) {\small $\alpha$};
\node at (-0.2,-0.2) {\small $\beta$};
\node at (-0.97,-0.6) {\small $\epsilon$};
\node at (-0.7,-0.2) {\small $\delta$};

\node at (0.6,-0.9) {\small $\beta$};  
\node at (0.2,-0.7) {\small $\alpha$};
\node at (0.2,-0.2) {\small $\gamma$};
\node at (0.97,-0.55) {\small $\delta$};
\node at (0.75,-0.2) {\small $\epsilon$};

\node at (1.8,0) {$\beta\gamma\epsilon^2[1]$};


\begin{scope}[yshift=-2.8cm]

\filldraw[fill=white] (0,0) circle (0.1);

\node at (0.55,-0.95) {\small $\gamma$};  
\node at (0.2,-0.7) {\small $\alpha$};
\node at (0.2,-0.2) {\small $\beta$};
\node at (0.97,-0.6) {\small $\epsilon$};
\node at (0.7,-0.2) {\small $\delta$};

\node at (-0.6,-0.9) {\small $\beta$};  
\node at (-0.2,-0.7) {\small $\alpha$};
\node at (-0.2,-0.2) {\small $\gamma$};
\node at (-0.97,-0.55) {\small $\delta$};
\node at (-0.75,-0.2) {\small $\epsilon$};
	
\node at (1.8,0) {$\beta\gamma\epsilon^2[2]$};

\end{scope}

\end{scope}


\draw[gray!50, line width=3]
	(-1.2,3.2) -- (-3.1,3.2) -- (-3.1,0.6) -- (-1.7,0.6) -- (-1.3,0) -- (1.3,0) -- (1.7,0.6) -- (2.4,1.5) -- (3.1,1.5) -- (3.1,2.5) -- (1.8,3.2) -- (1.8,3.9) -- (-1.2,3.9) -- cycle
	(-1.3,0) -- (-1.7,-0.6) -- (-2.4,-1.5) -- (-3.1,-1.5) -- ++(0,-0.3)
	(1.3,0) -- (1.7,-0.6) -- (3.1,-0.6) -- ++(0,-0.3);

\foreach \a in {-1,1}
{
\begin{scope}[scale=\a]

\draw
	(-0.6,-0.7) -- (-0.6,0.7)
	(0,0) -- (-1.3,0) -- (-1.7,0.6) -- (-1.2,1.1) -- (-0.6,0.7) -- (-0,1.1) -- (0.6,0.7) -- (1.2,1.1) -- (1.7,0.6) -- (1.3,0)
	(0,1.1) -- (0,1.8)
	(-1.2,1.1) -- (-1.2,1.8)
	(1.2,1.1) -- (1.2,1.8) 
	(2.4,1.5) -- (1.7,1.8) -- (1.2,1.8) -- (-2.4,1.8) -- (-2.4,-0.6)
	(1.7,0.6) -- (3.1,0.6) -- (3.5,0) -- (3.1,-0.6) -- (1.7,-0.6)
	(1.7,0.6) -- (2.4,1.5) -- (3.1,1.5) -- (3.1,0.6) 
	;

\draw[line width=1.5]
	(0,1.1) -- (0.6,0.7) -- (1.2,1.1)
	(0.6,0.7) -- (0.6,0)
	(-1.7,0.6) -- (-1.2,1.1) -- (-0.6,0.7)
	(-1.2,1.1) -- (-1.2,1.8)
	(1.7,0.6) -- (3.1,0.6)
	(2.4,-0.6) -- (2.4,0.6)
	(1.2,1.8) -- (1.8,1.8) -- (2.4,1.5);

\node at (-1.2,0.9) {\small $\alpha$}; 
\node at (-1.45,0.55) {\small $\beta$};
\node at (-0.75,0.55) {\small $\gamma$};	
\node at (-1.2,0.2) {\small $\delta$};	
\node at (-0.75,0.2) {\small $\epsilon$};

\node at (0.4,0.6) {\small $\alpha$};  
\node at (-0,0.85) {\small $\beta$};
\node at (-0.45,0.6) {\small $\delta$};
\node at (-0.45,0.2) {\small $\epsilon$};	
\node at (0.4,0.2) {\small $\gamma$};

\node at (0.8,0.6) {\small $\alpha$}; 
\node at (1.2,0.85) {\small $\gamma$}; 
\node at (1.5,0.6) {\small $\epsilon$};	
\node at (1.2,0.2) {\small $\delta$};	
\node at (0.8,0.2) {\small $\beta$};

\node at (-1.4,1.2) {\small $\alpha$}; 
\node at (-1.4,1.6) {\small $\beta$};
\node at (-1.75,0.8) {\small $\gamma$};	
\node at (-2.2,1.6) {\small $\delta$};	
\node at (-2.2,0.8) {\small $\epsilon$};

\node at (-1,1.2) {\small $\alpha$}; 
\node at (-0.6,0.95) {\small $\beta$};	
\node at (-1,1.6) {\small $\gamma$};
\node at (-0.15,1.25) {\small $\delta$};	
\node at (-0.15,1.6) {\small $\epsilon$};

\node at (0.15,1.2) {\small $\gamma$}; 
\node at (0.6,0.95) {\small $\alpha$};	
\node at (0.15,1.6) {\small $\epsilon$};
\node at (1,1.2) {\small $\beta$};	
\node at (1,1.6) {\small $\delta$};

\node at (2.2,0.4) {\small $\alpha$}; 
\node at (2.2,-0.4) {\small $\beta$};	
\node at (1.8,0.35) {\small $\gamma$};
\node at (1.8,-0.4) {\small $\delta$};	
\node at (1.5,0) {\small $\epsilon$};

\node at (2.4,0.8) {\small $\alpha$}; 
\node at (2.05,0.8) {\small $\beta$};	
\node at (2.9,0.8) {\small $\gamma$};
\node at (2.4,1.3) {\small $\delta$};	
\node at (2.9,1.3) {\small $\epsilon$};

\node at (1.75,1.6) {\small $\alpha$};
\node at (1.4,1.6) {\small $\beta$};
\node at (2.1,1.4) {\small $\gamma$};	
\node at (1.7,0.85) {\small $\epsilon$};	
\node at (1.4,1.2) {\small $\delta$};

\node at (2.6,0.4) {\small $\alpha$}; 
\node at (3,0.4) {\small $\beta$};
\node at (2.6,-0.4) {\small $\gamma$};
\node at (3,-0.4) {\small $\epsilon$};	
\node at (3.3,-0.05) {\small $\delta$};

\end{scope}
}

\node at (-2.6,2.3) {\small $\alpha$}; 
\node at (-2.9,2.25) {\small $\beta$};
\node at (-2.6,1.8) {\small $\gamma$};
\node at (-2.9,0.8) {\small $\delta$};
\node at (-2.6,0.8) {\small $\epsilon$};

\node at (1.6,2) {\small $\alpha$}; 
\node at (1.2,2) {\small $\gamma$};	
\node at (1.6,2.3) {\small $\beta$};
\node at (0.15,2.3) {\small $\delta$};	
\node at (0.15,2) {\small $\epsilon$};

\node at (-2.2,2.3) {\small $\alpha$}; 
\node at (-2.2,2) {\small $\beta$};	
\node at (-1.25,2.3) {\small $\gamma$};
\node at (-1.2,2) {\small $\delta$};	
\node at (-0.5,1.95) {\small $\epsilon$};

\node at (2,1.9) {\small $\alpha$};
\node at (2,2.3) {\small $\gamma$};	
\node at (2.4,1.7) {\small $\beta$};
\node at (2.9,2.3) {\small $\epsilon$};	
\node at (2.9,1.7) {\small $\delta$};	

\node at (-2.4,2.7) {\small $\alpha$};
\node at (-1.4,2.7) {\small $\beta$};
\node at (-2.9,2.7) {\small $\gamma$};
\node at (-1.4,3.05) {\small $\delta$};
\node at (-2.9,3.05) {\small $\epsilon$};

\node at (-0.2,3) {\small $\alpha$}; 
\node at (-1,2.95) {\small $\beta$};
\node at (-0.2,2.5) {\small $\gamma$};
\node at (-1,2.6) {\small $\delta$};
\node at (-0.15,2.1) {\small $\epsilon$};

\node at (0.2,3) {\small $\alpha$}; 
\node at (0.2,2.7) {\small $\beta$};
\node at (1.75,3) {\small $\gamma$};
\node at (1.8,2.7) {\small $\delta$};
\node at (2.5,2.65) {\small $\epsilon$};

\node at (0,3.4) {\small $\alpha$};
\node at (1.6,3.4) {\small $\beta$};
\node at (-1,3.4) {\small $\gamma$};
\node at (1.6,3.7) {\small $\delta$};
\node at (-1,3.7) {\small $\epsilon$};

\node at (0,-2) {\small $\epsilon^3$};

\draw
	(3.1,1.5) -- (3.1,2.5) -- (0,2.5) -- (0,1.8)  -- (-1.2,2.5) -- (-1.2,3.2) -- (-3.1,3.2) -- (-3.1,0.6)
	(-1.2,3.2) -- (-1.2,3.9) -- (1.8,3.9) -- (1.8,3.2) -- (3.1,2.5);

\draw[line width=1.5]
	(1.8,1.8) -- (1.8,2.5)
	(-2.4,1.8) -- (-2.4,2.5)
	(-1.2,2.5) -- (-3.1,2.5)
	(-1.2,3.2) -- (1.8,3.2)
	(0,3.2) -- (0,2.5);

\node[draw,shape=circle, inner sep=0.5] at (-1.1,0.55) {\small 1};
\node[draw,shape=circle, inner sep=0.5] at (0,0.4) {\small 2};
\node[draw,shape=circle, inner sep=0.5] at (1.1,0.55) {\small 3};
\node[draw,shape=circle, inner sep=0.5] at (-1.1,-0.55) {\small 4};
\node[draw,shape=circle, inner sep=0.5] at (0,-0.4) {\small 5};
\node[draw,shape=circle, inner sep=0.5] at (1.1,-0.55) {\small 6};
\node[draw,shape=circle, inner sep=0.5] at (-1.9,1.3) {\small 7};
\node[draw,shape=circle, inner sep=0.5] at (-0.6,1.4) {\small 8};
\node[draw,shape=circle, inner sep=0.5] at (0.6,1.4) {\small 9};
\node[draw,shape=circle, inner sep=0] at (1.9,-1.3) {\small 10};
\node[draw,shape=circle, inner sep=0] at (2,0) {\small 11};
\node[draw,shape=circle, inner sep=0] at (-2,0) {\small 12};
\node[draw,shape=circle, inner sep=0] at (1.75,1.2) {\small 13};
\node[draw,shape=circle, inner sep=0] at (-2.75,1.3) {\small 14};
\node[draw,shape=circle, inner sep=0] at (0.6,2.15) {\small 15};
\node[draw,shape=circle, inner sep=0] at (-1.7,2.15) {\small 16};
\node[draw,shape=circle, inner sep=0] at (-0.6,2.7) {\small 17};

\foreach \a in {-1,1}
{
\begin{scope}[scale=\a]

\fill 
	(0.6,0) circle (0.1);

\filldraw[fill=white] 
	(1.7,0.6) circle (0.1)
	(2.4,-0.6) circle (0.1);
	
\end{scope}
}

\end{tikzpicture}
\caption{$\{\alpha^3,\beta\gamma\delta,\delta^2\epsilon\}$: $\beta\gamma\epsilon^2[1]$, $\beta\gamma\epsilon^2[2]$, and tiling starting from $\beta\gamma\epsilon^2[1]$.}
\label{special2D}
\end{figure}

\subsubsection*{Case. There is $\beta\gamma\epsilon^2[1]$}

On the right of Figure \ref{special2D}, we start with $T_1,T_2,T_4,T_5$ around $\beta\gamma\epsilon^2[1]$. This determines three $n(\alpha^3)$, including $T_3,T_7,T_8,T_9$. Then $\thin^{\delta}\beta_3^{\alpha}\thick^{\alpha}\gamma_2^{\epsilon}\thin^{\gamma}\epsilon_5^{\delta}\thin\cdots=\beta\gamma\epsilon^2[1]$ determines $T_6$. This further determines one more $n(\alpha^3)$, including $T_{10}$. We find that $\beta\gamma\epsilon^2[1]$ appear in pairs, and $T_1,T_2,T_3,T_4,T_5,T_6$ has $180^{\circ}$ rotation symmetry. Therefore further development of the tiling has the same symmetry. Any new tile automatically gives another tile by $180^{\circ}$ rotation.

We have $\delta_3\delta_6\cdots=\delta^2\epsilon=\delta_3\delta_6\epsilon_{11}$. By $\epsilon_{11}$, we know $\beta_6\gamma_{10}\cdots=\beta\gamma\delta,\beta\gamma\epsilon^2$ is $\beta_6\gamma_{10}\delta_{11}$. Then $\delta_{11},\epsilon_{11}$ determine $T_{11}$, and further determines one more $n(\alpha^3)$. By the $180^{\circ}$ rotation symmetry, we get $T_{12}$ and the associated $n(\alpha^3)$. Then $\thin^{\delta}\beta^{\alpha}\thick^{\alpha}\gamma_{11}{}^{\epsilon}\thin^{\delta}\epsilon_3^{\gamma}\thin\cdots=\beta\gamma\epsilon^2[2]$ and  $\thin^{\epsilon}\gamma^{\alpha}\thick^{\alpha}\beta_{12}{}^{\delta}\thin^{\gamma}\epsilon_7^{\delta}\thin\cdots=\beta\gamma\epsilon^2[2]$ determine $T_{13}, T_{14}$, and further determine two $n(\alpha^3)$, including $T_{15},T_{16}$. Then we get $\epsilon_8\epsilon_9\epsilon_{15}\epsilon_{16}\cdots=\epsilon^5=\epsilon_8\epsilon_9\epsilon_{15}\epsilon_{16}\epsilon_{17}$. By $\epsilon_{17}$, we know $\gamma_{16}\cdots=\beta\gamma\delta,\beta\gamma\epsilon^2$ is $\beta\gamma_{16}\delta_{17}$. Then $\delta_{17},\epsilon_{17}$ determine $T_{17}$, and further determines one more $n(\alpha^3)$. The five $n(\alpha^3)$ around $\epsilon^5=\epsilon_8\epsilon_9\epsilon_{15}\epsilon_{16}\epsilon_{17}$ form the extended neighborhood $N(\epsilon^5)$ in Figure \ref{e-nd}. The boundary of $N(\epsilon^5)$ is given by shaded edges in Figure \ref{special2D}.

By the $180^{\circ}$ rotation symmetry, we have another $\epsilon^5$. Then we have another five $n(\alpha^3)$ around this $\epsilon^5$ that form another extended neighborhood $N(\epsilon^5)$. Part of the boundary of this $N(\epsilon^5)$ is given by shaded edges in Figure \ref{special2D}. 

The two $N(\epsilon^5)$ share the pair of $\beta\gamma\epsilon^2[1]$ along the common boundary part. Moreover, we have four $\beta\gamma\epsilon^2[2]$ along the non-common boundary part of two $N(\epsilon^5)$. The four $\beta\gamma\epsilon^2[2]$ also belong to two $n(\alpha^3)$ (at $\alpha_{11}\cdots$ and $\alpha_{12}\cdots$) between the two $N(\epsilon^5)$. See the first of Figure \ref{special2F}, where the two $n(\alpha^3)$ between the two $N(\epsilon^5)$ are shaded.

\begin{figure}[htp]
\centering
\begin{tikzpicture}[>=latex,scale=1]


\begin{scope}[xshift=-5cm]

\foreach \a in {1,-1}
\fill[gray!50, scale=\a]
	(1.5,1.2) -- (1.5,-1.2) -- (0.6,0);

\foreach \a in {1,-1}
{
\begin{scope}[scale=\a]

\draw
	(0,0) -- (-0.6,0) -- (-1.5,1.2) -- (0,1.8) -- (1.5,1.2) -- (0.6,0)
	(1.5,1.2) -- (1.5,-1.2);

\draw[line width=1.5]
	(1.2,-0.8) -- (1.2,0) -- (0.9,0.4)
	(1.2,0) -- (1.5,0.4);
	
\fill 
	(0.2,0) circle (0.1);

\node at (0,0.9) {$N(\epsilon^5)$};
	
\end{scope}
}

\foreach \a in {1,-1}
\filldraw[scale=\a,fill=white]
	(1.2,-0.8) circle (0.1)
	(0.9,0.4) circle (0.1);

\end{scope}


\foreach \a in {0,1,2}
\fill[gray!50, rotate=120*\a]
	(0,0) -- (90:0.9) -- (-30:0.9)
	(90:1.8) -- (30:1.8) to[out=90, in=0] 
	(90:2.2) to[out=180, in=90] 
	(150:1.8);
	
\foreach \a in {0,1,2}
{
\begin{scope}[rotate=120*\a]

\draw
	(30:1.8) -- (90:1.8) -- (90:0.9) -- (-30:0.9) -- (-30:1.8) -- (30:1.8)
	(30:1.8) to[out=90, in=0] 
	(90:2.2) to[out=180, in=90] 
	(150:1.8);

\draw[line width=1.5]
	(0,0) -- (180:0.52)
	(50:1.58) to[out=100, in=-20] 
	(80:1.9) to[out=160, in=80] (110:1.58)
	(80:1.9) to[out=100, in=-70] (110:2.18)
	(70:2.18) -- (70:2.5);

\fill 
	(90:1.2) circle (0.1)
	(90:1.5) circle (0.1);

\node at (90:0.65) {\scriptsize $\epsilon$};
\node at (80:0.9) {\scriptsize $\delta$};
\node at (100:0.9) {\scriptsize $\delta$};

\node at (30:1.6) {\scriptsize $\delta$};
\node at (22:1.8) {\scriptsize $\epsilon$};
\node at (38:1.8) {\scriptsize $\epsilon$};

\node at (30:2) {\scriptsize $\epsilon$};
	
\end{scope}
}

\foreach \a in {0,1,2}
\filldraw[rotate=120*\a, fill=white]
	(180:0.52) circle (0.1);

\foreach \a in {0,...,5}
{
	
\filldraw[rotate=60*\a, fill=white]
	(-10:1.58) circle (0.1);
}	
	
\node at (-90:1) {\small $N_1$};
\node at (30:1) {\small $N_2$};
\node at (150:1) {\small $N_3$};

\end{tikzpicture}
\caption{$\{\alpha^3,\beta\gamma\delta,\delta^2\epsilon\}$: Pairs of $\beta\gamma\epsilon^2[1]$ on the boundary of $N(\epsilon^5)$.}
\label{special2F}
\end{figure}

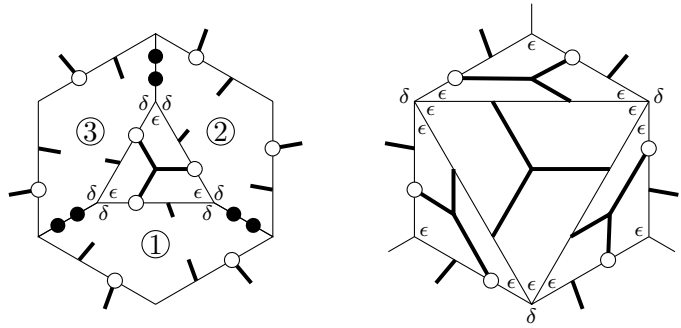
If two pairs of $\beta\gamma\epsilon^2[1]$ appear on the boundary of one $N(\epsilon^5)$, then we get the second of Figure \ref{special2F}, where $N_1(\epsilon^5)$ has two pairs of $\beta\gamma\epsilon^2[1]$, shared with $N_2(\epsilon^5)$ and $N_3(\epsilon^5)$. For the moment, we do not yet know that $N_2(\epsilon^5)$ and $N_3(\epsilon^5)$ touch each other. However, one $n(\alpha^3)$ between $N_1(\epsilon^5),N_2(\epsilon^5)$ is also the one $n(\alpha^3)$ between $N_1(\epsilon^5),N_3(\epsilon^5)$. This $n(\alpha^3)$ is the shaded triangle in the center, and forces $N_2(\epsilon^5)$ and $N_3(\epsilon^5)$ to be glued together. Moreover, $N_2(\epsilon^5)$ and $N_3(\epsilon^5)$ also share a pair of $\beta\gamma\epsilon^2[1]$. 

In the second of Figure \ref{special2F}, we also see the other $n(\alpha^3)$ between $N_1,N_2$, between $N_1,N_3$, and between $N_2,N_3$. They are the three shaded triangles outside the hexagon. Then we get three vertices $\delta\epsilon^2\cdots=\delta\epsilon^3$. They give three ``outside'' $\epsilon$, which are the three angles of the outside triangle. This triangle should be divided into $60-3\times 15 - 4\times 3=3$ pentagons. Therefore the tiling is completed by the pentagonal subdivision $n(\alpha^3)$ of this outside triangle. What we get is the tiling constructed according to the fourth of Figure \ref{mod60B}, or the tiling $T(15\beta\gamma\epsilon^2,3\delta\epsilon^3,3\epsilon^5)$ in the second of Figure \ref{subdivision_tiling_modify60B}. 

Now we may assume no two pairs of $\beta\gamma\epsilon^2[1]$ on the boundary of any $N(\epsilon^5)$. This means that there is no $\beta\gamma\epsilon^2[1]$ on the top two and bottom two edges of the first of Figure \ref{special2F}. We draw the boundary of the first of Figure \ref{special2G}, as the inside out of the first of Figure \ref{special2F}. Then $\beta\thick\gamma\cdots$ on the top two and bottom two edges of the first of Figure \ref{special2G} cannot be $\beta\gamma\epsilon^2[1]$. We label some $\delta$ on the boundary by numbers for further use.

\begin{figure}[htp]
\centering
\begin{tikzpicture}[>=latex,scale=1]


\foreach \a in {1,-1}
{
\begin{scope}[scale=\a]

\fill[gray!20]
	(0,-1.8) -- (0.9,0) -- (2.4,-1.2);

\fill[gray!50]
	(0,1.8) -- (0.9,0) -- (2.4,-1.2) -- (2.4,1.2);

\end{scope}
}

\foreach \a in {1,-1}
{
\begin{scope}[scale=\a]

\draw
	(0,1.8) -- (-0.9,0) -- (0.9,0) -- (0,1.8)	
	(0,1.8) -- (2.4,1.2) -- (2.4,-1.2) -- (0,-1.8)
	(2.4,1.2) -- (0.9,0) -- (2.4,-1.2)
	(2.4,1.2) -- ++(0.3,0.3)
	(2.4,-1.2) -- ++(0.3,-0.3)
	(-1.6,1.4) -- ++(-0.1,0.4)
	(0.8,1.6) -- ++(0.1,0.4);

\draw[line width=1.5]
	(0.3,0) -- (0,0.6) -- (0.3,1.2)
	(-0.6,0.6) -- (0,0.6)
	(0.6,0.6) -- (1.1,1) -- (1.9,0.8)
	(1.1,1) -- (0.8,1.6)
	(-0.3,1.2) -- (-1.1,1) -- (-1.4,0.4)
	(-1.1,1) -- (-1.6,1.4)
	(1.4,0.4) -- (1.9,0) -- (1.9,-0.8)
	(1.9,0) -- (2.4,0.4)
	(2.4,-0.4) -- ++(0.4,0)
	(1.6,1.4) -- ++(0.1,0.4)
	(-0.8,1.6) -- ++(-0.1,0.4);

\node at (2.6,0.5) {\small $\delta$};
\node at (2.6,-0.15) {\small $\beta$};
\node at (2.6,-0.6) {\small $\gamma$};

\node at (1,1.75) {\small $\epsilon$};
\node at (0.7,1.8) {\small $\epsilon$};

\node at (1.85,1.55) {\small $\beta$};
\node at (1.45,1.65) {\small $\gamma$};

\node at (-1.05,1.75) {\small $\gamma$};
\node at (-0.65,1.85) {\small $\beta$};

\node at (-1.8,1.55) {\small $\epsilon$};
\node at (-1.5,1.6) {\small $\epsilon$};

\node at (0,1.5) {\small $\epsilon$};
\node at (1.15,0) {\small $\epsilon$};

\end{scope}
}

\foreach \a in {1,-1}
\foreach \b in {1,-1}
{
\begin{scope}[xscale=\a, yscale=\b]

\node at (2.6,1.15) {\small $\epsilon$};
\node at (2.25,0.9) {\small $\epsilon$};
\node at (2.05,1.1) {\small $\epsilon$};
\node at (0.3,1.55) {\small $\epsilon$};
\node at (0.95,0.3) {\small $\epsilon$};
\node at (0.65,0.2) {\small $\epsilon$};

\end{scope}
}

\node at (2.35,1.5) {\small $\delta_1$};
\node at (-2.35,-1.5) {\small $\delta_2$};
\node at (-2.35,1.5) {\small $\delta_3$};
\node at (2.35,-1.5) {\small $\delta_4$};
\node at (0,2) {\small $\delta_5$};
\node at (0,-2) {\small $\delta_6$};


\begin{scope}[xshift=5cm]


\foreach \a in {0,...,4}
\draw
	(0,0) -- (72*\a:0.5);

\foreach \a in {1,2,3}
{
\draw[rotate=72*\a]
	(0:0.5) -- (20:0.75) -- (52:0.75) -- (72:0.5)
	(20:0.75) -- (16:1.1) -- (36:1.3) -- (56:1.1);
	
\draw[rotate=72*\a,line width=1.5]
	(20:0.75) -- (52:0.75) -- (72:0.5)
	(52:0.75) -- (56:1.1);
}

\foreach \a in {2,3}
\draw[rotate=72*\a]
	(16:1.1) -- (-16:1.1);
	
\draw
	(0.5,0) -- (0.7,0.3)
	(0.7,0.9) -- (88:1.1)
	(0.7,0.9) -- (1.05,0.78)
	(1.15,1.4) -- (88:1.1);
	
\draw[line width=1.5]
	(72:0.5) -- (0.5,0.6) -- (0.7,0.3)
	(0.5,0.6) -- (0.7,0.9);
	

\begin{scope}[shift={(1.2cm,0.3cm)}, xscale=-1]

\foreach \a in {0,...,4}
\draw
	(0,0) -- (72*\a:0.5);

\foreach \a in {1,2,3}
{
\draw[rotate=72*\a]
	(0:0.5) -- (20:0.75) -- (52:0.75) -- (72:0.5)
	(20:0.75) -- (16:1.1) -- (36:1.3) -- (56:1.1)
	(52:0.75) -- (56:1.1);
	
\draw[rotate=72*\a,line width=1.5]
	(0:0.5) -- (20:0.75) -- (52:0.75)
	(20:0.75) -- (16:1.1);
}

\foreach \a in {2,3}
\draw[rotate=72*\a]
	(16:1.1) -- (-16:1.1);
		
\foreach \a in {0,...,4}
\draw
	(0,0) -- (72*\a:0.5);
	
\draw
	(0.5,0) -- (0.7,-0.3)
	(0.7,-0.9) -- (-88:1.1)
	(0.7,-0.9) -- (1.05,-0.78)
	(1.15,-1.4) -- (-88:1.1);
		
\draw[line width=1.5]
	(-72:0.5) -- (0.5,-0.6) -- (0.7,-0.3)
	(0.5,-0.6) -- (0.7,-0.9);
			
\filldraw[fill=white]
	(232:1.1) circle (0.1)
	(88:1.1) circle (0.1);
				
\end{scope}

\filldraw[fill=white]
	(128:1.1) circle (0.1)
	(272:1.1) circle (0.1);

\end{scope}

\end{tikzpicture}
\caption{$\{\alpha^3,\beta\gamma\delta,\delta^2\epsilon\}$: Complete the tiling for the first of Figure \ref{special2F}.}
\label{special2G}
\end{figure}

We have $\delta_1\cdots=\delta\epsilon\cdots=\delta^2\epsilon,\delta\epsilon^3$. If $\delta_1\cdots=\delta^2\epsilon$, then the vertex $\beta\thick\gamma\cdots=\beta\gamma\delta,\beta\gamma\epsilon^2$ next to $\delta_1\cdots$ is $\beta\gamma\epsilon^2=\thin^{\gamma}\epsilon^{\delta}\thin^{\delta}\beta^{\alpha}\thick^{\alpha}\gamma^{\epsilon}\thin\cdots=\beta\gamma\epsilon^2[1]$. By the assumption $\beta\thick\gamma\cdots\ne\beta\gamma\epsilon^2[1]$, therefore, we conclude $\delta_1\cdots=\delta\epsilon^3$. The unique AAD $\thin^{\beta}\delta^{\epsilon}\thin^{\gamma}\epsilon^{\delta}\thin^{\gamma}\epsilon^{\delta}\thin^{\gamma}\epsilon^{\delta}\thin$ of $\delta\epsilon^3$ determines the tiles around $\delta_1\cdots$, and further determines two $n(\alpha^3)$. The same argument shows $\delta_2\cdots=\delta\epsilon^3$ and determines two more $n(\alpha^3)$. The four $n(\alpha^3)$ are the heavily shaded tiles.

Then we find $\delta_3\cdots,\delta_4\cdots=\delta\epsilon^2\cdots=\delta\epsilon^3$. Again the unique AAD of $\delta\epsilon^3$ determines the tiles around $\delta_3\cdots$, $\delta_4\cdots$, and further determines two more $n(\alpha^3)$, which are the lightly shaded tiles. 

Then we find $\delta_5\cdots,\delta_6\cdots=\delta\epsilon^2\cdots=\delta\epsilon^3$. The unique AAD of $\delta\epsilon^3$ determines the tiles around $\delta_5\cdots$, $\delta_6\cdots$, and further determines two more $n(\alpha^3)$, which are the remaining white tiles.

The tiling in the first of Figure \ref{special2G} is the union of two overlapping $N(\epsilon^5)$, illustrated by the second of Figure \ref{special2G}. The whole tiling is then obtained by glueing Figure \ref{special2G} and the first of Figure \ref{special2F} together. This is the tiling constructed according to the third of Figure \ref{mod60B}, or the tiling $T(10\beta\gamma\epsilon^2,6\delta\epsilon^3,4\epsilon^5)$ in the first of Figure \ref{subdivision_tiling_modify60B}.

\subsubsection*{Case. There is no $\beta\gamma\epsilon^2[1]$}

In Figure \ref{special2H}, we start with $T_1,T_2,T_3,T_4$ around $\beta\gamma\epsilon^2[2]$. They determine three $n(\alpha^3)$, including $T_5,T_6,T_7$. We have $\delta_2\epsilon_4\cdots=\delta^2\epsilon,\delta\epsilon^3$. If $\delta_2\epsilon_4\cdots=\delta^2\epsilon$, then we have $\beta_2\gamma_7\cdots=\thin^{\epsilon}\gamma_7^{\alpha}\thick^{\alpha}\beta_2^{\delta}\thin^{\delta}\epsilon^{\gamma}\thin\cdots=\beta\gamma\epsilon^2[1]$. By no $\beta\gamma\epsilon^2[1]$, this implies $\delta_2\epsilon_4\cdots=\delta\epsilon^3$. The unique AAD of $\delta\epsilon^3$ determines $T_8,T_9$. This further determines two more $n(\alpha^3)$, including $T_{10}$. Then $\beta_9\gamma_{10}\epsilon_7\cdots=\beta\gamma\epsilon^2[2]$ determines $T_{11}$. Now our argument that started with $\beta_3\gamma_4\epsilon_1\epsilon_2=\beta\gamma\epsilon^2[2]$ can be repeated at the new $\beta_9\gamma_{10}\epsilon_7\epsilon_{11}=\beta\gamma\epsilon^2[2]$. More repetitions give us an extended neighborhood $N(\epsilon^5)$ around $\epsilon_5\epsilon_6\cdots=\epsilon^5$, and an annulus ${\mc A}$ which is the five times repetition of the two $n(\alpha^3)$ around $\alpha_3\alpha_4\cdots$ and $\alpha_8\cdots$. The annulus consists of ten $n(\alpha^3)$.

Note that the tiles in ${\mc A}$ are positively oriented, and the tiles in $N(\epsilon^5)$ are negatively oriented. Therefore the extended neighborhood in Figure \ref{special2H} is the flip (with respect to $L_{\delta=2\epsilon}$) of the extended neighborhood in Figure \ref{e-nd}.

\begin{figure}[htp]
\centering
\begin{tikzpicture}[>=latex,scale=1]

\draw[gray!50, line width=3]
	(-2,0) -- (8.5,0)
	(-2,-2.4) -- (8.5,-2.4);

\draw
	(-2,0) -- (8.5,0)
	(-2,-2.4) -- (8.5,-2.4)
	(1,-2.4) -- (5,-2.4);

\draw[line width=1.5]
	(1,-1.6) -- ++(-0.4,0)
	(5.5,-0.8) -- ++(0.25,-0.4);

\foreach \a in {0,1}
{
\begin{scope}[xshift=3*\a cm]
	
\draw[line width=1.5]
	(2,0) -- (1.5,-0.8) -- (0.5,-0.8)
	(1.5,-0.8) -- (2,-1.6)
	;

\node at (1.4,-0.6) {\small $\alpha$};
\node at (1.7,-0.2) {\small $\beta$};
\node at (0.6,-0.6) {\small $\gamma$};
\node at (1,-0.2) {\small $\delta$};

\node at (1.7,-0.8) {\small $\alpha$};
\node at (2.05,-1.2) {\small $\beta$};
\node at (2.1,-0.25) {\small $\gamma$};
\node at (2.3,-0.8) {\small $\delta$};

\node at (1.4,-1) {\small $\alpha$};
\node at (0.9,-1.05) {\small $\beta$};
\node at (1.8,-1.65) {\small $\gamma$};
\node at (1.25,-1.6) {\small $\delta$};
\node at (1.5,-2.1) {\small $\epsilon$};

\end{scope}
}

\foreach \a in {0,1,2}
{
\begin{scope}[xshift=3*\a cm]

\draw
	(-1.5,-2.4) -- (0,0) -- (1.5,-2.4);
	
\draw[line width=1.5]
	(-1,0.8) -- (1,0.8) -- (1,0)
	(1,0.8) -- (1.5,1.2)
	(-0.5,-0.8) -- (0,-1.6) -- (-0.5,-2.4)
	(0,-1.6) -- (1,-1.6);
			
\node at (0.9,1) {\small $\alpha$};
\node at (-0.9,1) {\small $\beta$};
\node at (1.3,1.3) {\small $\gamma$};
\node at (-1.3,1.3) {\small $\delta$};
\node at (0,1.5) {\small $\epsilon$};

\node at (1.2,0.7) {\small $\alpha$};
\node at (1.5,0.9) {\small $\beta$};
\node at (1.2,0.2) {\small $\gamma$};
\node at (1.8,0.7) {\small $\delta$};
\node at (1.8,0.2) {\small $\epsilon$};

\node at (0.8,0.6) {\small $\alpha$};
\node at (0.8,0.2) {\small $\beta$};
\node at (-0.8,0.6) {\small $\gamma$};
\node at (0,0.25) {\small $\delta$};
\node at (-0.8,0.2) {\small $\epsilon$};

\node at (0.1,-1.4) {\small $\alpha$};
\node at (0.65,-1.4) {\small $\beta$};
\node at (-0.3,-0.8) {\small $\gamma$};
\node at (0.3,-0.8) {\small $\delta$};

\node at (0.1,-1.8) {\small $\alpha$};
\node at (-0.2,-2.2) {\small $\beta$};
\node at (0.9,-1.8) {\small $\gamma$};
\node at (0.5,-2.2) {\small $\delta$};
\node at (1.2,-2.25) {\small $\epsilon$};

\node at (-0.2,-1.6) {\small $\alpha$};
\node at (-0.55,-1.15) {\small $\beta$};
\node at (-0.6,-2.2) {\small $\gamma$};
\node at (-0.85,-1.6) {\small $\delta$};
\node at (-1.2,-2.25) {\small $\epsilon$};

\node at (0,-0.3) {\small $\epsilon$};
\node at (0.3,-0.15) {\small $\epsilon$};
\node at (-0.3,-0.15) {\small $\epsilon$};

\end{scope}
}

\foreach \a in {0,1,2,3}
{
\draw[xshift=3*\a cm]
	(-1,0) -- (-1,0.8) -- (-1.5,1.2) -- (-1.5,1.6);
		
\fill 
	(-1.5+3*\a,1.6) -- ++(-54:0.1) arc (-54:-126:0.1);
}

\foreach \a in {0,1,2}
{
\filldraw[fill=white] 
	(2+3*\a,0) circle (0.1);

\filldraw[fill=gray] 
	(3*\a,0) circle (0.1);
}

\node[draw,shape=circle, inner sep=0.5] at (1.5,0.45) {\small 1};
\node[draw,shape=circle, inner sep=0.5] at (2.6,0.4) {\small 2};
\node[draw,shape=circle, inner sep=0.5] at (1,-0.55) {\small 3};
\node[draw,shape=circle, inner sep=0.5] at (2,-0.6) {\small 4};
\node[draw,shape=circle, inner sep=0.5] at (0,1.1) {\small 5};
\node[draw,shape=circle, inner sep=0.5] at (3,1.1) {\small 6};
\node[draw,shape=circle, inner sep=0.5] at (4.5,0.45) {\small 7};
\node[draw,shape=circle, inner sep=0.5] at (3,-1) {\small 8};
\node[draw,shape=circle, inner sep=0.5] at (4,-0.55) {\small 9};
\node[draw,shape=circle, inner sep=0] at (5,-0.6) {\small 10};
\node[draw,shape=circle, inner sep=0] at (5.6,0.4) {\small 11};

\node at (-1.7,0.6) {\small $N(\epsilon^5)$};
\node at (-1.7,-1.2) {\small ${\mc A}$};

\end{tikzpicture}
\caption{$\{\alpha^3,\beta\gamma\delta,\delta^2\epsilon\}$: Tiling in the absence of $\beta\gamma\epsilon^2[1]$.}
\label{special2H}
\end{figure}

The annulus ${\mc A}$ has two boundaries, indicated by the shaded lines. The lower boundary consists of repetitions of $\beta\thick\gamma,\delta,\epsilon\thin\epsilon\thin\epsilon$, the same as the upper boundary. If one $\beta\thick\gamma\cdots=\beta\gamma\epsilon^2=\beta\gamma\epsilon^2[2]$, then the argument that gives us flipped $N(\epsilon^5)$ glued to the upper boundary of ${\mc A}$ can be applied to give another flipped $N(\epsilon^5)$ glued to the lower boundary of ${\mc A}$. If one $\beta\thick\gamma\cdots=\beta\gamma\delta$, then the adjacent $\epsilon\thin\epsilon\thin\epsilon\cdots=\delta\epsilon^3,\epsilon^5$ must be $\epsilon^5$, and we get the standard (i.e., non-flipped) $N(\epsilon^5)$ around $\epsilon^5$. This implies the next $\beta\thick\gamma\cdots=\beta\gamma\delta$, and the argument continues. More repetition gives us the standard $N(\epsilon^5)$ glued to the lower boundary of ${\mc A}$. 

In summary, if there is no $\beta\gamma\epsilon^2[1]$, and there is $\beta\gamma\epsilon^2[2]$, then we have two extended neighborhoods of $\epsilon^5$ glued to the two boundaries of the annulus ${\mc A}$. If both extended neighborhoods are flipped, then we get the double flip modification $T(10\beta\gamma\epsilon^2,10\delta\epsilon^3,2\epsilon^5)$ of the pentagonal subdivision tiling (the sixth of Figure \ref{subdivision_tiling_modify60A}). If one extended neighborhood is flipped, and the other is not, then we get the single flip modification $T(5\beta\gamma\epsilon^2,5\delta\epsilon^3,7\epsilon^5)$ of the pentagonal subdivision tiling (the fifth of Figure \ref{subdivision_tiling_modify60A}). 

Finally, we assume $\beta\gamma\epsilon^2$ is not a vertex. Since the unique AAD of $\delta\epsilon^3$ gives $\gamma\epsilon\cdots=\beta\gamma\epsilon^2$, this implies $\delta\epsilon^3$ is not a vertex. Therefore the AVC \eqref{special2_60A} is reduced to $\{\alpha^3,\beta\gamma\delta,\delta^2\epsilon,\epsilon^5\}$. Then the argument that uses Figure \ref{special2C} to derive the pentagonal subdivision $T(6\epsilon^4)$ of the octahedron for \eqref{special2_24A} can be used to gives us the pentagonal subdivision $T(12\epsilon^5)$ (the fifth of Figure \ref{subdivision_tiling}) of the icosahedron.

\section{General Cases}
\label{general}

\begin{proposition}\label{special3}
There is no tiling of the sphere by congruent non-symmetric pentagons in Figure \ref{pentagon}, such that $\alpha^3,\beta^2\delta,\gamma^2\epsilon$ are vertices. 
\end{proposition}

\begin{proof}
The angle sums of $\alpha^3, \beta^2\delta, \gamma^2\epsilon$ and the angle sum for pentagon imply 
\[
\alpha=\tfrac{2}{3}\pi,\;
\beta+\gamma=(\tfrac{5}{3}-\tfrac{4}{f})\pi,\;
\delta+\epsilon=(\tfrac{2}{3}+\tfrac{8}{f})\pi.
\]
Due to the symmetry of exchanging $\beta \leftrightarrow \gamma$ and $\delta \leftrightarrow \epsilon$, by Lemma \ref{geometry1}, we may assume $\beta<\gamma$ and $\delta>\epsilon$. By $f>12$, we have $\alpha+\beta+\gamma>2\pi$, $\beta+\gamma>\frac{4}{3}\pi>\delta+\epsilon$.

By $\beta^2\delta$, we have $3\beta+\gamma>2\beta+\delta+\epsilon>2\pi$. By $\beta<\gamma$ and the parity lemma, this implies $R(\beta\gamma\cdots)$ and $R(\gamma^2\cdots)$ have no $\beta,\gamma$. By $\beta^2\delta$, $\beta<\gamma$, and the parity lemma, we know $\gamma\delta\cdots$ is not a vertex. Therefore $R(\beta\gamma\cdots)$ and $R(\gamma^2\cdots)$ have no $\beta,\gamma,\delta$.

By Lemma \ref{aadlemma}, the vertex $\alpha^3$ implies $\beta\thick\gamma\cdots$ is a vertex. Then by no $\beta,\gamma,\delta$ in $R(\beta\gamma\cdots)$, and Lemma \ref{klem5}, we get $\beta\thick\gamma\cdots=\beta\gamma\epsilon^k$. By $\gamma^2\epsilon$ and $\beta<\gamma$, we get $k\ge 2$. Therefore we have $\thick\beta\thin\epsilon\thin\epsilon\thin$ at $\beta\gamma\epsilon^k$. By no $\gamma\delta\cdots$, the AAD of $\thick\beta\thin\epsilon\thin\epsilon\thin$ is $\thick^{\alpha}\beta^{\delta}\thin^{\delta}\epsilon^{\gamma}\thin^{\gamma}\epsilon^{\delta}\thin$. Therefore $\gamma\thin\gamma\cdots$ is a vertex. By no $\beta,\gamma$ in $R(\gamma^2\cdots)$, and Lemma \ref{klem5}, we get $\gamma\thin\gamma\cdots=\alpha^k\gamma^2$. This contradicts $\alpha+2\gamma>\alpha+\beta+\gamma>2\pi$. 
\end{proof}

\begin{proposition}\label{special4}
There is no tiling of the sphere by congruent non-symmetric pentagons in Figure \ref{pentagon}, such that $\alpha^3,\beta^2\delta,\delta^2\epsilon$ are vertices. 
\end{proposition}

\begin{proof}
The angle sums of $\alpha^3,\delta^2\epsilon$ imply $\alpha=\frac{2}{3}\pi$ and $\delta<\pi$. Then the angle sum of $\beta^2\delta$ further implies $\frac{1}{2}\pi<\beta<\pi$. Moreover, the angle sums of $\alpha^3,\beta^2\delta,\delta^2\epsilon$ and the angle sum for pentagon imply $\gamma=(\frac{5}{6}+\tfrac{4}{f})\pi-\frac{3}{4}\epsilon$. By Lemma \ref{geometry1}, we need to consider two cases: $\beta<\gamma,\delta>\epsilon$, and $\beta>\gamma,\delta<\epsilon$.

\subsubsection*{Case. $\beta<\gamma$ and $\delta>\epsilon$}

By $\gamma>\beta>\frac{1}{2}\pi$ and the parity lemma, there are at most two $ab$-angles at a vertex. By $\beta^2\delta$, $\beta<\gamma$, and the parity lemma, we know $\gamma\delta\cdots$ is not a vertex.

By $\alpha^3$, we have $R(\alpha^2\cdots)=\alpha<2\beta$. By $\beta<\gamma$ and the parity lemma, this implies $\alpha^2\cdots$ has no $\beta,\gamma$. Then by Lemma \ref{klem4}, we get $\alpha^2\cdots=\alpha^3$. 

The AAD $\thin^{\delta}\beta^{\alpha}\thick^{\alpha}\beta^{\delta}\thin$ at $\beta^2\delta$ gives $\thick^{\gamma}\alpha^{\beta}\thick^{\beta}\alpha^{\gamma}\thick\cdots=\alpha^3$. By Lemma \ref{aadlemma}, this implies $\beta\thick\gamma\cdots$ and $\gamma\thick\gamma\cdots$ are vertices. By at most two $ab$-angles at a vertex, we know $R(\beta\thick\gamma\cdots)$ and $R(\gamma\thick\gamma\cdots)$ have no $\beta,\gamma$. By $\beta^2\delta$ and $\beta<\gamma$, we have $R(\gamma\thick\gamma\cdots)<R(\beta\thick\gamma\cdots)<\delta$. Therefore the remainders have no $\delta$. By Lemma \ref{klem5}, we get $\beta\thick\gamma\cdots=\beta\gamma\epsilon^k$ and $\gamma\thick\gamma\cdots=\gamma^2\epsilon^l$, with $1\le l<k$. 

If $k=2$, then $l=1$. The angle sums of $\alpha^3,\beta^2\delta,\delta^2\epsilon,\beta\gamma\epsilon^2,\gamma^2\epsilon$ and the angle sum for pentagon imply $f=\tfrac{84}{5}$, a contradiction. (Another way to show $l\ne 1$ is by using Proposition \ref{special3}.) Therefore $k\ge 3$, and we have $\thin\epsilon\thin\epsilon\thin\epsilon\thin$ at $\beta\gamma\epsilon^k$. By no $\gamma\delta\cdots$, the AAD of $\thin\epsilon\thin\epsilon\thin\epsilon\thin$ implies a vertex $\gamma\thin\gamma\cdots$. By at most two $ab$-angles at a vertex, we may apply Lemma \ref{klem5} to $\gamma\thin\gamma\cdots$. Then by $\alpha^2\cdots=\alpha^3$, we have $\gamma\thin\gamma\cdots=\alpha\gamma^2$. Then the angle sums of $\alpha^3,\alpha\gamma^2,\beta^2\delta,\delta^2\epsilon$ and the angle sum for pentagon imply 
\[
\alpha=\gamma=\tfrac{2}{3}\pi,\;
\beta=(\tfrac{5}{9} + \tfrac{4}{3f})\pi, \;
\delta=(\tfrac{8}{9} - \tfrac{8}{3f})\pi,  \;
\epsilon=(\tfrac{2}{9} + \tfrac{16}{3f})\pi. 
\]
By $f>12$, we have $\epsilon<\tfrac{2}{3}\pi<3\epsilon$. This implies $\gamma\thick\gamma\cdots=\gamma^2 \epsilon^l=\gamma^2\epsilon^2$. The angle sum of $\gamma^2\epsilon^2$ further implies $\beta=\tfrac{7}{12}\pi$, $\epsilon=\tfrac{1}{3}\pi$, and $f=48$. Since the solution of $\beta+\gamma+k\epsilon=2\pi$ is $k=\frac{9}{4}$, we know $\beta\thick\gamma\cdots=\beta\gamma\epsilon^k$ is not a vertex. This is a contradiction.

\subsubsection*{Case. $\beta>\gamma$ and $\delta<\epsilon$}

The angle sums of $\beta^2\delta,\delta^2\epsilon$ imply $2\beta=\delta+\epsilon$. By $\delta<\epsilon$, this implies $\delta<\beta<\epsilon$. By $\beta^2\delta$ and $\beta>\delta$, we get $\beta>\frac{2}{3}\pi$. By $\delta^2\epsilon$ and $\delta<\epsilon$, we get $\delta<\frac{2}{3}\pi=\alpha<\epsilon$.

By $\delta^2\epsilon$ and $\delta<\epsilon$, we have $R(\epsilon^2\cdots)<\delta<\alpha,\beta,\epsilon$. By the parity lemma, this implies $\epsilon^2\cdots=\gamma^k\epsilon^2$ for even $k\ge 2$. Since this contradicts $\gamma+\epsilon
=(\tfrac{5}{6}+\tfrac{4}{f})\pi+\tfrac{1}{4}\epsilon
>\tfrac{5}{6}\pi+\tfrac{1}{4}\cdot\tfrac{2}{3}\pi
=\pi$, we conclude $\epsilon^2\cdots$ is not a vertex. By the unique AAD
$\thick^{\alpha}\gamma^{\epsilon}\thin^{\epsilon}\gamma^{\alpha}\thick$, this implies $\gamma\thin\gamma\cdots$ is not a vertex.

By $\beta^2\delta$, we have $R(\beta^2\cdots)=\delta<\alpha,\beta,\epsilon$. Therefore $\beta^2\cdots=\beta^2\gamma^k$. By $\beta^2\delta=\beta\thick\beta\cdots$ and no $\gamma\thin\gamma\cdots$, this implies no $\beta\thin\beta\cdots$.

By no $\beta\thin\beta\cdots$ and no $\epsilon^2\cdots$, we get the AAD $\thin^{\beta}\delta^{\epsilon}\thin^{\beta}\delta^{\epsilon}\thin$ at $\delta^2\epsilon$. This gives a vertex $\beta\thin\epsilon\cdots$. By $\beta^2\delta$, $\delta<\epsilon$, and the parity lemma, we get $\beta\thin\epsilon\cdots=\beta\gamma\epsilon\cdots$. By the angle sum for pentagon, we have $\beta+\gamma+\delta+\epsilon=(3+\frac{4}{f})\pi-\alpha=(\frac{7}{3}+\frac{4}{f})\pi>2\pi$. By $\delta<\alpha,\beta,\epsilon$, this implies $\beta\thin\epsilon\cdots=\beta\gamma^k\epsilon$ for odd $k$. By no $\gamma\thin\gamma\cdots$, we have $k=1$. Therefore $\alpha^3,\beta\gamma\epsilon,\delta^2\epsilon$ are vertices, and we may apply Proposition \ref{special1}'. As remarked after Proposition \ref{special1}, we know $\beta^2\delta$ should not be a vertex, contradicting the assumption of the current proposition.
\end{proof}

\begin{proposition}\label{special5}
There is no tiling of the sphere by congruent non-symmetric pentagons in Figure \ref{pentagon}, such that $\alpha^3,\beta^2\delta,\delta\epsilon^2$ are vertices and $f\ge 24$.
\end{proposition}

\begin{proof}
The angle sums of $\alpha^3,\beta^2\delta,\delta\epsilon^2$ and the angle sum for pentagon imply 
\[
\alpha=\tfrac{2}{3}\pi,\;
\beta=\epsilon,\;
\gamma=(\tfrac{1}{3}+\tfrac{4}{f})\pi,\;
\delta+2\epsilon=2\pi.
\]
By $f\ge 24$, we have $\gamma\le \frac{1}{2}\pi<\alpha<2\gamma$. By $\alpha^3$ and Lemma \ref{aadlemma}, we know $\beta\thick\gamma\cdots$ is a vertex. By Lemma \ref{geometry1}, we need to consider two cases: $\gamma>\beta=\epsilon<\delta$, and $\gamma<\beta=\epsilon>\delta$.

\subsubsection*{Case. $\gamma>\beta=\epsilon<\delta$}

By $\delta\epsilon^2$ and $\epsilon=\beta<\gamma\le \frac{1}{2}\pi$, we get $\delta>\pi$. This implies $\delta^2\cdots$ is not a vertex. By $\delta\epsilon^2$, we have $R(\delta\epsilon\cdots)=\epsilon=\beta<\alpha,\gamma,\delta$. By the parity lemma, this implies $\delta\epsilon\cdots=\delta\epsilon^2$. This verifies the conditions in the third part of Lemma \ref{klem2}, and we get $\thick\beta\thin\cdots=\thick\beta\thin\delta\thin\cdots$. This implies $\beta\thick\gamma\cdots=\delta\thin\beta\thick\gamma\cdots$ is a vertex, contradicting $\beta^2\delta$ and $\beta<\gamma$.

\subsubsection*{Case. $\gamma<\beta=\epsilon>\delta$}

By $\delta\epsilon^2$ and $\delta<\epsilon$, we have $\delta<\frac{2}{3}\pi=\alpha<\epsilon$. By $\delta\epsilon^2$ and $\beta=\epsilon$, we have $R(\beta^2\cdots)=R(\epsilon^2\cdots)=R(\beta\epsilon\cdots)=\delta<\alpha,\beta,2\gamma,\epsilon$. Then by the parity lemma, we have $\beta^2\cdots=\beta^2\delta$, $\epsilon^2\cdots=\delta\epsilon^2$, $\beta\epsilon\cdots=\beta\gamma\epsilon$. 

If $\beta\epsilon\cdots=\beta\gamma\epsilon$ is a vertex, then by $\alpha^3,\delta\epsilon^2$, we may apply Proposition \ref{special2}'. As remarked after Proposition \ref{special2}, we know $\beta^2\cdots$ should not be a vertex, contradicting the assumption that $\beta^2\delta$ is a vertex. Therefore $\beta\epsilon\cdots$ is not a vertex. 

By no $\beta\epsilon\cdots$, we get the unique AADs $\thick^{\alpha}\gamma^{\epsilon}\thin^{\epsilon}\delta^{\beta}\thin^{\beta}\delta^{\epsilon}\thin$ and $
\thin^{\beta}\delta^{\epsilon}\thin^{\epsilon}\delta^{\beta}\thin^{\beta}\delta^{\epsilon}\thin$, and we know there is no consecutive $\gamma\delta\gamma$. The unique AADs give a vertex $\beta\thin\beta\cdots$, contradicting $\beta^2\cdots=\beta^2\delta=\beta\thick\beta\cdots$. We conclude no consecutive $\gamma\delta\gamma,\gamma\delta\delta,\delta\delta\delta$.

By $\delta\epsilon^2$, $\delta<\epsilon$, and no consecutive $\delta\delta\delta$, we know a vertex without $b$-edge is $\delta\epsilon^2$. 

We know $\beta\thick\gamma\cdots$ is a vertex. If it has $\alpha$, then by Lemma \ref{klem5}, we know $R(\beta\thick\gamma\cdots)$ has at least two $ab$-angles. This contradicts $\beta>\gamma$ and $\alpha+4\gamma>2\pi$. Therefore $\beta\thick\gamma\cdots$ has no $\alpha$. By $\beta^2\cdots=\beta^2\delta$ and no $\beta\epsilon\cdots$, we know $R(\beta\thick\gamma\cdots)$ has no $\beta,\epsilon$. By $\beta+5\gamma>6\gamma>2\pi$ and the parity lemma, we get $\beta\thick\gamma\cdots=\beta\gamma\delta^k,\beta\gamma^3\delta^k$. By $\beta^2\delta$ and $\beta>\gamma$, we have $k\ge 2$ in $\beta\gamma\delta^k$, contradicting no consecutive $\gamma\delta\delta$. Therefore $\beta\gamma^3\delta^k$ is a vertex. If $k\ge 1$ in $\beta\gamma^3\delta^k$, then by $\beta^2\delta$, we get $\beta=3\gamma+(k-1)\delta>\pi$. However, this contradicts the vertex $\beta^2\delta$. Therefore $\beta\thick\gamma\cdots=\beta\gamma^3$. The angle sum of $\beta\gamma^3$ further implies 
\[
\alpha=\tfrac{2}{3}\pi,\;
\beta=\epsilon=(1-\tfrac{12}{f})\pi,\;
\gamma=(\tfrac{1}{3}+\tfrac{4}{f})\pi,\;
\delta=\tfrac{24}{f}\pi.
\]

If $\beta\cdots=\thin\beta\thick\theta\cdots$ is not $\beta^2\cdots=\beta^2\delta$ and not $\beta\thick\gamma\cdots=\beta\gamma^3$, then $\theta\ne\beta,\gamma$. By the parity lemma, we have $\beta\cdots=\thin\beta\thick\alpha\cdots=\alpha\beta\gamma\cdots$. By $0\ne R(\alpha\beta\gamma\cdots)=\frac{8}{f}\pi<$ all angles, we know $\alpha\beta\gamma\cdots$ is  not a vertex. We conclude $\beta\cdots=\beta^2\delta,\beta\gamma^3$. 

Consider a vertex $\gamma\cdots$ with no $\beta$. By $\alpha=\tfrac{2}{3}\pi$, $\gamma>\tfrac{1}{3}\pi$, and the parity lemma, the vertex is $\alpha\gamma^2\cdots,\gamma^4\cdots,\gamma^2\cdots$, with only $\delta,\epsilon$ in the remainders. We have $R(\gamma^4\cdots)<R(\alpha\gamma^2\cdots)=(\frac{2}{3}-\frac{8}{f})\pi<\epsilon$, and $R(\alpha\gamma^2\cdots)\ne 0$. We also have $\epsilon<R(\gamma^2\cdots)=(\frac{4}{3}-\frac{8}{f})\pi<2\epsilon$. This implies $\alpha\gamma^2\cdots=\alpha\gamma^2\delta^k(k\ge 1)$, $\gamma^4\cdots=\gamma^4\delta^k$, and $\gamma^2\cdots=\gamma^2\delta^k(k\ge 1),\gamma^2\delta^k\epsilon(k\ge 1)$. By no consecutive $\gamma\delta\gamma,\gamma\delta\delta$, we know $\alpha\gamma^2\cdots$ is not a vertex, and $\gamma^4\cdots=\gamma^4$, and  $\gamma^2\cdots=\gamma^2\delta\epsilon, \gamma^2\delta^2\epsilon$. The angle sum of $\gamma^4$ further implies $f=24$ and $\beta=\gamma$, a contradiction.

All the discussion leads to 
\[
\text{AVC}
=\{\alpha^3,\beta^2\delta,\delta\epsilon^2,\beta\gamma^3, 
\gamma^2\delta^k\epsilon \; (k=1,2)
\}.
\]
If $\gamma^2\delta^k\epsilon$ is not a vertex, then $\frac{f}{2}-6=v_4$ is the number of $\beta\gamma^3$. By counting the number of $\gamma$ in the whole tiling, we get $f=3\left(\frac{f}{2}-6\right)$. This implies $f=36$, a contradiction to $\delta<\epsilon$. Therefore $\gamma^2\delta^k\epsilon$ must be a vertex. 

We have $\gamma^2\delta^k\epsilon=\thin^{\epsilon}\gamma^{\alpha}\thick^{\alpha}\gamma^{\epsilon}\thin\cdots$. This gives $\thick^{\beta}\alpha^{\gamma}\thick^{\gamma}\alpha^{\beta}\thick\cdots=\alpha^3=\thick^{\beta}\alpha^{\gamma}\thick^{\gamma}\alpha^{\beta}\thick^{\gamma}\alpha^{\beta}\thick$, as in Figure \ref{special5A}. By the AVC, this gives $\thin^{\delta}\beta^{\alpha}\thick^{\alpha}\gamma^{\epsilon}\thin\cdots=\beta\gamma^3=\thin^{\delta}\beta^{\alpha}\thick^{\alpha}\gamma^{\epsilon}\thin^{\epsilon}\gamma^{\alpha}\thick^{\alpha}\gamma^{\epsilon}\thin$. Then the AAD $\thin^{\epsilon}\gamma^{\alpha}\thick^{\alpha}\gamma^{\epsilon}\thin$ at $\beta\gamma^3$ gives $\thick^{\beta}\alpha^{\gamma}\thick^{\gamma}\alpha^{\beta}\thick\cdots=\alpha^3=\thick^{\beta}\alpha^{\gamma}\thick^{\gamma}\alpha^{\beta}\thick^{\gamma}\alpha^{\beta}\thick$, and the argument continues. We end up with a sequence that starts from $\gamma^2\delta^k\epsilon$
\[
\delta^k\epsilon|\gamma^2
-\alpha^3
-\beta\gamma|\gamma^2
-\alpha^3
-\beta\gamma|\gamma^2
-\alpha^3
-\beta\gamma|\gamma^2
-\cdots
\]

\begin{figure}[htp]
\centering
\begin{tikzpicture}[>=latex,scale=1]


\foreach \a in {0,1,2}
\foreach \b in {0,1,2}
{
\begin{scope}[shift=(60+60*\b:1.732),rotate=60*\b,rotate=120*\a]

\draw
	(-30:1) -- (30:1) -- (90:1);	
\draw[line width=1.5]
	(0,0) -- (-0.866,0);

\node at (0:0.22) {\small $\alpha$};

\end{scope}
}


\begin{scope}[shift=(180:1.732)]

\node at (45:0.7) {\small $\gamma$};
\node at (73:0.67) {\small $\beta$};

\node at (-45:0.75) {\small $\beta$};
\node at (-75:0.75) {\small $\beta$};

\node at (165:0.75) {\small $\gamma$};
\node at (195:0.75) {\small $\gamma$};

\node at (-1.25,0.05) {\small $\delta^k\epsilon$};

\fill (-0.866,0) circle (0.1);

\end{scope}


\begin{scope}[shift=(120:1.732)]

\node at (105:0.72) {\small $\beta$};
\node at (137:0.72) {\small $\beta$};

\node at (-105:0.72) {\small $\gamma$};
\node at (-133:0.72) {\small $\gamma$};

\node at (15:0.75) {\small $\gamma$};
\node at (-17:0.75) {\small $\beta$};

\end{scope}


\begin{scope}[shift=(60:1.732)]

\node at (45:0.67) {\small $\beta$};
\node at (73:0.67) {\small $\beta$};

\node at (-47:0.75) {\small $\gamma$};
\node at (-75:0.75) {\small $\beta$};

\node at (165:0.75) {\small $\gamma$};
\node at (195:0.75) {\small $\gamma$};

\node at (-52:1.05) {\small $\gamma$};
\node at (-72:1.05) {\small $\gamma$};

\draw[line width=1.5]
	(-60:0.6) -- (-60:1.2);

\draw[line width=1.5,dotted]
	(-60:1.2) -- (-60:1.8);

\end{scope}

\end{tikzpicture}
\caption{Proposition \ref{special5}: A sequence starting from $\gamma^2\delta^k\epsilon$.}
\label{special5A}
\end{figure}

The sequence is a path in a finite graph. Therefore it must intersect itself, in the sense that two $\alpha^3$ are identified or two $\beta\gamma|\gamma^2$ are identified. Note that the identification also forces the identification of the neighbourhood tiles. Therefore the identification of the $m$-th and $n$-th $\alpha^3$ in the sequence implies the identification of the $(m-1)$-st and $(n-1)$-st $\beta\gamma|\gamma^2$ in the sequence. Similarly, the identification of the $m$-th and $n$-th $\beta\gamma|\gamma^2$ in the sequence implies the identification of the $(m-1)$-st and $(n-1)$-st $\alpha^3$ in the sequence. Eventually this leads to the identification of the first and the $n$-th $\alpha^3$ in the sequence, and then the identification of $\delta^k\epsilon|\gamma^2$ and the $(n-1)$-st $\beta\gamma|\gamma^2$. Since the two vertices are distinct, we get a contradiction. 
\end{proof}

\begin{proposition}\label{special6}
There is no tiling of the sphere by congruent non-symmetric pentagons in Figure \ref{pentagon}, such that $\alpha^3, \beta^2\delta,\beta\gamma\epsilon$ are vertices and $f\ge 24$. 
\end{proposition}

\begin{proof}

The angle sums of $\alpha^3, \beta^2\delta,\beta\gamma\epsilon$ and the angle sum for pentagon imply 
\[
\alpha=\tfrac{2}{3}\pi, \;
\beta=(\tfrac{5}{6}-\tfrac{2}{f})\pi,  \;
\delta=(\tfrac{1}{3}+\tfrac{4}{f})\pi,  \;
\gamma+\epsilon=(\tfrac{7}{6}+\tfrac{2}{f})\pi.
\]
We have $\alpha<\beta,2\delta$. 

By $\beta^2\delta$ and the balance lemma, we know $\gamma^2\cdots$ is a vertex. If $\beta<\gamma$, then by $\beta^2\delta,\beta\gamma\epsilon$, we have $R(\gamma^2\cdots)<\delta,\epsilon$. By $\gamma>\beta>\alpha=\frac{2}{3}\pi$, we also have $R(\gamma^2\cdots)<\alpha,\beta,\gamma$. Therefore $\gamma^2\cdots$ is not a vertex. The contradiction implies $\beta>\gamma$. By Lemma \ref{geometry1}, this implies $\delta<\epsilon$.

By $\beta\gamma\epsilon$, we have $R(\beta\epsilon\cdots)=\gamma<\beta$. By the parity lemma, this implies $\beta\epsilon\cdots=\beta\gamma\epsilon$.

By Lemma \ref{klem4}, the parity lemma, and $R(\alpha^2\cdots)=\alpha<\beta$, we have $\alpha^2\cdots=\alpha^3,\alpha^2\gamma^2\cdots$. If $\alpha^2\gamma^2\cdots$ is a vertex, then $\gamma\le \pi-\alpha=\frac{1}{3}\pi$, and $\epsilon=(\tfrac{7}{6}+\tfrac{2}{f})\pi-\gamma\ge (\tfrac{5}{6}+\tfrac{2}{f})\pi$. Then $R(\alpha^2\gamma^2\cdots)<\alpha<\beta,2\delta,\epsilon$, and $\alpha^2\gamma^2\cdots=\alpha^2\gamma^k,\alpha^2\gamma^k\delta$. Then we have $\thick\gamma\thin\gamma\thick$ or $\thick\gamma\thin\delta\thin\gamma\thick$ at $\alpha^2\gamma^2\cdots$, and their AAD imply $\epsilon\thin\epsilon\cdots$ is a vertex. By $\epsilon \ge (\tfrac{5}{6}+\tfrac{2}{f})\pi$, we have $R(\epsilon^2\cdots)\le (\tfrac{1}{3}-\tfrac{4}{f})\pi<\alpha,\beta,\delta,\epsilon$. By $\gamma+\epsilon>\pi$, we also have $R(\epsilon^2\cdots)<2\gamma$. Then by the parity lemma, we know $\epsilon^2\cdots$ is not a vertex. The contradiction proves $\alpha^2\cdots=\alpha^3$.

The unique AAD $\thin^{\delta}\beta^{\alpha}\thick^{\alpha}\beta^{\delta}\thin^{\epsilon}\delta^{\beta}\thin$ of $\beta^2\delta$ gives vertices $\thick^{\gamma}\alpha^{\beta}\thick^{\beta}\alpha^{\gamma}\thick\cdots=\alpha^3=\thick^{\gamma}\alpha^{\beta}\thick^{\beta}\alpha^{\gamma}\thick^{\beta}\alpha^{\gamma}\thick$ and $\thin^{\epsilon}\delta^{\beta}\thin^{\delta}\epsilon^{\gamma}\thin\cdots$. The AAD of $\alpha^3$ shows $\beta\thick\gamma\cdots$ and $\gamma\thick\gamma\cdots$ are vertices.

We divide further discussion into two cases: $\epsilon>\alpha=\tfrac{2}{3}\pi$ and $\epsilon\le \alpha=\tfrac{2}{3}\pi$.

\subsubsection*{Case. $\epsilon>\alpha=\frac{2}{3}\pi$}

We have $R(\epsilon^2\cdots)<\alpha<\beta,2\delta,\epsilon$. By $\gamma+\epsilon>\pi$ and the parity lemma, this implies $\epsilon^2\cdots=\delta \epsilon^2$. By Proposition \ref{special5} and $f\ge 24$, there is no tiling with vertices $\alpha^3,\beta^2\delta,\delta\epsilon^2$. Therefore $\epsilon^2\cdots$ is not a vertex. By the first part of Lemma \ref{klem2}', this implies no consecutive $\gamma\delta\cdots\delta\gamma$.

Recall that $\thin^{\epsilon}\delta^{\beta}\thin^{\delta}\epsilon^{\gamma}\thin\cdots$ is a vertex. If the vertex has no $\beta,\gamma$, then by Lemma \ref{klem4} and no $\epsilon^2\cdots$, we get $\thin^{\epsilon}\delta^{\beta}\thin^{\delta}\epsilon^{\gamma}\thin\cdots=\delta^k \epsilon$. By $\epsilon>\tfrac{2}{3}\pi$ and 
$\delta>\tfrac{1}{3}\pi$, we get $k=2,3$. By Proposition \ref{special4}, there is no tiling with vertices $\alpha^3,\beta^2\delta,\delta^2\epsilon$. Therefore $k=3$, and the angle sum of $\delta^3\epsilon$ further implies 
\[
\alpha=\tfrac{2}{3}\pi, \;
\beta=(\tfrac{5}{6}-\tfrac{2}{f})\pi,  \;
\gamma=(\tfrac{1}{6}+\tfrac{14}{f})\pi, \;
\delta=(\tfrac{1}{3}+\tfrac{4}{f})\pi,  \;
\epsilon=(1-\tfrac{12}{f})\pi.
\]

By $\beta^2\delta$, no $\epsilon^2\cdots$, and Lemma \ref{klem3}, we know $\gamma^2\epsilon\cdots$ is a vertex. Then by $R(\gamma^2\epsilon\cdots)=(\frac{2}{3}-\frac{16}{f})\pi<\alpha<\beta,2\delta,\epsilon$, we get $\gamma^2\epsilon\cdots=\gamma^k\epsilon,\gamma^k\delta\epsilon$. By no consecutive $\gamma\delta\cdots\delta\gamma$, we get $k=2$ in both cases. By Proposition \ref{special3}, there is no tiling with vertices $\alpha^3,\beta^2\delta,\gamma^2\epsilon$. Therefore $\gamma^2\epsilon\cdots=\gamma^2\delta\epsilon$. The angle sum of $\gamma^2\delta\epsilon$ further implies
\[
f=60\colon
\alpha=\tfrac{2}{3}\pi,\;
\beta=\epsilon=\tfrac{4}{5}\pi,\;
\gamma=\delta=\tfrac{2}{5}\pi.
\]
After exchanging $\beta \leftrightarrow \gamma$ and $\delta\leftrightarrow  \epsilon$, this is the pentagon with angles \eqref{special2_60B}. The proof of Proposition \ref{special2} shows that such pentagon does not exist. 

We conclude that $\thin^{\epsilon}\delta^{\beta}\thin^{\delta}\epsilon^{\gamma}\thin\cdots$ must have $\beta,\gamma$. By $\beta>\gamma$, $\gamma+\epsilon=\beta+\delta>\pi$ and the parity lemma, we have $\delta\epsilon\cdots=\gamma^2\delta\epsilon\cdots$. By $(\gamma+\epsilon)+\delta=(\frac{3}{2}+\frac{6}{f})\pi$, we have $R(\gamma^2\delta\epsilon\cdots)<(\frac{1}{2}-\frac{6}{f})\pi<\alpha<\beta,2\delta,\epsilon$. Therefore $\delta\epsilon\cdots=\gamma^k\delta\epsilon,\gamma^k\delta^2\epsilon$. By no consecutive $\gamma\delta\cdots\delta\gamma$, we get $k=2$. Then
\[
\thin^{\epsilon}\delta^{\beta}\thin^{\delta}\epsilon^{\gamma}\thin\cdots
=\begin{cases}
\gamma^2\delta\epsilon \\
\gamma^2\delta^2\epsilon
\end{cases}
=\begin{cases}
\thick^{\alpha}\gamma^{\epsilon}\thin^{\epsilon}\delta^{\beta}\thin^{\delta}\epsilon^{\gamma}\thin^{\epsilon}\gamma^{\alpha}\thick, \\
\thick^{\alpha}\gamma^{\epsilon}\thin\delta\thin^{\epsilon}\delta^{\beta}\thin^{\delta}\epsilon^{\gamma}\thin^{\epsilon}\gamma^{\alpha}\thick, \\
\thick^{\alpha}\gamma^{\epsilon}\thin^{\epsilon}\delta^{\beta}\thin^{\delta}\epsilon^{\gamma}\thin\delta\thin^{\epsilon}\gamma^{\alpha}\thick.
\end{cases}
\]
We find $\epsilon^2\cdots$ is always a vertex, a contradiction.

\subsubsection*{Case. $\epsilon\le \alpha=\frac{2}{3}\pi$}

We have $\gamma=(\tfrac{7}{6}+\tfrac{2}{f})\pi-\epsilon\ge(\frac{1}{2}+\frac{2}{f})\pi$. 

By $\beta\gamma\epsilon$, $\beta>\gamma$, and $\gamma+\epsilon>\pi$, we know $R(\gamma^2\epsilon\cdots)\ne 0$ and has no $\epsilon$. By $2\gamma+\epsilon
=\gamma+(\gamma+\epsilon)
\ge (\tfrac{1}{2}+\tfrac{2}{f})\pi+(\tfrac{7}{6}+\tfrac{2}{f})\pi\ge (\tfrac{5}{3}+\tfrac{4}{f})\pi$, we know $R(\gamma^2\epsilon\cdots)\le (\tfrac{1}{3}-\tfrac{4}{f})\pi<\alpha,\beta,\gamma,\delta$. Therefore $\gamma^2\epsilon\cdots$ is not a vertex.

Recall that $\gamma\thick\gamma\cdots$ is a vertex. By $\beta>\gamma>\frac{1}{2}\pi$ and the parity lemma, we know $R(\gamma\thick\gamma\cdots)$ has no $\beta,\gamma$. By Lemma \ref{klem5}, and no $\gamma^2\epsilon\cdots$, this implies $\gamma\thick\gamma\cdots=\gamma^2\delta^k$. By $\beta^2\delta$, $\gamma>\tfrac{1}{2}\pi$, $\delta>\tfrac{1}{3}\pi$, we get $k=2$. Then the angle sum of $\gamma^2\delta^2$ further implies 
\[
\alpha=\tfrac{2}{3}\pi, \;
\beta=(\tfrac{5}{6}-\tfrac{2}{f})\pi,  \;
\gamma = (\tfrac{2}{3}-\tfrac{4}{f})\pi, \; 
\delta=(\tfrac{1}{3}+\tfrac{4}{f})\pi,  \;
\epsilon = (\tfrac{1}{2}+\tfrac{6}{f})\pi.
\]

By $\beta^2\delta$, no $\gamma^2\epsilon\cdots$, and Lemma \ref{klem3}, we know $\epsilon^2\cdots$ is a vertex. By $\beta>\gamma$, $\gamma+\epsilon>\pi$, and the parity lemma, we know $\epsilon^2\cdots$ has no $\beta,\gamma$. Then by Lemma \ref{klem4}, and $\delta>\frac{1}{3}\pi$, $\epsilon>\frac{1}{2}\pi$, we get $\epsilon^2\cdots = \delta\epsilon^2, \epsilon^3, \delta^2\epsilon^2, \delta\epsilon^3$. The angle sums of $\delta\epsilon^2, \epsilon^3, \delta^2\epsilon^2, \delta\epsilon^3$ further imply
\begin{align*}
f=24 &\colon 
\alpha=\tfrac{2}{3}\pi, \;
\beta=\tfrac{3}{4}\pi, \;
\gamma=\tfrac{1}{2}\pi, \;
\delta=\tfrac{1}{2}\pi, \;
\epsilon=\tfrac{3}{4}\pi; \\
f=36  &\colon 
\alpha=\tfrac{2}{3}\pi, \;
\beta=\tfrac{7}{9}\pi, \;
\gamma=\tfrac{5}{9}\pi, \;
\delta=\tfrac{4}{9}\pi, \;
\epsilon=\tfrac{2}{3}\pi; \\
f=60  &\colon 
\alpha=\tfrac{2}{3}\pi, \;
\beta=\tfrac{4}{5}\pi, \;
\gamma=\tfrac{3}{5}\pi, \;
\delta=\tfrac{2}{5}\pi, \;
\epsilon=\tfrac{3}{5}\pi; \\
f=132  &\colon
\alpha=\tfrac{2}{3}\pi, \;
\beta=\tfrac{9}{11}\pi, \;
\gamma=\tfrac{7}{11}\pi, \;
\delta=\tfrac{4}{11}\pi, \;
\epsilon=\tfrac{6}{11}\pi. 
\end{align*}

The case $f=24$ has $\delta\epsilon^2$. By Proposition \ref{special5}, there is no tiling.

For the case $f=36$, by the angle values and the edge length consideration, we may calculate all the possible angle combinations
\[
\text{AVC}
=\{\alpha^3,  
\alpha\beta\gamma,  
\beta^2\delta, 
\beta\gamma\epsilon,
\epsilon^3,   
\gamma^2\delta^2,
\delta^3 \epsilon\}.
\]
This implies $\beta\thin\gamma\cdots=\alpha\beta\gamma$, $\beta\thick\gamma\cdots=\beta\gamma\epsilon$, and $\beta\thin\beta\cdots$ is not a vertex. 

The unique AAD $\thin^{\delta}\beta^{\alpha}\thick^{\alpha}\beta^{\delta}\thin^{\epsilon}\delta^{\beta}\thin$ determines $T_1,T_2,T_3$ around $\beta^2\delta$ in the first of Figure \ref{special6A}. Then $\beta_2\delta_3\cdots=\beta^2\delta$ determines $T_4$, and $T_1,T_2,T_3,T_4$ has $180^{\circ}$ rotation symmetry. Then 
$\alpha_1\alpha_3\cdots=\alpha_1\alpha_3\alpha_5$ and $\delta_4\epsilon_3\cdots=\delta_4\delta_6\delta_7\epsilon_3$. By $\alpha_5$, we have $\gamma_3\cdots=\beta\thick\gamma\cdots,\gamma\thick\gamma\cdots$, which is $\beta\gamma\epsilon,\gamma^2\delta^2$. By $\delta_6$, we get $\gamma_3\cdots=\beta\gamma\epsilon=\beta_5\gamma_3\epsilon_6$. Then $\alpha_5,\beta_5$ determine $T_5$, and $\delta_6,\epsilon_6$ determine $T_6$. By $\delta_7$ and no $\beta_6\thin\beta_7\cdots$, we determine $T_7$. By the $180^{\circ}$ rotation symmetry, we also determine $T_8,T_9$. Then $\beta_9\epsilon_1\cdots=\beta\gamma\epsilon$ and $\gamma_1\gamma_5\cdots=\gamma^2\delta^2$. This implies $\gamma,\delta$ are adjacent in a tile, a contradiction.

\begin{figure}[htp]
\centering
\begin{tikzpicture}[>=latex,scale=1]


\foreach \a in {1,-1}
{
\begin{scope}[scale=\a]

\draw
	(0,0.3) -- (2,0.3)
	(1,1) -- (1,-1)
	(0,0) -- (0,1) -- (0.5,1.4) -- (1,1) -- (1.5,1.4) -- (2,1) -- (2,-1) -- (1.5,-1.4) -- (1,-1) -- (0,-1);

\draw[line width=1.5]
	(0,0.3) -- (0,1) -- (0.5,1.4)
	(1,1) -- (1.5,1.4) -- (2,1)
	(0,1) -- (-1,1)
	(2,0.3) -- (2,-1) -- (1.5,-1.4);

\node at (0.2,0.95) {\small $\alpha$}; 
\node at (0.2,0.5) {\small $\beta$};
\node at (0.5,1.15) {\small $\gamma$};	
\node at (0.85,0.5) {\small $\delta$};	
\node at (0.85,0.95) {\small $\epsilon$};
	
\node at (0.2,-0.8) {\small $\alpha$};	 
\node at (0.2,-0.3) {\small $\beta$};	
\node at (0.85,-0.8) {\small $\gamma$};
\node at (0.2,0.1) {\small $\delta$};
\node at (0.85,0.1) {\small $\epsilon$};

\node at (1.5,-1.15) {\small $\gamma$}; 
\node at (1.2,0.1) {\small $\delta$};
\node at (1.15,-0.9) {\small $\epsilon$};	
\node at (1.8,0.05) {\small $\beta$};	
\node at (1.8,-0.9) {\small $\alpha$};

\node at (1.5,1.15) {\small $\alpha$}; 
\node at (1.2,0.9) {\small $\beta$};	
\node at (1.85,0.85) {\small $\gamma$};
\node at (1.2,0.5) {\small $\delta$};	
\node at (1.85,0.5) {\small $\epsilon$};

\end{scope}
}

\draw
	(0.5,1.4) -- (-0.5,2.1) -- (-1.5,1.4);

\node at (-0.1,1.2) {\small $\alpha$}; 
\node at (0.2,1.4) {\small $\gamma$}; 
\node at (-0.9,1.2) {\small $\beta$}; 
\node at (-1.2,1.4) {\small $\delta$}; 
\node at (-0.5,1.9) {\small $\epsilon$}; 

\node at (0.7,1.6) {\small $\delta^2$}; 

\node at (1,1.2) {\small $\gamma$};	

\node[draw,shape=circle, inner sep=1] at (0.5,0.75) {\small 1};
\node[draw,shape=circle, inner sep=1] at (0.5,-0.5) {\small 2};
\node[draw,shape=circle, inner sep=1] at (-0.5,0.5) {\small 3};
\node[draw,shape=circle, inner sep=1] at (-0.5,-0.75) {\small 4};
\node[draw,shape=circle, inner sep=1] at (1.5,0.75) {\small 9};
\node[draw,shape=circle, inner sep=1] at (1.5,-0.5) {\small 8};
\node[draw,shape=circle, inner sep=1] at (-1.5,-0.75) {\small 7};
\node[draw,shape=circle, inner sep=1] at (-1.5,0.5) {\small 6};
\node[draw,shape=circle, inner sep=1] at (-0.5,1.5) {\small 5};


\begin{scope}[shift={(5cm,-0.5cm)}]

\foreach \a in {0,...,3}
{
\begin{scope}[rotate=90*\a]

\draw
	(0,0) -- (0.8,0) -- (1.2,0.6) -- (0.6,1.2) -- (0,0.8);

\draw[line width=1.5]
    (1.2,0.6) -- (0.6,1.2) -- (0,0.8);

\end{scope}
}

\draw
	(-1.2,0.6) -- (-1.8,0.6) -- (-1.8,1.8) -- (0.6,1.8) -- (1.2,2.2) -- (1.8,1.6)
	(1.2,-0.6) -- (1.8,-0.6) -- (1.8,0.6) 
	(1.2,0.6) -- (2.4,0.6) -- (2.4,1.6) 
	(1.8,1.6) -- (3,1.6) -- (3,2.6)
	(1.2,0.6) -- (2.4,2.6) -- (3,2.6)
	(1.2,2.2) -- (1.2,2.8) -- (1.8,3.2) -- (2.4,2.6)
	(-0.6,1.2) -- (-0.6,1.8) -- (-0.6,2.7) -- (0.6,2.7) -- (1.2,2.2)
	(-0.6,1.8) -- (-1.2,2.4)
	;

\draw[line width=1.5]
	(-1.2,0.6) -- (-1.8,0.6)
	(0.6,1.2) -- (0.6,1.8)
	(1.2,-0.6) -- (1.8,-0.6)
	(1.8,0.6) -- (2.4,0.6) -- (2.4,1.6)
	(1.2,2.8) -- (1.8,3.2) -- (2.4,2.6)
	(-0.6,2.7) -- (0.6,2.7) -- (1.2,2.2);

\node at (0.75,0.2) {\small $\epsilon$}; 
\node at (0.95,0.55) {\small $\gamma$};
\node at (0.6,0.95) {\small $\alpha$};
\node at (0.2,0.7) {\small $\beta$};
\node at (0.2,0.2) {\small $\delta$};

\foreach \a in {0,...,2}
{
\begin{scope}[rotate=90*\a]

\node at (-0.7,0.2) {\small $\gamma$}; 
\node at (-0.95,0.6) {\small $\alpha$};
\node at (-0.6,0.95) {\small $\beta$};
\node at (-0.2,0.7) {\small $\delta$};
\node at (-0.2,0.2) {\small $\epsilon$};

\end{scope}
}

\node at (0,1) {\small $\beta$}; 
\node at (0.4,1.3) {\small $\alpha$};
\node at (-0.4,1.3) {\small $\delta$};
\node at (0.4,1.55) {\small $\gamma$};
\node at (-0.4,1.6) {\small $\epsilon$};

\node at (1.6,-0.4) {\small $\gamma$};  
\node at (1.6,0.4) {\small $\epsilon$};
\node at (1.3,-0.4) {\small $\alpha$};
\node at (1.05,0) {\small $\beta$};
\node at (1.3,0.4) {\small $\delta$};

\node at (-1.3,0.8) {\small $\alpha$}; 
\node at (-1.6,0.8) {\small $\gamma$};
\node at (-1.6,1.6) {\small $\epsilon$};
\node at (-0.8,1.6) {\small $\delta$};
\node at (-0.8,1.25) {\small $\beta$};

\node at (1.2,0.85) {\small $\gamma$}; 
\node at (1.6,1.55) {\small $\epsilon$};
\node at (1.15,2) {\small $\delta$};
\node at (0.8,1.7) {\small $\beta$};
\node at (0.8,1.25) {\small $\alpha$};

\node at (1.45,0.8) {\small $\delta$};  
\node at (1.9,1.45) {\small $\epsilon$};
\node at (1.8,0.75) {\small $\beta$};
\node at (2.2,0.8) {\small $\alpha$};
\node at (2.2,1.4) {\small $\gamma$};

\node at (0.6,2.5) {\small $\alpha$};  
\node at (-0.4,2.45) {\small $\beta$};
\node at (0.9,2.2) {\small $\gamma$};
\node at (0.6,2) {\small $\epsilon$};
\node at (-0.4,2) {\small $\delta$};

\node at (-1.05,2) {\small $\delta$};
\node at (-0.75,2.2) {\small $\delta$};
\node at (1.05,2.5) {\small $\gamma$};

\node at (1.75,3) {\small $\alpha$}; 
\node at (1.35,2.65) {\small $\beta$}; 
\node at (2.2,2.55) {\small $\gamma$};
\node at (1.35,2.35) {\small $\delta$}; 
\node at (1.75,1.85) {\small $\epsilon$};

\node at (2.1,1.8) {\small $\delta$}; 

\node[inner sep=0.5,draw,shape=circle] at (0.55,0.55) {\small $1$};
\node[inner sep=0.5,draw,shape=circle] at (-0.55,0.55) {\small $2$};
\node[inner sep=0.5,draw,shape=circle] at (-0.55,-0.55) {\small $3$};
\node[inner sep=0.5,draw,shape=circle] at (0.55,-0.55) {\small $4$};
\node[inner sep=0.5,draw,shape=circle] at (0,1.45) {\small $5$};
\node[inner sep=0.5,draw,shape=circle] at (1.45,0) {\small $6$};
\node[inner sep=0.5,draw,shape=circle] at (-1.3,1.3) {\small $7$};
\node[inner sep=0.5,draw,shape=circle] at (1.2,1.5) {\small $8$};
\node[inner sep=0.5,draw,shape=circle] at (2,1.1) {\small $9$};
\node[inner sep=0,draw,shape=circle] at (0,2.25) {\small $10$};
\node[inner sep=0,draw,shape=circle] at (1.8,2.4) {\small $11$};
\node[inner sep=0,draw,shape=circle] at (2.6,2.1) {\small $12$};

\end{scope}

\end{tikzpicture}
\caption{Proposition \ref{special6}: Tiling for $f=36$ and $132$.}
\label{special6A}
\end{figure}

For the case $f=132$, by the angle values and the edge length consideration, we may calculate all the possible angle combinations
\[
\text{AVC}
=\{\alpha^3, 
\beta^2\delta,  
\beta\gamma\epsilon,
\gamma^2\delta^2,  
\delta\epsilon^3,   
\delta^4 \epsilon
\}.
\]
The case assumes $\delta\epsilon^3$ is a vertex. By the AVC and the edge length consideration, we know $\beta\thin\gamma\cdots,\gamma\thin\gamma\cdots$ are not vertices. This implies the unique AAD $\thin^{\epsilon}\delta^{\beta}\thin^{\delta}\epsilon^{\gamma}\thin^{\delta}\epsilon^{\gamma}\thin^{\delta}\epsilon^{\gamma}\thin$ of $\delta\epsilon^3$, which determines four tiles $T_1,T_2,T_3,T_4$ around $\delta\epsilon^3$ in the second of Figure \ref{special6A}. Then $\beta_1\delta_2\cdots=\beta_1\beta_5\delta_2$ and $\gamma_4\epsilon_1\cdots=\gamma_4\beta_6\epsilon_1$ determine $T_5,T_6$. Then $\beta_2\delta_5\cdots=\beta_2\beta_7\delta_5$ and $\gamma_1\delta_6\cdots=\gamma_1\gamma_8\delta_6\delta_9$ determine $T_7,T_8$. Then $\beta_8\gamma_5\cdots=\beta_8\gamma_5\epsilon_{10}$. By $\epsilon_{10}$, we find $\delta_7\epsilon_5\cdots=\delta\epsilon^3,\delta^4\epsilon$ is $\delta^4\epsilon=\delta_7\delta_{10}\epsilon_5\cdots$. Then $\delta_{10},\epsilon_{10}$ determine $T_{10}$. Then $\gamma_{10}\delta_8\cdots=\gamma_{10}\gamma\delta_8\delta_{11}$. 

Note that we already know $\epsilon_8,\delta_9,\delta_{11}$. If $\epsilon_8\cdots$ has $\beta,\gamma$, then by the AVC, we get $\epsilon_8\cdots=\beta\gamma\epsilon$. This implies $\gamma$ at $\epsilon_8\cdots$ is adjacent to either $\delta_9$ or $\delta_{11}$. The contradiction implies $\epsilon_8\cdots$ has no $\beta,\gamma$. This determines $\beta_9,\beta_{11}$, and then determines $T_9,T_{11}$. Then $\epsilon_8\epsilon_9\epsilon_{11}\cdots=\delta_{12}\epsilon_8\epsilon_9\epsilon_{11}$. Then we have either $\beta_{12}\thin\gamma_9\cdots$ or $\beta_{12}\thin\gamma_{11}\cdots$. Since $\beta\thin\gamma\cdots$ is not a vertex, we get a contradiction.

It remains to deal with the case $f=60$. It turns out that a tiling is combinatorially possible, in which all $\alpha^3, \beta^2\delta,\beta\gamma\epsilon,\gamma^2\delta^2,\delta^2\epsilon^2,\delta^5$ appear as vertices. We will eliminate the case by using the spherical trigonometry to show that the pentagon with the edge combination $a^3b^2$ and the prescribed angles does not exist. 

By
\[
\alpha=\tfrac{2}{3}\pi,\;
\beta+\gamma+\epsilon=2\pi,\;
\delta=\tfrac{2}{5}\pi,
\]
the pentagon would tile a pentagonal subdivision of the regular icosahedron. Therefore we may apply the idea for calculating \eqref{special1_60A} and \eqref{special2_60A}. We notice that $\delta$ and $\epsilon$ are exchanged (and therefore $\beta$ and $\gamma$ are also exchanged) compared with the earlier calculations. We get 
\[
\cos^3 a + (5-2\sqrt{5})\cos^2 a + (13-6\sqrt{5})\cos a - \tfrac{4}{5}\sqrt{5}+1 =0.
\]
The cubic equation has only one real root $\cos a\approx 0.9023$, which implies $a \approx 0.1418\pi$. Then we get $\beta \approx 0.8100\pi \ne \frac{4}{5}\pi$. This implies that the case $f=60$ is not possible.
\end{proof}

\begin{proposition}\label{special7}
There is no tiling of the sphere by congruent non-symmetric pentagons in Figure \ref{pentagon}, such that $\alpha\beta\gamma,\delta\epsilon^2$ are vertices, and there is a vertex of degree $4$ or $5$.
\end{proposition}

\begin{proof}
The angle sums of $\alpha\beta\gamma,\delta\epsilon^2$ and the angle sum for pentagon imply 
\[
\alpha+\beta+\gamma=2\pi,\;
\delta=\tfrac{8}{f}\pi,\;
\epsilon=(1-\tfrac{4}{f})\pi.
\]
By $f>12$, we have $\delta<\frac{2}{3}\pi<\epsilon$. By Lemma \ref{geometry1}, we get $\beta>\gamma$. 

By $\alpha\beta\gamma$ and Lemma \ref{klem6}, we know $\alpha^2\cdots=\alpha^k$. By Lemma \ref{klem4}, a vertex $\alpha\cdots\ne\alpha^k$ has $ab$-angles. By $\alpha\beta\gamma$, $\beta>\gamma$, and the parity lemma, we conclude $\alpha\cdots=\alpha\beta\gamma,\alpha^k,\alpha\gamma^2\cdots$, with no $\alpha,\beta$ in the remainder. 

Let $p,q,r$ be the respective numbers of $\alpha\beta\gamma,\alpha^k,\alpha\gamma^2\cdots$ in the tiling. Then the total number of $\alpha$ in the tiling is $f=p+kq+r$. We also know that the total number of $\gamma$ is $f\ge p+2r$. Therefore $kq\ge r$.

If $q=0$, then $r=0$. This means $\alpha\cdots=\alpha\beta\gamma$ and $p=f$. Since the number of $\beta$ and $\gamma$ at $\alpha\beta\gamma$ is already $p=f$, and their total number is also $f$, we conclude that $\beta\cdots=\gamma\cdots=\alpha\beta\gamma$. By Lemma \ref{basic}, there is a tile with four degree $3$ vertices, such as $T_2$ in Figure \ref{special7A}. By $\beta_2\cdots=\gamma_2\cdots=\alpha\beta\gamma$, we determine $T_1,T_3$. Then one of $\delta_1\epsilon_2\cdots$ and $\delta_2\epsilon_3\cdots$ has degree $3$. If $\delta_1\epsilon_2\cdots$ has degree $3$, then we have a tile outside $T_1,T_2$. The tile has three vertices that cannot involve $\alpha,\beta,\gamma$, a contradiction. If $\delta_2\epsilon_3\cdots$ has degree $3$, we have the same contradiction. 

\begin{figure}[htp]
\centering
\begin{tikzpicture}[>=latex,scale=1]

\foreach \a in {0,...,3}
\draw[xshift=1.2*\a cm]
	(-0.6,1.1) -- (-0.6,0.4);

\foreach \a in {0,1,2}
{
\begin{scope}[xshift=1.2*\a cm]

\draw
	(0,0) -- (0,-0.3);
	
\draw[line width=1.5]
	(-0.6,0.4) -- (0,0) -- (0.6,0.4);

\node at (0,0.25) {\small $\alpha$};	
\node at (0.45,0.5) {\small $\beta$};
\node at (-0.45,0.5) {\small $\gamma$};
\node at (0.45,0.9) {\small $\delta$};
\node at (-0.45,0.9) {\small $\epsilon$};	

\end{scope}
}

\draw
	(-0.6,1.1) -- ++(3.6,0)
	(-0.6,1.1) -- (0.2,1.8) -- (1,1.8) -- (1.8,1.1);

\foreach \a in {0,1}
\node at (0.6 cm +1.2*\a cm,0.2) {\small $\alpha$};

\node[draw,shape=circle, inner sep=0.5] at (0,0.7) {\small $1$};
\node[draw,shape=circle, inner sep=0.5] at (1.2,0.7) {\small $2$};
\node[draw,shape=circle, inner sep=0.5] at (2.4,0.7) {\small $3$};
	
\end{tikzpicture}
\caption{$\{\alpha\beta\gamma,\delta\epsilon^2\}$: Cannot have $\alpha\cdots=\alpha\beta\gamma$.}
\label{special7A}
\end{figure}

We just proved that $q>0$. This means $\alpha^k$ is a vertex. This implies $\alpha\le\frac{2}{3}\pi$. Then $\alpha\beta\gamma$ further implies $\beta+\gamma=2\pi-\alpha\ge \frac{4}{3}\pi>\pi$. Moreover, by $\alpha\beta\gamma$, $\beta>\gamma$, $\alpha\le\frac{2}{3}\pi<\epsilon$, and the parity lemma, we know $\beta\epsilon\cdots$ is not a vertex.

By $\alpha\beta\gamma$, and $\alpha^2\cdots=\alpha^k$, we know $R(\alpha\gamma^2\cdots)$ has no $\alpha,\beta$. If $R(\alpha\gamma^2\cdots)$ has $\gamma$, then by the parity lemma, we know $\alpha\gamma^4\cdots$ is a vertex. Compared with the angle sum of $\alpha\beta\gamma$, we get $\beta\ge 3\gamma$. Then $\frac{4}{3}\pi\le \beta+\gamma\le \beta+\frac{1}{3}\beta=\frac{4}{3}\beta$. This implies $\beta\ge \pi$. Then $\beta^2\cdots$ is not a vertex, and $\alpha\gamma^4\cdots$ violates the balance lemma. Therefore $R(\alpha\gamma^2\cdots)$ has only $\delta,\epsilon$. Applying the balance lemma to $\alpha\gamma^2\cdots$, we know $\beta^2\cdots$ is a vertex. This implies $\beta<\pi$. Then by $\alpha\beta\gamma$, we have $R(\alpha\gamma^2\cdots)<\beta<\pi$. By $\delta+\epsilon>\pi$ and $\delta<\epsilon$, we conclude $\alpha\gamma^2\cdots=\alpha\gamma^2\delta^k,\alpha\gamma^2\epsilon$.

By $\alpha\beta\gamma$ and $\beta>\gamma$, we know $\beta^2\cdots$ has no $\alpha$. By $\beta>\gamma$, $\beta+\gamma>\pi$, and the parity lemma, we know $R(\beta^2\cdots)$ has no $\beta,\gamma$. By no $\beta\epsilon\cdots$, we know $R(\beta^2\cdots)$ has no $\epsilon$. Therefore $\beta^2\cdots=\beta^2\delta^k=\beta\thick\beta\cdots$. This implies $\beta\thin\beta\cdots$ is not a vertex. 

By no $\beta\thin\beta\cdots$ and no $\beta\epsilon\cdots$, the AAD shows that we cannot have consecutive $\delta\delta\delta$. Therefore $k\le 2$ in $\beta^2\delta^k$, and $\beta^2\cdots=\beta^2\delta,\beta^2\delta^2$.

The proposition assumes the existence of a vertex $H$ of degree $4$ or $5$. If $H$ has no $b$-edge, then by $\delta\epsilon^2$ and $\delta<\epsilon$, we get $H=\delta^3\epsilon,\delta^4\epsilon$. This contradicts no consecutive $\delta\delta\delta$. Therefore $H$ must have $b$-edge.

Suppose there is a vertex $\beta\gamma\cdots$ without $\alpha$. By $\beta+\gamma>\pi$ and the parity lemma, if $R(\beta\gamma\cdots)$ has $\beta,\gamma$, then $\beta\gamma\cdots=\beta\gamma^3\cdots$. This implies $\beta+3\gamma\le 2\pi=\alpha+\beta+\gamma$. Then $2\gamma\le \alpha$, $\alpha+\gamma\le \frac{3}{2}\alpha\le\pi$, and $\beta=2\pi-\alpha-\gamma\ge \pi$. This implies $\beta^2\cdots$ is not a vertex. Then $\beta\gamma^3\cdots$ violates the balance lemma. Therefore $R(\beta\gamma\cdots)$ has no $\beta,\gamma$. By no $\beta\epsilon\cdots$, we know $R(\beta\gamma\cdots)$ has no $\epsilon$. Therefore $\beta\gamma\cdots=\beta\gamma\delta^k$. If $k\ge 2$, then the AAD of $\beta\gamma\delta^k=\thick^{\alpha}\gamma^{\epsilon}\thin\delta\thin\delta\thin\cdots$ implies either $\beta\thin\beta\cdots$ or $\beta\thin\epsilon\cdots$ is a vertex, a contradiction. Therefore $\beta\gamma\cdots=\beta\gamma\delta$. By $\beta>\gamma$, this implies $\beta^2\cdots=\beta^2\delta^k$ is not a vertex. By the balance lemma, this implies that $\alpha\beta\gamma,\beta\gamma\delta$ are the only vertices involving $\beta,\gamma$. Therefore the vertex $H$ of degree $4$ or $5$ has no $\beta,\gamma$. Since $H$ must have $b$-edge, we get $H=\alpha^4,\alpha^5$. The angle sums of $\alpha\beta\gamma,\beta\gamma\delta$ and $H$ imply $\frac{8}{f}\pi=\delta=\alpha=\frac{1}{2}\pi$ or $\frac{2}{5}\pi$, contradicting $f\ge 24$. Therefore we conclude $\beta\gamma\cdots=\alpha\beta\gamma$. 

Recall that $\alpha^k$ is a vertex. There are $k$ tiles around the vertex. By $\beta\gamma\cdots=\alpha\beta\gamma=\beta\thin\gamma\cdots$, we know $\beta\thick\gamma\cdots$ is not a vertex. By the second part of Lemma \ref{aadlemma}, we know $k$ is even. Then we have $k\ge 4$, which implies $\alpha\le\frac{1}{2}\pi$. 

The total numbers of $\alpha,\beta,\gamma$ are all $f$. Then by $\alpha^k$ and $\beta\gamma\cdots=\alpha\beta\gamma$, there must be vertices other than $\beta\gamma\cdots$ that also involve $\beta,\gamma$. By the balance lemma, this implies that $\beta^2\cdots=\beta^2\delta,\beta^2\delta^2$ is a vertex, and $\gamma^2\cdots$ is also a vertex. The angle sum of $\beta^2\cdots$ further implies
\begin{align*}
\beta^2\delta&\colon
\alpha+\gamma=(1+\tfrac{4}{f})\pi,\;
\beta=(1-\tfrac{4}{f})\pi,\;
\delta=\tfrac{8}{f}\pi,\;
\epsilon=(1-\tfrac{4}{f})\pi; \\
\beta^2\delta^2&\colon
\alpha+\gamma=(1+\tfrac{8}{f})\pi,\;
\beta=(1-\tfrac{8}{f})\pi,\;
\delta=\tfrac{8}{f}\pi,\;
\epsilon=(1-\tfrac{4}{f})\pi.
\end{align*}
By $\alpha\le\frac{1}{2}\pi$, we get $\gamma>(\frac{1}{2}+\tfrac{4}{f})\pi$. 

We know $\gamma^2\cdots$ is a vertex. By $\alpha^2\cdots=\alpha^k$, we know $\gamma^2\cdots$ has at most one $\alpha$. By $\gamma>(\frac{1}{2}+\tfrac{4}{f})\pi$, we have $R(\gamma^2\cdots)<(1-\tfrac{8}{f})\pi<2\gamma,\epsilon$. By $\beta>\gamma$ and the parity lemma, we get $\gamma^2\cdots=\alpha\gamma^2\delta^k,\gamma^2\delta^k$. By $\alpha\beta\gamma$ and $\beta>\gamma$, we get $k\ge 1$. Then the AAD of $\alpha\gamma^2\delta^k$ or $\gamma^2\delta^k$ gives either $\beta\thin\beta\cdots$ or $\beta\thin\epsilon\cdots$, a contradiction.
\end{proof}

\begin{proposition}\label{special8}
There is no tiling of the sphere by congruent non-symmetric pentagons in Figure \ref{pentagon}, such that 
$\alpha\beta^2$ is a vertex, and $\delta^3$ or $\delta^2\epsilon$ is also a vertex. 
\end{proposition}

\begin{proof}
By the unique AAD $\thin^{\delta}\beta^{\alpha}\thick^{\beta}\alpha^{\gamma}\thick^{\alpha}\beta^{\delta}\thin$ of $\alpha\beta^2$, we know $\alpha\thick\gamma\cdots$ is a vertex. By the parity lemma, we have $\alpha\thick\gamma\cdots=\alpha\beta\gamma\cdots,\alpha\gamma^2\cdots$. By $\alpha\beta^2$, this implies $\beta>\gamma$. By Lemma \ref{geometry1}, we get $\delta<\epsilon$. Since $\delta^3$ or $\delta^2\epsilon$ is a vertex, we get $\delta\le\frac{2}{3}\pi<\epsilon$ and $\delta+\epsilon>\frac{4}{3}\pi$. 

By $\alpha\beta^2$, we know degree $3$ vertices involving $\gamma$ are $\beta\gamma\delta,\beta\gamma\epsilon,\gamma^2\delta,\gamma^2\epsilon$. If none of these are vertices, then by Lemma \ref{hdeg} and the parity lemma, we know one of $\beta\gamma^3,\gamma^4$ is a vertex. All six vertices are of the form $\beta\thick\gamma\cdots$ or $\gamma\thick\gamma\cdots$, and the AAD of $\thin\beta\thick\gamma\thin$ or $\thin\gamma\thick\gamma\thin$ gives a vertex $\alpha\thick\alpha\cdots$. By $\alpha\beta^2$ and Lemma \ref{klem4}, we get $\alpha^2\cdots=\alpha^k,\alpha^2\beta\gamma\cdots,\alpha^2\gamma^2\cdots$, with no $\beta$ in the remainders. 

By either $\delta^3$ or $\delta^2\epsilon$, and by $\delta<\epsilon$, we know $R(\epsilon^2\cdots)$ has no $\delta,\epsilon$. Then by Lemma \ref{klem4} and the parity lemma, we know $\epsilon^2\cdots$ has at least two $ab$-angles. Then by $\beta>\gamma$, the angle sum of $\epsilon^2\cdots$ implies $\gamma+\epsilon\le\pi$. By $\delta\le\frac{2}{3}\pi$ and the angle sum for pentagon, we get
\[
\alpha+\beta
=(3+\tfrac{4}{f})\pi-(\gamma+\epsilon)-\delta 
\ge 
(\tfrac{4}{3}+\tfrac{4}{f})\pi. 
\]
Comparing with the angle sum of $\alpha\beta^2$, we get $\beta\le(\frac{2}{3}-\frac{4}{f})\pi$ and $\alpha \ge(\frac{2}{3}+\frac{8}{f})\pi$. Therefore $\alpha^k$ is not a vertex, and $\alpha\thick\alpha\cdots=\alpha^2\beta\gamma\cdots,\alpha^2\gamma^2\cdots$ is a vertex. By $\beta>\gamma$, this implies $\alpha+\gamma\le\pi$, and 
further implies $\gamma\le \pi-\alpha \le(\frac{1}{3}-\frac{8}{f})\pi$. The vertex $\epsilon^2\cdots$ also implies $\epsilon<\pi$. By $\beta\le(\frac{2}{3}-\frac{4}{f})\pi,\gamma\le(\frac{1}{3}-\frac{8}{f})\pi,\delta<\epsilon<\pi$, all $\beta\gamma\delta,\beta\gamma\epsilon,\gamma^2\delta,\gamma^2\epsilon,\beta\gamma^3,\gamma^4$ have angle sums $<2\pi$ and therefore cannot be vertices. The contradiction proves that $\epsilon^2\cdots$ is not a vertex. Then by the first part of Lemma \ref{klem2}', we do not have consecutive $\gamma\delta\cdots\delta\gamma$. 

By no $\gamma\delta\cdots\delta\gamma$, we know $\gamma^2\delta,\beta\gamma^3,\gamma^4$ are not vertices. Then one of $\beta\gamma\delta,\beta\gamma\epsilon,\gamma^2\epsilon$ is a vertex. By $\beta>\gamma$ and $\delta<\epsilon$, this implies $\beta+\gamma+\epsilon\ge 2\pi$.

We divide the further discussion according to which of $\delta^3,\delta^2\epsilon$ is a vertex.

\subsubsection*{Case. $\alpha\beta^2,\delta^3$ are vertices}

By $\delta^3$ and Lemma \ref{aadlemma}, we know $\beta\epsilon\cdots$ is a vertex. By $\beta>\gamma$ and the parity lemma, the angle sum of $\beta\epsilon\cdots$ implies $\beta+\gamma+\epsilon\le 2\pi$. Since we already proved $\beta+\gamma+\epsilon\ge 2\pi$, we get $\beta\epsilon\cdots=\beta\gamma\epsilon$. The angle sums of $\alpha\beta^2,\delta^3,\beta\gamma\epsilon$ and the angle sum for pentagon imply
\[
\alpha=(\tfrac{1}{3}+\tfrac{4}{f})\pi, \;
\beta=(\tfrac{5}{6}+\tfrac{2}{f})\pi, \;
\delta=\tfrac{2}{3}\pi, \;
\gamma+\epsilon=(\tfrac{7}{6}+\tfrac{2}{f})\pi.
\]

By $\delta^3$, we have $R(\delta^2\cdots)=\delta<\beta,\epsilon$. This implies $\delta^2\cdots=\alpha^k\gamma^l\delta^2,\delta^3$. By no $\gamma\delta\cdots\delta\gamma$, we have $l=0$ in $\alpha^k\gamma^l\delta^2$. Then by Lemma \ref{klem4}, we know $\alpha^k\delta^2$ is not a vertex. Therefore $\delta^2\cdots=\delta^3$. Now the AAD $\thick^{\alpha}\beta^{\delta}\thin^{\delta}\beta^{\alpha}\thick$ at $\alpha\beta^2$ gives a vertex $\thin^{\epsilon}\delta^{\beta}\thin^{\beta}\delta^{\epsilon}\thin\cdots=\delta^3=\thin^{\beta}\delta^{\epsilon}\thin\delta\thin^{\epsilon}\delta^{\beta}\thin$. This implies $\epsilon^2\cdots$ is a vertex, a contradiction.

\subsubsection*{Case. $\alpha\beta^2,\delta^2\epsilon$ are vertices}

The angle sums of $\alpha\beta^2$, $\delta^2\epsilon$ and the angle sum for pentagon imply  
\[
\alpha+2\gamma+\epsilon
=2(\alpha+\beta+\gamma+\delta+\epsilon)-(\alpha+2\beta)-(2\delta+\epsilon)
=(2+\tfrac{8}{f})\pi.
\]
By $\alpha\beta^2$, no $\epsilon^2\cdots$, and
Lemma \ref{klem3}, we know $\gamma^2\epsilon\cdots$ is a vertex. By $\alpha+2\gamma+\epsilon>2\pi$ and $\beta+\gamma+\epsilon\ge 2\pi$, we know $\gamma^2\epsilon\cdots$ has no $\alpha,\beta$. By $\delta^2\epsilon$, we know $\gamma^2\epsilon\cdots$ has at most one $\delta$. By no $\epsilon^2\cdots$, we know $\gamma^2\epsilon\cdots$ has no $\epsilon$. Therefore $\gamma^2\epsilon\cdots=\gamma^k\epsilon,\gamma^k\delta\epsilon$. By no $\gamma\delta\cdots\delta\gamma$, and the parity lemma, we know $k=2$, and $\gamma^2\epsilon\cdots=\gamma^2\epsilon,\gamma^2\delta\epsilon$.

Suppose $\gamma^2\epsilon\cdots=\gamma^2\epsilon$. Then the angle sums of $\alpha\beta^2,\gamma^2\epsilon,\delta^2\epsilon$ and the angle sum for pentagon imply 
\[
\alpha=\tfrac{8}{f}\pi,\;
\beta=(1-\tfrac{4}{f})\pi,\;
\gamma=\delta=\pi-\tfrac{1}{2}\epsilon.
\]
By $\gamma^2\epsilon$, $\beta>\gamma$, and the parity lemma, we know $\beta\epsilon\cdots$ is not a vertex, and $\gamma\epsilon\cdots=\gamma^2\epsilon$. This implies $\delta\epsilon\cdots$ has no $\beta,\gamma$. Then by $\delta^2\epsilon$, $\delta<\epsilon$, and Lemma \ref{klem4}, we have $\delta\epsilon\cdots=\delta^2\epsilon$. 

By $\delta\epsilon\cdots=\delta^2\epsilon$, and no $\beta\epsilon\cdots,\epsilon^2\cdots$, we may apply the third part of Lemma \ref{klem2}' to get $\gamma\cdots=\gamma\thin\epsilon\cdots=\gamma^2\epsilon$. However, we know at the beginning of the proof that $\alpha\gamma\cdots$ is a vertex. This is a contradiction. 

We conclude that $\gamma^2\epsilon\cdots=\gamma^2\delta\epsilon$ is a vertex. By no $\epsilon^2\cdots$, we have the AAD $\thick^{\alpha}\gamma^{\epsilon}\thin^{\beta}\delta^{\epsilon}\thin$ at $\gamma^2\delta\epsilon$. This gives a vertex $\beta\epsilon\cdots$. By $\beta>\gamma$, $\beta+\gamma+\epsilon\ge 2\pi$, and the parity lemma, we get $\beta\epsilon\cdots=\beta\gamma\epsilon$. The angle sums of $\alpha\beta^2,\delta^2\epsilon,\gamma^2\delta\epsilon,\beta\gamma\epsilon$ and the angle sum for pentagon imply 
\[
\alpha=(\tfrac{1}{2}+\tfrac{6}{f})\pi,\;
\beta=(\tfrac{3}{4}-\tfrac{3}{f})\pi,\;
\gamma=(\tfrac{1}{4}-\tfrac{1}{f})\pi,\;
\delta=(\tfrac{1}{2}-\tfrac{2}{f})\pi,\;
\epsilon=(1+\tfrac{4}{f})\pi.
\]

Recall that $\alpha\thick\alpha\cdots=\alpha^k,\alpha^2\beta\gamma\cdots,\alpha^2\gamma^2\cdots$ is a vertex. The angle sum implies that $\alpha^2\beta\gamma\cdots$ is not a vertex. By $R(\alpha^2\gamma^2\cdots)=(\tfrac{1}{2}-\tfrac{10}{f})\pi<\alpha,\beta,2\gamma,\delta,\epsilon$, and the parity lemma, we get $\alpha^2\gamma^2\cdots=\alpha^2\gamma^2$. By no $\gamma\delta\cdots\delta\gamma$, this is a contradiction. Therefore $\alpha^k$ must be a vertex. By $4\alpha>2\pi$, we get $\alpha^k=\alpha^3$. This further implies 
\[
f=36\colon
\alpha=\tfrac{2}{3}\pi, \; 
\beta=\tfrac{2}{3}\pi, \; 
\gamma=\tfrac{2}{9}\pi, \; 
\delta=\tfrac{4}{9}\pi, \; 
\epsilon=\tfrac{10}{9}\pi.
\]
By specific angles and the edge length consideration, we get $\alpha^2\cdots=\alpha^3$, and $\beta\thick\beta\cdots$ is not a vertex. Now the AAD $\thin^{\epsilon}\gamma^{\alpha}\thick^{\alpha}\gamma^{\epsilon}\thin$ at $\gamma^2\delta\epsilon$ gives a vertex $\thick^{\beta}\alpha^{\gamma}\thick^{\gamma}\alpha^{\beta}\thick\cdots=\alpha^3$. By Lemma \ref{aadlemma}, this implies $\beta\thick\beta\cdots$ is a vertex, a contradiction.
\end{proof}

\begin{proposition}\label{special9}
There is no tiling of the sphere by congruent non-symmetric pentagons in Figure \ref{pentagon}, such that $\beta\gamma\delta,\delta\epsilon^2$ are vertices, and there is a vertex $H=\alpha^3\cdots$ of degree $4$ or $5$, and $f\ge 24$. 
\end{proposition}

\begin{proof}
The angle sums of $\beta\gamma\delta,\delta\epsilon^2$ and the angle sum for pentagon imply $\alpha+\epsilon=(1+\tfrac{4}{f})\pi$ and $\beta+\gamma=2\epsilon$.

If $H$ has $ab$-angle, then by the parity lemma, we have $H=\alpha^3\beta^2,\alpha^3\gamma^2,\alpha^3\beta\gamma$. If $H$ has no $ab$-angle, then by Lemma \ref{klem4}, we have $H=\alpha^4,\alpha^5$. By $2\alpha+\beta+\gamma=2\alpha+2\epsilon>2\pi$, we have $H\ne \alpha^3\beta\gamma$.

\subsubsection*{Case. $H=\alpha^4$}

The angle sums of $\beta\gamma\delta,\delta\epsilon^2,\alpha^4$ and the angle sum for pentagon imply
\[
\alpha=\tfrac{1}{2}\pi,\;
\beta+\gamma=(1+\tfrac{8}{f})\pi,\;
\delta=(1-\tfrac{8}{f})\pi,\;
\epsilon=(\tfrac{1}{2}+\tfrac{4}{f})\pi.
\]
By $f\ge 24$ and non-symmetry, we have $\delta>\epsilon$. This implies $f>24$, and by Lemma \ref{geometry1}, also implies $\beta<\gamma$. By $\beta+\gamma=2\epsilon$, we have $\beta<\epsilon<\gamma$.

By $\beta\gamma\delta$ and Lemma \ref{klem6}, we know $\delta^2\cdots=\delta^k\epsilon^l$. By $\delta\epsilon^2$ and $\delta>\epsilon$, this implies $\delta^2\cdots$ is not a vertex. By the first part of Lemma \ref{klem2}, there is no consecutive $\beta\epsilon\cdots\epsilon\beta$. 

By $\alpha<R(\beta\thin\gamma\cdots)=\delta<2\alpha$, we know $\beta\thin\gamma\cdots\ne\alpha^k\beta\gamma$. By Lemma \ref{klem5} and the parity lemma, this implies $R(\beta\thin\gamma\cdots)$ has at least two $ab$-angles. By $\beta\gamma\delta=\thin\beta\thick\gamma\thin\delta\thin$, we know $R(\beta\thin\gamma\cdots)$ has no $\delta$. Then by $\beta<\gamma$ and $\beta+\gamma>\pi$, we have $\beta\thin\gamma\cdots=\beta^3\gamma\cdots$, with no $\gamma,\delta$ in the remainder. This implies $\beta\epsilon\cdots\epsilon\beta$ at $\beta\thin\gamma\cdots$, a contradiction. Therefore $\beta\thin\gamma\cdots$ is not a vertex. 

By $\beta<\gamma$, $\beta+\gamma>\pi$, and the parity lemma, we know $R(\gamma\thin\gamma\cdots)$ has no $\beta,\gamma$. By Lemma \ref{klem5}, this implies $\gamma\thin\gamma\cdots=\alpha^k\gamma^2$. By $\alpha+\gamma>\alpha+\epsilon>\pi$, we get $k=1$. By Proposition \ref{special8}', there is no tiling with vertices $\alpha\gamma^2, \delta\epsilon^2$. Therefore $\gamma\thin\gamma\cdots$ is not a vertex. 

By no $\beta\thin\gamma\cdots,\gamma\thin\gamma\cdots$, we get the unique AAD $\thin^{\gamma}\epsilon^{\delta}\thin^{\beta}\delta^{\epsilon}\thin^{\gamma}\epsilon^{\delta}\thin$ of $\delta\epsilon^2$. This gives a vertex $\gamma\epsilon\cdots$. By the parity lemma, we have $\gamma\epsilon\cdots=\beta\gamma\epsilon\cdots,\gamma^2\epsilon\cdots$. By $\beta<\gamma$, we have  $R(\gamma^2\epsilon\cdots)<R(\beta\gamma\epsilon\cdots)= (\tfrac{1}{2}-\tfrac{12}{f})\pi<\alpha<\epsilon<\gamma,\delta$. Moreover, by $\beta+2\gamma+\epsilon>\frac{3}{2}(\beta+\gamma)+\epsilon=(2+\tfrac{16}{f})\pi>2\pi$, we have $R(\gamma^2\epsilon\cdots)<\beta$. Therefore $\gamma\epsilon\cdots=\beta^k\gamma\epsilon,\gamma^2\epsilon$. 

By no $\beta\epsilon\cdots\epsilon\beta$, we have $k=1$ in $\beta^k\gamma\epsilon$. Then $\beta\gamma\epsilon$ contradicts the existing vertex $\beta\gamma\delta$. 

If $\gamma^2\epsilon$ is a vertex, then the angle sum of $\gamma^2\epsilon$ further implies 
\[
\alpha=\tfrac{1}{2}\pi,\;
\beta=(\tfrac{1}{4}+\tfrac{10}{f})\pi,\;\gamma=(\tfrac{3}{4}-\tfrac{2}{f})\pi,\;
\delta=(1-\tfrac{8}{f})\pi,\;
\epsilon=(\tfrac{1}{2}+\tfrac{4}{f})\pi.
\]
By $\gamma^2\epsilon$, no $\delta^2\cdots$, and Lemma \ref{klem3}, we know $\beta^2\delta\cdots$ is a vertex. We have $0<R(\beta^2\delta\cdots)=(\tfrac{1}{2}-\tfrac{12}{f})\pi<\alpha<\epsilon<2\beta,\delta$. By $\beta<\gamma$ and the parity lemma, we get a contradiction. 

\subsubsection*{Case. $H=\alpha^5$}

The angle sums of $\beta\gamma\delta,\delta\epsilon^2,\alpha^5$ and the angle sum for pentagon imply
\[
\alpha=\tfrac{2}{5}\pi,\;
\beta+\gamma=(\tfrac{6}{5}+\tfrac{8}{f})\pi,\;
\delta=(\tfrac{4}{5}-\tfrac{8}{f})\pi,\;
\epsilon=(\tfrac{3}{5}+\tfrac{4}{f})\pi.
\]

By $\alpha<R(\beta\thin\gamma\cdots)=\delta<2\alpha$, we know $\beta\thin\gamma\cdots\ne\alpha^k\beta\gamma$. By Lemma \ref{klem5} and the parity lemma, this implies $R(\beta\thin\gamma\cdots)$ has at least two $ab$-angles. By $\beta\gamma\delta=\thin\beta\thick\gamma\thin\delta\thin$, we know $R(\beta\thin\gamma\cdots)$ has no $\delta$. 

If $\beta<\gamma$, then by $\beta+\gamma>\pi$, we have $\beta\thin\gamma\cdots=\beta^3\gamma\cdots$, with no $\gamma,\delta$ in the remainder. By $3\beta+\gamma\le 2\pi$ and $\beta+\gamma=(\tfrac{6}{5}+\tfrac{8}{f})\pi$, we get $\beta\le(\frac{2}{5}-\frac{4}{f})\pi$ and $\gamma\ge (\frac{4}{5}+\frac{12}{f})\pi$. By $\beta^3\gamma\cdots$ and  the balance lemma, we know $\gamma^2\cdots$ is a vertex. By $R(\gamma^2\cdots)\le(\frac{2}{5}-\frac{24}{f})\pi<\alpha,\gamma,\delta,\epsilon$ and the parity lemma, we get $\gamma^2\cdots=\beta^k\gamma^2$ ($k\ge 2$), contradicting $\beta+\gamma>\pi$. Similarly, if $\beta>\gamma$, then $\beta\thin\gamma\cdots=\beta\gamma^3\cdots$, with no $\beta,\delta$ in the remainder. Then we get the same contradiction. We conclude $\beta\thin\gamma\cdots$ is not a vertex. 

Suppose $\beta>\gamma$. By Lemma \ref{geometry1}, this implies $\delta<\epsilon$. By no $\beta\thin\gamma\cdots$, we have $\thin^{\epsilon}\delta^{\beta}\thin^{\delta}\epsilon^{\gamma}\thin$ at $\delta\epsilon^2$. This gives a vertex $\thick^{\alpha}\beta^{\delta}\thin^{\epsilon}\delta^{\beta}\thin\cdots$. By $\beta\gamma\delta$, $\beta>\gamma$, and the parity lemma, we get $\beta\delta\cdots=\beta\gamma\delta$. Therefore $\thick^{\alpha}\beta^{\delta}\thin^{\epsilon}\delta^{\beta}\thin\cdots=\thick^{\alpha}\beta^{\delta}\thin^{\epsilon}\delta^{\beta}\thin^{\epsilon}\gamma^{\alpha}\thick$. This gives a vertex $\beta\epsilon\cdots$. By the parity lemma, we have $\beta\epsilon\cdots=\beta^2\epsilon\cdots,\beta\gamma\epsilon\cdots$. By $2\beta+\epsilon>\beta+\gamma+\epsilon>\beta+\gamma+\delta=2\pi$, we get a contradiction. 

We conclude $\beta<\gamma$. By Lemma \ref{geometry1}, this implies $\delta>\epsilon$. By $\beta+\gamma=2\epsilon$, we also have $\beta<\epsilon<\gamma$. The rest of the argument is the same as the case $H=\alpha^4$. 

By $\beta\gamma\delta$, and Lemmas \ref{klem6} and \ref{klem2}, we conclude no consecutive $\beta\epsilon\cdots\epsilon\beta$. We already proved that $\beta\thin\gamma\cdots$ is not a vertex. We still have $\alpha+\epsilon>\pi$. Then the argument for $H=\alpha^4$ proves $\gamma\thin\gamma\cdots$ is not a vertex. Then we find $\gamma\epsilon\cdots=\beta\gamma\epsilon\cdots,\gamma^2\epsilon\cdots$ is a vertex. We have $R(\gamma^2\epsilon\cdots)<R(\beta\gamma\epsilon\cdots)=(\tfrac{1}{5}-\tfrac{12}{f})\pi<\alpha<\epsilon<\gamma,\delta$. We also have $\beta+2\gamma+\epsilon\ge \frac{3}{2}(\beta+\gamma)+\epsilon=(\tfrac{12}{5}+\tfrac{16}{f})\pi>2\pi$. This implies $\gamma\epsilon\cdots=\beta^k\gamma\epsilon,\gamma^2\epsilon$. 

By no $\beta\epsilon\cdots\epsilon\beta$, we have $\gamma\epsilon\cdots=\gamma^2\epsilon$. The angle sum of $\gamma^2\epsilon$ further implies 
\[
\alpha=\tfrac{2}{5}\pi,\;
\beta=(\tfrac{1}{2}+\tfrac{10}{f})\pi,\;
\gamma=(\tfrac{7}{10}-\tfrac{2}{f})\pi,\;
\delta=(\tfrac{4}{5}-\tfrac{8}{f})\pi,\;
\epsilon=(\tfrac{3}{5}+\tfrac{4}{f})\pi.
\]
By the same reason, we know $\beta^2\delta\cdots$ is a vertex, and $0<R(\beta^2\delta\cdots)=(\tfrac{1}{5}-\tfrac{12}{f})\pi<$ all angles. We get a contradiction. 

\subsubsection*{Case. $H=\alpha^3\beta^2$}

By $2\alpha+\beta+\gamma=2\alpha+2\epsilon>2\pi=3\alpha+2\beta$, we get $\gamma>\alpha+\beta$. Then $\alpha+2\gamma>3\alpha+2\beta=2\pi$, and $\alpha\gamma^2$ is not a vertex. 

If $\alpha\beta\gamma$ is a vertex, then the angle sums of $\beta\gamma\delta,\delta\epsilon^2,\alpha\beta\gamma,\alpha^3\beta^2$ and the angle sum for pentagon imply
\[
\alpha=\delta=\tfrac{8}{f}\pi,\;
\beta=(1-\tfrac{12}{f})\pi,\;
\gamma=(1+\tfrac{4}{f})\pi,\;
\epsilon=(1-\tfrac{16}{f})\pi.
\]
This implies $\gamma^2\cdots$ is not a vertex. Then $H=\alpha^3\beta^2$ contradicts the balance lemma. Therefore $\alpha\beta\gamma$ is not a vertex. 

By $H=\alpha^3\beta^2$, we know $\alpha^3,\alpha\beta^2$ are not vertices. Therefore $\alpha$ does not appear at degree $3$ vertices. By Lemma \ref{hdeg} and the edge length consideration, one of $\alpha^4,\alpha^5$ is a vertex. However, we already finished the cases $\alpha^4$ or $\alpha^5$ is a vertex. 

\subsubsection*{Case. $H=\alpha^3\gamma^2$}

The argument is the same as the case $H=\alpha^3\beta^2$, by exchanging $\beta$ and $\gamma$ (but not exchanging $\delta$ and $\epsilon$).
\end{proof}

\section{Classification of Tiling}
\label{classification}

Our classification of edge-to-edge tilings of the sphere by congruent pentagons with the edge combination $a^3b^2$ starts with the partial neighbourhood of a special tile. The vertex $\bullet$ in Figure \ref{nhd} is the fifth vertex $H$ of degree $3$, $4$ or $5$. Up to the symmetry of horizontal flip, there are three ways of arranging the $b$-edges of the center tile $T_1$. We also assume the angles $\alpha,\beta,\gamma,\delta,\epsilon$ of $T_1$ are given by the first of Figure \ref{pentagon}. We label the three cases as Cases 1, 2, 3.

\begin{figure}[htp]
\centering
\begin{tikzpicture}[>=latex,]

\foreach \a in {0,1,2}
{
\begin{scope}[xshift=3*\a cm]

\foreach \x in {0,1,2}
\draw[rotate=-72*\x]
	(18:0.5) -- (18:0.9) -- (-18:1.2) -- (-54:0.9) -- (-54:0.5);

\draw
	(90:0.5) -- (65:1) -- (35:1.2) -- (18:0.9)
	(90:0.5) -- (115:1) -- (145:1.2) -- (162:0.9);
	
\foreach \x in {0,...,4}
\draw[rotate=-72*\x]
	(18:0.5) -- (90:0.5);

\draw[rotate=72*\a,line width=1.5]
	(18:0.5) -- (90:0.5) -- (162:0.5);
	
\fill 
	(90:0.5) circle (0.1);

\node[draw,shape=circle, inner sep=0.5] at (0,0) {\small $1$};
\node[draw,shape=circle, inner sep=0.5] at (48:0.75) {\small $2$};
\node[draw,shape=circle, inner sep=0.5] at (-18:0.75) {\small $3$};
\node[draw,shape=circle, inner sep=0.5] at (-90:0.75) {\small $4$};
\node[draw,shape=circle, inner sep=0.5] at (198:0.75) {\small $5$};
\node[draw,shape=circle, inner sep=0.5] at (135:0.75) {\small $6$};

\node at (90:0.8) {\small $H$};
	
\end{scope}
}

\foreach \a in {1,2,3}
\node[xshift=-3cm + 3*\a cm] at (-54:1.3) {\a};

\end{tikzpicture}
\caption{Partial neighbourhood of the special tile.}
\label{nhd}
\end{figure}

Next we assign $b$-edges to the partial neighbourhoods, such that each tile has two adjacent $b$-edges. The assignments can be divided into 11 cases, given by Figure \ref{edge_nhd}. We have not yet assigned $b$-edges to the dotted ones, because we do not need to distinguish them in the subsequent discussion.

\begin{figure}[htp]
\centering
\begin{tikzpicture}[>=latex,scale=1]


\foreach \a in {1,2,3}
{
\begin{scope}[xshift=-2.5cm+2.5*\a cm]

\foreach \x in {0,...,4} 
\draw[rotate=-72*\x]
	(18:0.5) -- (-54:0.5);
	
\foreach \x in {0,1,2} 
\draw[rotate=-72*\x]
	(18:0.5) -- (18:0.9) -- (-18:1.2) -- (-54:0.9) -- (-54:0.5);

\draw
	(90:0.5) -- (60:1) -- (40:1.2) -- (18:0.9)
	(90:0.5) -- (120:1) -- (140:1.2) -- (162:0.9);

\fill (0,0.5) circle (0.1);
	
\node at (0,0) {1.\a};

\end{scope}
}

\begin{scope}[line width=1.5]


\draw
	(18:0.5) -- (90:0.5) -- (162:0.5)
	(60:1) -- (90:0.5) -- (120:1)
	(18:0.9) -- (-18:1.2) -- (-54:0.9) -- (-90:1.2) -- (234:0.9) -- (198:1.2) -- (162:0.9);
	

\draw[xshift=2.5cm]
	(18:0.5) -- (90:0.5) -- (162:0.5)
	(60:1) -- (90:0.5) -- (120:1)
	(162:0.9) -- (198:1.2) -- (234:0.9)
	(-90:1.2) -- (-54:0.9) -- (-18:1.2)
	(-54:0.9) -- (-54:0.5);


\draw[xshift=5cm]
	(-18:1.2) -- (18:0.9) -- (18:0.5) -- (90:0.5) -- (162:0.5) -- (162:0.9) -- (198:1.2)
	(-54:0.9) -- (-90:1.2) -- (234:0.9);
	
\end{scope}


\begin{scope}[xshift=7.5cm]

\draw
	(18:0.5) -- (-54:0.5) -- (234:0.5) -- (234:0.9) -- (198:1.2) 
	(90:0.5) -- (120:1) -- (140:1.2) -- (162:0.9) -- (162:0.8) -- (162:0.5) -- (234:0.5)
	(60:1) -- (40:1.2) -- (18:0.9) -- (18:0.5);

\draw[line width=1.5]
	(198:1.2) -- (162:0.9) -- (162:0.5) -- (90:0.5) -- (18:0.5)
	(90:0.5) -- (60:1);
		
\draw[dotted]
	(198:1.2) -- (162:0.9) -- (162:0.5) -- (90:0.5) -- (18:0.5)
	(-54:0.5) -- (-54:0.9)
	(234:0.9) -- (-90:1.2) -- (-54:0.9) -- (-18:1.2) -- (18:0.9);

\fill (0,0.5) circle (0.1);

\node at (0,0) {1.4};

\end{scope}


\begin{scope}[yshift=-2.3cm]

	
\foreach \a in {1,2,3}
{
\begin{scope}[xshift=-2.5cm+2.5*\a cm]

\foreach \x in {0,...,4} 
\draw[rotate=-72*\x]
	(18:0.5) -- (-54:0.5);
	
\foreach \x in {0,1,2} 
\draw[rotate=-72*\x]
	(18:0.5) -- (18:0.9) -- (-18:1.2) -- (-54:0.9) -- (-54:0.5);

\draw
	(90:0.5) -- (120:1) -- (140:1.2) -- (162:0.9);

\draw[line width=1.5]
	(90:0.5) -- (162:0.5) -- (234:0.5)
	(162:0.5) -- (162:0.9);
	
\fill (0,0.5) circle (0.1);
	
\node at (0,0) {2.\a};

\end{scope}
}

\foreach \a in {0,2}
\draw[xshift=2.5*\a cm,dotted]
	(90:0.5) -- (60:1) -- (40:1.2) -- (18:0.9);	

\draw[line width=1.5]
	(18:0.9) -- (-18:1.2) -- (-54:0.9) -- (-90:1.2) -- (234:0.9);

\draw[xshift=2.5cm]
	(90:0.5) -- (60:1) -- (40:1.2) -- (18:0.9);
	
\draw[xshift=2.5cm,line width=1.5]
	(40:1.2) -- (18:0.9) -- (-18:1.2)
	(18:0.5) -- (18:0.9)
	(-54:0.9) -- (-90:1.2) -- (234:0.9);

\draw[xshift=5cm,line width=1.5]
	(-90:1.2) -- (-54:0.9) -- (-18:1.2)
	(-54:0.5) -- (-54:0.9);


\begin{scope}[xshift=7.5cm]

\draw
	(90:0.5) -- (18:0.5) -- (-54:0.5) -- (234:0.5)
	(-54:0.5) -- (-54:0.9) -- (-90:1.2)
	(-54:0.5) -- (-54:0.9)
	(162:0.5) -- (162:0.9)
	(120:1) -- (140:1.2) -- (162:0.9) -- (198:1.2)  -- (234:0.9);

\draw[dotted]
	(90:0.5) -- (60:1) -- (40:1.2) -- (18:0.9) -- (-18:1.2) -- (-54:0.9)
	(18:0.5) -- (18:0.9);
	
\draw[line width=1.5]
	(120:1) -- (90:0.5) -- (162:0.5) -- (234:0.5) -- (234:0.9) -- (-90:1.2);

\fill (0,0.5) circle (0.1);

\node at (0,0) {2.4};

\end{scope}

\end{scope}


\begin{scope}[yshift=-4.6cm]

\foreach \a in {1,2,3}
{
\begin{scope}[xshift=-2.5cm+2.5*\a cm]

\foreach \x in {0,...,4} 
\draw[rotate=-72*\x]
	(18:0.5) -- (-54:0.5);
	
\foreach \x in {0,1,2} 
\draw[rotate=-72*\x]
	(18:0.5) -- (18:0.9) -- (-18:1.2) -- (-54:0.9) -- (-54:0.5);

\fill (0,0.5) circle (0.1);

\node at (0,0) {3.\a};

\end{scope}
}

\draw[line width=1.5]
	(18:0.9) -- (-18:1.2) -- (-54:0.9);

\draw[xshift=2.5cm,line width=1.5]
	(40:1.2) -- (18:0.9) -- (-18:1.2)
	(18:0.5) -- (18:0.9);
	
\draw[xshift=2.5cm]
	(90:0.5) -- (60:1) -- (40:1.2) -- (18:0.9);
		
\foreach \a in {0,2}
\draw[xshift=2.5*\a cm,dotted]
	(90:0.5) -- (60:1) -- (40:1.2) -- (18:0.9);		
	
\foreach \a in {0,1}
\draw[xshift=2.5*\a cm,dotted]
	(90:0.5) -- (120:1) -- (140:1.2) -- (162:0.9);	

\foreach \a in {0,1}
\draw[xshift=2.5*\a cm,line width=1.5]
	(162:0.5) -- (234:0.5) -- (-54:0.5)
	(234:0.5) -- (234:0.9);	


\begin{scope}[xshift=5cm]

\draw
	(90:0.5) -- (120:1) -- (140:1.2) -- (162:0.9);

\draw[line width=1.5]
	(140:1.2) -- (162:0.9) -- (162:0.5) -- (234:0.5) -- (-54:0.5) -- (-54:0.9) -- (-18:1.2);

\end{scope}

\end{scope}	
			
\end{tikzpicture}
\caption{Edge congruent tilings of the partial neighbourhood.}
\label{edge_nhd}
\end{figure}

We will first consider the symmetric case, second the case with $3^5$-tile, and third the general case.

\subsection{Tiling by Symmetric Pentagons}
\label{symmetric}

By symmetric pentagon with edge combination $a^3b^2$, we mean $\beta=\gamma$. By Lemma \ref{geometry1}, this is equivalent to $\delta=\epsilon$. Therefore a symmetric pentagon has angles $\alpha,\beta,\beta,\delta,\delta$. Moreover, the angles are determined by the bounding edges. Therefore we know all the angles in the tiles without dotted edges in Figure \ref{edge_nhd}. 

\begin{proposition}\label{symmetric_tiling}
An edge-to-edge tiling of the sphere by congruent symmetric pentagons with the edge combination $a^3b^2$, $a\ne b$, is the reduction of the pentagonal subdivision tiling. 
\end{proposition}

\begin{proof}
For $f=12$, the proposition follows from \cite{ay1,gsy}. Therefore we assume $f>12$.

We have $\alpha\beta^2,\delta^3$ in Cases 1.3, 1.4, 2.4, 3.3. We also have $\alpha^3,\beta^2\delta,\delta^3$ in Cases 2.1, 2.2, 2.3, 3.1. The angle sums of these vertices and the angle sum for pentagon always imply $f=12$. Therefore we dismiss these cases.

We have $\beta^2\delta,\delta^3$ in Cases 1.1, 1.2. The angle sums of $\beta^2\delta,\delta^3$ and the angle sum for pentagon imply 
\[
\alpha=(\tfrac{1}{3}+\tfrac{4}{f})\pi,\;
\beta=\delta=\tfrac{2}{3}\pi.
\]
In Case 1.1, we have a vertex $\beta\thin\beta\cdots$. By the parity lemma and $\alpha<R(\beta\thin\beta\cdots)=\beta=\delta<2\alpha$, we know $R(\beta\thin\beta\cdots)$ has no $\alpha,\beta$. Therefore $\beta\thin\beta\cdots=\beta^2\delta^k$, a contradiction because $\delta$ is $a^2$-angle. In Case 1.2, if $H$ has no $a$-edge, then $H=\alpha^3\cdots=\alpha^3,\alpha^4,\alpha^5$. If $H=\alpha^3$, then $\tfrac{2}{3}\pi=\alpha=(\tfrac{1}{3}+\tfrac{4}{f})\pi$ implies $f=12$. If $H$ has $b$-edge, then $H=\alpha^3\beta^2$, contradicting $\alpha+\beta>\pi$. 

In Case 3.2, we have $\alpha^3,\beta^2\delta$. By $\beta^2\delta$ and the edge length consideration, we have $H=\delta^2\cdots=\delta^3,\delta^4,\delta^5$. If $H=\delta^3$, then the angle sums of $\alpha^3,\beta^2\delta,\delta^3$ and the angle sum for pentagon imply $f=12$.

In summary, for $f>12$ and symmetric pentagon, the following are all the partial neighborhoods of special tiles: Case 1.2, $H=\alpha^4$ or $\alpha^5$, and Case 3.2, $H=\delta^4$ or $\delta^5$. This implies no $3^5$-tile.

\subsubsection*{Case 1.2, $H=\alpha^4$} 

The angle sum of $\beta^2\delta,\delta^3,\alpha^4$ and the angle sum for pentagon imply
\[
f=24\colon
\alpha=\tfrac{1}{2}\pi,\;
\beta=\delta=\tfrac{2}{3}\pi.
\]
By no $3^5$-tile and the second part of Lemma \ref{basic}, every tile is a $3^44$-tile. By the angle values and the edge length consideration, we get AVC$=\{\beta^2\delta,\delta^3,\alpha^4\}$. 

The first of Figure \ref{classify1} shows three tiles $T_1,T_2,T_3$ around $\alpha^4$. Their edges determine the angles. Since every tile is a $3^44$-tile, all the vertices of $T_1,T_2,T_3$ except $\alpha^4$ have degree $3$. This gives $T_4,T_5,T_6$ as indicated. 

By the AVC, we have $V_{124}=V_{236}=\beta^2\cdots=\beta^2\delta$. This gives one $\delta$ in each of $T_4,T_6$. Up to the symmetry of $\alpha^4$ (horizontal flip) and using this $\delta$ in $T_4$, we may assume that $T_4$ is arranged as indicated. This determines $T_5$. Then $V_{256}=\delta^2\cdots=\delta^3$ gives another $\delta$ in $T_6$. The two $\delta$ in $T_6$ determine $T_6$.

We see that $T_6$ can be derived from $T_4$. The process may continue with $T_6$ as the new $T_4$. Eventually we get all the ``second layer'' of tiles $T_4,T_5,T_6,\dots$ around $\alpha^4$, in addition to the ``first layer'' $T_1,T_2,T_3,\dots$. Moreover, we find $\alpha^2\cdots=\alpha^4$ shared by $T_4,T_5$, and the other similar $\alpha^4$ shared between tiles in the second layer. Then our argument around the initial $\alpha^4$ may be repeated at these new $\alpha^4$. More repetitions lead to the pentagonal subdivision of the octahedron, given by the second of Figure \ref{subdivision_tiling}.

\subsubsection*{Case 1.2, $H=\alpha^5$} 

The angle sums of $\beta^2\delta,\delta^3,\alpha^5$ and the angle sum for pentagon imply
\[
f=60\colon
\alpha=\tfrac{2}{5}\pi,\;
\beta=\delta=\tfrac{2}{3}\pi.
\]
Then we get AVC=$\{\beta^2\delta,\delta^3,\alpha^5\}$ similar to the case $H=\alpha^4$. This implies no vertex of degree $4$. Combined with no $3^5$-tile and the third part of Lemma \ref{basic}, we know every tile is a $3^45$-tile. Then the further proof for the case $H=\alpha^4$ is still valid because only three consecutive $\alpha$ at $H$ was used. The same argument gives the pentagonal subdivision of the icosahedron, given by the fourth of Figure \ref{subdivision_tiling}.

\begin{figure}[htp]
\centering
\begin{tikzpicture}[>=latex,scale=1]

\foreach \b in {0,1}
{
\begin{scope}[xshift=4*\b cm]

\foreach \a in {0,...,2}
\draw[rotate=90*\a]
	(0,0) -- (0.8,0) -- (1.2,0.6) -- (0.6,1.2) -- (0,0.8) -- (0,0);

\draw 
	(-1.2,-0.6) -- (-1.8,-0.6) -- (-1.8,1.8) -- (0.6,1.8) -- (0.6,1.2)
	(-0.6,1.8) -- (-0.6,1.2)
	(-1.2,0.6) -- (-1.8,0.6);

\node[inner sep=1,draw,shape=circle] at (0.55,0.55) {\small $1$};
\node[inner sep=1,draw,shape=circle] at (-0.55,0.55) {\small $2$};
\node[inner sep=1,draw,shape=circle] at (-0.55,-0.55) {\small $3$};
\node[inner sep=1,draw,shape=circle] at (0,1.45) {\small $4$};
\node[inner sep=1,draw,shape=circle] at (-1.3,1.3) {\small $5$};
\node[inner sep=1,draw,shape=circle] at (-1.45,0) {\small $6$};

\end{scope}	
}


\foreach \a in {0,...,2}
{
\begin{scope}[rotate=90*\a]
	
\node at (0.2,0.2) {\small $\alpha$};  
\node at (0.7,0.2) {\small $\beta$};
\node at (0.2,0.7) {\small $\beta$};
\node at (1,0.6) {\small $\delta$};
\node at (0.6,0.95) {\small $\delta$};
	
\end{scope}
}

\draw[line width=1.5]
	(-0.8,0) -- (0.8,0)
	(0,-0.8) -- (0,0.8)
	(-0.6,1.2) -- (-0.6,1.8)
	(-1.8,1.8) -- (0.6,1.8)
	(-1.8,0.6) -- (-1.8,-0.6) -- (-1.2,-0.6);

\node at (-0.4,1.6) {\small $\alpha$};  
\node at (0.4,1.55) {\small $\beta$};
\node at (-0.4,1.3) {\small $\beta$};
\node at (0,1) {\small $\delta$};
\node at (0.4,1.3) {\small $\delta$};

\node at (-0.8,1.6) {\small $\alpha$};  
\node at (-1.6,1.55) {\small $\beta$};
\node at (-0.8,1.25) {\small $\beta$};
\node at (-1.25,0.8) {\small $\delta$};
\node at (-1.6,0.8) {\small $\delta$};

\node at (-1.6,-0.4) {\small $\alpha$};  
\node at (-1.6,0.4) {\small $\beta$};
\node at (-1.3,-0.4) {\small $\beta$};
\node at (-1.3,0.4) {\small $\delta$};
\node at (-1,0) {\small $\delta$};


\begin{scope}[xshift=4cm]

\foreach \a in {0,...,2}
{
\begin{scope}[rotate=90*\a]

\draw[line width=1.5]
	(1.2,0.6) -- (0.6,1.2) -- (0,0.8);

\node at (0.2,0.2) {\small $\delta$};  
\node at (0.7,0.2) {\small $\delta$};
\node at (0.2,0.7) {\small $\beta$};
\node at (0.95,0.55) {\small $\beta$};
\node at (0.6,0.95) {\small $\alpha$};
	
\end{scope}
}

\draw[line width=1.5]
	(0.6,1.2) -- (0.6,1.8)
	(-1.8,0.6) -- (-1.2,0.6);

\node at (-0.4,1.6) {\small $\delta$};  
\node at (0.4,1.55) {\small $\beta$};
\node at (-0.4,1.3) {\small $\delta$};
\node at (0,1.05) {\small $\beta$};
\node at (0.4,1.3) {\small $\alpha$};

\node at (-0.8,1.6) {\small $\delta$};  
\node at (-1.6,1.6) {\small $\delta$};
\node at (-0.8,1.25) {\small $\beta$};
\node at (-1.25,0.8) {\small $\alpha$};
\node at (-1.6,0.8) {\small $\beta$};

\node at (-1.6,-0.4) {\small $\delta$};  
\node at (-1.6,0.35) {\small $\beta$};
\node at (-1.3,-0.4) {\small $\delta$};
\node at (-1.3,0.4) {\small $\alpha$};
\node at (-1.05,-0.05) {\small $\beta$};

\end{scope}

\end{tikzpicture}
\caption{Tiling by congruent symmetric pentagons.}
\label{classify1}
\end{figure}

\subsubsection*{Case 3.2, $H=\delta^4$} 

The angle sums of $\alpha^3,\beta^2\delta,\delta^4$ and the angle sum for pentagon imply
\[
f=24\colon
\alpha=\tfrac{2}{3}\pi,\;
\beta=\tfrac{3}{4}\pi,\;
\delta=\tfrac{1}{2}\pi.
\]
By no $3^5$-tile and the second part of Lemma \ref{basic}, every tile is a $3^44$-tile. By the angle values and the edge length consideration, we get AVC=$\{\alpha^3,\beta^2\delta,\delta^4\}$. 

The second of Figure \ref{classify1} shows three tiles $T_1,T_2,T_3$ around $\delta^4$. Since every tile is a $3^44$-tile, we have tiles $T_4,T_5,T_6$ as indicated. 

Without loss of generality, we may assume the edges of $T_1$ are given as indicated. This determines the angles of $T_1$. By the AVC, we have $V_{124}=\beta\cdots=\beta^2\delta$. This gives one $\delta$ of $T_2$ and one $\beta$ of $T_4$, and further determines $T_2,T_4$. The way $T_1$ determines $T_2,T_4$ can be repeated on $T_2$, and we determine $T_3,T_6$. Then the edges of $T_5$ shared with $T_2,T_4,T_6$ determine $T_5$. The process continues and determines the first and second layers around $\delta^4$. Moreover, we find $\delta^2\cdots=\delta^4$ shared by $T_4,T_5$, and the other similar $\delta^4$ shared between tiles in the second layer. Then our argument around the initial $\delta^4$ may be repeated at these new $\delta^4$. More repetitions lead to the pentagonal subdivision of the octahedron.

\subsubsection*{Case 3.2, $H=\delta^5$}

The angle sums of $\alpha^3,\beta^2\delta,\delta^5$ and the angle sum for pentagon imply
\[
f=60\colon
\alpha=\tfrac{2}{3}\pi,\;
\beta=\tfrac{4}{5}\pi,\;
\delta=\tfrac{2}{5}\pi.
\]
We can similarly get AVC=$\{\alpha^3,\beta^2\delta,\delta^5\}$. Then the same argument for $H=\delta^4$ can be used to prove that the tiling is the pentagonal subdivision of the icosahedron. This is similar to that the argument for Case 1.2, $H=\alpha^4$ can be applied to Case 1.2, $H=\alpha^5$.

\subsubsection*{Calculate the Pentagon}

The tilings are the reductions of the pentagonal subdivision tilings by congruent pentagons with the edge combination $\bar{a}^2\bar{b}^2\bar{c}$ in \cite[Section 3.1]{wy1}. In the first of Figure \ref{classify6}, we draw part of the pentagonal subdivision of the octahedron ($n=4$) or icosahedron ($n=5$), with angle notations from \cite{wy1}. The pentagon in the general subdivision tiling is the second of Figure \ref{classify6}. The four tilings are reductions in the following way:
\begin{enumerate}
\item Case 1.2: $\bar{a}=\bar{c}$, $\bar{\gamma}=\frac{2}{n}\pi$, $\bar{\alpha}=\bar{\beta}=\bar{\delta}=\bar{\epsilon}=\frac{2}{3}\pi$. Moreover, $\bar{a},\bar{b},\bar{\gamma},\bar{\alpha},\bar{\beta}$ in $\bar{a}^2\bar{b}^2\bar{c}$ become $a,b,\alpha,\beta,\delta$ in the symmetric $a^3b^2$.
\item Case 3.2: $\bar{b}=\bar{c}$, $\bar{\beta}=\frac{2}{3}\pi$, $\bar{\alpha}=\bar{\delta}=(1-\frac{1}{n})\pi$, $\bar{\gamma}=\bar{\epsilon}=\frac{2}{n}\pi$. Moreover, $\bar{b},\bar{a},\bar{\beta},\bar{\alpha},\bar{\gamma}$ in $\bar{a}^2\bar{b}^2\bar{c}$ become $a,b,\alpha,\beta,\delta$ in the symmetric $a^3b^2$.
\end{enumerate}
We also recall the key values about the regular octahedron and icosahedron.
\renewcommand{\arraystretch}{1.3}
\begin{center}
 \begin{tabular}{|c| c|| c c c| c c |} 
 \hline
 $f$ & $n$ & $\cos x$ & $\cos y$ & $\cos z$ & $\cos 2y$ & $\cos 2z$ \\   
 \hline\hline
 24 & 4 
 & $\frac{1}{\sqrt{3}}$
 & $\frac{\sqrt{2}}{\sqrt{3}}$ 
 & $\frac{1}{\sqrt{2}}$
 & $\frac{1}{3}$
 & $0$ \\
 \hline
 60 & 5  
 & $\frac{\sqrt{5}+1}{\sqrt{6(5-\sqrt{5})}}$
 & $\frac{\sqrt{5}+1}{2\sqrt{3}}$
 & $\frac{\sqrt{2}}{\sqrt{5-\sqrt{5}}}$ 
 & $\frac{\sqrt{5}}{3}$
 & $\frac{1}{\sqrt{5}}$ \\ 
 \hline
\end{tabular}
\end{center}

\begin{figure}[htp]
\centering
\begin{tikzpicture}[>=latex]

\begin{scope}[scale=1.4]

\draw[gray!50]	
	(0,1) -- ++(30:1);

\foreach \x in {1,-1}
{
\begin{scope}[scale=\x, yshift=1 cm]

\draw[dotted]
	(30:1) -- (-30:2)
	(150:1) -- (210:2);

\draw[gray!50]	
	(0,0) -- (-30:2) -- (210:2) -- (0,0) -- (0,-1);
	
\draw
	(0,0) -- (224:0.86)
	(0,0) -- (-16:0.86);
	
\draw[line width=1.5]
	(-30:2) -- ++(140:1.2)
	(-30:2) -- ++(200:1.2);

\node at (0.75,-0.35) {\small $\bar{\alpha}$};
\node at (0.05,-0.25) {\small $\bar{\beta}$};
\node at (1.4,-0.95) {\small $\bar{\gamma}$};	
\node at (-0.4,-0.55) {\small $\bar{\delta}$};
\node at (0.65,-1.25) {\small $\bar{\epsilon}$};

\node at (0.65,-1.55) {\small $\bar{\alpha}$};

\end{scope}
}	

\draw[dashed,yshift=1 cm]
	(-66:1.53) -- (-136:0.86);

\fill
	 (1.732,0) circle (0.07)
	 (-1.732,0) circle (0.07); 

\filldraw[fill=white]
	 (0,1) circle (0.07)
	 (0,-1) circle (0.07);
	 
\node[fill=white,inner sep=1] at (1,0.4) {\small $x$};
\node[fill=white,inner sep=1] at (0,0.4) {\small $y$};
\node[fill=white,inner sep=1] at (0.8,0) {\small $z$};


\end{scope}

\filldraw[fill=gray]
	 (0,0) circle (0.1);
	 
\node at (0.6,0.5) {\small $\triangle$};


\begin{scope}[xshift=4.5cm]

\draw
	(234:1) -- (162:1) -- (90:1);

\draw[line width=1.5]
	(-54:1) -- (18:1) -- (90:1);

\draw[densely dashed]
	(234:1) -- (-54:1);

\node at (54:1) {\small $\bar{b}$};
\node at (-18:1) {\small $\bar{b}$};
\node at (126:1) {\small $\bar{a}$};
\node at (198:1) {\small $\bar{a}$};
\node at (-90:1) {\small $\bar{c}$};

\node at (90:0.7) {$\bar{\alpha}$};
\node at (162:0.7) {$\bar{\beta}$};
\node at (15:0.7) {$\bar{\gamma}$};
\node at (234:0.7) {$\bar{\delta}$};
\node at (-54:0.7) {$\bar{\epsilon}$};

\end{scope}


\begin{scope}[shift={(6.5cm,-1cm)}]

\draw[gray!50]
	(0,0) -- (0,2.6) -- (4,0) -- (0,0);
	
\draw[line width=1.5]
	(1.2,0.8) -- (4,0);

\draw[dashed]
	(0,0) -- (1.2,0.8);

\draw
	(0,2.6) -- (1.2,0.8);

\fill
	 (4,0) circle (0.1);
	 
\filldraw[fill=white]
	 (0,2.6) circle (0.1);
	 
\filldraw[fill=gray]
	 (0,0) circle (0.1);
	 
\node[fill=white,inner sep=1] at (2,1.3) {\small $x$};
\node[fill=white,inner sep=1] at (0,1.3) {\small $y$};
\node[fill=white,inner sep=1] at (2,0) {\small $z$};

\node[gray] at (0.35,0.35) {\small $\frac{1}{2}\pi$};
\node[gray] at (0.35,2.05) {\small $\frac{1}{3}\pi$};
\node[gray] at (3,0.35) {\small $\frac{1}{n}\pi$};

\node at (1.4,1.4) {\small $\triangle_x$};
\node at (0.5,1) {\small $\triangle_y$};
\node at (1.6,0.3) {\small $\triangle_z$};

\node at (1.35,1) {\small $\bar{\alpha}$};
\node at (0.9,0.85) {\small $\bar{\delta}$};
\node at (1.2,0.55) {\small $\bar{\epsilon}$};

\end{scope}

\end{tikzpicture}
\caption{Calculate symmetric pentagon.}
\label{classify6}
\end{figure}

\subsubsection*{Case 1.2. $a=\bar{a}=\bar{c}$, $b=\bar{b}$}

For the case $\bar{a}=\bar{c}$, we see the edge $2y$ between two $\circ$ is connected by three edges of length $\bar{a}$ at alternating angles $\bar{\delta},\bar{\delta}$. By Lemma \ref{fourth}, we get
\[
\cos 2y
=(1-\cos\bar{\delta})^2\cos^3\bar{a}
+\sin^2\bar{\delta}\cos^2\bar{a}
+(2\cos\bar{\delta}-\cos^2\bar{\delta})\cos \bar{a}
-\sin^2\bar{\delta}.
\]
This gives the equation for the edge length $a=\bar{a}$ in Case 2.1
\begin{align*}
f=24 &\colon
\tfrac{1}{3}
=\tfrac{9}{4}\cos^3a+\tfrac{3}{4}\cos^2a-\tfrac{5}{4}\cos a-\tfrac{3}{4};  \\
f=60 &\colon
\tfrac{\sqrt{5}}{3}
=\tfrac{9}{4}\cos^3a+\tfrac{3}{4}\cos^2a-\tfrac{5}{4}\cos a-\tfrac{3}{4}.
\end{align*}
The cubic equation has single root
\begin{align*}
f=24 &\colon 
\cos a
=\tfrac{1}{9}(2R^{\frac{1}{3}}+8R^{-\frac{1}{3}}-1),
\quad a=0.14869\pi; \\
f=60 &\colon 
\cos a
=\tfrac{1}{9}(R^{\frac{1}{3}}+16R^{-\frac{1}{3}}-1),
\quad a=0.08791\pi,
\end{align*}
where
\[
R=\begin{cases}
19+3\sqrt{33}, & f=24; \\
2(3+\sqrt{5})(3+8\sqrt{5}+3{\textstyle \sqrt{48\sqrt{5}-63}}), & f=60.
\end{cases}
\]

Note that $\bar{a},\bar{b},x$ form a triangle, and the angle between $\bar{a},\bar{b}$ is $\bar{\alpha}=\frac{2}{3}\pi$. Therefore we get
\begin{equation}\label{symmetric_eq1}
\cos x=\cos a\cos b+\sin a\sin b\cos\bar{\alpha}
=\cos a\cos b-\tfrac{1}{2}\sin a\sin b.
\end{equation}
Using the value of $a,x$, the equation determines $b$
\begin{align*}
f=24 &\colon
\cos b=\tfrac{1}{12}((3\sqrt{3}-\sqrt{11})R^{\frac{1}{3}}
+4(3\sqrt{3}+\sqrt{11})R^{-\frac{1}{3}}
-4\sqrt{3}), \\
&\qquad b=0.20564\pi;  \\
f=60 &\colon
\cos b=\tfrac{1}{48\sqrt{30}}(3+\sqrt{5}){\textstyle \sqrt{5-\sqrt{5}}}
\left((9-{\textstyle \sqrt{48\sqrt{5}-63}})R^{\frac{1}{3}} \right. \\
&\qquad\qquad
\left. +16(9+{\textstyle \sqrt{48\sqrt{5}-63}})R^{-\frac{1}{3}} -24\right), \\
&\qquad  b=0.15038\pi.
\end{align*}

In \cite{wy1}, we remarked that the double pentagonal subdivision of the tetrahedron (corresponding to $f=24$ and $n=3$) is a pentagonal subdivision of the octahedron (corresponding to $f=24$ and $n=4$) with the edge combination $a^3b^2$. This exactly the case $f=24$ here. In particular, the values here for $f=24$ are the same as the values in \cite[Section 3.2]{wy1}.

\subsubsection*{Case 3.2. $a=\bar{b}=\bar{c}$, $b=\bar{a}$}

For the case $\bar{b}=\bar{c}$, we see the edge $2z$ between two $\bullet$ is connected by three edges of length $\bar{b}$ at alternating angles $\bar{\epsilon},\bar{\epsilon}$. By Lemma \ref{fourth}, we get
\[
\cos 2z
=(1-\cos\bar{\epsilon})^2\cos^3\bar{b}
+\sin^2\bar{\epsilon}\cos^2\bar{b}
+(2\cos\bar{\epsilon}-\cos^2\bar{\epsilon})\cos \bar{b}
-\sin^2\bar{\epsilon}.
\]
This gives the following equations for the edge length $a$ ($=\bar{b}$) in Case 3.2
\begin{align*}
f=24 &\colon
0
=\cos^3a+\cos^2a-1;  \\
f=60 &\colon
\tfrac{1}{\sqrt{5}}
=\tfrac{5}{8}(3-\sqrt{5})\cos^3a
+\tfrac{1}{8}(5+\sqrt{5})\cos^2a \\
& \quad\quad +\tfrac{1}{8}(5\sqrt{5}-7)\cos a-\tfrac{1}{8}(5+\sqrt{5}).
\end{align*}
The cubic equations have single roots
\begin{align*}
f=24 &\colon 
\cos a
=\tfrac{1}{3}(R^{\frac{1}{3}}+R^{-\frac{1}{3}}-\sqrt{3}),
\quad a=0.22769\pi; \\
f=60 &\colon 
\cos a
=\tfrac{1}{3\sqrt{5}}(R^{\frac{1}{3}}
+2(3-\sqrt{5})R^{-\frac{1}{3}}
-2-\sqrt{5}),
\quad a=0.18417\pi.
\end{align*}
where
\[
R=\begin{cases}
\tfrac{1}{2}(25+3\sqrt{69}), & f=24; \\
4(2+\sqrt{5})(9+8\sqrt{5}+3\sqrt{48\sqrt{5}-27}), & f=60.
\end{cases}
\]
Then we get precise values of $b$ by the an equation similar to \eqref{symmetric_eq1}
\begin{align*}
f=24 &\colon 
\cos b
=\tfrac{1}{3}((3\sqrt{3}-\sqrt{23})R^{\frac{1}{3}}+(3\sqrt{3}+\sqrt{23})R^{-\frac{1}{3}}-\sqrt{3}), \\
&\qquad b=0.09870\pi; \\
f=60 &\colon 
\cos b
=\tfrac{1}{12\sqrt{30}}(2+\sqrt{5}){\textstyle \sqrt{5-\sqrt{5}}}
\left(
((3+\sqrt{5})(9-{\textstyle \sqrt{48\sqrt{5}-27}})R^{\frac{1}{3}} \right. \\
&\qquad \qquad 
\left. +8(9+{\textstyle \sqrt{48\sqrt{5}-27}})R^{-\frac{1}{3}}
-6(\sqrt{5}-1)\right), \\
&\qquad b=0.02829\pi. \qedhere
\end{align*}
\end{proof}

\subsubsection*{Verify the pentagon}

The right triangle $\triangle$ with vertices $\circ$, ${\color{gray} \bullet}$, $\bullet$ in the first and third of Figure \ref{classify6} is one sixth of a regular triangular face of the regular octahedron or icosahedron. It is supposed to be divided into three triangles $\triangle_x,\triangle_y,\triangle_z$. Therefore we reconstruct the picture by first using the known values of $\bar{a},\bar{b},\frac{1}{2}\bar{c},\bar{\alpha},\bar{\delta},\bar{\epsilon}$ ($\bar{c}=\bar{a}$ or $\bar{c}=\bar{b}$) to construct $\triangle_x,\triangle_y,\triangle_z$. since the three triangles share the same edge lengths, we may glue three together as in the third of Figure \ref{classify6}. 

For Case 1.2, we use the equality \eqref{symmetric_eq1} to get the precise value of $b=\bar{b}$. This means the third edge of $\triangle_x$ constructed above is $x$. We may symbolically verify the similar equalities for $\triangle_y,\triangle_z$, and conclude their third edges are respectively $y,z$. For Case 3.2, we may use the similar argument to symbolically verify that the third edges of $\triangle_x,\triangle_y,\triangle_z$ are respectively $x,y,z$. This means that the union of the three triangles is exactly the right triangle $\triangle$, which is the one sixth of a regular triangular face of the octahedron or icosahedron.

The first of Figure \ref{classify6} shows that the pentagon is obtained by attaching copies of $\triangle_x,\triangle_y,\triangle_z$ to the three sides $x,y,z$ of the right triangle $\triangle$. Since the copies of $\triangle_x,\triangle_y,\triangle_z$ are in three non-overlapping copies of $\triangle$, we conclude the pentagon is simple.

\subsection{Tiling with $3^5$-tile}

After Proposition \ref{symmetric_tiling}, we may assume the pentagon is not symmetric. 

\begin{proposition}\label{3to5_tiling}
There is no tiling of the sphere by congruent non-symmetric pentagons in Figure \ref{pentagon}, such that $f>12$ and there is a $3^5$-tile. 
\end{proposition}

Unlike $3^44$- and $3^45$-tiles, the neighbourhood of a $3^5$-tile has the (combinatorial) rotation symmetry. Therefore we only need to consider Case 1 for the proposition. Then by the edge length consideration, we only need to consider Cases 1.1, 1.2, 1.3. Since we need to study these cases for $3^44$- and $3^45$-tiles anyway, we first tile the partial neighbourhood for the three cases, and then specialize to $3^5$-tile.

Figure \ref{classify2} describes some partial neighbourhood tilings for Cases 1.1, 1.2, 1.3. We fix angles of $T_1$ like the first of Figure \ref{pentagon}. 

The edges in the first of Figure \ref{classify2} is Case 1.1. By Lemma \ref{geometry4}, we determine $T_4$ and get the $a^2$-angle $\theta=A_{3,14}=A_{5,14}$. Since we know all the edges, $\theta$ determines $T_3,T_5$. The first of Figure \ref{classify2} describes the case $\theta=\epsilon$.  The same angle sums of $V_{123}=\gamma_1\delta_3\cdots$ and $V_{156}=\beta_1\delta_5\cdots$ imply $V_{123}=\beta_2\gamma_1\delta_3$ and $V_{156}=\beta_1\gamma_6\delta_5$. This determines $T_2,T_6$. If $\theta=\delta$, then we get the similar result, with the arrangements of $T_3,T_5$ changed. We conclude the following two possibilities:
\begin{itemize}
\item Case 1.1.1: $V_{123}=V_{156}=\beta\gamma\delta$, $V_{134}=V_{145}=\delta\epsilon^2$.
\item Case 1.1.2: $V_{123}=V_{156}=\beta\gamma\epsilon$, $V_{134}=V_{145}=\delta^2\epsilon$.
\end{itemize}
Case 1.1.1 is the first of Figure \ref{classify2}. Case 1.1.2 is obtained from Case 1.1.1 by exchanging $\beta\leftrightarrow \gamma$ and $\delta\leftrightarrow \epsilon$ in $T_3,T_5$. 

The edges in the second and third of Figure \ref{classify2} is Case 1.2. We note that $V_{123}=\gamma\cdots, V_{156}=\beta\cdots, V_{134}=\epsilon\cdots$ are degree $3$ vertices with only one $b$-edge. If the tiling has only one such vertex, then the vertex is $\beta\gamma\epsilon$. If the tiling has more than one such vertex, then by Lemma \ref{geometry3}, we have $V_{134}=\beta\gamma\epsilon$ or $\gamma^2\epsilon$. 

Suppose $V_{134}=\beta\gamma\epsilon$. By Lemma \ref{geometry3}, the only other degree $3$ vertex with only one $b$-edge is $\beta^2\delta$. Therefore $V_{123}=\gamma\cdots=\beta_2\gamma_1\epsilon_3$. This determines $T_2,T_3$. Then we have $V_{134}=\beta_4\gamma_3\epsilon_1$, which further determines $T_4$. On the other hand, we have $V_{156}=\beta\cdots=\beta\gamma\epsilon$ or $\beta^2\delta$. Either way determines $T_5,T_6$. We conclude the following two possibilities:
\begin{itemize}
\item Case 1.2.1: $V_{134}=V_{123}=\beta\gamma\epsilon$, $V_{156}=\beta\gamma\epsilon$, $V_{145}=\delta^3$.
\item Case 1.2.2: $V_{134}=V_{123}=\beta\gamma\epsilon$, $V_{156}=\beta^2\delta$, $V_{145}=\delta^2\epsilon$.
\end{itemize}
The second of Figure \ref{classify2} is Case 1.2.1.

Suppose $V_{134}=\gamma^2\epsilon$. This determines $T_3,T_4$. By Lemma \ref{geometry3}, the other degree $3$ vertex with only one $b$-edge is $\beta\gamma\delta$ or $\beta^2\delta$. Therefore $V_{123}=\gamma_1\epsilon_3\cdots=\gamma_1\gamma_2\epsilon_3$. This determines $T_2$. Moreover, we have $V_{156}=\beta\cdots=\beta\gamma\delta$ or $\beta^2\delta$. Either way determines $T_5,T_6$. We conclude the following two possibilities:
\begin{itemize}
\item Case 1.2.3: $V_{134}=V_{123}=\gamma^2\epsilon$, $V_{156}=\beta\gamma\delta$, $V_{145}=\delta\epsilon^2$.
\item Case 1.2.4: $V_{134}=V_{123}=\gamma^2\epsilon$, $V_{156}=\beta^2\delta$, $V_{145}=\delta\epsilon^2$.
\end{itemize}
The third of Figure \ref{classify2} is Case 1.2.4.

The edges in the fourth of Figure \ref{classify2} is Case 1.3. By Lemma \ref{geometry4}, we determine $T_4$. Then the same angle sums of $\alpha_2\gamma_1\cdots$ and $\alpha_6\beta_1\cdots$ imply $\alpha_2\gamma_1\cdots=\alpha_2\beta_3\gamma_1$ and $\alpha_6\beta_1\cdots=\alpha_6\beta_1\gamma_5$. This determines $T_3,T_5$. Then the same angle sums of $\delta_3\delta_4\epsilon_1$ and $\delta_1\epsilon_4\epsilon_5$ imply $\delta=\epsilon$. Therefore the pentagon is symmetric, and we may dismiss Case 1.3, no matter there is a $3^5$-tile or not.

\begin{figure}[htp]
\centering
\begin{tikzpicture}[>=latex,scale=1]

\foreach \a in {0,1,2,3}
{
\begin{scope}[xshift=3.5*\a cm]

\foreach \x in {0,...,4} 
\draw[rotate=-72*\x]
	(18:0.7) -- (-54:0.7);
	
\foreach \x in {0,1,2} 
\draw[rotate=-72*\x]
	(18:0.7) -- (18:1.3) -- (-18:1.7) -- (-54:1.3) -- (-54:0.7);

\draw
	(90:0.7) -- (60:1.5) -- (40:1.7) -- (18:1.3)
	(90:0.7) -- (120:1.5) -- (140:1.7) -- (162:1.3);

\draw[line width=1.5]
	(18:0.7) -- (90:0.7) -- (162:0.7);

\fill (0,0.7) circle (0.1);

\node at (90:0.45) {\small $\alpha$}; 
\node at (162:0.45) {\small $\beta$};
\node at (18:0.45) {\small $\gamma$};
\node at (234:0.45) {\small $\delta$};
\node at (-54:0.45) {\small $\epsilon$};

\node[draw,shape=circle, inner sep=0.5] at (0,0) {\small $1$};
\node[draw,shape=circle, inner sep=0.5] at (48:1.05) {\small $2$};
\node[draw,shape=circle, inner sep=0.5] at (-18:1.05) {\small $3$};
\node[draw,shape=circle, inner sep=0.5] at (-90:1.05) {\small $4$};
\node[draw,shape=circle, inner sep=0.5] at (198:1.05) {\small $5$};
\node[draw,shape=circle, inner sep=0.5] at (135:1.05) {\small $6$};

\end{scope}
}


\draw[line width=1.5]
	(60:1.5) -- (90:0.7) -- (120:1.5)
	(162:1.3) -- (198:1.7) -- (234:1.3) -- (-90:1.7) -- (-54:1.3) -- (-18:1.7) -- (18:1.3);

\node at (70:0.78) {\small $\alpha$}; 
\node at (30:0.85) {\small $\beta$};
\node at (26:1.25) {\small $\delta$};
\node at (40:1.5) {\small $\epsilon$};
\node at (55:1.35) {\small $\gamma$};

\node at (-18:1.45) {\small $\alpha$}; 
\node at (8:1.15) {\small $\beta$};	
\node at (-45:1.15) {\small $\gamma$};
\node at (6:0.75) {\small $\delta$};
\node at (-38:0.75) {\small $\epsilon$};

\node at (-68:0.8) {\small $\delta$}; 
\node at (-114:0.8) {\small $\epsilon$};
\node at (-116:1.15) {\small $\gamma$};
\node at (-90:1.45) {\small $\alpha$};
\node at (-64:1.15) {\small $\beta$};	

\node at (198:1.45) {\small $\alpha$}; 
\node at (173:1.2) {\small $\beta$};
\node at (224:1.15) {\small $\gamma$};
\node at (175:0.8) {\small $\delta$};
\node at (220:0.75) {\small $\epsilon$};

\node at (112:0.75) {\small $\alpha$}; 
\node at (150:0.8) {\small $\gamma$};
\node at (155:1.2) {\small $\epsilon$};
\node at (140:1.5) {\small $\delta$};
\node at (128:1.35) {\small $\beta$};

\node at (-54:1.8) {1.1.1};


\foreach \a in {1,2}
\draw[xshift=3.5*\a cm,line width=1.5]
	(60:1.5) -- (90:0.7) -- (120:1.5)
	(234:1.3) -- (198:1.7) -- (162:1.3)
	(-54:1.3) -- (-54:0.7)	
	(-18:1.7) -- (-54:1.3) -- (-90:1.7);


\begin{scope}[xshift=3.5cm]

\node at (70:0.78) {\small $\alpha$}; 
\node at (30:0.85) {\small $\beta$};
\node at (26:1.25) {\small $\delta$};
\node at (40:1.5) {\small $\epsilon$};
\node at (55:1.35) {\small $\gamma$};

\node at (-18:1.45) {\small $\beta$}; 
\node at (8:1.15) {\small $\delta$};	
\node at (-45:1.15) {\small $\alpha$};
\node at (6:0.75) {\small $\epsilon$};
\node at (-42:0.8) {\small $\gamma$};

\node at (-64:1.15) {\small $\alpha$}; 
\node at (-70:0.8) {\small $\beta$};
\node at (-90:1.45) {\small $\gamma$};
\node at (-114:0.8) {\small $\delta$};
\node at (-118:1.15) {\small $\epsilon$};

\node at (198:1.45) {\small $\alpha$}; 
\node at (172:1.15) {\small $\gamma$};
\node at (224:1.15) {\small $\beta$};
\node at (175:0.8) {\small $\epsilon$};
\node at (220:0.8) {\small $\delta$};

\node at (112:0.75) {\small $\alpha$}; 
\node at (150:0.8) {\small $\gamma$};
\node at (155:1.2) {\small $\epsilon$};
\node at (140:1.5) {\small $\delta$};
\node at (128:1.35) {\small $\beta$};
	
\node at (-54:1.8) {1.2.1};
	
\end{scope}


\begin{scope}[xshift=7cm]

\node at (70:0.78) {\small $\alpha$}; 
\node at (30:0.8) {\small $\gamma$};
\node at (26:1.25) {\small $\epsilon$};
\node at (40:1.5) {\small $\delta$};
\node at (55:1.3) {\small $\beta$};

\node at (-18:1.45) {\small $\beta$}; 
\node at (8:1.15) {\small $\delta$};	
\node at (-45:1.15) {\small $\alpha$};
\node at (6:0.75) {\small $\epsilon$};
\node at (-42:0.8) {\small $\gamma$};

\node at (-70:0.8) {\small $\gamma$}; 
\node at (-114:0.8) {\small $\epsilon$};
\node at (-118:1.15) {\small $\delta$};
\node at (-90:1.45) {\small $\beta$};
\node at (-64:1.15) {\small $\alpha$};	

\node at (198:1.45) {\small $\alpha$}; 
\node at (173:1.2) {\small $\beta$};
\node at (224:1.15) {\small $\gamma$};
\node at (175:0.8) {\small $\delta$};
\node at (220:0.75) {\small $\epsilon$};

\node at (112:0.75) {\small $\alpha$}; 
\node at (150:0.8) {\small $\beta$};
\node at (127:1.35) {\small $\gamma$};
\node at (153:1.2) {\small $\delta$};
\node at (140:1.5) {\small $\epsilon$};
	
\node at (-54:1.8) {1.2.4};
	
\end{scope}


\begin{scope}[xshift=10.5 cm]

\draw[line width=1.5]
	(-18:1.7) -- (18:1.3) -- (18:0.7)
	(162:0.7) -- (162:1.3)  -- (198:1.7) 	
	(234:1.3) -- (-90:1.7) -- (-54:1.3);

\node at (-18:1.45) {\small $\gamma$}; 
\node at (10:1.15) {\small $\alpha$};	
\node at (-45:1.15) {\small $\epsilon$};
\node at (3:0.75) {\small $\beta$};
\node at (-42:0.8) {\small $\delta$};

\node at (-68:0.8) {\small $\delta$}; 
\node at (-114:0.8) {\small $\epsilon$};
\node at (-116:1.15) {\small $\gamma$};
\node at (-90:1.45) {\small $\alpha$};
\node at (-64:1.15) {\small $\beta$};

\node at (196:1.45) {\small $\beta$}; 
\node at (171:1.15) {\small $\alpha$};
\node at (224:1.15) {\small $\delta$};
\node at (175:0.8) {\small $\gamma$};
\node at (220:0.75) {\small $\epsilon$};

\node at (30:0.85) {\small $\alpha$}; 
\node at (150:0.85) {\small $\alpha$}; 

\node at (-54:1.8) {1.3};

\end{scope}

\end{tikzpicture}
\caption{Some partial neighbourhoods in Cases 1.1, 1.2, 1.3.}
\label{classify2}
\end{figure}

\begin{proof}[Proof of Proposition \ref{3to5_tiling}]
We only need to consider $\deg H=3$ in Cases 1.1 and 1.2. Then $H=\alpha^3$, and all six subcases are impossible for various reasons.

Case 1.1.1: $\alpha^3,\beta\gamma\delta,\delta\epsilon^2$ are vertices. Proposition \ref{special1} gives all the tilings, and none has $3^5$-tile.

Cases 1.1.2 and 1.2.2: $\alpha^3,\beta\gamma\epsilon,\delta^2\epsilon$ are vertices. Proposition \ref{special1}' gives all the tilings, and none has $3^5$-tile.

Case 1.2.1: $\alpha^3,\beta\gamma\epsilon,\delta^3$ are vertices. The angle sums of $\alpha^3,\beta\gamma\epsilon,\delta^3$ and the angle sum for pentagon imply $f=12$.

Cases 1.2.3 and 1.2.4: $\alpha^3,\gamma^2\epsilon, \delta\epsilon^2$ are vertices. There is no tiling by Proposition \ref{special4}'.
\end{proof}

\subsection{Tiling without $3^5$-tile}

After Propositions \ref{symmetric_tiling} and \ref{3to5_tiling}, we only need to deal with all eleven cases in Figure \ref{edge_nhd}, and without $3^5$-tile. By the second part of Lemma \ref{basic}, this implies $f\ge 24$. Therefore all the Propositions \ref{special1} through \ref{special9} are valid. 

For Cases 1.1, 1.2, 1.3, we already argued that Case 1.3 has no tiling. By Propositions \ref{special9} and \ref{special9}', Cases 1.1.1, 1.1.2, 1.2.2, 1.2.3 have no tiling. It remains to study Cases 1.2.1, 1.2.4, without $3^5$-tile.

For Case 1.4, we note that $T_1,T_5$ share an edge with both ends having degree $3$, as in Figure \ref{classify3}. Moreover, $T_1$ is already arranged as indicated, and $T_5$ has two possible arrangements. In the first case, we see $\beta_1\beta_5\cdots=\alpha\beta^2$ and $\delta_1\delta_5\cdots=\delta^3,\delta^2\epsilon$ are vertices. By Proposition \ref{special8}, there is no tiling. In the second case, we see $\beta_1\gamma_5\cdots=\alpha\beta\gamma$ and $\delta_1\epsilon_5\cdots=\delta^2\epsilon,\delta\epsilon^2$ are vertices. By Propositions \ref{special7}' and \ref{special7}, there is no tiling.

\begin{figure}[htp]
\centering
\begin{tikzpicture}[>=latex,scale=1]

\foreach \a in {-1,1}
\foreach \b in {0,1}
{
\begin{scope}[xshift=-0.566*\a cm, xshift=3.5*\b cm, xscale=\a]

\draw
	(36:0.7) -- (-36:0.7) -- (-108:0.7) -- (180:0.7);

\draw[line width=1.5]
	(36:0.7) -- (108:0.7) -- (180:0.7);

\end{scope}
}

\foreach \b in {0,1}
{
\begin{scope}[xshift=3.5*\b cm]

\node at (0.7,0.45) {\small $\alpha$};
\node at (0.15,0.2) {\small $\beta$};
\node at (1,0) {\small $\gamma$};
\node at (0.15,-0.25) {\small $\delta$};
\node at (0.7,-0.5) {\small $\epsilon$};

\node at (0,0.6) {\small $\alpha$};

\node[draw,shape=circle, inner sep=1] at (0.566,0) {\small 1};
\node[draw,shape=circle, inner sep=1] at (-0.566,0) {\small 5};

\end{scope}
}

\node at (-0.7,0.45) {\small $\alpha$};
\node at (-0.15,0.2) {\small $\beta$};
\node at (-1,0) {\small $\gamma$};
\node at (-0.15,-0.25) {\small $\delta$};
\node at (-0.7,-0.5) {\small $\epsilon$};

\begin{scope}[xshift=3.5 cm]

\node at (-0.7,0.45) {\small $\alpha$};
\node at (-1,0) {\small $\beta$};
\node at (-0.15,0.2) {\small $\gamma$};
\node at (-0.7,-0.45) {\small $\delta$};
\node at (-0.15,-0.25) {\small $\epsilon$};

\end{scope}

\end{tikzpicture}
\caption{Cases 1.4, 2.4, 3.3.}
\label{classify3}
\end{figure}

The argument for Case 1.4 can also be applied to Cases 2.4 and 3.3. The only modification is $T_5$ changed to $T_4$ for Case 2.4, and $T_5$ changed to $T_3$ for Case 3.3. There is no tiling for Cases 2.4 and 3.3.

The remaining cases are 2.1, 2.2, 2.3, 3.1, 3.2, all having the vertex $\alpha^3$. By Propositions \ref{special3}, \ref{special6}, \ref{special6}', the combination of $\alpha^3$ with any pair in Lemma \ref{geometry3} has no tiling. Therefore we may assume all degree $3$ vertices with one $b$-edge are the same in all the subsequent discussions.

For Case 2.1, we first determine $T_3$ by Lemma \ref{geometry4}. Then we successively consider the possibilities for $V_{145}=\beta_1\cdots$ and $V_{134}=\delta_1\epsilon_3\cdots$, subject to the constraint in Lemma \ref{geometry3}. We get the following list, together with the applicable propositions that either give the tilings or show that there is no tiling.
\begin{itemize}
\item $V_{156}=\alpha^3$, $V_{145}=\beta^2\delta$, $V_{134}=\delta\epsilon^2$: Proposition \ref{special5}.
\item $V_{156}=\alpha^3$, $V_{145}=\beta^2\epsilon$, $V_{134}=\delta^2\epsilon$: Proposition \ref{special5}'.
\item $V_{156}=\alpha^3$, $V_{145}=\beta\gamma\delta$, $V_{134}=\delta\epsilon^2$: Proposition \ref{special1}.
\item $V_{156}=\alpha^3$, $V_{145}=\beta\gamma\epsilon$, $V_{134}=\delta^2\epsilon$: Proposition \ref{special1}'.
\end{itemize}
For Case 2.2, we may assume $V_{123}=V_{145}$ and get the following:
\begin{itemize}
\item $V_{156}=\alpha^3$, $V_{123}=V_{145}=\beta\gamma\epsilon$, $V_{134}=\delta^2\epsilon$: Proposition \ref{special1}'.
\item $V_{156}=\alpha^3$, $V_{123}=V_{145}=\beta\gamma\epsilon$, $V_{134}=\delta^3$: $f=12$.
\item $V_{156}=\alpha^3$, $V_{123}=V_{145}=\beta^2\epsilon$, $V_{134}=\delta^3$. 
\end{itemize}
For Case 2.3, we may assume $V_{134}=V_{145}$ and get the following:
\begin{itemize}
\item $V_{156}=\alpha^3$, $V_{134}=V_{145}=\beta\gamma\delta$, $V_{123}=\delta\epsilon^2$: Proposition \ref{special1}.
\item $V_{156}=\alpha^3$, $V_{134}=V_{145}=\beta\gamma\delta$, $V_{123}=\epsilon^3$: $f=12$.
\item $V_{156}=\alpha^3$, $V_{134}=V_{145}=\beta^2\delta$, $V_{123}=\delta^2\epsilon$: Proposition \ref{special4}.
\item $V_{156}=\alpha^3$, $V_{134}=V_{145}=\beta^2\delta$, $V_{123}=\delta\epsilon^2$: Proposition \ref{special5}.
\end{itemize}
For Case 3.1, we may assume $V_{134}=V_{156}$ and get the following:
\begin{itemize}
\item $V_{145}=\alpha^3$, $V_{134}=V_{156}=\beta\gamma\delta$, $V_{123}=\delta\epsilon^2$: Proposition \ref{special1}.
\item $V_{145}=\alpha^3$, $V_{134}=V_{156}=\beta\gamma\delta$, $V_{123}=\delta^2\epsilon$: Proposition \ref{special2}.
\item $V_{145}=\alpha^3$, $V_{134}=V_{156}=\beta\gamma\epsilon$, $V_{123}=\delta^2\epsilon$: Proposition \ref{special1}'.
\item $V_{145}=\alpha^3$, $V_{134}=V_{156}=\beta\gamma\epsilon$, $V_{123}=\delta^3$: $f=12$.
\end{itemize}
For Case 3.2, we may assume $V_{123}=V_{134}=V_{156}$ and get the following:
\begin{itemize}
\item $V_{145}=\alpha^3$, $V_{123}=V_{134}=V_{156}=\beta\gamma\delta$. 
\end{itemize}

Case 2.2, $V_{123}=V_{145}=\beta^2\epsilon$, $V_{134}=\delta^3$ is the first of Figure \ref{classify5}. Case 3.2, $V_{123}=V_{134}=V_{156}=\beta\gamma\delta$ is the second of Figure \ref{classify5} ($T_7,T_8,T_9$ not yet included). Both require further study. Besides this, the only other cases requiring further study are Cases 1.2.1 and 1.2.4, described by the second and third of Figure \ref{classify2}.

\subsubsection*{Case 1.2.1}

We have vertices $\beta\gamma\epsilon,\delta^3$, and $H=\alpha^3\cdots=\alpha^4,\alpha^5,\alpha^3\beta^2,\alpha^3\gamma^2,\alpha^3\beta\gamma$. 

We will prove that $H=\alpha^3\beta^2,\alpha^3\gamma^2,\alpha^3\beta\gamma$ admit no tiling. We will handle $H=\alpha^4,\alpha^5$ in Section \ref{division}.

\subsubsection*{Subcase. $H=\alpha^3\gamma^2$}

The angle sums of $\beta\gamma\epsilon,\delta^3,\alpha^3\gamma^2$ and the angle sum for pentagon imply 
\[
\alpha=(\tfrac{1}{3}+\tfrac{4}{f})\pi,\;
\beta+\epsilon=(\tfrac{3}{2}+\tfrac{6}{f})\pi,\;
\gamma=(\tfrac{1}{2}-\tfrac{6}{f})\pi,\;
\delta=\tfrac{2}{3}\pi.
\]
We have $\beta+\epsilon>\gamma+\delta$. By Lemma \ref{geometry1}, this implies $\beta>\gamma$ and $\delta<\epsilon$. The AAD $\thick^{\alpha}\gamma^{\epsilon}\thin^{\epsilon}\gamma^{\alpha}\thick$ at $H$ gives a vertex $\epsilon^2\cdots$. By $\beta\gamma\epsilon$ and Lemma \ref{klem6}, we know $\epsilon^2\cdots$ has only $\delta,\epsilon$. By $\delta^3$ and $\delta<\epsilon$, we have $R(\epsilon^2\cdots)<\delta<\epsilon$. This implies $\epsilon^2\cdots$ has no $\delta,\epsilon$, a contradiction.

\subsubsection*{Subcase. $H=\alpha^3\beta\gamma$}

The angle sums of $\beta\gamma\epsilon,\delta^3,\alpha^3\beta\gamma$ and the angle sum for pentagon imply 
\[
\alpha=(\tfrac{1}{3}+\tfrac{4}{f})\pi,\;
\beta+\gamma=(1-\tfrac{12}{f})\pi,\;
\delta=\tfrac{2}{3}\pi,\;
\epsilon=(1+\tfrac{12}{f})\pi.
\]
We have $\delta<\epsilon$. By Lemma \ref{geometry1}, this implies $\beta>\gamma$, and further implies $\beta>(\frac{1}{2}-\tfrac{6}{f})\pi>\gamma$.  By $\epsilon>\pi$, we also know $\epsilon^2\cdots$ is not a vertex. By the AAD, this implies no $\gamma\thin\gamma\cdots$. 

The AAD $\thick^{\alpha}\beta^{\delta}\thin^{\epsilon}\gamma^{\alpha}\thick$ at $H$ gives a vertex $\thin^{\epsilon}\delta^{\beta}\thin^{\gamma}\epsilon^{\delta}\thin\cdots$. By $R(\thin^{\epsilon}\delta^{\beta}\thin^{\gamma}\epsilon^{\delta}\thin\cdots)=(\frac{1}{3}-\frac{12}{f})\pi<\alpha,\beta,\delta,\epsilon$, we get $\thin^{\epsilon}\delta^{\beta}\thin^{\gamma}\epsilon^{\delta}\thin\cdots=\gamma^k\delta\epsilon$. By no $\gamma\thin\gamma\cdots$ and the parity lemma, we get $k=2$. Therefore $\thin^{\epsilon}\delta^{\beta}\thin^{\gamma}\epsilon^{\delta}\thin\cdots=\gamma^2\delta\epsilon=\thick^{\alpha}\gamma^{\epsilon}\thin^{\epsilon}\delta^{\beta}\thin^{\gamma}\epsilon^{\delta}\thin^{\epsilon}\gamma^{\alpha}\thick$, contradicting no $\epsilon^2\cdots$.

\subsubsection*{Subcase. $H=\alpha^3\beta^2$}

The angle sums of $\beta\gamma\epsilon,\delta^3,\alpha^3\beta^2$ and the angle sum for pentagon imply
\[
\alpha=(\tfrac{1}{3}+\tfrac{4}{f})\pi,\;
\beta=(\tfrac{1}{2}-\tfrac{6}{f})\pi,\;
\gamma+\epsilon=(\tfrac{3}{2}+\tfrac{6}{f})\pi,\;
\delta=\tfrac{2}{3}\pi.
\]
If $\beta<\gamma$ and $\delta>\epsilon$, then $\gamma>(\frac{3}{2}+\frac{6}{f})\pi-\delta=(\frac{5}{6}+\frac{6}{f})\pi$, and $R(\gamma^2\cdots)<(\frac{1}{3}-\frac{12}{f})\pi<\alpha,\beta,\gamma,\delta$. By $\beta\gamma\epsilon$ and $\beta<\gamma$, we also have $R(\gamma^2\cdots)<\epsilon<\delta$. Therefore $\gamma^2\cdots$ is not a vertex. Then $\alpha^3\beta^2$ contradicts the balance lemma. By Lemma \ref{geometry1}, we conclude $\beta>\gamma$ and $\delta<\epsilon$. This implies $\epsilon>(\frac{3}{2}+\frac{6}{f})\pi-\beta=(1+\frac{12}{f})\pi>\pi$. Therefore $\epsilon^2\cdots$ is not a vertex. 

The AAD $\thick^{\alpha}\beta^{\delta}\thin^{\delta}\beta^{\alpha}\thick$ at $H$ gives a vertex $\thin^{\epsilon}\delta^{\beta}\thin^{\beta}\delta^{\epsilon}\thin\cdots$. By no $\epsilon^2\cdots$, the AAD implies $\thin^{\epsilon}\delta^{\beta}\thin^{\beta}\delta^{\epsilon}\thin\cdots\ne\delta^3$. Then by $R(\thin^{\epsilon}\delta^{\beta}\thin^{\beta}\delta^{\epsilon}\thin\cdots)=\delta<\epsilon$, we know $R(\thin^{\epsilon}\delta^{\beta}\thin^{\beta}\delta^{\epsilon}\thin\cdots)$ has no $\delta,\epsilon$. By Lemma \ref{klem4}, this implies $\thin^{\epsilon}\delta^{\beta}\thin^{\beta}\delta^{\epsilon}\thin\cdots=\thick^{\alpha}\beta^{\delta}\thin^{\epsilon}\delta^{\beta}\thin^{\beta}\delta^{\epsilon}\thin\cdots, \thick^{\alpha}\gamma^{\epsilon}\thin^{\epsilon}\delta^{\beta}\thin^{\beta}\delta^{\epsilon}\thin\cdots$. Then by no $\epsilon^2\cdots$, we have $\thin^{\epsilon}\delta^{\beta}\thin^{\beta}\delta^{\epsilon}\thin\cdots=\thick^{\alpha}\beta^{\delta}\thin^{\epsilon}\delta^{\beta}\thin^{\beta}\delta^{\epsilon}\thin\cdots$. This gives a vertex $\thin^{\epsilon}\delta^{\beta}\thin^{\delta}\epsilon^{\gamma}\thin\cdots$. 
 
By $\delta=\tfrac{2}{3}\pi$ and $\epsilon>(1+\frac{12}{f})\pi$, we have $R(\thin^{\epsilon}\delta^{\beta}\thin^{\delta}\epsilon^{\gamma}\thin\cdots)<(\frac{1}{3}-\frac{12}{f})\pi<\alpha,\beta,\delta,\epsilon$. This implies $\thin^{\epsilon}\delta^{\beta}\thin^{\delta}\epsilon^{\gamma}\thin\cdots=\gamma^k\delta\epsilon=\thick^{\alpha}\gamma^{\epsilon}\thin^{\epsilon}\delta^{\beta}\thin^{\delta}\epsilon^{\gamma}\thin\cdots$. This gives a vertex $\epsilon^2\cdots$, a contradiction.

\subsubsection*{Case 1.2.4}

We have vertices $\beta^2\delta,\gamma^2\epsilon,\delta\epsilon^2$, and $H=\alpha^3\cdots=\alpha^4,\alpha^5,\alpha^3\beta^2,\alpha^3\gamma^2,\alpha^3\beta\gamma$. 

By the parity lemma, the vertex $\thick^{\alpha}\gamma_5^{\epsilon}\thin^{\epsilon}\delta_4^{\beta}\thin\cdots$ in the third of Figure \ref{classify2} has one more $\beta$ or $\gamma$. By comparing the angle sums of $\gamma\delta\cdots$ and $\beta^2\delta$, we get $\beta>\gamma$. By Lemma \ref{geometry1}, this implies $\delta<\epsilon$. 

The angle sums of $\beta^2\delta,\delta\epsilon^2$ imply $\beta+\delta+\epsilon=2\pi$. Then the angle sum for pentagon imply $\alpha+\gamma=(1+\tfrac{4}{f})\pi>\pi$. By $\beta>\gamma$, this implies $\alpha^3\beta^2,\alpha^3\beta\gamma,\alpha^3\gamma^2$ are not vertices. Therefore $H=\alpha^4,\alpha^5$.

\subsubsection*{Subcase. $H=\alpha^4$}

The angle sums of $\beta^2\delta,\gamma^2\epsilon,\delta\epsilon^2,\alpha^4$ and the angle sum for pentagon imply
\[
\alpha=\tfrac{1}{2}\pi,\;
\beta=\epsilon=(1-\tfrac{8}{f})\pi,\;
\gamma=(\tfrac{1}{2}+\tfrac{4}{f})\pi,\;
\delta=\tfrac{16}{f}\pi.
\]
By $\beta>\gamma$, we have $f>24$. 

By $R(\beta\epsilon\cdots)=\tfrac{16}{f}\pi<\beta,\gamma$ ($f>24$ used), and the parity lemma, we know $\beta\epsilon\cdots$ is not a vertex. 

By $\beta>\gamma>\frac{1}{2}\pi$ and the parity lemma, we have $\gamma_5\delta_4\cdots=\beta\gamma\delta\cdots,\gamma^2\delta\cdots$, with no $\beta,\gamma$ in the remainders. By $\beta>\gamma$, we have $R(\beta\gamma\delta\cdots)<R(\gamma^2\delta\cdots )=(1-\frac{24}{f})\pi<2\alpha,\epsilon$. Therefore $\gamma_5\delta_4\cdots$ has at most one $\alpha$, and has no $\epsilon$. We conclude $\thick^{\alpha}\gamma_5^{\epsilon}\thin^{\epsilon}\delta_4^{\beta}\thin\cdots=\beta\gamma\delta^k,\gamma^2\delta^k,\alpha\beta\gamma\delta^k,\alpha\gamma^2\delta^k$. By $\alpha+\beta+\gamma+\delta=(2+\tfrac{12}{f})\pi>2\pi$, we know $\alpha\beta\gamma\delta^k$ is not a vertex. 

If $k\ge 2$, then by no $\beta\epsilon\cdots$, we have $\thick^{\alpha}\gamma^{\epsilon}\thin^{\epsilon}\delta^{\beta}\thin^{\beta}\delta^{\epsilon}\thin$ at the vertex. This gives a vertex $\beta\thin\beta\cdots$. By $f>24$, we have $R(\beta\thin\beta\cdots)=\tfrac{16}{f}\pi<2\alpha,\beta,\gamma$. Then by Lemma \ref{klem5}, we know $\beta\thin\beta\cdots=\alpha\beta^2$. The angle sum of $\alpha\beta^2$ further implies $f=32$. Then $\delta=\frac{1}{2}\pi$, and $\gamma+\delta>\pi$. By $\beta>\gamma$, this implies $\beta\gamma\delta^k,\gamma^2\delta^k,\alpha\gamma^2\delta^k$ are not vertices.

We conclude $k=1$, and $\thick^{\alpha}\gamma_5^{\epsilon}\thin^{\epsilon}\delta_4^{\beta}\thin\cdots=\beta\gamma\delta,\gamma^2\delta,\alpha\gamma^2\delta$. By $\beta^2\delta,\gamma^2\epsilon$, $\beta>\gamma$, and $\delta<\epsilon$, we know $\beta\gamma\delta,\gamma^2\delta$ are not vertices. Therefore $\thick^{\alpha}\gamma_5^{\epsilon}\thin^{\epsilon}\delta_4^{\beta}\thin\cdots=\alpha\gamma^2\delta=\thick\alpha\thick^{\alpha}\gamma^{\epsilon}\thin^{\epsilon}\delta^{\beta}\thin^{\epsilon}\gamma^{\alpha}\thick$, contradicting no $\beta\epsilon\cdots$. 

\subsubsection*{Subcase. $H=\alpha^5$}

The angle sums of $\beta^2\delta,\gamma^2\epsilon,\delta\epsilon^2,\alpha^5$ and the angle sum for pentagon imply
\[
\alpha=\tfrac{2}{5}\pi,\;
\beta=\epsilon=(\tfrac{4}{5}-\tfrac{8}{f})\pi,\;
\gamma=(\tfrac{3}{5}+\tfrac{4}{f})\pi,\;
\delta=(\tfrac{2}{5}+\tfrac{16}{f})\pi.
\]
By $\beta>\gamma$, we get $f>60$. By Lemma \ref{aadlemma}, the AAD of $\alpha^5$ gives a vertex $\beta\thick\gamma\cdots$. By $f>60$, we have $\delta<R(\beta\thick\gamma\cdots)=(\tfrac{3}{5}+\tfrac{4}{f})\pi=\gamma<\beta,2\delta,\epsilon$. By the parity lemma, this implies $R(\beta\thick\gamma\cdots)$ has no $\beta,\gamma,\delta,\epsilon$, contradicting Lemma \ref{klem5}.

\subsubsection*{Case 2.2}

We have vertices $\alpha^3,\beta^2\epsilon,\delta^3$, and $H=\gamma\delta\cdots$. The angle sums of $\alpha^3,\beta^2\epsilon,\delta^3$ and the angle sum for pentagon imply
\[
\alpha=\delta=\tfrac{2}{3}\pi,\;
\beta-\gamma=(\tfrac{1}{3}-\tfrac{4}{f})\pi.
\] 
We have $\beta>\gamma$. By Lemma \ref{geometry1}, this implies $\delta<\epsilon$.

By $\beta^2\epsilon$ and Lemma \ref{geometry3}, a degree $3$ vertex $\gamma\cdots$ cannot have only one $b$-edge. By the parity lemma, a degree $3$ vertex $\gamma\cdots=\alpha\beta\gamma,\alpha\gamma^2$. The angle sums of $\alpha^3,\beta^2\epsilon,\delta^3$, the angle sum of $\alpha\beta\gamma$ or $\alpha\gamma^2$, and the angle sum for pentagon imply
\begin{align*}
\alpha\beta\gamma &\colon
\alpha=\delta=\tfrac{2}{3}\pi,\;
\beta=(\tfrac{5}{6}-\tfrac{2}{f})\pi,\;
\gamma=(\tfrac{1}{2}+\tfrac{2}{f})\pi,\;
\epsilon=(\tfrac{1}{3}+\tfrac{4}{f})\pi; \\
\alpha\gamma^2 &\colon
\alpha=\gamma=\delta=\tfrac{2}{3}\pi,\;
\beta=(1-\tfrac{4}{f})\pi,\;
\epsilon=\tfrac{8}{f}\pi.
\end{align*}
In both cases, we have $\delta>\epsilon$, a contradiction. Therefore $\gamma$ does not appear at degree $3$ vertices.  By Lemma \ref{hdeg} and the parity lemma, one of $\beta\gamma^3,\gamma^4$ is a vertex. The angle sums of $\alpha^3,\beta^2\epsilon,\delta^3$, the angle sum of $\beta\gamma^3$ or $\gamma^4$, and the angle sum for pentagon imply
\begin{align*}
\beta\gamma^3 &\colon
\alpha=\delta=\tfrac{2}{3}\pi,\;
\beta=(\tfrac{3}{4}-\tfrac{3}{f})\pi,\;
\gamma=(\tfrac{5}{12}+\tfrac{1}{f})\pi,\;
\epsilon=(\tfrac{1}{2}+\tfrac{6}{f})\pi; \\
\gamma^4 &\colon
\alpha=\delta=\tfrac{2}{3}\pi,\;
\beta=(\tfrac{5}{6}-\tfrac{4}{f})\pi,\;
\gamma=\tfrac{1}{2}\pi,\;
\epsilon=(\tfrac{1}{3}+\tfrac{8}{f})\pi.
\end{align*}
In the first of Figure \ref{classify5}, we see $H=\beta\gamma\delta\cdots,\gamma^2\delta\cdots$. Then
\[
R(\beta\gamma\delta\cdots)
<R(\gamma^2\delta\cdots)
=\begin{cases}
(\frac{1}{2}-\frac{2}{f})\pi, 
& \text{for $\beta\gamma^3$} \\
\frac{1}{3}\pi, 
& \text{for $\gamma^4$}
\end{cases}
<\alpha,\beta,2\gamma,\delta,\epsilon.
\]
Moreover, by $\beta^2\epsilon$, $\beta>\gamma$, and $\delta<\epsilon$, we know $R(\beta\gamma\delta\cdots)>0$. Then by the parity lemma, we know $H=\beta\gamma\delta\cdots,\gamma^2\delta\cdots$ is not a vertex.

\begin{figure}[htp]
\centering
\begin{tikzpicture}[>=latex,scale=1]


\foreach \a in {-1,1}
\draw[xscale=\a]
	(90:0.7) -- (18:0.7) -- (-54:0.7) -- (234:0.7)
	(18:0.7) -- (18:1.3) 
	(90:0.7) -- (60:1.5) -- (40:1.7) -- (18:1.3) -- (-18:1.7) -- (-54:1.3) -- (-90:1.7)
	(-54:0.7) -- (-54:1.3);

\draw[line width=1.5]
	(90:0.7) -- (162:0.7) -- (234:0.7)
	(162:0.7) -- (162:1.3)
	(-54:1.3) -- (-90:1.7) -- (234:1.3)
	(40:1.7) -- (18:1.3) -- (-18:1.7)
	(18:0.7) -- (18:1.3);

\fill (90:0.7) circle (0.1);

\node at (90:0.45) {\small $\gamma$}; 
\node at (162:0.45) {\small $\alpha$};
\node at (18:0.45) {\small $\epsilon$};
\node at (234:0.45) {\small $\beta$};
\node at (-54:0.45) {\small $\delta$};

\node at (70:0.75) {\small $\delta$}; 
\node at (30:0.8) {\small $\beta$};
\node at (26:1.2) {\small $\alpha$};
\node at (40:1.45) {\small $\gamma$};
\node at (56:1.35) {\small $\epsilon$};

\node at (-18:1.45) {\small $\gamma$}; 
\node at (8:1.15) {\small $\alpha$};	
\node at (-47:1.15) {\small $\epsilon$};
\node at (2:0.75) {\small $\beta$};
\node at (-38:0.75) {\small $\delta$};

\node at (-68:0.8) {\small $\delta$}; 
\node at (-114:0.8) {\small $\epsilon$};
\node at (-116:1.15) {\small $\gamma$};
\node at (-90:1.45) {\small $\alpha$};
\node at (-64:1.15) {\small $\beta$};	

\node at (198:1.5) {\small $\epsilon$}; 
\node at (172:1.15) {\small $\gamma$};
\node at (226:1.15) {\small $\delta$};
\node at (178:0.8) {\small $\alpha$};
\node at (220:0.8) {\small $\beta$};

\node at (150:0.8) {\small $\alpha$};  

\node[draw,shape=circle, inner sep=0.5] at (0,0) {\small $1$};
\node[draw,shape=circle, inner sep=0.5] at (48:1.05) {\small $2$};
\node[draw,shape=circle, inner sep=0.5] at (-18:1.05) {\small $3$};
\node[draw,shape=circle, inner sep=0.5] at (-90:1.05) {\small $4$};
\node[draw,shape=circle, inner sep=0.5] at (198:1.05) {\small $5$};
\node[draw,shape=circle, inner sep=0.5] at (135:1.05) {\small $6$};

\node at (-54:1.8) {2.2};


\begin{scope}[xshift=4.5cm]

\draw
	(40:1.7) -- (60:1.5) -- (90:0.7) -- (18:0.7) -- (-54:0.7) -- (-54:2.1)
	(90:0.7) -- (162:0.7) -- (162:1.3) -- (198:1.7) -- (234:1.3) -- (-90:1.7) -- (-54:1.3) -- (-18:1.7) -- (18:1.3)
	(-38:2.4) -- (-54:2.1) -- (-70:2.4)
	(198:1.7) -- (198:2.4);

\draw[line width=1.5]
	(40:1.7) -- (18:1.3) -- (-18:1.7) 
	(18:0.7) -- (18:1.3)
	(162:0.7) -- (234:0.7) -- (-54:0.7)
	(234:0.7) -- (234:1.3)
	(-90:2.4) -- (198:2.4)
	(-70:2.4) -- (-90:2.4) -- (-90:1.7)
	(-38:2.4) -- (-18:2.4) -- (-18:1.7)
	;

\draw[dotted]
	(162:1.3) -- (140:1.7) -- (120:1.5) -- (90:0.7);

\fill (90:0.7) circle (0.1);

\node at (90:0.45) {\small $\epsilon$}; 
\node at (162:0.45) {\small $\gamma$};
\node at (18:0.45) {\small $\delta$};
\node at (234:0.45) {\small $\alpha$};
\node at (-54:0.45) {\small $\beta$};

\node at (70:0.75) {\small $\epsilon$}; 
\node at (30:0.8) {\small $\gamma$};
\node at (26:1.2) {\small $\alpha$};
\node at (40:1.45) {\small $\beta$};
\node at (56:1.35) {\small $\delta$};

\node at (-18:1.45) {\small $\gamma$}; 
\node at (8:1.15) {\small $\alpha$};	
\node at (-47:1.15) {\small $\epsilon$};
\node at (2:0.75) {\small $\beta$};
\node at (-38:0.75) {\small $\delta$};

\node at (-68:0.8) {\small $\gamma$}; 
\node at (-114:0.8) {\small $\alpha$};
\node at (-116:1.15) {\small $\beta$};
\node at (-90:1.45) {\small $\delta$};
\node at (-62:1.15) {\small $\epsilon$};	

\node at (198:1.5) {\small $\epsilon$}; 
\node at (172:1.2) {\small $\delta$};
\node at (226:1.15) {\small $\gamma$};
\node at (178:0.8) {\small $\beta$};
\node at (220:0.8) {\small $\alpha$};

\node at (150:0.8) {\small $\delta$};  

\node at (-95:2.05) {\small $\alpha$}; 
\node at (-97:1.8) {\small $\beta$};
\node at (203:2.05) {\small $\gamma$};
\node at (234:1.5) {\small $\delta$};
\node at (202:1.8) {\small $\epsilon$};

\node at (-60:1.5) {\small $\epsilon$}; 
\node at (-84:1.8) {\small $\gamma$};
\node at (-59:2) {\small $\delta$};
\node at (-85:2.2) {\small $\alpha$};
\node at (-70:2.2) {\small $\beta$};

\node at (-47:1.5) {\small $\delta$}; 
\node at (-26:1.75) {\small $\beta$};
\node at (-49:1.95) {\small $\epsilon$};
\node at (-23:2.2) {\small $\alpha$};
\node at (-36:2.2) {\small $\gamma$};

\node at (18:1.5) {\small $\alpha$};
\node at (-12:1.8) {\small $\alpha$};

\node at (-54:2.6) {3.2};

\node[draw,shape=circle, inner sep=0.5] at (0,0) {\small $1$};
\node[draw,shape=circle, inner sep=0.5] at (48:1.05) {\small $2$};
\node[draw,shape=circle, inner sep=0.5] at (-18:1.05) {\small $3$};
\node[draw,shape=circle, inner sep=0.5] at (-90:1.05) {\small $4$};
\node[draw,shape=circle, inner sep=0.5] at (198:1.05) {\small $5$};
\node[draw,shape=circle, inner sep=0.5] at (135:1.05) {\small $6$};
\node[draw,shape=circle, inner sep=0.5] at (220:1.7) {\small $7$};
\node[draw,shape=circle, inner sep=0.5] at (-70:1.85) {\small $8$};
\node[draw,shape=circle, inner sep=0.5] at (-38:1.85) {\small $9$};

\end{scope}

\end{tikzpicture}
\caption{Cases 2.2, and Case 3.2.}
\label{classify5}
\end{figure}

\subsubsection*{Case 3.2}

We have vertices $\alpha^3, \beta\gamma\delta$, and $H=\beta\epsilon^2\cdots,\epsilon^3\cdots$. The angle sums of $\alpha^3, \beta\gamma\delta$ and the angle sum for pentagon imply 
\[
\alpha=\tfrac{2}{3}\pi,\;
\epsilon=(\tfrac{1}{3}+\tfrac{4}{f})\pi.
\]

By $\beta\gamma\delta$ and Lemma \ref{geometry3}, a degree $3$ vertex $\epsilon\cdots=\gamma^2\epsilon,\delta^2\epsilon,\delta\epsilon^2,\epsilon^3$. Given that $\alpha^3, \beta\gamma\delta$ are vertices, for $\gamma^2\epsilon,\delta^2\epsilon,\delta\epsilon^2$, we may respectively apply Propositions \ref{special6}', \ref{special2}, \ref{special1}, to either get all the tilings or show there is no tiling. By $f\ge 24$, we have $\epsilon\le \tfrac{1}{2}\pi$. This implies $\epsilon^3$ is not a vertex.

Therefore we may assume $\epsilon$ does not appear at degree $3$ vertices. By Lemma \ref{hdeg} and the parity lemma, one of $\delta\epsilon^3,\epsilon^4,\epsilon^5$ is a vertex. We will handle the vertices $\epsilon^4,\epsilon^5$ in Section \ref{division}. We handle the vertex $\delta\epsilon^3$ below. 

The angle sums of $\alpha^3, \beta\gamma\delta,\delta\epsilon^3$ and the angle sum for pentagon imply
\[
\alpha=\tfrac{2}{3}\pi,\;
\beta+\gamma=(1+\tfrac{12}{f})\pi,\;
\delta=(1-\tfrac{12}{f})\pi,\;
\epsilon=(\tfrac{1}{3}+\tfrac{4}{f})\pi.
\]
By $f\ge 24$ and the non-symmetry, we have $\delta>\epsilon$. By Lemma \ref{geometry1}, this implies $\beta<\gamma$. By the angle values of $\delta,\epsilon$, and no $\epsilon$ at degree $3$ vertices, we know $\delta^3,\delta\epsilon^3,\epsilon^5$ are the only vertices without $b$-edge. 

By $\beta<\gamma$, $\beta+\gamma>\pi$, and the parity lemma, we know $R(\gamma^2\cdots)$ has no $\beta,\gamma$. By $2\alpha+2\gamma>2\alpha+\beta+\gamma>2\pi$, we know $\gamma^2\cdots$ has at most one $\alpha$. By $\beta<\gamma$, we have $R(\gamma^2\cdots)<R(\beta\gamma\cdots)=\delta<\alpha+\epsilon,3\epsilon$. Then by no $\epsilon$ at degree $3$ vertices, we get $\gamma^2\cdots=\alpha\gamma^2,\gamma^2\epsilon^2$. 

Suppose $\alpha\gamma^2$ is a vertex. The angle sum of $\alpha\gamma^2$ further implies
\[
\alpha=\gamma=\tfrac{2}{3}\pi,\;
\beta=(\tfrac{1}{3}+\tfrac{12}{f})\pi,\;
\delta=(1-\tfrac{12}{f})\pi,\;
\epsilon=(\tfrac{1}{3}+\tfrac{4}{f})\pi.
\]
The unique AAD $\thin^{\epsilon}\gamma^{\alpha}\thick^{\beta}\alpha^{\gamma}\thick^{\alpha}\gamma^{\epsilon}\thin$ of $\alpha\gamma^2$ gives a vertex $\alpha\beta\cdots$. By the parity lemma, we have $\alpha\beta\cdots=\alpha\beta^2\cdots,\alpha\beta\gamma\cdots$. Then $R(\alpha\beta\gamma\cdots)=(\tfrac{1}{3}-\tfrac{12}{f})\pi<$ all angles, and $R(\alpha\beta^2\cdots)=(\tfrac{2}{3}-\tfrac{24}{f})\pi<\alpha,2\beta,\delta,2\epsilon$. Moreover, by $\alpha\gamma^2$, both remainders are nonzero. Then by $\beta<\gamma$ and the parity lemma, we have $\alpha\beta\cdots=\alpha\beta^2\epsilon$. The angle sum of $\alpha\beta^2\epsilon$ further implies
\[
f=84\colon
\alpha=\gamma=\tfrac{2}{3}\pi,\;
\beta=\tfrac{10}{21}\pi,\;
\delta=\tfrac{6}{7}\pi,\;
\epsilon=\tfrac{8}{21}\pi.
\]
The AAD $\thick^{\alpha}\beta^{\delta}\thin^{\gamma}\epsilon^{\delta}\thin^{\delta}\beta^{\alpha}\thick$ at $\alpha\beta^2\epsilon$ gives a vertex $\delta^2\cdots$. By $R(\delta^2\cdots)=\tfrac{2}{7}\pi<$ all angles, we get a contradiction. 

Suppose $\gamma^2\epsilon^2$ is a vertex. The angle sum of $\gamma^2\epsilon^2$ further implies 
\[
\alpha=\tfrac{2}{3}\pi,\;
\beta=(\tfrac{1}{3}+\tfrac{16}{f})\pi,\;
\gamma=(\tfrac{2}{3}-\tfrac{4}{f})\pi,\;
\delta=(1-\tfrac{12}{f})\pi,\;
\epsilon=(\tfrac{1}{3}+\tfrac{4}{f})\pi.
\]
We have $R(\beta^2\delta\cdots)=(\tfrac{1}{3}-\tfrac{20}{f})\pi<$ all angles. By $\beta\gamma\delta$, the remainder is nonzero. Therefore $\beta^2\delta\cdots$ is not a vertex. Then the vertex $\gamma^2\epsilon^2$ and Lemma \ref{klem3} imply $\delta^2\cdots$ is a vertex. 

By $\beta\gamma\delta$ and Lemma \ref{klem6}, we know $\delta^2\cdots$ has only $\delta,\epsilon$. Since such vertices are $\delta^3,\delta\epsilon^3$, we get $\delta^2\cdots=\delta^3$. The angle sum of $\delta^3$ further implies 
\[
f=36\colon
\alpha=\delta=\tfrac{2}{3}\pi,\;
\beta+\gamma=\tfrac{4}{3}\pi,\;
\epsilon=\tfrac{4}{9}\pi.
\]
By $\beta<\gamma$, we have $\beta<\frac{2}{3}\pi<\gamma$. Then $\gamma+\epsilon>\pi$, which implies that $\gamma^2\cdots=\gamma^2\epsilon^2$ is not a vertex. 

We just completed the proof that $\gamma^2\cdots=\alpha\gamma^2,\gamma^2\epsilon^2$ is not a vertex. By the balance lemma, this implies that the only vertex involving $\beta,\gamma$ is $\beta\gamma\cdots$, with no $\beta,\gamma$ in the remainder. By $
\beta\gamma\delta$, we have $R(\beta\gamma\cdots)=\delta<2\alpha,\alpha+\epsilon,3\epsilon$. Therefore $\beta\gamma\cdots=\alpha\beta\gamma,\beta\gamma\delta,\beta\gamma\epsilon^2$. By Lemma \ref{klem4}, vertices without $\beta,\gamma$ are $\alpha^3,\delta^3,\delta\epsilon^3,\epsilon^5$. We also note that the angle sum of $\alpha\beta\gamma$ or $\delta^3$ further implies $f=36$, and the angle sum of $\beta\gamma\epsilon^2$ or $\epsilon^5$ further implies $f=60$. Therefore we get two cases
\begin{align*}
f=36 &\colon 
\alpha=\tfrac{2}{3}\pi,\;
\beta+\gamma=\tfrac{4}{3}\pi,\;
\delta=\tfrac{2}{3}\pi,\;
\epsilon=\tfrac{4}{9}\pi, \\
&\text{AVC}=\{\alpha^3,\alpha\beta\gamma,\beta\gamma\delta,\delta^3,\delta\epsilon^3\}; \\
f=60 &\colon 
\alpha=\tfrac{2}{3}\pi,\;
\beta+\gamma=\tfrac{6}{5}\pi,\;
\delta=\tfrac{4}{5}\pi,\;
\epsilon=\tfrac{2}{5}\pi, \\
&\text{AVC}=\{\alpha^3,\beta\gamma\delta,\beta\gamma\epsilon^2,\delta\epsilon^3,\epsilon^5\}.
\end{align*}

Next we use the AVCs to construct tilings for the two cases. For $f=36$, we go back to the partial neighbourhood given by the second of Figure \ref{classify5}. We already have $T_1,T_2,T_3,T_4,T_5$. By the AVC, we have $\beta_4\thick\gamma_5\cdots=\beta_4\gamma_5\delta_7$. Then by no $\beta_7\epsilon_5\cdots$, we determine $T_7$. Then $\beta_7\delta_4\cdots=\beta_7\gamma_8\delta_4$ determines $T_8$. Then $\epsilon_3\epsilon_4\epsilon_8\cdots=\delta_9\epsilon_3\epsilon_4\epsilon_8$. By no $\gamma_3\epsilon_9\cdots$, we determine $T_9$. Then $\beta_9\thin\gamma_3\cdots=\alpha\beta\gamma$ and $\alpha_2\alpha_3\cdots=\alpha^3$ give two $\alpha$ in a tile. The contradiction shows that there is no tiling for $f=36$. 

For $f=60$, we have $2\delta+\epsilon=2\pi$. Together with the existing vertices $\alpha^3,\beta\gamma\delta$, we know all the tilings from Proposition \ref{special2}.

\subsection{Pentagonal Subdivision}
\label{division}

After all the previous classifications, only the following cases remain: 
\begin{itemize}
\item Case 1.2.1 (second of Figure \ref{classify2}), $\beta\gamma\epsilon,\delta^3$ are vertices, and $H=\alpha^4$ or $\alpha^5$.
\item Case 3.2 (second of Figure \ref{classify5}, not including $T_7,T_8,T_9$), $\alpha^3,\beta\gamma\delta$ are vertices, and $\epsilon^4$ or $\epsilon^5$ is a vertex. Moreover, $\delta\epsilon^3$ is not a vertex, and $\epsilon$ does not appear at degree $3$ vertices.
\end{itemize}
In both cases, the pentagon is not symmetric, and there is no $3^5$-tile.

\subsubsection*{Case 1.2.1, $H=\alpha^4$}

The angle sums of $\beta\gamma\epsilon,\delta^3,\alpha^4$ and the angle sum for pentagon imply $f=24$. By no $3^5$-tile, and the second part of Lemma \ref{basic}, every tile is a $3^44$-tile. We start with three tiles $T_1,T_2,T_3$ around $\alpha^4$ as in the first of Figure \ref{classify4}. Then we get $T_4,T_5,T_6,T_7$ as indicated. 

Without loss of generality, we may assume the angles of $T_1$ are arranged as indicated. By $\beta\gamma\epsilon$ and Lemma \ref{geometry3}, the degree $3$ vertex $\gamma_1\cdots=\beta_2\gamma_1\epsilon_5$. Then $\beta_2$ determines $T_2$. By $\delta^3$, the degree $3$ vertex $\epsilon_1\cdots$ cannot have $\delta,\epsilon$ only. By Lemma \ref{klem4}, this implies $\epsilon_1\cdots$ has $ab$-angles. Then by $\beta\gamma\epsilon$ and $\epsilon_5$, we get $\epsilon_1\cdots=\beta\gamma\epsilon=\beta_4\gamma_5\epsilon_1$. This determines $T_4,T_5$. 

The way $T_1$ determines $T_2,T_4,T_5$ can be repeated with $T_2$ in place of $T_1$. Then we determine $T_3,T_6,T_7$. More repetitions determine all the tiles in the two layers around $\alpha^4$. 

By $\alpha^4$ and the edge length consideration, $\alpha_4\alpha_5\cdots$ is not a degree $3$ vertex. Since every tile is a $3^44$-tile, this implies $\alpha_4\alpha_5\cdots$ has degree $4$, and further implies $\beta_5\epsilon_6\cdots$ has degree $3$. By a degree $3$ vertex $\epsilon\cdots=\beta\gamma\epsilon$, we have $\beta_5\epsilon_6\cdots=\beta_5\gamma_8\epsilon_6$. This determines $T_8$. 

The degree $4$ vertex $\alpha_4\alpha_5\alpha_8\cdots$ is completed by a $b^2$-angle $\alpha$, and is $\alpha^4$. Then the argument starting from the initial $\alpha^4$ can be repeated for this new $\alpha^4$. More repetitions give the second pentagonal subdivision tiling in Figure \ref{subdivision_tiling}.

\begin{figure}[htp]
\centering
\begin{tikzpicture}[>=latex,scale=1]

\foreach \b in {0,1}
{
\begin{scope}[xshift=5*\b cm]

\foreach \a in {0,1,2}
\draw[rotate=90*\a]
	(0,0) -- (0.8,0) -- (1.2,0.6) -- (0.6,1.2) -- (0,0.8) -- (0,0);

\foreach \a in {0,1}
\draw[rotate=90*\a] 
	(-0.6,1.2) -- (-0.6,1.8) -- (1.8,1.8) -- (1.8,0.6) -- (1.2,0.6)
	(0.6,1.2) -- (0.6,1.8);
	
\draw
	(0.6,1.8) -- (0.6,2.5) -- (-1.8,2.5) -- (-1.8,1.8);

\node[inner sep=1,draw,shape=circle] at (0.55,0.55) {\small $1$};
\node[inner sep=1,draw,shape=circle] at (-0.55,0.55) {\small $2$};
\node[inner sep=1,draw,shape=circle] at (-0.55,-0.55) {\small $3$};
\node[inner sep=1,draw,shape=circle] at (1.3,1.3) {\small $4$};
\node[inner sep=1,draw,shape=circle] at (0,1.45) {\small $5$};
\node[inner sep=1,draw,shape=circle] at (-1.3,1.3) {\small $6$};
\node[inner sep=1,draw,shape=circle] at (-1.45,0) {\small $7$};
\node[inner sep=1,draw,shape=circle] at (0,2.15) {\small $8$};

\end{scope}	
}


\foreach \a in {0,1,2}
{
\begin{scope}[rotate=90*\a]
	
\node at (0.2,0.2) {\small $\alpha$};  
\node at (0.7,0.2) {\small $\beta$};
\node at (0.2,0.7) {\small $\gamma$};
\node at (1,0.6) {\small $\delta$};
\node at (0.6,0.95) {\small $\epsilon$};
	
\end{scope}
}

\foreach \a in {0,1}
{
\begin{scope}[rotate=90*\a]

\draw[line width=1.5]
	(0.8,0) -- (-0.8,0)
	(0.6,1.2) -- (0.6,1.8)
	(-0.6,1.8) -- (1.8,1.8);

\node at (0.8,1.6) {\small $\alpha$};  
\node at (0.8,1.3) {\small $\beta$};
\node at (1.6,1.6) {\small $\gamma$};
\node at (1.3,0.8) {\small $\delta$};
\node at (1.6,0.8) {\small $\epsilon$};
	
\node at (0.4,1.6) {\small $\alpha$};  
\node at (-0.4,1.6) {\small $\beta$};
\node at (0.4,1.3) {\small $\gamma$};
\node at (-0.4,1.3) {\small $\delta$};
\node at (0,1) {\small $\epsilon$};

\end{scope}
}

\draw[line width=1.5]
	(0.6,1.8) -- (0.6,2.5);
	
\node at (0.4,2) {\small $\alpha$}; 
\node at (0.4,2.3) {\small $\beta$};
\node at (-0.6,2) {\small $\gamma$};
\node at (-1.6,2.3) {\small $\delta$};
\node at (-1.6,2) {\small $\epsilon$};


\begin{scope}[xshift=5cm]

\foreach \a in {0,1,2}
{
\begin{scope}[rotate=90*\a]

\draw[line width=1.5]
	(0.8,0) -- (1.2,0.6) -- (0.6,1.2)
	(1.2,0.6) -- (1.8,0.6);

\node at (0.2,0.2) {\small $\epsilon$};  
\node at (0.7,0.2) {\small $\gamma$};
\node at (0.2,0.7) {\small $\delta$};
\node at (0.95,0.55) {\small $\alpha$};
\node at (0.55,0.95) {\small $\beta$};
	
\end{scope}
}

\foreach \a in {0,1}
{
\begin{scope}[rotate=90*\a]

\node at (0.8,1.6) {\small $\epsilon$};  
\node at (1.6,1.6) {\small $\delta$};
\node at (0.8,1.3) {\small $\gamma$};
\node at (1.3,0.8) {\small $\alpha$};
\node at (1.6,0.8) {\small $\beta$};

\node at (-0.4,1.6) {\small $\gamma$};  
\node at (0.4,1.6) {\small $\epsilon$};
\node at (-0.4,1.3) {\small $\alpha$};
\node at (0,1.05) {\small $\beta$};
\node at (0.4,1.3) {\small $\delta$};
	
\end{scope}
}

\draw[line width=1.5]
	(0.6,2.5) -- (-1.8,2.5) -- (-1.8,1.8);
	
\node at (0.4,2) {\small $\epsilon$}; 
\node at (-1.6,2) {\small $\beta$};
\node at (-1.6,2.35) {\small $\alpha$};
\node at (-0.6,2) {\small $\delta$};
\node at (0.4,2.3) {\small $\gamma$};

\end{scope}

\end{tikzpicture}
\caption{Pentagonal subdivision tiling.}
\label{classify4}
\end{figure}

\subsubsection*{Case 3.2, $\epsilon^4$ is a vertex}

The angle sums of $\alpha^3,\beta\gamma\delta,\epsilon^4$ and the angle sum for pentagon imply $f=24$. By no $3^5$-tile, and the second part of Lemma \ref{basic}, every tile is a $3^44$-tile. Then we get tiles $T_1,T_2,T_3,T_4,T_5,T_6,T_7$ around a vertex $\epsilon^4$ as in the second of Figure \ref{classify4}.

Without loss of generality, we may assume the angles of $T_3$ are arranged as indicated. If $T_2$ is not arranged as indicated, then we have the AAD $\thin^{\delta}\epsilon_2^{\gamma}\thin^{\gamma}\epsilon_3^{\delta}\thin$ at a degree $4$ vertex $\epsilon^4$. Since every tile is a $3^44$-tile, this gives a degree $3$ vertex $\thick^{\alpha}\gamma^{\epsilon}\thin^{\epsilon}\gamma^{\alpha}\thick\cdots=\alpha\gamma^2=\thin^{\epsilon}\gamma^{\alpha}\thick^{\beta}\alpha^{\gamma}\thick^{\alpha}\gamma^{\epsilon}\thin$. This induces a vertex $\alpha\thick\beta\cdots$, such that $\alpha\thick\beta\cdots$ and $\epsilon^4$ are the vertices of the same tile. Since every tile is a $3^44$-tile, we know $\alpha\thick\beta\cdots$ has degree $3$. This implies $\alpha\thick\beta\cdots=\thick\alpha\thick\beta\thin\beta\thick, \thick\alpha\thick\beta\thin\gamma\thick$. Comparing the angle sums of $\alpha\gamma^2$ and $\alpha\thick\beta\cdots$, we get $\beta=\gamma$, contradicting the non-symmetry assumption. This proves that $T_2$ must be arranged as in the second of Figure \ref{classify4}. Then by $\beta\gamma\delta$, we know the degree $3$ vertex $\gamma_3\delta_2\cdots=\beta_7\gamma_3\delta_2$ determines $T_7$. Then the degree $3$ vertex $\beta_2\delta_7\cdots=\beta_2\gamma_6\delta_7$ determines $T_6$.

The way $T_3$ determines $T_2,T_6,T_7$ can be repeated with $T_2$ in place of $T_3$. Then we determine $T_1,T_5,T_4$. More repetitions determine all the tiles in the two layers around $\epsilon^4$.  

By no $\epsilon$ at degree $3$ vertices, and every tile is a $3^44$-tile, the vertex $\epsilon_4\epsilon_5\cdots$ has degree $4$, and $\beta_6\gamma_5\cdots$ has degree $3$. By $\beta\gamma\delta$, we get $\beta_6\gamma_5\cdots=\beta_6\gamma_5\delta_8$. By no $\epsilon$ at degree $3$ vertices, this implies $\epsilon_8$ is at $\epsilon_4\epsilon_5\cdots$. Then $\delta_8,\epsilon_8$ determine $T_8$. 

By $\epsilon^4$, the degree $4$ vertex $\epsilon_4\epsilon_5\epsilon_8\cdots=\epsilon^4$. Then the argument starting from the initial $\epsilon^4$ can be repeated for this new $\epsilon^4$. More repetitions give the third pentagonal subdivision tiling in Figure \ref{subdivision_tiling}.

\subsubsection*{Case 1.2.1, $H=\alpha^5$}

The angle sums of $\beta\gamma\epsilon,\delta^3,\alpha^5$ and the angle sum for pentagon imply $f=60$. To repeat the argument above for the case $H=\alpha^4$, we need to solve two problems. The first is no $3^44$-tile. The second is to make sure the vertex $\alpha_4\alpha_5\alpha_8\cdots$ in the first of Figure \ref{classify4} is $\alpha^5$. 

For the problem of no $3^44$-tile, after all the previous works, we may assume that the partial neighbourhood of a $3^44$-tile is the second of Figure \ref{classify2}, or the second of Figure \ref{classify5} (not including $T_7,T_8,T_9$), with the possibility of exchanging $\beta\leftrightarrow\gamma$ and $\delta\leftrightarrow\epsilon$. If the vertex $H$ has degree $4$ in the second of Figure \ref{classify2}, then by the edge length consideration, we get $H=\alpha^4$, contradicting the vertex $\alpha^5$. In the second of Figure \ref{classify5}, we find $\alpha^3$ is a vertex, also contradicting the vertex $\alpha^5$. This completes the argument that we may assume there is no $3^44$-tile. Then by $f=60$ and the third part of Lemma \ref{basic}, every tile is a $3^45$-tile. 

We may repeat the argument given by the first of Figure \ref{classify4}, until we reach the vertex $\alpha_4\alpha_5\alpha_8\cdots$. The vertex has degree $5$. 

Suppose $\alpha_4\alpha_5\alpha_8\cdots\ne\alpha^5$. Then by Lemma \ref{klem4}, the vertex has two $ab$-angles. This implies the AAD ${}^{\alpha}\thick^{\gamma}\alpha_4^{\beta}\thick^{\gamma}\alpha_5^{\beta}\thick^{\gamma}\alpha_8^{\beta}\thick^{\alpha}$ at the degree $5$ vertex $\alpha_4\alpha_5\alpha_8\cdots$. Since every tile is a $3^45$-tile, all the vertices induced from this AAD have degree $3$. Therefore the AAD gives a degree $3$ vertex $\alpha\thick\gamma\cdots$. By the parity lemma, we have $\alpha\thick\gamma\cdots=\alpha\beta\gamma,\alpha\gamma^2$. The angle sums of $\beta\gamma\epsilon,\delta^3,\alpha^5$, the angle sum of $\alpha\beta\gamma$ or $\alpha\gamma^2$, and the angle sum for pentagon imply
\begin{align*}
\alpha\beta\gamma &\colon 
\alpha=\epsilon=\tfrac{2}{5}\pi,\;
\beta+\gamma=\tfrac{8}{5}\pi,\;
\delta=\tfrac{2}{3}\pi; \\
\alpha\gamma^2 &\colon
\alpha=\tfrac{2}{5}\pi,\;
\beta+\epsilon=\tfrac{6}{5}\pi,\;
\gamma=\tfrac{4}{5}\pi,\;
\delta=\tfrac{2}{3}\pi.
\end{align*}
In the first case, we have $\delta>\epsilon$. By Lemma \ref{geometry1}, this implies $\beta<\gamma$. By $\beta+\gamma=\tfrac{8}{5}\pi$, we further get $R(\alpha_4\alpha_5\alpha_8\cdots)=\frac{4}{5}\pi<\gamma$. In the second case, we have $R(\alpha_4\alpha_5\alpha_8\cdots)=\frac{4}{5}\pi=\gamma$. Since $R(\alpha_4\alpha_5\alpha_8\cdots)$ consists of two $ab$-angles, we conclude $\alpha_4\alpha_5\alpha_8\cdots=\alpha^3\beta^2$. 

The angle sum of $\alpha^3\beta^2$ further implies 
\begin{align*}
\alpha\beta\gamma &\colon 
\alpha=\beta=\epsilon=\tfrac{2}{5}\pi,\;
\gamma=\tfrac{6}{5}\pi,\;
\delta=\tfrac{2}{3}\pi; \\
\alpha\gamma^2 &\colon
\alpha=\beta=\tfrac{2}{5}\pi,\;
\gamma=\tfrac{4}{5}\pi,\;
\delta=\tfrac{2}{3}\pi,\;
\epsilon=\tfrac{4}{5}\pi.
\end{align*}
The second case violates Lemma \ref{geometry1}. Therefore we are in the first case, and a degree $3$ vertex $\alpha\thick\gamma\cdots=\alpha\beta\gamma$. Now the AAD of $\alpha_4\alpha_5\alpha_8\cdots=\alpha^3\beta^2$ is $\thin^{\delta}\beta^{\alpha}\thick^{\gamma}\alpha_4^{\beta}\thick^{\gamma}\alpha_5^{\beta}\thick^{\gamma}\alpha_8^{\beta}\thick^{\alpha}\beta^{\delta}\thin$. This gives a degree $3$ vertex $\thick^{\gamma}\alpha^{\beta}\thick^{\alpha}\gamma^{\epsilon}\thin\cdots=\alpha\beta\gamma=\thick^{\gamma}\alpha^{\beta}\thick^{\alpha}\gamma^{\epsilon}\thin^{\delta}\beta^{\alpha}\thick$. This further gives a vertex $\delta\epsilon\cdots$. However, by the angle values above for the case $\alpha\beta\gamma$, we know $\delta\epsilon\cdots$ is not a vertex.

This completes the proof that $\alpha_4\alpha_5\alpha_8\cdots=\alpha^5$. Then the argument for the case $H=\alpha^4$ may continue to give the fourth pentagonal subdivision tiling in Figure \ref{subdivision_tiling}.

\subsubsection*{Case 3.2, $\epsilon^5$ is a vertex}

After finishing Case 1.2.1, the only partial neighbourhood we need to consider is the second of Figure \ref{classify5} ($T_7,T_8,T_9$ not included). The angle sums of $\alpha^3,\beta\gamma\delta,\epsilon^5$ and the angle sum for pentagon imply $f=60$. 

Similar to Case 1.2.1, $H=\alpha^5$, we first argue for no $3^44$-tile. We only need to consider the possibility that the vertex $H$ in the second of Figure \ref{classify5} has degree $4$. These are $H=\beta^2\epsilon^2,\beta\gamma\epsilon^2,\delta\epsilon^3,\epsilon^4$. We already assumed $\delta\epsilon^3$ is not a vertex, and $H=\epsilon^4$ contradicts the vertex $\epsilon^5$. Therefore $H=\beta_6\epsilon_1\epsilon_2\cdots=\beta^2\epsilon^2,\beta\gamma\epsilon^2$. This also implies a vertex $\delta_5\epsilon_6\cdots$. 

If $H=\beta^2\epsilon^2$, then the angle sums of $\alpha^3,\beta\gamma\delta,\epsilon^5,\beta^2\epsilon^2$ and the angle sum for pentagon imply
\[
\alpha=\tfrac{2}{3}\pi,\;
\beta=\tfrac{3}{5}\pi,\;
\gamma+\delta=\tfrac{7}{5}\pi,\;
\epsilon=\tfrac{2}{5}\pi.
\]
By $\beta+\epsilon<\gamma+\delta$ and Lemma \ref{geometry1}, we get $\beta<\gamma$ and $\delta>\epsilon$. By $\beta^2\epsilon^2$ and the balance lemma, we know $\gamma^2\cdots$ is a vertex. By $\beta<\gamma$, we have $R(\gamma^2\cdots)<R(\beta^2\cdots)=\tfrac{4}{5}\pi=2\epsilon<2\alpha,\alpha+\epsilon,2\beta$. By $\beta<\gamma$ and $\delta>\epsilon$, this implies $\gamma^2\cdots=\alpha\gamma^2,\gamma^2\epsilon,\gamma^2\delta$. By $\beta\gamma\delta$, we know $\gamma^2\delta$ is not a vertex. Given $\alpha^3,\beta\gamma\delta$, if $\gamma^2\epsilon$ is a vertex, then Proposition \ref{special6}' implies no tiling. If $\alpha\gamma^2$ is a vertex, then the angle sum of $\alpha\gamma^2$ further implies
\[
\alpha=\gamma=\tfrac{2}{3}\pi,\;
\beta=\tfrac{3}{5}\pi,\;
\delta=\tfrac{11}{15}\pi,\;
\epsilon=\tfrac{2}{5}\pi.
\]
Then $R(\delta_5\epsilon_6\cdots)=\tfrac{13}{15}\pi$ is not a sum of angle values. We get a contradiction.

If $H=\beta\gamma\epsilon^2$, then the angle sums of $\alpha^3,\beta\gamma\delta,\epsilon^5,\beta\gamma\epsilon^2$ and the angle sum for pentagon imply
\[
\alpha=\tfrac{2}{3}\pi,\;
\beta+\gamma=\tfrac{6}{5}\pi,\;
\delta=\tfrac{4}{5}\pi,\;
\epsilon=\tfrac{2}{5}\pi.
\]
We have $\delta>\epsilon$. By Lemma \ref{geometry1}, this implies $\beta<\gamma$. Then $R(\delta_5\epsilon_6\cdots)=\tfrac{4}{5}\pi=\delta=2\epsilon<\beta+\gamma$. By $\beta<\gamma$ and Lemma \ref{klem4}, we get $\delta\epsilon\cdots=\delta^2\epsilon,\delta\epsilon^3,\alpha\beta^k\delta\epsilon,\beta^k\delta\epsilon$, $k$ even. Given $\alpha^3,\beta\gamma\delta$, the case $\delta^2\epsilon$ is a vertex is handled by Proposition \ref{special2}. We have also assumed $\delta\epsilon^3$ is not a vertex. 

If $k\ge 2$ in $\alpha\beta^k\delta\epsilon$, or $k\ge 4$ in $\beta^k\delta\epsilon$, then the angle sum implies $\beta\le\frac{1}{5}\pi$. This implies $\gamma=\frac{6}{5}\pi-\beta\ge \pi$, which further implies $\gamma^2\cdots$ is not a vertex. Then $\alpha\beta^k\delta\epsilon$ or $\beta^k\delta\epsilon$ contradicts the balance lemma. By Lemma \ref{klem4} and the parity lemma, we cannot have $k<2$ in $\alpha\beta^k\delta\epsilon$. By the parity lemma, therefore, it remains to consider $\beta^2\delta\epsilon$. The angle sum of $\beta^2\delta\epsilon$ further implies
\[
\alpha=\tfrac{2}{3}\pi,\;
\beta=\epsilon=\tfrac{2}{5}\pi,\;
\gamma=\delta=\tfrac{4}{5}\pi.
\]
The angles are the same as \eqref{special2_60B}, and the proof of Proposition \ref{special2} already showed that there is no pentagon with such angles.

This completes the argument that we may assume there is no $3^44$-tile. Then by $f=60$ and the third part of Lemma \ref{basic}, every tile is a $3^45$-tile.  

Similar to Case 1.2.1, $H=\alpha^5$, next we need to argue that $\epsilon_4\epsilon_5\epsilon_8\cdots=\epsilon^5$ in the second of Figure \ref{classify4}. This allows us to adopt the argument for Case 3.2, $\epsilon^4$ is a vertex, and conclude the pentagonal subdivision tiling. 

If the degree $5$ vertex $\epsilon_4\epsilon_5\epsilon_8\cdots\ne \epsilon^5$, then by the edge length consideration, we have $\epsilon_4\epsilon_5\epsilon_8\cdots=\beta^2\epsilon^3,\gamma^2\epsilon^3,\beta\gamma\epsilon^3$. We note the AAD $\thin^{\gamma}\epsilon_4^{\delta}\thin^{\gamma}\epsilon_5^{\delta}\thin^{\gamma}\epsilon_8^{\delta}\thin$ in the second of Figure \ref{classify4} and get
\begin{align*}
\beta^2\epsilon^3
&=\thin^{\gamma}\epsilon_8^{\delta}\thin^{\delta}\beta^{\alpha}\thick^{\alpha}\beta^{\delta}\thin^{\gamma}\epsilon_4^{\delta}\thin\cdots, \\
\gamma^2\epsilon^3
&=\thin^{\gamma}\epsilon_8^{\delta}\thin^{\epsilon}\gamma^{\alpha}\thick^{\alpha}\gamma^{\epsilon}\thin^{\gamma}\epsilon_4^{\delta}\thin\cdots, \\
\beta\gamma\epsilon^3
&=\thin^{\gamma}\epsilon_8^{\delta}\thin^{\delta}\beta^{\alpha}\thick^{\alpha}\gamma^{\epsilon}\thin^{\gamma}\epsilon_4^{\delta}\thin\cdots.
\end{align*}
Since every tile is a $3^45$-tile, all the vertices induced from the AAD above have degree $3$. In $\gamma^2\epsilon^3$ and $\beta\gamma\epsilon^3$, by the edge length consideration, we have a degree $3$ vertex $\gamma\epsilon\cdots=\beta\gamma\epsilon,\gamma^2\epsilon$. By $\beta\gamma\delta$, we know $\beta\gamma\epsilon$ is not a vertex. Therefore $\gamma^2\epsilon$ is a vertex. In $\beta^2\epsilon^3$, by the edge length consideration, we have a degree $3$ vertex $\thick^{\gamma}\alpha^{\beta}\thick^{\beta}\alpha^{\gamma}\thick\cdots=\alpha^3$. By Lemma \ref{aadlemma}, this implies a vertex $\gamma\thick\gamma\cdots$. Moreover, the vertex $\gamma\thick\gamma\cdots$ has degree $3$, because it and $\beta^2\epsilon^3$ are vertices of the same tile, and every tile is a $3^45$-tile. By $\beta\gamma\delta$, Lemma \ref{geometry3}, and the edge length consideration, we get $\gamma\thick\gamma\cdots=\gamma^2\epsilon$.

In summary, if $\epsilon_4\epsilon_5\epsilon_8\cdots\ne \epsilon^5$, then $\gamma^2\epsilon$ is a vertex. 
Given $\alpha^3,\beta\gamma\delta$, by Proposition \ref{special6}', the vertex $\gamma^2\epsilon$ implies no tiling. This completes the argument that we may assume $\epsilon_4\epsilon_5\epsilon_8\cdots=\epsilon^5$. Then the argument for Case 3.2, $\epsilon^4$ is a vertex, can be adopted to get the fifth pentagonal subdivision tiling in Figure \ref{subdivision_tiling}.


\begin{thebibliography}{1}

\bibitem{ay1}
Y.~Akama, M.~Yan.
\newblock On deformed dodecahedron tiling.
\newblock {\em preprint}, arXiv:1403.6907, 2014.

\bibitem{awy}
Y.~Akama, E.~X.~Wang, M.~Yan.
\newblock Tilings of the sphere by congruent pentagons III: edge combination $a^5$.
\newblock {\em preprint}, arXiv:1805.07217, 2021.

\bibitem{gsy}
H.~H.~Gao, N.~Shi, M.~Yan.
\newblock Spherical tiling by $12$ congruent pentagons.
\newblock {\em J. Combinatorial Theory Ser. A}, 120(4):744--776, 2013.

\bibitem{ly1}
H.~P.~Luk, M.~Yan.
\newblock Tilings of the sphere by congruent almost equilateral pentagons I: five distinct angles.
\newblock {\em preprint}, 2021.

\bibitem{ly2}
H.~P.~Luk, M.~Yan.
\newblock Tilings of the sphere by congruent almost equilateral pentagons II: three angles.
\newblock {\em preprint}, 2021.

\bibitem{wy1}
E.~X.~Wang, M.~Yan.
\newblock Tilings of the sphere by congruent pentagons I: edge combinations $a^2b^2c$ and $a^3bc$.
\newblock {\em preprint}, arXiv:1804.03770, 2021.

\bibitem{wy3}
E.~X.~Wang, M.~Yan.
\newblock Moduli of pentagonal subdivision tiling.
\newblock preprint, arXiv: 1907.08776, 2019.

\bibitem{yan2}
M.~Yan.
\newblock Pentagonal subdivision.
\newblock {\em Elec. J. of Combi.}, 26:\#P4.19, 2019.


\end{thebibliography}
\end{document}